\documentclass[11pt, reqno]{amsart}

\usepackage{latexsym, amssymb, amsmath}
\usepackage{tikz, url, verbatim, skull}

\usepackage{amsfonts, graphicx}
\usepackage{graphicx,color}
\renewcommand{\leq}{\leqslant}
\renewcommand{\geq}{\geqslant}

\newtheorem{thm}{Theorem}[section]
\newtheorem{conj}[thm]{Conjecture}
\newtheorem{lma}[thm]{Lemma}
\newtheorem{prop}[thm]{Proposition}

\theoremstyle{definition}
\newtheorem{defn}[thm]{Definition}
\newtheorem{rmk}[thm]{Remark}

\usepackage{mathrsfs}

\newtheorem{claim}[thm]{Claim}
\numberwithin{equation}{section}
\def\D{\mathbb{D}}

\def\K{K\"ahler }
\def\KE{K\"ahler-Einstein }
\def\KR{K\"ahler-Ricci }
\def\Ric{\text{\rm Ric}}
\def\ddbar{i\partial\bar\partial}
\newcommand{\de}{\partial}
\newcommand{\db}{\overline{\partial}}
\def\tr{\operatorname{tr}}
\newcommand{\vp}{\varphi}
\newcommand{\ov}[1]{\overline{#1}}
\newcommand{\ti}[1]{\tilde{#1}}
\renewcommand{\diamond}{\diamondsuit}
\def\e{\varepsilon}
\def\ve{\varepsilon}
\def\a{{\alpha}}
\def\b{{\beta}}
\def\R{\mathbb{R}}
\def\C{\mathbb{C}}
\def\P{\mathbb{P}}
\newcommand{\DD}{\mathfrak{D}}

\begin{document}

\title[Collapsing immortal K\"ahler-Ricci flows]{Collapsing immortal K\"ahler-Ricci flows}

\author{Hans-Joachim Hein}
\address[Hans-Joachim Hein]{Mathematisches Institut, Universit\"at M\"unster, 48149 M\"unster, Germany
}
\email{hhein@uni-muenster.de}

\author{Man-Chun Lee}
\address[Man-Chun Lee]{Department of Mathematics, The Chinese University of Hong Kong, Shatin, N.T., Hong Kong
}
\email{mclee@math.cuhk.edu.hk}

\author{Valentino Tosatti}
\address[Valentino Tosatti]{Courant Institute of Mathematical Sciences, New York University, New York, NY 10012, USA
}
\email{tosatti@cims.nyu.edu}

\date{\today}

\begin{abstract} We consider the K\"ahler-Ricci flow on compact K\"ahler manifolds with semiample canonical bundle and intermediate Kodaira dimension, and show that the flow collapses to a canonical metric on the base of the Iitaka fibration in the locally smooth topology and with bounded Ricci curvature away from the singular fibers. This follows from an asymptotic expansion for the evolving metrics, in the spirit of recent work of the first and third-named authors on collapsing Calabi-Yau metrics, and proves two conjectures of Song and Tian.
\end{abstract}

\maketitle

\markboth{Collapsing immortal K\"ahler-Ricci flows}{Hans-Joachim Hein, Man-Chun Lee and Valentino Tosatti}

\setcounter{tocdepth}{3}
\tableofcontents

\section{Introduction}
\subsection{Background and motivation}
Let $(X,\omega_0)$ be a compact \K manifold, and let $\omega^\bullet(t),t\in[0,T),$ be a family of K\"ahler metrics on $X$ which solve the K\"ahler-Ricci flow
\begin{equation}\label{kkrf}
\partial_t\omega^\bullet (t)=-\Ric(\omega^\bullet (t))-\omega^\bullet(t), \quad \omega^\bullet (0)=\omega_0,
\end{equation}
for some $0<T\leq+\infty$. In this paper we are interested in the case when $\omega^\bullet(t)$ is an {\em immortal solution}, i.e. when $T=+\infty$. Thanks to a result of Tian-Zhang \cite{TiZ} (see also \cite{Ts}), we know that the solution $\omega^\bullet(t)$ is immortal if and only if the canonical bundle $K_X$ is nef, which means that $c_1(K_X)$ lies in the closure of the cone of K\"ahler classes in $H^{1,1}(X,\mathbb{R})$. This condition does not depend on $\omega_0$, and manifolds with $K_X$ nef are also known as smooth minimal models.

The Abundance Conjecture in birational geometry, and its natural extension to K\"ahler manifolds, predicts that if the canonical bundle of a compact K\"ahler manifold is nef, then it must be semiample, which means that $K_X^p$ is base-point-free for some $p\geq 1$. This conjecture is known when $\dim X\leq 3$ by \cite{CHP,DO,DO2}.

Throughout the rest of the paper we will assume that $K_X$ is semiample. It is then known (see e.g. \cite[Theorem 2.1.27]{Laz}) that global sections of $K_X^p$ for $p\geq 1$ sufficiently divisible define a surjective holomorphic map $f:X\to B\subset\mathbb{CP}^N$ (the Iitaka fibration of $X$) with connected fibers onto a normal projective variety $B$ (known as the canonical model of $X$), of dimension $m$ equal to the Kodaira dimension of $X$ (in particular, we have $0\leq m\leq\dim X$). The smooth fibers of $f$ are then Calabi-Yau manifolds, of dimension $n:=\dim X-m$, which are pairwise diffeomorphic but in general are not pairwise biholomorphic.

In the two extreme cases when $m=0$ or $m=\dim X$, the behavior of the flow is completely understood thanks to the work of many people (see e.g. the recent survey \cite{To2} and references therein), so we will furthermore assume from now on that $0<m<\dim X$, which is known as ``intermediate Kodaira dimension''. Thus we have $\dim X=m+n$, and both the fibers and the base of $f$ are positive-dimensional.

The simplest examples of this setup arise when $m=n=1$, where $X$ is a minimal properly elliptic surface, $B$ is a compact Riemann surface, and $f:X\to B$ is an elliptic fibration. In this case, the behavior of the K\"ahler-Ricci flow \eqref{kkrf} was first studied by Song-Tian \cite{ST}, who shortly afterwards also considered the case of general $m,n$ in \cite{ST2}. A major difficulty in this setting is that the total volume of $(X,\omega^\bullet(t))$ is easily seen to converge to zero as $t\to+\infty$, and this ``collapsing'' behavior makes it extremely hard to analyze the flow. As we will now explain, in \cite{ST,ST2} it was shown that the metrics $\omega^\bullet(t)$ collapse in the weak topology to the pullback of a canonical metric on $B$, and our main goal is to obtain higher order regularity and a uniform Ricci curvature bound for $\omega^\bullet(t)$ (away from the singular fibers of $f$) and thus prove two conjectures of Song-Tian.

When $X$ is projective, the condition that $K_X$ be nef means that $X$ is a smooth minimal model. The connection between the Minimal Model Program in birational geometry and the behavior of the K\"ahler-Ricci flow was first discovered independently by Cascini-La Nave \cite{CaL} and Song-Tian \cite{ST}, and remains an area of current research. These works outlined a conjectural picture for the behavior of the K\"ahler-Ricci flow on any projective (or more generally compact K\"ahler) manifold. When $K_X$ is not nef, singularities must develop in finite time, and the flow should implement the corresponding birational contractions or collapse the fibers of a Mori fiber space. On the other hand, the case when $K_X$ is nef (so the manifold is a smooth minimal model) is the topic of our paper.

\subsection{Setup}\label{set}
We now describe our setup in more detail. As mentioned above, we have a compact K\"ahler manifold  $(X^{m+n},\omega_0)$ with semiample canonical bundle and intermediate Kodaira dimension $m$ (so $m,n>0$), and $\omega^\bullet(t)$ denotes the immortal solution of the K\"ahler-Ricci flow \eqref{kkrf}. Let $f:X\to B$ be the Iitaka fibration of $X$, and let $S\subset X$ be the preimage of the union of the set of singular values of $f$ and the singular set of $B$. Thus, by construction, $f:X\backslash S\to B\backslash f(S)$ is a proper holomorphic submersion with $n$-dimensional connected Calabi-Yau fibers $X_{z}=f^{-1}(z), z\in B\backslash f(S)$. By Ehresmann's Lemma (and the connectedness of $B\backslash f(S)$), the fibers $X_z$ are pairwise diffeomorphic, but in general their complex structure varies with $z$, and this variation can be encoded in a smooth semipositive Weil-Petersson form $\omega_{\rm WP}\geq 0$ on $B\backslash f(S)$, defined in \cite{ST} (see also \cite[\S 5.6]{To}).

By \cite{ST2}, there exists a smooth \K metric $\omega_{\mathrm{can}}$ on $B\backslash f(S)$ satisfying the twisted \KE equation
\begin{equation}
\Ric(\omega_{\mathrm{can}})=-\omega_{\mathrm{can}}+\omega_{\rm WP}.
\end{equation}
The pullback of $\omega_{\rm can}$ to $X\backslash S$ will also be denoted by the same symbol, for convenience. The metric $\omega_{\rm can}$ extends to a closed positive current on $B$, and in \cite{ST,ST2} it is shown that as $t\to+\infty$ we have
\begin{equation}\label{golo}
\omega^\bullet(t)\to \omega_{\rm can},
\end{equation}
weakly as currents on $X$ as well as in the  $C^{0}_{\rm loc}(X\backslash S)$ topology of K\"ahler potentials. Motivated by this, in \cite[p.612]{ST}, \cite[p.306]{ST2}, \cite[Conjecture 4.5.7]{Ti}, \cite[p.258]{Ti2}, Song-Tian posed the following:

\begin{conj}\label{c1}
Let $(X,\omega_0)$ be a compact K\"ahler manifold with $K_X$ semiample and intermediate Kodaira dimension $0<m<\dim X$, and let $\omega^\bullet(t)$ solve \eqref{kkrf}. Then the convergence \eqref{golo} happens in the locally smooth topology as tensors on $X\backslash S$.
\end{conj}
Explicitly, Conjecture \ref{c1} asks to show that given any $K\Subset X\backslash S$ and $k\in\mathbb{N}$ we have
\begin{equation}
\|\omega^\bullet(t)-\omega_{\rm can}\|_{C^k(K,g_0)}\to 0.
\end{equation}
There have been a number of partial results towards Conjecture \ref{c1}, often using techniques that were first developed for a family of elliptic PDEs that describe the collapsing of families of Ricci-flat K\"ahler metrics on a Calabi-Yau manifold with a fibration structure, and which share some of the features of \eqref{kkrf}, see e.g. the survey \cite{To3}. Indeed, Fong-Zhang \cite{FZ} adapted work of the third-named author \cite{To4} to prove that \eqref{golo} holds in the $C^{1,\alpha}_{\rm loc}(X\backslash S)$ topology of K\"ahler potentials ($\alpha<1$), and the works \cite{FZ,HT,TZ} proved Conjecture \ref{c1} when the smooth fibers of $f$ are tori or finite quotients of tori (see also \cite{Gi} and \cite[\S 5.14]{To}), using and improving a method of Gross-Tosatti-Zhang \cite{GTZ}. Later, Tosatti-Weinkove-Yang proved that \eqref{golo} holds in $C^0_{\rm loc}(X\backslash S)$, and this was improved to $C^\alpha_{\rm loc}(X\backslash S),\alpha<1$ by Chu-Lee \cite{CL} adapting the techniques of Hein-Tosatti \cite{HT2}, which also allowed Fong-Lee \cite{FL} to prove Conjecture \ref{c1} when all smooth fibers are pairwise biholomorphic.

In a later work \cite{ST3}, Song-Tian proved that the scalar curvature of $\omega^\bullet(t)$ remains uniformly bounded on $X$, independent of $t\geq 0$. They then conjectured a similar statement for the Ricci curvature, away from the singular fibers of $f$ (see \cite[Conjecture 4.7]{Ti2}):

\begin{conj}\label{c2}
Let $(X,\omega_0)$ be a compact K\"ahler manifold with $K_X$ semiample and intermediate Kodaira dimension $0<m<\dim X$, and let $\omega^\bullet(t)$ solve \eqref{kkrf}. Then the Ricci curvature of $\omega^\bullet(t)$ remains uniformly bounded on compact subsets of $X\backslash S$, independent of $t$.
\end{conj}
This is only known when the smooth fibers of $f$ are tori, or finite quotients of tori \cite{FL} (hence in particular it holds on minimal properly elliptic surfaces), or when the smooth fibers are pairwise biholomorphic \cite{CL}. It is known that in general the conjectural Ricci bound cannot be improved to a full Riemann curvature bound (on compact subsets of $X\backslash S$): by \cite{TZ}, this holds if and only if the smooth fibers are tori or finite quotients.

It is well known that the K\"ahler-Ricci flow \eqref{kkrf} reduces to a scalar PDE, of parabolic complex Monge-Amp\`ere type, for a family of evolving K\"ahler potentials. Following \cite{ST2}, we construct a closed real $(1,1)$-form $\omega_F$ on $X\backslash S$, which is of the form $\omega_F=\omega_0+\ddbar\rho$, such that for every $z\in B\backslash f(S)$ we have that $\omega_F|_{X_z}$ is the unique Ricci-flat K\"ahler metric on $X_z$ cohomologous to $\omega_0|_{X_z}$. While $\omega_F$ is not semipositive definite in general (see \cite{CGPT} for a counterexample), given any compact set $K\Subset X\backslash S$ we can find $t_0$ such that for all $t\geq t_0$,
\begin{equation}
\omega^\natural(t):=(1-e^{-t})\omega_{\mathrm{can}}+e^{-t}\omega_F,
\end{equation}
is a K\"ahler metric on $K$, with fibers of size $\approx e^{-t/2}$ and base of size $\approx 1$. On $X\backslash S$ we can then write $\omega^\bullet(t)=\omega^\natural(t)+\ddbar\vp(t)$, where the potentials $\vp(t)$ satisfy
\begin{equation}\label{ma}
\left\{
                \begin{aligned}
                  &\frac{\de}{\de t}\vp(t)=\log\frac{e^{nt}(\omega^\natural(t)+\ddbar\vp(t))^n}{\binom{m+n}{n}\omega_{\rm can}^m\wedge\omega_F^n}-\vp(t),\\
                  &\vp(0)=-\rho,\\
                  &\omega^\bullet(t)=\omega^\natural(t)+\ddbar\vp(t)>0,
                \end{aligned}
              \right.
\end{equation}
for $t\geq 0$, see e.g. \cite[\S 5.7]{To} and \cite[\S 3.1]{TWY}.
Then, since we know the weak convergence in \eqref{golo}, Conjecture \ref{c1} is equivalent to the a priori estimates
\begin{equation}\label{desi1}
\|\omega^\bullet(t)\|_{C^k(K,g_0)}\leq C_{K,k},
\end{equation}
for all $k\in\mathbb{N}$ and all $t\geq 0$. Furthermore, since $\vp(t)$ is uniformly bounded in $L^\infty(X)$ by \cite{ST2} (which uses \cite{DPa, EGZ2}, see also \cite{GPT} for a new proof), these estimates are also equivalent to
\begin{equation}\label{desi2}
\|\vp(t)\|_{C^k(K,g_0)}\leq C_{K,k},
\end{equation}
for all $k\in\mathbb{N}$ and all $t\geq 0$.

\subsection{Main result}
The main result of this paper gives a complete solution of Conjectures \ref{c1} and \ref{c2}:
\begin{thm}\label{mainthm}
Conjectures \ref{c1} and \ref{c2} are true.
\end{thm}
In fact, in both conjectures we prove much more precise statements. The higher order estimates for $\omega^\bullet(t)$ are derived as consequences of a very detailed asymptotic expansion for $\omega^\bullet(t)$, which is in the same spirit as the expansion recently obtained in \cite{HT3} by Hein-Tosatti for collapsing Ricci-flat metrics on Calabi-Yau manifolds. As for the Ricci curvature bound, we show that on $X\backslash S$ we have
\begin{equation}
\Ric(\omega^\bullet(t))=-\omega_{\rm can}+{\rm Err},
\end{equation}
where on any fixed compact subset of $X\backslash S$ we have $|{\rm Err}|_{g^\bullet(t)}\to 0$, as $t\to+\infty$. Thus, in a strong sense, the Ricci curvature of the evolving metrics $\omega^\bullet(t)$ is asymptotic to $-\omega_{\rm can}$. Furthermore,  our bound on the Ricci curvature (and on all of the pieces of the asymptotic expansion of the metric) is an {\em a priori} bound: it only depends on the uniform constants in Lemma \ref{lma:earlier work}, which are due to \cite{FZ,ST3,TWY}. $\qquad$\\

The starting point of our analysis, which was proved in \cite{FZ} by adapting \cite{To4} in the elliptic setting, is the following estimate: given $K\Subset X\backslash S$, there is $C>0$ such that on $K$ we have
\begin{equation}\label{equiv}
C^{-1}\omega^\natural(t)\leq \omega^\bullet(t)\leq C\omega^\natural(t),
\end{equation}
for all $t\geq t_0$. In other words, $\omega^\bullet(t)$ is shrinking in the fiber directions, and remains of bounded size in the base directions. Since the linearized operator of \eqref{ma} is the time-dependent heat operator of $\omega^\bullet(t)$, we see from \eqref{equiv} that the ellipticity is degenerating in the fiber directions as $t\to+\infty$, and so there is no clear way to approach the a priori estimates \eqref{desi2}. Indeed, the local analog of such estimates are false, see the discussion in \cite{HT2} in the elliptic case.

However it turns out that we can work locally on the base (but using crucially that the fibers are compact without boundary), and since $f$ is differentiably a locally trivial fiber bundle over $B\backslash f(S)$, we may without loss assume that our base $B$ is now simply the Euclidean unit ball in $\mathbb{C}^m$, and $f:B\times Y\to B$ is just the projection onto the first factor, where $Y$ is a closed manifold and $B\times Y$ is equipped with a complex structure $J$ (not necessarily a product) such that $f$ is $(J,J_{\mathbb{C}^m})$ holomorphic. The fibers $\{z\}\times Y, z\in B,$ are then compact $n$-dimensional Calabi-Yau manifolds diffeomorphic to $Y$. Under this trivialization, the Ricci-flat K\"ahler metric $\omega_F|_{X_z}$ defines a Riemannian metric $g_{Y,z}$ on $\{z\}\times Y$, which we extend trivially to $B\times Y$, and use these to define a family of shrinking Riemannian product metrics
\begin{equation}
g_z(t)=g_{\mathbb{C}^m}+e^{-t}g_{Y,z},
\end{equation}
on $B\times Y$, which are uniformly equivalent to $\omega^\natural(t)$ and hence to $\omega^\bullet(t)$. We will also denote by $g(t):=g_0(t)$ the shrinking product metrics with $z$ equal to the origin in $B$.

\subsection{Overview of the proof}
As in \cite{HT2,HT3}, the first attempt to overcome the issue of degenerating ellipticity is to try to prove much more, namely try to prove uniform bounds for $\vp(t)$ or $\omega^\bullet(t)$ in the {\em shrinking} norms $C^k(K,g(t))$, since $g^\bullet(t)$ is uniformly equivalent to $g(t)$. This however cannot be proved in general, since we know from \cite{TWY} that $e^t\omega^\bullet(t)|_{X_z}$ converge smoothly to $\omega_F|_{X_z}$, and since $g_{Y,z}$ and $g_{Y,z'}$ are not in general parallel with respect to each other, the shrinking $C^k$ norms of $g_{z}(t)$ and $g_{z'}(t)$ are not uniformly equivalent as $t\to+\infty$. To address this issue, the first and third-named authors defined in \cite{HT3} a connection $\D$ on $B\times Y$ which on each fiber $\{z\}\times Y$ acts like the Levi-Civita connection of $g_z(t)$, and using its parallel transport operator they defined new shrinking $C^{k,\alpha}$ norms, $0<\alpha<1$. We will consider the natural parabolic extension of these norms to space-time derivatives in Section \ref{s1} below. Since parabolic H\"older seminorms behave differently according to the parity of $k$, we will only work with $k=2j$ even (cf. Remark \ref{remm}).

The hope would then be to show that $\omega^\bullet(t)-\omega^\natural(t)=\ddbar\vp(t)$ is uniformly bounded in these shrinking $C^{2j+\alpha,j+\a/2}$ norms. This turns out to be true when $j=0$, but false starting from $j=1$. This phenomenon, which was discovered in \cite{HT3} in the elliptic setting, manifests itself only when the complex structure $J$ is not a product and the fibers are not tori or quotients. In a nutshell, the variation of complex structures, and the non-flatness of $g_{z}(t)$, destroy these desired shrinking norm bounds. However, with much work, we are able to construct a collection of ``obstruction functions'' on $B\times Y$ (up to shrinking $B$), and decompose the solution $\ddbar\vp(t)$ into a sum of finitely many terms $\gamma_1(t),\dots,\gamma_j(t)$ (constructed roughly speaking using the fiberwise $L^2$ projections of $\Delta^{g^\natural(t)}\vp(t)$ onto the space of obstructions), and a remainder $\eta_j(t)$. We then show via a contradiction and blowup argument that the remainder $\eta_j(t)$ {\em is} bounded in the shrinking $C^{2j+\alpha,j+\a/2}$ norm, while the terms $\gamma_1(t),\dots,\gamma_j(t)$ are not, but they satisfy strong enough estimates which guarantee that they are bounded in the $C^{2j+\alpha,j+\a/2}$ norm of a {\em fixed} metric $\omega_X$ on $X$. As mentioned earlier, the higher order estimates on all these pieces depend only on the constant in the $C^0$ estimate \eqref{equiv}, and on the other constants that appear in Lemma \ref{lma:earlier work} (including the uniform bound on the scalar curvature of $\omega^\bullet(t)$ from \cite{ST3}), and thus ultimately they depend only on the geometry of $X$ and on the initial metric $\omega_0$.

This procedure is the iterated by replacing $j$ with $j+1$, and new obstruction functions are constructed by measuring the failure of the remainder $\eta_j(t)$ to be bounded in the shrinking $C^{2(j+1)+\alpha,j+1+\a/2}$ norm. This way, we can split $\eta_j(t)=\gamma_{j+1}(t)+\eta_{j+1}(t)$, and obtain the next order in the expansion. As in \cite{HT3}, there is an extra technical difficulty, which arises from the fact that the terms $\gamma_j(t)$ are constructed by plugging in $\eta_{j-1}(t)$ and the obstruction functions into an approximate elliptic Green operator, which has an extra parameter $k\in\mathbb{N}$ that measures the quality of the approximation. Thus, all the terms in the expansion also end up depending on $k$, which is large and chosen a priori, and the procedure works for $j\leq k$.

The resulting asymptotic expansion of $\omega^\bullet(t)$ is described in detail in Theorem \ref{thm:decomposition} below, which is the main technical result of the paper. It is the parabolic analog of \cite[Theorem 4.1]{HT3}, and its proof follows the same overall method via blowup and contradiction, but there are some new key difficulties. First, as mentioned earlier, the (shrinking) parabolic H\"older norms that we use are better-behaved when the order of derivatives is even, which compels us to use $C^{2j+\alpha,j+\a/2}$ norms instead of $C^{j+\alpha,(j+\a)/2}$ (see e.g. Lemma \ref{lma:jet-interpolation} and Remark \ref{remm}).
More importantly, since the approximate Green operator that we use in this paper is the same as in \cite{HT3}, it provides and approximate parametrix for the Laplacian of $\omega^\natural(t)$ (in a rough sense) but not for the heat operator (it seems far from clear that a similar strategy could be implemented with an approximate heat kernel construction). Because of this, to obtain a contradiction at the end of the blowup argument (which is divided into 3 cases, with the last case itself divided into 3 subcases A, B and C), we now have to deal with new terms that come from taking time derivatives of the solution, which are not taken care of by construction, unlike \cite{HT3}. To make matters worse, in the blowup argument the evolving K\"ahler potential has $L^\infty$ norm that is blowing up, so it cannot be passed to a limit to obtain a contradiction. Dealing with these issues requires substantial work.

Another new difficulty, compared to \cite{HT3}, is that the case $j=0$ (i.e. where we prove $C^{\a,\a/2}$ estimates) does not behave in the same way as the cases $j\geq 1$, because the parabolic complex Monge-Amp\`ere equation also involves $\vp(t)$ without derivatives landing on it, unlike the elliptic complex Monge-Amp\`ere equation where only $\ddbar\vp$ enters. To deal with this issue, we employ a different blowup quantity for $j=0$, which is closer in spirit to our earlier works \cite{HT2, FL}. As a result, different ideas will be required to close the blowup argument, according to whether $j=0$ or $j\geq 1$. Furthermore, when $j\geq 1$ we are forced to add one new term to the main blowup quantity (when compared to \cite{HT3}), to gain better control on the fiber average of the K\"ahler potential and its time derivative, and we later have to show that this new term can be dealt with in the blowup argument. Next, in subcase A, dealing with these terms forces us to refine the Selection Theorem \ref{thm:Selection} where the obstruction functions are chosen, and when $j=0$ we need a whole new argument. In subcase B, we employ an energy argument inspired by \cite[Claim 3.2]{FL}, and in subcase C a different energy argument has to be applied fiber by fiber.

Once the asymptotic expansion is established, the smooth convergence of Conjecture \ref{c1} follows easily. On the other hand, proving the Ricci curvature bound for $\omega^\bullet(t)$ in Conjecture \ref{c2} requires substantial work, by plugging in the expansion with $j=1$ into the formula for the Ricci curvature (as two derivatives of the logarithm of the volume form), and use our explicit a priori estimates for the terms of the expansion to deduce boundedness of Ricci. Here again we encounter a new difficulty compared to \cite{HT3}, which arises from the fact that one of the estimates in Theorem \ref{thm:decomposition} is weaker than the corresponding one in the elliptic setting, because of the fact that we can only work with even order norms. During the course of the proof of the Ricci bound we also prove a fact of independent interest in Proposition \ref{trew}, by showing that $\vp+\dot{\vp}$ minus its fiberwise average decays to zero (away from the singular fibers) faster than $e^{-t}$ (see \eqref{suxx2}). This improves on earlier work of Fong-Zhang \cite[p.110]{FZ} (see also \cite[Lemma 5.13]{To}) and Tosatti-Weinkove-Yang \cite[Lemma 3.1 (iv)]{TWY}.

\begin{rmk}\label{glue}
We conjecture that the Ricci curvature of $\omega^\bullet(t)$ remains uniformly bounded also near the singular fibers of $f$. One could imagine settling this for some minimal elliptic surfaces by developing a parabolic version of the Gross-Wilson gluing result in \cite{GW} (thanks to J. Lott for this suggestion), and for some Lefschetz fibered $3$-folds by developing a parabolic version of Li's gluing result in \cite{Li}.
\end{rmk}

\begin{rmk}\label{spiro}
It is natural to ask whether we really need to assume that our compact Kähler manifold $X$ with $K_X$ nef satisfies the Abundance Conjecture (thanks to S. Karigiannis and J. Cheng for raising this point).  The reader can verify that the results in \cite{FZ,ST2,ST3,TWY} on which we rely, as well as our main theorems, are also valid under the a priori weaker assumption that $c_1(K_X)$ is a \emph{semiample $(1,1)$-class} \cite[Def.3.4]{To5}: there is a surjective holomorphic map $f:X\to B$ with connected fibers onto a normal compact K\"ahler analytic space $B$ such that $c_1(K_X)=f^*[\omega]$ for some K\"ahler class $[\omega]$ on $B$. However, a very recent result of Das-Hacon \cite[Theorem 4.4]{DH}, which was prompted by our questions to C. Hacon as well as the related \cite[Question 3.5]{To5}, shows that under this hypothesis $K_X$ is already semiample, and it is elementary to deduce from this that $f$ is the Iitaka fibration of $K_X$. We thank also M. P\u{a}un for discussions about this point.
\end{rmk}

\subsection{Organization of the paper} In section \ref{s1} we introduce our parabolic shrinking norms and seminorms and prove an interpolation inequality, the crucial Proposition \ref{prop:braindead}, and a Schauder estimate. Section \ref{s2} contains the proof of the Selection Theorem \ref{thm:Selection} where the obstruction functions are selected. Section \ref{sekt} is the main part of the paper, and is where the asymptotic expansion is proved in Theorem \ref{thm:decomposition}. Lastly, in section \ref{s4} we give the proof of our main Theorem \ref{mainthm}.

\subsection*{Acknowledgements} Part of this work was carried out during visits of the second-named author to the Courant Institute of Mathematical Sciences, and of the third-named author to the Center of Mathematical Sciences and Applications at Harvard University, which we would like to thank for the hospitality.  We thank J. Cheng, C. Hacon, J. Lott, S. Karigiannis, and M. P\u{a}un for discussions about Remarks \ref{glue} and \ref{spiro}, and the referee for helpful comments on a previous version. The first-named author was partially supported by the German Research Foundation (DFG) under Germany's Excellence Strategy EXC 2044-390685587 ``Mathematics M\"unster: Dynamics-Geometry-Structure" and by the CRC 1442 ``Geometry: Deformations and Rigidity'' of the DFG. The second-named author was partially supported by Hong Kong RGC grant (Early Career Scheme) of Hong Kong No. 24304222, No. 14300623,  and a NSFC grant No. 12222122. The third-named author was partially supported by NSF grants DMS-2231783 and DMS-2404599.

\section{Parabolic H\"older norms and interpolation}\label{s1}
The setup where we are working in was described in the Introduction.
\subsection{$\D$-derivatives}
Recall that our main goal is to establish higher order estimates for the metrics $\omega^\bullet(t)$ on $B\times Y$ which evolve by the normalized K\"ahler-Ricci flow \eqref{kkrf}. We know from Lemma \ref{lma:earlier work} (i) below that $\omega^\bullet(t)$ is uniformly equivalent to $\omega^\natural(t)=(1-e^{-t})\omega_{\rm can}+e^{-t}\omega_F$, which is shrinking in the fiber directions as $t\to+\infty$. As mentioned above in the overview of proof, the fiberwise Ricci-flat metrics $g_{Y,z}$ are in general quite different from each other as $z\in B$ varies, and this forces us to define a new connection $\D$ which along each fiber $\{z\}\times Y$ acts like the Levi-Civita connection of $g_z(t)=g_{\C^m}+e^{-t}g_{Y,z}$. This is what was achieved by the first and third-named authors in \cite[\S 2.1]{HT3}, and we now recall their construction.

\begin{defn}
For $z\in B\subset\mathbb{C}^m$, we let $\nabla^z$ be the Levi-Civita connection of the product metric $g_{z}(t)=g_{\mathbb{C}^m}+e^{-t} g_{Y,z}$ on $B\times Y$, which is independent of $t\geq 0$. Let $\D$ be the connection on the tangent bundle of $B\times Y$ and on all of its tensor bundles defined by
\begin{equation}
(\D \eta)(x)= (\nabla^{\mathrm{pr}_{B}(x)} \eta)(x),
\end{equation}
for all tensors $\eta$ on $B\times Y$ and $x\in B\times Y$.
\end{defn}

For the detailed discussion of the properties of $\D$, we refer readers to \cite[\S 2.1]{HT3}.  Given a curve $\gamma$ in $B\times Y$ which contains the points $a,b$, we let $\P^\gamma_{ab}$ denote the $\D$-parallel transport from $a$ to $b$ along the $\gamma$. A curve $\gamma$ is called a $\mathbb{P}$-geodesic if $\dot{\gamma}$ is $\D$-parallel along $\gamma$. Two examples of $\P$-geodesics are horizontal paths $(z(t),y_0)$ where $z(t)$ is an affine segment in $\C^m$, and vertical paths $(z_0,y(t))$ where $y(t)$ is a $g_{Y,z_0}$-geodesic in $\{z_0\}\times Y$. These are the only $\P$-geodesics that we will use in the paper, as every two points in $B\times Y$ can be connected by concatenating two of these $\P$-geodesics, where the vertical one is minimal. We may also write $\P_{ab}$ instead of $\P^\gamma_{ab}$ if the $\P$-geodesic $\gamma$ joining $a$ and $b$ is not emphasized.

\subsection{$\mathfrak{D}$-derivatives} $\D$-derivatives that we just defined are spatial derivatives. It will be very convenient to use a similar shorthand notation when we also allow time derivatives. Thus, given a time-dependent contravariant tensor $\eta$ and $k\in\mathbb{N}$, we define
\begin{equation}
\mathfrak{D}^k\eta:=\sum_{p+2q=k}\D^p \partial_t^q \eta,
\end{equation}
which is a sum of tensors of different types.
We will also use the notation
\begin{equation}
\mathfrak{D}^k_{\mathbf{bt}}\eta:=\sum_{p+2q=k}\D^p_{\mathbf{b}\cdots\mathbf{b}} \partial_t^q \eta,
\end{equation}
when we only take spatial base derivatives, as well as time derivatives. Observe also that if $g$ is any Riemannian product metric on $B\times Y$ then we have the pointwise equality
\begin{equation}
|\DD^k\eta|^2_g=\sum_{p+2q=k}|\D^p \partial_t^q \eta|^2_g,
\end{equation}
which we will use implicitly many times.

In our setting, $\{g_{Y,z}\}_{z\in B},$ is a smooth family of Riemannian metrics on $Y$, so (up to shrinking $B$ slightly) we can find $\Lambda>1$ so that
\begin{equation}\label{equiv-product-ref}
\left\{
\begin{array}{ll}
\Lambda^{-1} g_{Y,0}\leq g_{Y,z}\leq \Lambda g_{Y,0},\\
\Lambda^{-\frac{1}{2}} \leq \mathrm{inj}(Y,g_{Y,z}) \leq \mathrm{diam}(Y,g_{Y,z})\leq \Lambda^{\frac{1}{2}}.
\end{array}
\right.
\end{equation}
In particular, the norm measured with respect to $g_{Y,0}$ is uniformly comparable to that of $g_{Y,z}$ for $z\in B$.

\subsection{H\"older seminorms}We now use the connection $\D$ to define a parabolic H\"older norm on $B\times Y\times [0,+\infty)$.
For $p=(z,y)\in B\times Y,t\geq 0, 0<R\leq \sqrt{t}$ and (shrinking) product metrics $g_\zeta(\tau)=g_{\mathbb{C}^m}+e^{-\tau} g_{Y,\zeta}$, we define the parabolic domain
\begin{equation}
Q_{g_\zeta(\tau),R}(p,t)= B_{\mathbb{C}^m}(z,R)\times  B_{e^{-\tau}g_{Y,\zeta}}(y,R)\times [t-R^2,t].
\end{equation}
The parabolic domain with respect to any other product metric is defined analogously. We will very often simply take $\zeta=0\in B$.

\begin{defn}
For any $0<\a<1$, $R>0$, $p\in B\times Y$,  $t\geq 0$ and smooth tensor field $\eta$ on $B\times Y\times [t-R^2,t]$, given a product metric $g$ (such as $g=g_z(\tau)$ for some $z\in B$ and $\tau\geq 0$), we define
\begin{equation}\label{kzk}
[\eta]_{\a,\a/2, Q_{g,R}(p,t),g}=\sup \left\{ \frac{|\eta(x,s)-\P_{x'x}\eta(x',s')|_g}{(d^g(x,x')+|s-s'|^{\frac{1}{2}})^\a}\right\},
\end{equation}
where the supremum is taken among all $(x,s)$ and $(x',s')$ in $Q_{g,R}(p,t)$ in which $x$ and $x'$ are either horizontally or vertically joined by a $\P$-geodesic.
\end{defn}

In case when we use $g=g_z(\tau)$ and $\tau$ is allowed to go to $+\infty$, we will refer to these as shrinking parabolic H\"older seminorms. Nevertheless for each fixed $R>0$, we will have
\begin{equation}
B_{\mathbb{C}^m}(z,R)\times  B_{e^{-\tau}g_{Y,z}}(y,R)=B_{\mathbb{C}^m}(z,R)\times  Y,
\end{equation}
for all $\tau>\tau_0(R,Y)$. This will be the setting where the parabolic H\"older seminorm are applied in the whole paper.  In this case, we will simply denote it by
\begin{equation}
Q_R(z,t)=B_{\mathbb{C}^m}(z,R)\times Y \times [t-R^2,t],
\end{equation}
when the metric $g$ and the shrinking rate $\tau$ play no role.

Lastly, as in \cite[(4.101)]{HT3}, it will also be useful to consider (shrinking) parabolic H\"older seminorms $[\eta]_{\a,\a/2,{\rm base}, Q_{g,R}(p,t),g}$ which are defined as in \eqref{kzk} but where the supremum is taken only among $(x,s)$ and $(x',s')$ in $Q_{g,R}(p,t)$ such that $x$ and $x'$ are horizontally joined by a $\P$-geodesic.

\subsection{Parabolic interpolation}

We need an interpolation inequality between the highest order (i.e. $C^{k+\a,(k+\a)/2}$) and the lowest order (i.e. $L^\infty$) norms of a tensor.  In the parabolic framework, it will be more convenient to interpolate with the top \textit{even} order (cf. Remark \ref{remm}). This can be viewed as a parabolic version of \cite[Proposition 2.8]{HT3}, and as in there it is crucial that the constants in the interpolation inequality are independent of the shrinking size parameter $\tau\geq 0$.

\begin{prop}\label{prop:interpolation}
For any $k\in \mathbb{N}_{>0}$ and $\a\in (0,1)$, there exists $C_k=C_k(\a,\Lambda)>0$ (where $\Lambda$ is given in \eqref{equiv-product-ref}) such that the following holds.  Let $\eta$ be a smooth contravariant $p$-tensor on $B\times Y$.  Then for all $(x_0, t_0)\in B\times Y\times \mathbb{R}$, $ 0<\rho<R$ and $\tau\geq 0$ such that $ Q_{g_0(\tau),R}(x_0,t_0)\Subset  B\times Y\times \mathbb{R}$, we have
\begin{equation}\label{interpolation-Linfty}
\begin{split}
\sum_{j=1}^{2k}(R-\rho)^{j} \| \mathfrak{D}^j \eta \|_{\infty, Q_{g_0(\tau),\rho}(x_0,t_0),g_0(\tau)}&\leq C_k\Big((R-\rho)^{2k+\a}[\mathfrak{D}^{2k} \eta] _{\a,\a/2, Q_{g_0(\tau),R}(x_0,t_0),g_0(\tau)}\\
&\quad \;\;+\| \eta \|_{\infty, Q_{g_0(\tau),R}(x_0,t_0),g_0(\tau)}\Big).
\end{split}
\end{equation}
Moreover  for any $j\in \mathbb{N}$ and $\b\in (0,1)$ with $j+\b<2k+\a$, we have
\begin{equation}\label{interpolation-Holder}
\begin{split}
(R-\rho)^{j+\b} [\mathfrak{D}^j \eta ]_{\b,\b/2, Q_{g_0(\tau),\rho}(x_0,t_0),g_0(\tau)}&\leq C_k\Big((R-\rho)^{2k+\a}[\mathfrak{D}^{2k} \eta] _{\a,\a/2, Q_{g_0(\tau),R}(x_0,t_0),g_0(\tau)}\\
&\quad \;\;+\| \eta \|_{\infty, Q_{g_0(\tau),R}(x_0,t_0),g_0(\tau)}\Big).
\end{split}
\end{equation}
\end{prop}
\begin{proof}
We first show \eqref{interpolation-Linfty}.  Fix a pair of $(p,q)$ such that $0<j=p+2q\leq 2k$, and assume first that $p>0$. Since $d^{g_0(\tau)}(x,x_0)<\rho$ for $(x,t)\in Q_{g_0(\tau),\rho}(x_0,t_0)$, we can treat $\partial_t^q \eta|_{(x,t)}$ as a smooth tensor on $B_{\mathbb{C}^m}(z_0,\rho)\times Y$ by freezing $t$ so that \cite[Proposition 2.8]{HT3} applies to conclude
\begin{equation}\label{spatial-interpolation}
\begin{split}
&\quad  (R-\rho)^{p}\| \D^p \partial_t^q \eta \|_{\infty, Q_{g_0(\tau),\rho}(x_0,t_0),g_0(\tau)}\\
&\leq C \left((R-\rho)^{2k-2q+\a} [\mathfrak{D}^{2k} \eta] _{\a,\a/2, Q_{g_0(\tau),R}(x_0,t_0),g_0(\tau)} +\| \partial_t^q\eta \|_{\infty, Q_{g_0(\tau),R}(x_0,t_0),g_0(\tau)} \right).
\end{split}
\end{equation}
Thus, it remains to show the interpolation on time derivatives, i.e. we assume in the rest that $p=0$, so $q>0$. For each $(x,t)\in Q_{g_0(\tau),\rho}(x_0,t_0)$,  fix $s=(R-\rho)^2>0$ so that $(x,t-s)\in  Q_{g_0(\tau),R}(x_0,t_0)$. Then there exists $t_s\in [t-s,t]$ so that
\begin{equation}
\frac{1}{s}\left(\partial_t^{q-1}\eta(x,t)-\partial_t^{q-1}\eta(x,t-s)\right)=\partial_t^{q} \eta(x,t_s),
\end{equation}
which allows us to estimate
\begin{equation}
\begin{split}
|\partial_t^q\eta (x,t)|&\leq \left|\partial_t^q\eta (x,t)-\frac1s\left(\partial_t^{q-1}\eta(x,t)-\partial_t^{q-1}\eta(x,t-s) \right) \right|\\
&\quad +\frac1s \left|\partial_t^{q-1}\eta(x,t)-\partial_t^{q-1}\eta(x,t-s) \right|\\
&\leq  \left|\partial_t^{q}\eta (x,t)-\partial_t^{q} \eta(x,t_s) \right|+ \frac2s \|\partial_t^{q-1} \eta \|_{\infty, Q_{g_0(\tau),R}(x_0,t_0),g_0(\tau)}.
\end{split}
\end{equation}
If $q=k$,  we arrive at
\begin{equation*}
\begin{split}
\|  \partial_t^k \eta \|_{\infty, Q_{g_0(\tau),\rho}(x_0,t_0),g_0(\tau)}&\leq
 (R-\rho)^{\a}[\partial_t^k \eta]_{\a,\a/2, Q_{g_0(\tau),R}(x_0,t_0),g_0(\tau)} \\
 &\quad + \frac2{(R-\rho)^2} \|\partial_t^{k-1} \eta \|_{\infty, Q_{g_0(\tau),R}(x_0,t_0),g_0(\tau)}.
 \end{split}
\end{equation*}
Otherwise, $q<k$ and  we have
\begin{equation*}
\begin{split}
\|  \partial_t^q \eta \|_{\infty, Q_{g_0(\tau),\rho}(x_0,t_0),g_0(\tau)}&\leq
 (R-\rho)^{2}\|\partial_t^{q+1} \eta\|_{\infty, Q_{g_0(\tau),R}(x_0,t_0),g_0(\tau)} \\
 &\quad +2(R-\rho)^{-2} \|\partial_t^{q-1} \eta \|_{\infty, Q_{g_0(\tau),R}(x_0,t_0),g_0(\tau)}.
 \end{split}
\end{equation*}
Applying this dichotomy inductively, with suitable replacements of $\rho$ and $R$ at each step, we conclude that there exists $C>0$ so that for each $1\leq q\leq k$,
\begin{equation}
\begin{split}
&\quad  (R-\rho)^{2q}\|  \partial_t^q \eta \|_{\infty, Q_{g_0(\tau),\rho}(x_0,t_0),g_0(\tau)}\\
&\leq C \left((R-\rho)^{2k+\a} [\partial_t^k \eta] _{\a,\a/2, Q_{g_0(\tau),R}(x_0,t_0),g_0(\tau)} +\|\eta \|_{\infty, Q_{g_0(\tau),R}(x_0,t_0),g_0(\tau)} \right).
\end{split}
\end{equation}
By combining this with \eqref{spatial-interpolation},  we see that \eqref{interpolation-Linfty} follows.

It remains to prove \eqref{interpolation-Holder}. Fix $(x,t),(x',s)\in Q_{g_0(\tau),\rho}(x_0,t_0)$ such that $x$ and $x'$ are joined either horizontally or vertically by a $\P$-geodesic.  Denote $d=d^{g_0(\tau)}(x,x')+|t-s|^{\frac{1}{2}}$. We want to estimate $|\mathfrak{D}^j \eta(x,t)-\P_{x'x} \mathfrak{D}^j\eta(x',s)|_{g_0(\tau)}$.  Fix a pair of $(p,q)$ such that $0<p+2q=j\leq 2k$ and $\b\in (0,1)$ with $j+\beta<2k+\alpha$.

If $d\geq \frac1{4\Lambda}  (R-\rho)$ where $\Lambda$ is the constant in \eqref{equiv-product-ref}, then using the triangle inequality and the boundedness of the operator norm of $\P$ from \cite[\S 2.1.1]{HT3},  we deduce that
\begin{equation}
\begin{split}
\frac{|\D^p\partial_t^q \eta(x,t)-\P_{x'x} \D^p\partial_t^q\eta(x',s)|_{g_0(\tau)}}{\left(d^{g_0(\tau)}(x,x')+|t-s|^{\frac{1}{2}} \right)^\b}&\leq \frac{C}{(R-\rho)^\b} \| \mathfrak{D}^j \eta \|_{\infty, Q_{g_0(\tau),\rho}(x_0,t_0),g_0(\tau)},
\end{split}
\end{equation}
so that the conclusion follows from \eqref{interpolation-Linfty}.

If $d<\frac1{4\Lambda}  (R-\rho)$,  $j=2k$ and $\b<\a$, then
\begin{equation}
\begin{split}
 \frac{|\D^p\partial_t^q \eta(x,t)-\P_{x'x} \D^p\partial_t^q\eta(x',s)|_{g_0(\tau)}}{\left(d^{g_0(\tau)}(x,x')+|t-s|^{\frac{1}{2}} \right)^\b}&=\frac{|\D^p\partial_t^q \eta(x,t)-\P_{x'x} \D^p\partial_t^q\eta(x',s)|_{g_0(\tau)}}{\left(d^{g_0(\tau)}(x,x')+|t-s|^{\frac{1}{2}} \right)^\a} d^{\a-\b}\\
&\leq  (R-\rho)^{\a-\b}[\mathfrak{D}^{2k} \eta] _{\a,\a/2, Q_{g_0(\tau),R}(x_0,t_0),g_0(\tau)},
\end{split}
\end{equation}
which is acceptable.
It remains to consider the case when $d<\frac1{4\Lambda}  (R-\rho)$ and $j=p+2q<2k$.  Here, using again the boundedness of $\P$, we can estimate
\begin{equation}\label{zlz}
\begin{split}
&\quad |\D^p\partial_t^q \eta(x,t)-\P_{x'x} \D^p\partial_t^q\eta(x',s)|_{g_0(\tau)}\\
&\leq  |\D^p\partial_t^q \eta(x,t)-\P_{x'x} \D^p\partial_t^q\eta(x',t)|_{g_0(\tau)}+|\P_{x'x} \D^p\partial_t^q\eta(x',t)-\P_{x'x} \D^p\partial_t^q\eta(x',s)|_{g_0(\tau)}\\
&\leq  Cd^\b [\D^p\partial_t^q \eta(t)]_{\b, B_{\mathbb{C}^m}(z,d)\times  B_{g_{Y,0}}(y,d),g_0(\tau)}
+C| \D^p\partial_t^q\eta(x,t)-\D^p\partial_t^q\eta(x,s)|_{g_0(\tau)},
\end{split}
\end{equation}
where the first term is the spatial H\"older seminorm of the tensor $\partial_t^q\eta(t)$ with $t$ frozen, and $x=(z,y)\in \mathbb{C}^m\times Y$.  Applying \cite[Lemma 2.5]{HT3} to the first term in the last line of \eqref{zlz} gives
\begin{equation}\label{zlz1}\begin{split}
 [\D^p\partial_t^q \eta(t)]_{\b, B_{\mathbb{C}^m}(z,d)\times  B_{g_{Y,0}}(y,d),g_0(\tau)}
 &\leq C d^{1-\b} \| \D^{p+1} \partial_t^q\eta(t)\|_{\infty, B_{\mathbb{C}^m}(z,d)\times  B_{g_{Y,0}}(y,d),g_0(\tau)}\\
 &\leq C d^{1-\b} \| \mathfrak{D}^{j+1} \eta\|_{\infty, Q_{g_0(\tau),R}(x_0,t_0),g_0(\tau)}.
\end{split}
\end{equation}
As for the second term in the last line of \eqref{zlz}, assume first that $j+2=p+2(q+1)\leq 2k$. In this case we can argue similarly by estimating the difference in term of time derivatives
\begin{equation}\label{zlz2}
| \D^p\partial_t^q\eta(x,t)-\D^p\partial_t^q\eta(x,s)|_{g_0(\tau)}\leq  d^2  \| \mathfrak{D}^{j+2} \eta\|_{\infty, Q_{g_0(\tau),R}(x_0,t_0),g_0(\tau)}.
\end{equation}
Hence, under the assumption that $j+2=p+2(q+1)\leq 2k$, we can combine \eqref{zlz}, \eqref{zlz1} and \eqref{zlz2} to get
\begin{equation}
\begin{split}
&\quad  \frac{|\D^p\partial_t^q \eta(x,t)-\P_{x'x} \D^p\partial_t^q\eta(x',s)|_{g_0(\tau)}}{\left(d^{g_0(\tau)}(x,x')+|t-s|^{\frac{1}{2}} \right)^\b}\\
&\leq  C \left(d^{1-\b} \| \mathfrak{D}^{j+1} \eta\|_{\infty, Q_{g_0(\tau),R}(x_0,t_0),g_0(\tau)}+ d^{2-\b}\| \mathfrak{D}^{j+2} \eta\|_{\infty, Q_{g_0(\tau),R}(x_0,t_0),g_0(\tau)} \right).
\end{split}
\end{equation}
The conclusion then follows by combining with \eqref{interpolation-Linfty} since $d\leq \frac1{4\Lambda}(R-\rho)$.

It remains to consider the case when $p+2q<2k<p+2q+2$, i.e. $p+2q=2k-1$. This implies that $p$ is odd, hence $p\geq 1$.  Let $v$ run over a $g_0(\tau)$-orthonormal basis of tangent vectors which are either horizontal or vertical. Let $\gamma(u)$ be the unique $\P$-geodesic with $\gamma(0)=x$ and $\dot\gamma(0)=v$ with $u\in (0,R-\rho)$.  Denote $\sigma(u,\cdot )=\P^{-1}_{\gamma(0),\gamma(u)}\D^{p-1}\partial_t^{q} \eta(\gamma(u),\cdot )$ so that $\D_v \D^{p-1}\partial_t^{q}\eta(x,\cdot )=\partial_u|_{u=0}\sigma(u,\cdot )$.
By the mean value theorem, there exists $\theta\in [0,1]$ such that
\begin{equation}
\begin{split}
&\quad | \D_v \D^{p-1}\partial_t^{q}\eta(x,t)-\D_v\D^{p-1}\partial_t^{q}\eta(x,s)|_{g_0(\tau)}\\
&\leq    \left| \D_v \D^{p-1}\partial_t^{q}\eta(x,t)-\frac1d( \sigma(d,t)- \sigma(0,t))\right|_{g_0(\tau)}\\
&\quad +\left| \D_v \D^{p-1}\partial_t^{q}\eta(x,s)-\frac1d( \sigma(d,s)- \sigma(0,s))\right|_{g_0(\tau)}\\
&\quad +\frac1d \left| ( \sigma(d,t)- \sigma(d,s))-( \sigma(0,t)- \sigma(0,s))\right|_{g_0(\tau)}\\
&\leq |\sigma'(0,t)-\sigma'(\theta d,t)|+ |\sigma'(0,s)-\sigma'(\theta d,s)|
\\
&\quad +\frac1d \left|\sigma(d,t)- \sigma(d,s)\right|_{g_0(\tau)}+\frac1d \left|\sigma(0,t)- \sigma(0,s)\right|_{g_0(\tau)}\\
&\leq C d \|\D^2_v \D^{p-1}\partial_t^{q} \eta\|_{\infty, Q_{g_0(\tau),R}(x_0,t_0),g_0(\tau)}+ C d \| \D^{p-1}\partial_t^{q+1} \eta\|_{\infty, Q_{g_0(\tau),R}(x_0,t_0),g_0(\tau)}\\
&\leq C d\|\mathfrak{D}^{2k} \eta\| _{\infty, Q_{g_0(\tau),R}(x_0,t_0),g_0(\tau)},
\end{split}
\end{equation}
where $\sigma'$ denotes the $u$-derivative. Since $v$ is arbitrary, we conclude that
\begin{equation}
\frac{|\D^p\partial_t^{q} \eta(x,t)-\P_{x'x} \D^{p}\partial_t^{q}\eta(x',s)|_{g_0(\tau)}}{\left(d^{g_0(\tau)}(x,x')+|t-s|^{\frac{1}{2}} \right)^\b}
\leq  C d^{1-\b}\|\mathfrak{D}^{2k} \eta\| _{\infty, Q_{g_0(\tau),R}(x_0,t_0),g_0(\tau)}.
\end{equation}
The result then follows using \eqref{interpolation-Linfty} again. This completes the proof.
\end{proof}

We end this subsection by showing that a function on $\mathbb{R}^n\times (-\infty,0]$ with bounded $(2k+\a)$ parabolic H\"older seminorm and vanishing parabolic $2k$ jet at $(x,t)=(0,0)$ will be bounded in $C^{2k+\a,k+\a/2}_{\rm loc}$.
\begin{lma}\label{lma:jet-interpolation}
Let $u$ be a smooth function on $\mathbb{R}^n\times (-\infty,0]$ such that
\begin{equation}
[\mathfrak{D}^{2k}u]_{\a,\a/2,B_{\mathbb{R}^n}(R)\times [-R^2,0]}\leq \Lambda_0,
\end{equation}
for some $R,\Lambda_0>0,k\in \mathbb{N}$ and $\mathfrak{D}^{\ell}u|_{(0,0)}=0$ for all $0\leq \ell\leq 2k$. Then for all $0<r\leq R$ and $0\leq m\leq 2k$, there exists $C_0(n,m)>0$ such that
\begin{equation}\label{kzk1} \|\mathfrak{D}^{m}u\|_{\infty,B_{\mathbb{R}^n}(r)\times [-r^2,0]}\leq C_0\Lambda_0  r^{2k+\a-m}.\end{equation}
Moreover, for all $\b\in (0,\a)$ there exists $C_1(n,\Lambda_0,\b)>0$ such that
\begin{equation}\label{kzk2} [\mathfrak{D}^{m}u]_{\b,\b/2,B_{\mathbb{R}^n}(r)\times [-r^2,0]}\leq C_1\Lambda_0  r^{2k+\a-m-\b}.\end{equation}
\end{lma}
\begin{rmk}\label{remm}
This Lemma is false as stated if we replace $2k$ with an odd integer, and this is the main reason why in our main Theorem \ref{thm:decomposition} we will restrict to even order derivatives. The simplest counterexample is the function $u(x,t)=t$ in $\mathbb{R}\times\mathbb{R}$, which satisfies
$u(0,0)=0,$ $\DD u|_{(0,0)}=0$ and $[\DD u]_{\a,\a/2,\mathbb{R}\times\mathbb{R}}=0$ but \eqref{kzk1} fails for $m=0$.

To fix this, one has to redefine the parabolic H\"older seminorms of odd order by adding an additional term, see \cite[p.46]{Lie}. If one was to do this, then the statement of Lemma \ref{lma:jet-interpolation} would also hold when $2k$ is replaced by an odd integer. However, the additional term that one would need to add would not be compatible with our blowup arguments in section \ref{sekt}, especially with the ``non-escaping property'' in Section \ref{nonesc}.
\end{rmk}

\begin{proof}
Write $Q_r=B_{\mathbb{R}^n}(r)\times [-r^2,0]$ for notational convenience. We only prove the bound for $\|\mathfrak{D}^m u\|_{\infty}$ in \eqref{kzk1}, since the bound for H\"older seminorm in \eqref{kzk2} is similar.

By considering $\tilde u(x,t)=\Lambda_0 ^{-1}R^{-2k-\a}u(Rx,R^2t)$ for $(x,t)\in B_{\mathbb{R}^n}(1)\times [-1,0]$, we can assume $\Lambda_0=1=R=1$ and $0<r\leq 1$.
We prove the result by induction on $k$.  In case $k=0$, the jet assumption is equivalent to $u(0,0)=0$ and hence for all $0<r\leq 1$,
\begin{equation}
\|u\|_{\infty,Q_r}\leq r^\a,
\end{equation}
so that the conclusion holds.

Next we consider the induction step, so we assume that the conclusion holds for all $0\leq \ell\leq k$, and prove it for $k+1\geq 1$. Given a smooth function $u$ with $[\mathfrak{D}^{2k+2}u]_{\a,\a/2,Q_1}\leq 1$, given any $0\leq m\leq 2k$, every derivative $\DD^{m+2}u$ can be written as $\DD^m v$ where $v=\partial_t u$ or $v=\D^2 u$ (evaluated at some pair of tangent vectors). The function $v$ satisfies $[\mathfrak{D}^{2k}v]_{\a,\a/2,Q_1}\leq 1$ and $\mathfrak{D}^{\ell}v|_{(0,0)}=0$ for all $0\leq \ell \leq 2k$.  The induction hypothesis then implies $\|\mathfrak{D}^{m+2} u\|_{\infty,Q_r}\leq C_{k} r^{2k+\a-m}$ for all $0\leq m\leq 2k$ and $0<r\leq 1$. It remains to extend it to $m=-1,-2$, i.e. to bound $u$ and $\D u$. Let $(x,t)\in Q_r$ and fix a unit vector $e_1$, and estimate
\begin{equation}
\begin{split}
|\D_1 u(x,t)|&\leq \left|\D_1 u(x,t)-\frac{u(x+re_1,t)-u(x,t)}{r}\right|\\
&\quad +\left|\frac{u(x+re_1,t)-u(x,t)}{r}-\frac{u(x+re_1,0)-u(x,0)}{r}\right|\\
&\quad +\left|\frac{u(x+re_1,0)-u(x,0)}{r}-\D_1u(x,0)\right|\\
&\quad  +\left|\D_1u(x,0)-\D_1u(0,0)\right|\\
&=\mathbf{I}+\mathbf{II}+\mathbf{III}+\mathbf{IV},
\end{split}
\end{equation}
and we bound each of the numbered terms as follows. By the mean value theorem, there exists $\theta\in [0,1]$ such that
\begin{equation}
\mathbf{I}=\left|\D_1 u(x,t)-\D_1 u(x+\theta r e_1,t)\right|,
\end{equation}
and using \cite[Lemma 2.5]{HT3} we can bound this by $r\|\mathfrak{D}^2 u\|_{\infty,Q_r}\leq C r^{2k+\a+1}.$ For the second term, the mean value theorem again shows that there is $\theta'\in [0,1]$ so that
\begin{equation}
\begin{split}
\mathbf{II}&=\frac{|t|}r \left|\partial_t u(x+re_1,\theta' t)- \partial_t u(x,\theta' t)\right|\leq r\|\de_t u\|_{\infty, Q_r}\leq r\|\mathfrak{D}^2 u\|_{\infty,Q_r}\leq Cr^{2k+\a+1}.
\end{split}
\end{equation}
For the third term, using the mean value theorem and \cite[Lemma 2.5]{HT3}, we can find $\theta''\in [0,1]$ so that
\begin{equation}
\begin{split}
\mathbf{III}&=\left|\D_1 u(x+r\theta'' e_1,0)-\D_1u(x,0)\right|\leq  r\|\D^2 u\|_{\infty,Q_r}\leq r\|\DD^2 u\|_{\infty, Q_r}\leq Cr^{2k+\a+1},
\end{split}
\end{equation}
and for the fourth term we again use \cite[Lemma 2.5]{HT3} to bound
\begin{equation}
\mathbf{IV}\leq r\|\D^2 u\|_{\infty,Q_r}\leq r\|\DD^2 u\|_{\infty, Q_r}\leq Cr^{2k+\a+1},
\end{equation}
and putting these all together proves that $|\D_1 u(x,t)|\leq Cr^{2k+\a+1}$, and hence  $\|\mathfrak{D} u\|_{\infty,Q_r}\leq C_{k} r^{2k+\a+1}$ since $e_1$ is arbitrary. The upper bound for $|u|$ is now straightforward using the bounds on $\D u$ and $\partial_t u$. This completes the proof of the inductive step.
\end{proof}

\subsection{Bounds on H\"older seminorms imply decay}
In this section we establish a generalization of \cite[Theorem 2.11]{HT3} to our setting.
Recall that at each point $x=(z,y)\in B\times  Y$, $\omega_{F,z}$ is the unique K\"ahler-Ricci flat metric on each fiber $X_z$ which is cohomologous to $\omega_0|_{X_z}$. We can assume that $\int_{\{z\}\times Y} \omega_{F,z}^n=1$ for all $z\in  B$. For any function $f$ in space-time $B\times  Y\times \mathbb{R}$, we will use $\underline{f}(z,t)$ to denote its fiberwise average:
\begin{equation}
\underline{f}(z,t)=\int_{\{z\}\times Y} f(z,\cdot,t) \,\omega_{F,z}^n.
\end{equation}
The following result will be crucial for us:
\begin{prop}\label{prop:braindead}
Suppose $g=g_{\mathbb{C}^m}+\delta^2 g_{Y,0}$ is a metric on $B\times Y$ where $0<\delta\leq 1$ is arbitrary. For any $k\in \mathbb{N}, \a\in (0,1)$ and $0<\rho<R<1$ with $\rho\geq \Lambda \delta$, there exists $C(k,\a,\rho,R)>0$ such that for all smooth function $\varphi$ on $B\times Y\times \mathbb{R}$ with $\underline{\varphi}=0$, $x_0=(0,y_0), t_0\in \mathbb{R}$ and for all $0\leq j\leq 2k$, we have
\begin{equation}\label{BD-Linfty}
\|\mathfrak{D}^j \varphi\|_{\infty,Q_{g,\rho}(x_0,t_0),g}\leq C \delta^{2k+\a-j}[ \mathfrak{D}^{2k} \varphi]_{\a,\a/2,Q_{g,R}(x_0,t_0),g}.
\end{equation}
Moreover for all $\b\in (0,1)$ such that $j+\b<2k+\a$,
\begin{equation}\label{BD-holder}
[ \mathfrak{D}^{j} \varphi]_{\b,\b/2,Q_{g,\rho}(x_0,t_0),g}\leq C \delta^{2k+\a-j-\b}[ \mathfrak{D}^{2k} \varphi]_{\a,\a/2,Q_{g,R}(x_0,t_0),g}.
\end{equation}
Moreover, the same estimates hold if $\varphi$ is replaced by $\eta=\ddbar\varphi$ where $\underline{\varphi}=0$.
\end{prop}
\begin{rmk}
We require $\rho\geq \Lambda \delta$ here so as to ensure that
\begin{equation}
Q_{g,\rho}(x_0,t_0)=\left(B_{\mathbb{C}^m}(\rho)\times Y\right)\times [t_0-\rho^2,t_0]=Q_\rho(x_0,t_0),
\end{equation}
which is needed in order to apply \cite[Theorem 2.11]{HT3}.
\end{rmk}
\begin{proof}
Suppose we can show that
\begin{equation}\label{burzum}
\|\vp\|_{\infty,Q_{g,\rho}(x_0,t_0),g}\leq C \delta^{2k+\a}[ \mathfrak{D}^{2k} \varphi]_{\a,\a/2,Q_{g,R}(x_0,t_0),g},
\end{equation}
then \eqref{BD-Linfty} and \eqref{BD-holder} would follow from this and the interpolation Proposition \ref{prop:interpolation}, as in \cite[(2.61)--(2.62)]{HT3}. To prove \eqref{burzum}, given
$(x,t)\in Q_{g,\rho}(x_0,t_0)$, write as usual $x=(z,y)$ and freeze the time variable in $\vp(\cdot,t)$. Assuming first that $k\geq 1$, similarly to \cite[(2.81)]{HT3} we can use $\underline{\vp}(t)=0$ to estimate
\begin{equation}\begin{split}
\sup_{\{z\}\times Y}|\vp(t)|&\leq C\sup_{\{z\}\times Y}|\nabla_{\mathbf{f}}\vp(t)|_{\{z\}\times Y}|_{g_{Y,z}}\leq C[\nabla^{2k}_{\mathbf{f\cdots f}}\vp(t)|_{\{z\}\times Y}]_{C^\alpha(\{z\}\times Y, g_{Y,z})}\\
&\leq C\delta^{2k+\alpha}[\D^{2k}_{\mathbf{f\cdots f}}\vp(t)|_{\{z\}\times Y}]_{C^\alpha(\{z\}\times Y,g)}
\leq C\delta^{2k+\alpha}[\D^{2k}\vp(t)]_{C^\alpha(B_R\times Y,g)}\\
&\leq C\delta^{2k+\alpha}[\DD^{2k}\vp]_{\a,\a/2,Q_{g,R}(x_0,t_0),g},
\end{split}\end{equation}
where in the second inequality we used \cite[Lemma 2.10]{HT3}, and \eqref{burzum} follows when $k\geq 1$. Lastly, when $k=0$ the argument is straightforward: using again that $\underline{\vp}(t)=0$ we bound
\begin{equation}\begin{split}
\sup_{\{z\}\times Y}|\vp(t)|&\leq C[\vp(t)|_{\{z\}\times Y}]_{C^\alpha(\{z\}\times Y, g_{Y,z})}\leq C\delta^{\alpha}[\vp(t)|_{\{z\}\times Y}]_{C^\alpha(\{z\}\times Y,g)}\\
&\leq C\delta^{\alpha}[\vp(t)]_{C^\alpha(B_R\times Y,g)}\leq C\delta^{\alpha}[\vp]_{\a,\a/2,Q_{g,R}(x_0,t_0),g},
\end{split}\end{equation}
which completes the proof of \eqref{BD-Linfty} and \eqref{BD-holder} for $\vp$. Lastly, the analogous estimates for $\eta=\ddbar\vp$ follow in a similar fashion, by first using interpolation Proposition \ref{prop:interpolation} to reduce ourselves to proving
\begin{equation}\label{marduk}
\|\eta\|_{\infty,Q_{g,\rho}(x_0,t_0),g}\leq C \delta^{2k+\a}[ \mathfrak{D}^{2k} \eta]_{\a,\a/2,Q_{g,R}(x_0,t_0),g},
\end{equation}
and then proving \eqref{marduk} by freezing the time variable $t$ and applying \cite[(2.61)]{HT3} to $\eta(\cdot, t)$ and get
\begin{equation}
\|\eta(t)\|_{\infty,B_R\times Y,g}\leq C \delta^{2k+\a}[ \D^{2k} \eta]_{C^\alpha(B_R\times Y,g)}\leq C \delta^{2k+\a}[ \mathfrak{D}^{2k} \eta]_{\a,\a/2,Q_{g,R}(x_0,t_0),g},
\end{equation}
which concludes the proof.
\end{proof}

\subsection{Parabolic Schauder estimates}

In the course of the proof of our main Theorem, we also need two parabolic Schauder estimates on cylinders, which will be used when linearizing the \KR flow equation, and which are analogs of \cite[Proposition 2.15]{HT3}.  Let $(z_\ell,y_\ell)\to (z_\infty,y_\infty)$  be a convergent family of points in $B\times Y$. For $\ell\geq 1$, consider the diffeomorphism $\Lambda_\ell: (\check z,\check y)\mapsto (z_\ell+e^{-t_\ell/2}\check z,\check y)$, and let $\check J_\ell$ be the pullback of the complex structure $J$ via $\Lambda_\ell$, which converges to $\check J_\infty=J_{\mathbb{C}^m}+J_{Y,z_\infty}$ locally smoothly. Similarly, we let $\check{\D}_\ell$ denotes the pullback of the connection $\D$ so that $\check{\D}_\ell\to \check{\D}_\infty=\nabla^{\mathbb{C}^m}+\nabla^{g_{Y,z_\infty}}$ locally smoothly in spacetime.  By the translation, we may assume that our new base point is $\check p_\ell=(\check z_\ell,\check y_\ell)=(0,\check y_\ell)\to (0,\check y_\infty)$. We rescale the geometric quantities in a parabolic way centered at $t_\ell$, such as for example $\check g_\ell(\check t)=e^{t_\ell}\Lambda_\ell^* g(t_\ell+e^{-t_\ell}\check t),$ where recall that we denote by $g(t)=g_{\C^m}+e^{-t}g_{Y,0}$, so that $\check{g}_\ell(0)=g_{\C^m}+g_{Y,0}$.

The first Schauder estimate is for scalar functions:

\begin{prop}\label{prop:schauder}
 Let $U\subset \mathbb{C}^m\times Y$ be an open set containing $(0,y_\infty)$.  Let $\check g_\ell, \check{\omega}_\ell^\sharp$ be Riemannian resp. $\check J_\ell$-\K metrics on $U$ that converges locally smoothly to a Riemannian resp. $\check J_\infty$-\K metric $\check{g}_\infty,\check{\omega}_\infty^\sharp$ on $U$.
Then for all $a\in \mathbb{N}_{>0}$, $\a\in (0,1)$ and $R>0$, there exists $\ell_0>0$ and $C>0$ such that  for all $0<\rho<R$, $\ell\geq \ell_0$ and all smooth function $u$ defined on $U\times \mathbb{R}$, we have that
\begin{equation}
\begin{split}
[\mathfrak{D}^{2a}  u]_{\a,\a/2,\check Q_{\check g_\ell(0),\rho}(\check p_\ell,0 ),\check g_\ell(0)}&\leq C[ \mathfrak{D}^{2a-2}(\partial_t-\Delta_{\check \omega_\ell^\sharp})u]_{\a,\a/2,\check Q_{\check g_\ell(0),R}(\check p_\ell,0 ),\check g_\ell(0)}\\
&\quad \quad +C(R-\rho)^{-2a-\a} \|u\|_{\infty,\check Q_{\check g_\ell(0),R}(\check p_\ell,0 )},
\end{split}
\end{equation}
whenever $\check Q_{\check g_\ell(0),R}(\check p_\ell,0 )\subset U\times \mathbb{R}$.
\end{prop}
\begin{proof}
We let $\sigma$ be small enough so that with respect to $\check g_\infty=g_{\mathbb{C}^m}+g_{Y,z_\infty}$, the $\check g_\infty$-geodesic ball of size $\sigma$ is geodesically convex and admits a normal coordinate chart centered at any $(z,y)\in \mathbb{C}^m\times Y$. This is possible since $Y$ is compact and $\mathbb{C}^m$ is flat with respect to $\check g_\infty$. Since $\check g_\ell(t)\to g_{\mathbb{C}^m}+g_{Y,z_\infty}$ as $\ell\to+\infty$ and $\check \D_\ell\to \check \D_\infty$ locally  smoothly, for any large $R>0$ there exists $\ell_0$ such that for all $\ell>\ell_0$ and $(\check p,\check t)\in \check Q_{\check g_\ell(0),2R}(\check p_\ell,0)$, $B_{\check g_\ell(0)}(\check p,\sigma)$ is geodesically convex and is compactly contained in a Euclidean ball of radius $2\sigma$. Moreover, we can assume that the \K structure $\check{\omega}^\sharp_\ell, \check J_\ell$ is smoothly close to the product structure $\check{\omega}^\sharp_\infty, \check{J}_\infty$ in the parabolic domain $\check Q_{\check g_\ell(0),2R}(\check p_\ell,0 )$.

The product metric $\check g_\ell(0)$ is uniformly comparable to the Euclidean metric on $\check Q_{\check g_\ell(0),\sigma}(\check p,\check t)$, for all $(\check p,\check t)\in \check Q_{\check g_\ell(0),R}(\check p_\ell,0 )$. By the standard Euclidean parabolic Schauder estimates (see e.g. \cite[Theorem 4.9 and p.84]{Lie}), there exists $C>0$ such that for all $0<\rho_1<\rho_2<\sigma$, and functions $u$ on $\check Q_{\check g_\ell(0),\sigma}(\check p,\check t)$ where $(\check p,\check t)\in \check Q_{\check g_\ell(0),R}(\check p_\ell,0 )$, we have
\begin{equation}\label{eqn:standSchau}
\begin{split}
&\quad \sum_{p+2q=2a}[\partial^{p}\partial_t^qu]_{\a,\a/2,\check Q_{\check g_\ell(0),\rho_1}(\check p,\check t)}+ \sum_{p+2q\leq 2a}(\rho_2-\rho_1)^{-2a-\a+p+2q}\|\partial^{p}\partial_t^qu\|_{\infty,\check Q_{\check g_\ell(0),\rho_1}(\check p,\check t)} \\
&\leq C \sum_{p+2q=2a-2}\left[\partial^{p}\partial_t^q \left(\frac{\de}{\de t}-\Delta_{\check\omega_\ell^\sharp}\right)u\right]_{\a,\a/2,\check Q_{\check g_\ell(0),\rho_2}(\check p,\check t)} +  C(\rho_2-\rho_1)^{-2a-\a}\| u\|_{\infty,\check Q_{\check g_\ell(0),\rho_2}(\check p,\check t)}\\
&\quad +C\sum_{p+2q\leq 2a-2}(\rho_2-\rho_1)^{-2a-\a+p+2q+2}\left|\left|\partial^{p}\partial_t^q \left(\frac{\de}{\de t}-\Delta_{\check\omega_\ell^\sharp}\right)u\right|\right|_{\infty,\check Q_{\check g_\ell(0),\rho_2}(\check p,\check t)},
\end{split}
\end{equation}
where the derivatives, H\"older norms are computed using the standard Euclidean metric and $\Delta_{\check\omega_\ell^\sharp}$ denotes the Laplacian of the metric $\check \omega_\ell^\sharp$.

We first use interpolation to eliminate the
terms in the last line of \eqref{eqn:standSchau}. For $0<\rho_1<\rho_2<\sigma$,  we let $\rho_2'=\frac12 (\rho_1+\rho_2)$ so that \eqref{eqn:standSchau} holds with $\rho_2$ replaced by $\rho_2'$ and $C$ replaced by a slightly larger $C$.  Hence the  standard Euclidean interpolation (or Proposition \ref{prop:interpolation}) yields
\begin{equation}\label{eqn:standSchau-2}
\begin{split}
&\sum_{p+2q\leq 2a-2}(\rho_2-\rho_1)^{-2a-\a+p+2q+2}\left|\left|\partial^{p}\partial_t^q \left(\frac{\de}{\de t}-\Delta_{\check\omega_\ell^\sharp}\right)u\right|\right|_{\infty,\check Q_{\check g_\ell(0),\rho_2'}(\check p,\check t)}\\
&\leq C \sum_{p+2q=2a-2}\left[\partial^{p}\partial_t^q \left(\frac{\de}{\de t}-\Delta_{\check\omega_\ell^\sharp}\right)u\right]_{\a,\a/2,\check Q_{\check g_\ell(0),\rho_2}(\check p,\check t)} +  C(\rho_2-\rho_1)^{-2a-\a+2}\left|\left| \left(\frac{\de}{\de t}-\Delta_{\check\omega_\ell^\sharp}\right)u\right|\right|_{\infty,\check Q_{\check g_\ell(0),\rho_2}(\check p,\check t)},
\end{split}
\end{equation}
where we have used $\rho_2'-\rho_2=\frac12 (\rho_2-\rho_1)$.

We now want to estimate $L^\infty$ norm of $\left(\frac{\de}{\de t}-\Delta_{\check\omega_\ell^\sharp}\right)u$ in terms of the $L^\infty$ norm of $u$ and of the H\"older seminorm of top order derivatives of $u$. To do this, we interpolate again. Let $\e\in (0,1)$ be a small constant to be determined, and given $0<\rho_1<\rho_2<\sigma$, denote $\rho_2''=\rho_1+(1-\e) (\rho_2-\rho_1)$. We consider \eqref{eqn:standSchau-2} with $\rho_2$ replaced by $\rho_2''$. Since $\check \omega_\ell^\sharp$ is uniformly bounded in $C^\infty$, the standard Euclidean interpolation (or Proposition \ref{prop:interpolation}) yields
\begin{equation}\label{spandau}
\begin{split}
&\quad (\rho_2''-\rho_1)^{-2a-\a+2}\left|\left| \left(\frac{\de}{\de t}-\Delta_{\check\omega_\ell^\sharp}\right)u\right|\right|_{\infty,\check Q_{\check g_\ell(0),\rho_2''}(\check p,\check t)}\\
&=\left(\frac{\ve}{1-\ve} \right)^{2a+\a-2}(\rho_2-\rho_2'')^{-2a-\a+2}\left|\left| \left(\frac{\de}{\de t}-\Delta_{\check\omega_\ell^\sharp}\right)u\right|\right|_{\infty,\check Q_{\check g_\ell(0),\rho_2''}(\check p,\check t)}\\
&\leq C\left(\frac{\ve}{1-\e} \right)^{2a+\a-2}\left(\sum_{p+2q=2a}\left[\partial^p\partial_t^q u \right]_{\a,\a/2,\check Q_{\check g_\ell(0),\rho_2}(\check p,\check t)}+(\rho_2-\rho_2'')^{-2a-\a} \|u\|_{\infty,\check Q_{\check g_\ell(0),\rho_2}(\check p,\check t)}\right).
\end{split}
\end{equation}
Therefore, inserting \eqref{eqn:standSchau-2} and \eqref{spandau} into \eqref{eqn:standSchau}, we can choose $\e$ sufficiently small so that
\begin{equation}
\begin{split}
&\quad \sum_{p+2q=2a}[\partial^{p}\partial_t^qu]_{\a,\a/2,\check Q_{\check g_\ell(0),\rho_1}(\check p,\check t)}+ \sum_{p+2q\leq 2a}(\rho_2-\rho_1)^{-2a-\a+p+2q}\|\partial^{p}\partial_t^qu\|_{\infty,\check Q_{\check g_\ell(0),\rho_1}(\check p,\check t)} \\
&\leq C \sum_{p+2q=2a-2}\left[\partial^{p}\partial_t^q \left(\frac{\de}{\de t}-\Delta_{\check\omega_\ell^\sharp}\right)u\right]_{\a,\a/2,\check Q_{\check g_\ell(0),\rho_2}(\check p,\check t)} +  C(\rho_2-\rho_1)^{-2a-\a}\| u\|_{\infty,\check Q_{\check g_\ell(0),\rho_2}(\check p,\check t)}\\
&\quad +\frac12 \sum_{p+2q=2a}[\partial^{p}\partial_t^qu]_{\a,\a/2,\check Q_{\check g_\ell(0),\rho_2}(\check p,\check t)},
\end{split}
\end{equation}
for all $0<\rho_1<\rho_2<\sigma$. We can then apply the iteration lemma in \cite[Lemma 2.9]{HT3} to obtain
\begin{equation}\label{eqn:standSchau-3}
\begin{split}
&\quad \sum_{p+2q=2a}[\partial^{p}\partial_t^qu]_{\a,\a/2,\check Q_{\check g_\ell(0),\rho_1}(\check p,\check t)}+ \sum_{p+2q\leq 2a}(\rho_2-\rho_1)^{-2a-\a+p+2q}\|\partial^{p}\partial_t^qu\|_{\infty,\check Q_{\check g_\ell(0),\rho_1}(\check p,\check t)} \\
&\leq C \sum_{p+2q=2a-2}\left[\partial^{p}\partial_t^q \left(\frac{\de}{\de t}-\Delta_{\check\omega_\ell^\sharp}\right)u\right]_{\a,\a/2,\check Q_{\check g_\ell(0),\rho_2}(\check p,\check t)} +  C(\rho_2-\rho_1)^{-2a-\a}\| u\|_{\infty,\check Q_{\check g_\ell(0),\rho_2}(\check p,\check t)}.
\end{split}
\end{equation}

We now claim that one can interchange the Euclidean derivatives $\partial$ and Euclidean parallel transport $P$ in the definition of H\"older norms with $\check\D_\ell$ and $\check{\mathbb{P}}_\ell$. Although in our definition of H\"older norms we only consider horizontal and vertical $\mathbb{P}$-geodesics while the standard Euclidean H\"older norms consider all possible segments, it is immediate to see from \cite[(2.31)]{HT3} that the difference is harmless.  To compare $\partial$ with $\check \D_\ell$, we note that on each $\check Q_{\check g_\ell(0),\sigma}(\check p,\check t)$ we can write $\partial=\check \D_\ell+\check \Gamma_\ell$ (indeed independent of $\check{t}$) so that $\check \Gamma_\ell\to \check \Gamma_\infty$ where $\check \Gamma_\infty(\check p)=0$ by our choice of normal coordinate centered at $\check p$. In particular, the local smooth convergence implies that $\check \Gamma_\ell\to 0$ in $C^k$ for fixed $k$ uniformly as $\ell\to +\infty$ and $\sigma\to 0$, and then ODE estimates show that $P-\check{\P}_\ell\to 0$ in $C^k$. In particular, switching from $\partial^p$ to $\check \D_\ell^p$ will generate an error of the form (ignoring combinatorial constants):
\begin{equation}\label{eqn:error}
\left(\check \D_\ell^a-\partial^a\right) u=\sum_{i+j_1+\dots+j_q+q=a,q>0} \partial^{i}u \circledast   \partial^{j_1}\check \Gamma_\ell\circledast \dots\circledast  \partial^{j_q}\check \Gamma_\ell.
\end{equation}
Since $0<\rho_1<\rho_2\leq \sigma<1$, and since all terms in \eqref{eqn:error} with $\check{\Gamma}_\ell$ are $o(1)$,  the right hand side in \eqref{eqn:error} can be absorbed in the second term in the left hand side of \eqref{eqn:standSchau}. Thus, going back to \eqref{eqn:standSchau-3}, we can change the $\de,P$ with $\check{D}_\ell,\check{\P}_\ell$, and using again interpolation we obtain
\begin{equation}\label{eqn:standSchau-switch}
\begin{split}
&\quad [\mathfrak{D}^{2a} u]_{\a,\a/2,\check Q_{\check g_\ell(0),\rho_1}(\check p,\check t),\check g_\ell(0)}+ \sum_{b\leq 2a}(\rho_2-\rho_1)^{-2a-\a+b}\|\mathfrak{D}^bu\|_{\infty,\check Q_{\check g_\ell(0),\rho_1}(\check p,\check t),\check g_\ell(0)} \\
&\leq C \sum_{p+2q=2a-2}\left[\partial^{p}\partial_t^q \left(\frac{\de}{\de t}-\Delta_{\check \omega_\ell^\sharp}\right) u\right]_{\a,\a/2,\check Q_{\check g_\ell(0),\rho_2}(\check p,\check t)} +  C(\rho_2-\rho_1)^{-2a-\a}\| u\|_{\infty,\check Q_{\check g_\ell(0),\rho_2}(\check p,\check t)}\\
&\leq  C \left[\mathfrak{D}^{2a-2} (\partial_t-\Delta_{\check \omega_\ell^\sharp}) u\right]_{\a,\a/2,\check Q_{\check g_\ell(0),\rho_2}(\check p,\check t),\check g_\ell(0)} +  C(\rho_2-\rho_1)^{-2a-\a}\| u\|_{\infty,\check Q_{\check g_\ell(0),\rho_2}(\check p,\check t)}\\
&\quad + \frac12 \sum_{b\leq 2a}(\rho_2-\rho_1)^{-2a-\a+b}\|\mathfrak{D}^bu\|_{\infty,\check Q_{\check g_\ell(0),\rho_2}(\check p,\check t),\check g_\ell(0)},
\end{split}
\end{equation}
where the coefficient $\frac12$ in the last term is achieved by choosing $\sigma$ sufficiently small thanks to the local smooth convergence of $\check \Gamma_\ell$.  We fix $\sigma$ from now on.  It now follows from \cite[Lemma 2.9]{HT3} again that we have
\begin{equation}\label{eqn:standSchau-switched}
\begin{split}
&\quad [\mathfrak{D}^{2a} u]_{\a,\a/2,\check Q_{\check g_\ell(0),\rho_1}(\check p,\check t),\check g_\ell(0)}+ \sum_{b\leq 2a}(\rho_2-\rho_1)^{-2a-\a+b}\|\mathfrak{D}^bu\|_{\infty,\check Q_{\check g_\ell(0),\rho}(\check p,\check t),\check g_\ell(0)} \\
&\leq  C \left[\mathfrak{D}^{2a-2} (\partial_t-\Delta_{\check \omega_\ell^\sharp}) u\right]_{\a,\a/2,\check Q_{\check g_\ell(0),\rho_2}(\check p,\check t),\check g_\ell(0)} +  C(\rho_2-\rho_1)^{-2a-\a}\| u\|_{\infty,\check Q_{\check g_\ell(0),\rho_2}(\check p,\check t)},
\end{split}
\end{equation}
for all $0<\rho_1<\rho_2<\sigma$. This in particular shows the desired conclusion for all small $0<\rho_1<\rho_2<\sigma$ with arbitrary center $(\check p,\check t)$ in the compact set $ \check Q_{\check g_\ell(0),R}(\check p,0 )$.

Now we prove the H\"older control on $ \check Q_{\check g_\ell(0),\rho}(\check p_\ell,0 )$ for $0<\rho<R$.  We can assume $R\geq \sigma$. If $\rho<\frac12 \sigma\leq \frac12 R$ so that $\sigma-\rho\geq \frac12\sigma\geq C^{-1}(R-\rho)$, then \eqref{eqn:standSchau-switched} implies
\begin{equation}
\begin{split}
[\mathfrak{D}^{2a} u]_{\a,\a/2,\check Q_{\check g_\ell(0),\rho}(\check p,\check t),\check g_\ell(0)}&\leq C \left[\mathfrak{D}^{2a-2} (\partial_t-\Delta_{\check \omega_\ell^\sharp}) u\right]_{\a,\a/2,\check Q_{\check g_\ell(0),\sigma}(\check p,\check t),\check g_\ell(0)} \\
&\quad +  C(\sigma-\rho)^{-2a-\a}\| u\|_{\infty,\check Q_{\check g_\ell(0),\sigma}(\check p,\check t)}\\
&\leq  C \left[\mathfrak{D}^{2a-2} (\partial_t-\Delta_{\check \omega_\ell^\sharp}) u\right]_{\a,\a/2,\check Q_{\check g_\ell(0),R}(\check p,\check t),\check g_\ell(0)} \\
&\quad +  C(R-\rho)^{-2a-\a}\| u\|_{\infty,\check Q_{\check g_\ell(0),\sigma}(\check p,\check t)}.
\end{split}
\end{equation}

Hence, it remains to consider the case $\rho\geq \frac12\sigma$. Given any  two given points $(\check p,\check t)$ and $(\check q,\check s)$ in  $ \check Q_{\check g_\ell(0),\rho}(\check p_\ell,0 )$ with  $r=d^{\check g_\ell(0)}(\check p,\check q)+|\check t-\check s|^{\frac{1}{2}}$.  We choose a sequence of points $\{(\check p_i,\check t_i) \}_{i=1}^N$ inside $\check Q_{\check g_\ell(0),\rho}(\check p_\ell,0 )$ such that $(\check p_1,\check t_1)=(\check p,\check t)$ and $(\check p_N,\check t_N)=(\check q,\check s)$.  We can choose them in a way so that $r_i=d^{\check g_\ell(0)}(\check p_i,\check p_{i+1})+|\check t_i-\check t_{i+1}|^{\frac{1}{2}}\leq \frac14 \sigma$ is uniformly comparable to $r$ and $N\leq C \sigma^{-2}$.  Here the square comes from the time direction.  Moreover, we can assume $ \check Q_{\check g_\ell(0),\frac14\sigma}(\check p_i,\check t_i)\subset  \check Q_{\check g_\ell(0),\rho}(\check p_\ell,0 ) $ for all $i=2,\dots,N-1$ (i.e.  except  $(\check p,\check t)$ and $(\check q,\check s)$). For each $i=1,\dots,N-1$, we can apply \eqref{eqn:standSchau-switched} again to obtain
\begin{equation}\label{langa}
\begin{split}
&\quad \frac{|\mathfrak{D}^{2a}u(\check p_{i+1},\check t_{i+1})- \check{\mathbb{P}_\ell}\mathfrak{D}^{2a}u(\check p_i,\check t_i)|_{\check g_\ell(0)}}{r_i^\a}\\
&\leq C \left[\mathfrak{D}^{2a-2} (\partial_t-\Delta_{\check \omega_\ell^\sharp}) u\right]_{\a,\a/2,\check Q_{\check g_\ell(0),\frac14\sigma+\min(R-\rho,\frac14\sigma )}(\check p_i,\check t_i),\check g_\ell(0)}\\
&\quad  +  C\min\left(R-\rho,\frac14\sigma \right)^{-2a-\a}\| u\|_{\infty,\check Q_{\check g_\ell(0),\frac14\sigma+\min(R-\rho,\frac14\sigma )}(\check p_i,\check t_i)}\\
&\leq C \left[\mathfrak{D}^{2a-2} (\partial_t-\Delta_{\check \omega_\ell^\sharp}) u\right]_{\a,\a/2,\check Q_{\check g_\ell(0),R}(\check p_i,\check t_i),\check g_\ell(0)}\\
&\quad  +  C(R-\rho)^{-2a-\a}\| u\|_{\infty,\check Q_{\check g_\ell(0),R}(\check p_\ell,0)},
\end{split}
\end{equation}
for all $i=1,\dots,N-1$.  Here we have used the fact that $\sigma\geq C^{-1}R\geq C^{-1}(R-\rho)$. Since
\begin{equation}
 \frac{|\mathfrak{D}^{2a}u(\check p,\check t)- \check{\mathbb{P}_\ell}\mathfrak{D}^{2a}u(\check q,\check s)|_{\check g_\ell(0)}}{r^\a}\leq C\sum_{i=0}^{N-1} \frac{|\mathfrak{D}^{2a}u(\check p_{i+1},\check t_{i+1})- \check{\mathbb{P}_\ell}\mathfrak{D}^{2a}u(\check p_i,\check t_i)|_{\check g_\ell(0)}}{r_i^\a},
\end{equation}
using \eqref{langa} completes the proof after taking supremum over all $(\check p,\check t)$ and $(\check q,\check s)$ in $\check Q_{\check g_\ell(0),\rho}(\check p_\ell,0 )$.

\end{proof}

The second Schauder estimate is for real $(1,1)$-forms:

\begin{prop}\label{prop:schauder2}
 Let $U\subset \mathbb{C}^m\times Y$ be an open set containing $(0,y_\infty)$.  Let $\check g_\ell, \check{\omega}_\ell^\sharp$ be Riemannian resp. $\check J_\ell$-\K metrics on $U$ that converges locally smoothly to a Riemannian resp. $\check J_\infty$-\K metric $\check{g}_\infty,\check{\omega}_\infty^\sharp$ on $U$.
Then for all $a\in \mathbb{N}_{>0}$, $\a\in (0,1)$ and $R>0$, there exists $\ell_0>0$ and $C>0$ such that  for all $0<\rho<R$, $\ell\geq \ell_0$ and all real $\check{J}_\ell$-$(1,1)$ form $\eta$ defined on $U\times \mathbb{R}$, we have that
\begin{equation}\label{JSB}
\begin{split}
[\mathfrak{D}^{2a}  \eta]_{\a,\a/2,\check Q_{\check g_\ell(0),\rho}(\check p_\ell,0 ),\check g_\ell(0)}&\leq C[ \mathfrak{D}^{2a-2}(\partial_t-\Delta_{\check \omega_\ell^\sharp})\eta]_{\a,\a/2,\check Q_{\check g_\ell(0),R}(\check p_\ell,0 ),\check g_\ell(0)}\\
&\quad \quad +C(R-\rho)^{-2a-\a} \|\eta\|_{\infty,\check Q_{\check g_\ell(0),R}(\check p_\ell,0 )},
\end{split}
\end{equation}
whenever $\check Q_{\check g_\ell(0),R}(\check p_\ell,0 )\subset U\times \mathbb{R}$, and where $\Delta_{\check \omega_\ell^\sharp}\eta$ denotes the Hodge Laplacian of $\check \omega_\ell^\sharp$ acting on differential forms.
\end{prop}

\begin{proof}
The proof shares some similarities with the proof of Proposition \ref{prop:schauder}. After the same preliminary remarks as there, we first work in the Euclidean setting, with the Hodge Laplacian $\Delta_{\check \omega_\ell^\sharp}\eta$. We can then apply standard parabolic Schauder estimate to the uniformly parabolic system given by $(\partial_t-\Delta_{\check \omega_\ell^\sharp})$ acting on real $(1,1)$-forms \cite[Chapter 7]{LSU}, which shows that there exists $C>0$ such that for all $0<\rho_1<\rho_2<\sigma$, and $\eta$ lives in $\check Q_{\check g_\ell(0),\sigma}(\check p,\check t)$ where $(\check p,\check t)\in \check Q_{\check g_\ell(0),R}(\check p_\ell,0 )$, we have
\begin{equation}\label{eqn:standSchaux}
\begin{split}
&\quad \sum_{p+2q=2a}[\partial^{p}\partial_t^q\eta]_{\a,\a/2,\check Q_{\check g_\ell(0),\rho_1}(\check p,\check t)}+ \sum_{p+2q\leq 2a}(\rho_2-\rho_1)^{-2a-\a+p+2q}\|\partial^{p}\partial_t^q \eta\|_{\infty,\check Q_{\check g_\ell(0),\rho_1}(\check p,\check t)} \\
&\leq C \sum_{p+2q=2a-2}\left[\partial^{p}\partial_t^q \left(\partial_t-\Delta_{\check\omega_\ell^\sharp} \right) \eta\right]_{\a,\a/2,\check Q_{\check g_\ell(0),\rho_2}(\check p,\check t)} +  C(\rho_2-\rho_1)^{-2a-\a}\| \eta\|_{\infty,\check Q_{\check g_\ell(0),\rho_2}(\check p,\check t)}\\
&\quad +C\sum_{p+2q\leq 2a-2}(\rho_2-\rho_1)^{-2a-\a+p+2q+2}\left|\left|\partial^{p}\partial_t^q\left(\partial_t-\Delta_{\check\omega_\ell^\sharp} \right) \eta\right|\right|_{\infty,\check Q_{\check g_\ell(0),\rho_2}(\check p,\check t)},
\end{split}
\end{equation}
where the derivatives and H\"older norms are computed using standard Euclidean metric.

We first eliminate the terms in the last line of \eqref{eqn:standSchaux} by interpolation: given $0<\rho_1<\rho_2<\sigma$, we let $\rho_2'=\frac{1}{2}(\rho_1+\rho_2)$, so that \eqref{eqn:standSchaux} holds with $\rho_2$ replaced by $\rho_2'$, and so standard interpolation gives
\begin{equation}\label{krkrk}
\begin{split}
&\quad \sum_{p+2q\leq 2a-2}(\rho_2'-\rho_1)^{-2a-\a+p+2q+2}\left|\left|\partial^{p}\partial_t^q\left(\partial_t-\Delta_{\check\omega_\ell^\sharp} \right) \eta\right|\right|_{\infty,\check Q_{\check g_\ell(0),\rho_2'}(\check p,\check t)} \\
&\leq \sum_{p+2q= 2a-2}\left[\partial^{p}\partial_t^q\left(\partial_t-\Delta_{\check\omega_\ell^\sharp} \right) \eta\right]_{\alpha,\alpha/2,\check Q_{\check g_\ell(0),\rho_2}(\check p,\check t)}\\
&\quad  +C(\rho_2-\rho_1)^{-2a-\a+2}\left\|\left(\partial_t-\Delta_{\check\omega_\ell^\sharp} \right) \eta\right\|_{\infty,\check Q_{\check g_\ell(0),\rho_2}(\check p,\check t)},
\end{split}
\end{equation}
using that $\rho'_2-\rho_2=\frac{1}{2}(\rho_2-\rho_1)$.
We now want to estimate the $L^\infty$ norm of $\left(\partial_t-\Delta_{\check\omega_\ell^\sharp} \right) \eta$.  Let $\e\in (0,1)$ be a constant to be determined, and denote by $\rho_2''=\rho_2+\e(\rho_2-\rho_1)$. By interpolation and the $C^\infty$ boundedness of $\check g_\ell^\sharp$, we have
\begin{equation}
\begin{split}
&\quad C(\rho_2-\rho_1)^{-2a-\a+2}\left|\left|\left(\partial_t-\Delta_{\check\omega_\ell^\sharp} \right) \eta\right|\right|_{\infty,\check Q_{\check g_\ell(0),\rho_2}(\check p,\check t)}\\
&=C\e^{2a+\a-2}(\rho_2''-\rho_2)^{-2a-\a+2}\left|\left|\left(\partial_t-\Delta_{\check\omega_\ell^\sharp} \right) \eta\right|\right|_{\infty,\check Q_{\check g_\ell(0),\rho_2}(\check p,\check t)}\\
&\leq C\e^{2a+\a-2}(\rho_2''-\rho_2)^{-2a-\a+2}\sum_{p+q=2}\|\partial^{p}\partial_t^q\eta\|_{\infty,\check Q_{\check g_\ell(0),\rho_2}(\check p,\check t)}\\
&\leq C\e^{2a+\a-2} \sum_{p+2q=2a}[\partial^{p}\partial_t^q\eta]_{\a,\a/2,\check Q_{\check g_\ell(0),\rho_2''}(\check p,\check t)}+ C_\e(\rho_2-\rho_1)^{-2a-\a}\| \eta\|_{\infty,\check Q_{\check g_\ell(0),\rho_2''}(\check p,\check t)}.
\end{split}
\end{equation}
Since $a\geq 1$, we can choose $\e$ small enough so that $C\e^{2a+\a-2}\leq \frac12$. Inserting this in \eqref{krkrk} (replacing $\rho_2''$ by $\rho_2$), and plugging into \eqref{eqn:standSchaux} gives
\begin{equation}
\begin{split}
&\quad \sum_{p+2q=2a}[\partial^{p}\partial_t^q\eta]_{\a,\a/2,\check Q_{\check g_\ell(0),\rho_1}(\check p,\check t)}+ \sum_{p+2q\leq 2a}(\rho_2-\rho_1)^{-2a-\a+p+2q}\|\partial^{p}\partial_t^q\eta\|_{\infty,\check Q_{\check g_\ell(0),\rho_1}(\check p,\check t)} \\
&\leq C \sum_{p+2q=2a-2}\left[\partial^{p}\partial_t^q \left(\partial_t-\Delta_{\check\omega_\ell^\sharp} \right) \eta\right]_{\a,\a/2,\check Q_{\check g_\ell(0),\rho_2}(\check p,\check t)} +  C(\rho_2-\rho_1)^{-2a-\a}\| \eta\|_{\infty,\check Q_{\check g_\ell(0),\rho_2}(\check p,\check t)}\\
&\quad +\frac{1}{2}\sum_{p+2q=2a}[\partial^{p}\partial_t^q\eta]_{\a,\a/2,\check Q_{\check g_\ell(0),\rho_2''}(\check p,\check t)},
\end{split}
\end{equation}
and the iteration lemma in \cite[Lemma 2.9]{HT3} then gives
\begin{equation}\label{eqn:1}
\begin{split}
&\quad \sum_{p+2q=2a}[\partial^{p}\partial_t^q\eta]_{\a,\a/2,\check Q_{\check g_\ell(0),\rho_1}(\check p,\check t)}+ \sum_{p+2q\leq 2a}(\rho_2-\rho_1)^{-2a-\a+p+2q}\|\partial^{p}\partial_t^q\eta\|_{\infty,\check Q_{\check g_\ell(0),\rho_1}(\check p,\check t)} \\
&\leq C \sum_{p+2q=2a-2}\left[\partial^{p}\partial_t^q \left(\partial_t-\Delta_{\check\omega_\ell^\sharp} \right) \eta\right]_{\a,\a/2,\check Q_{\check g_\ell(0),\rho_2}(\check p,\check t)} +  C(\rho_2-\rho_1)^{-2a-\a}\| \eta\|_{\infty,\check Q_{\check g_\ell(0),\rho_2}(\check p,\check t)}.
\end{split}
\end{equation}
This is the direct analog of \eqref{eqn:standSchau-3}. After this, the rest of the proof of \eqref{JSB} proceeds exactly as in the proof of Proposition \ref{prop:schauder}.
\end{proof}

\section{The Selection Theorem}\label{s2}
In our main theorem we will need the analog of the Selection Theorem \cite[Theorem 3.11]{HT3}, adapted to our parabolic setting, and to the specific structure of the parabolic complex Monge-Amp\`ere equation that we are dealing with. As in \cite{HT3}, to state this we will need some preparatory notation. First, we fix two natural numbers $0\leq j\leq k$ and a Euclidean ball $B$ centered at the origin, and as usual we have fixed the fiberwise Calabi-Yau volume forms on the fibers $\{z\}\times Y$. Given $t,k$ and two smooth functions $A\in C^\infty(B,\mathbb{R})$ and $G\in C^\infty(B\times Y,\mathbb{R})$ with fiberwise average zero and fiberwise $L^2$ norm $1$, we constructed in \cite[\S 3.2]{HT3} a function $\mathfrak{G}_{t,k}(A,G)\in C^\infty(B\times Y,\mathbb{R})$ also with fiberwise average zero, which is in some appropriate sense an approximate right inverse of $\Delta^{\omega^\natural_t}$ applied to $AG$, see \cite[Lemma 3.7]{HT3} for a precise statement. We will need the following quasi-explicit formula from \cite[Lemma 3.8]{HT3}: given any $A,G$ as above, there are $t$-independent smooth functions $\Phi_{\iota, r}(G)$ on $B\times Y$ such that for all $t\geq 0$ we have
\begin{equation}\label{plop}
\mathfrak{G}_{t,k}(A,G)=\sum_{\iota=0}^{2k}\sum_{r=\lceil \frac{\iota}{2} \rceil}^{k} e^{-rt}(\Phi_{\iota,r}(G)\circledast \D^\iota A),
\end{equation}
where here and in the rest, $\circledast$ denotes some tensorial contraction, and we have
\begin{equation}\label{qsangennaro2}
\Phi_{0,0}(G)=(\Delta^{\omega_F|_{\{\cdot\}\times Y}})^{-1}G.
\end{equation}

For all $i<j$ we suppose that we have smooth functions $G_{i,p,k}\in C^\infty(B\times Y,\mathbb{R}), 1\leq p\leq N_{i,k}$, which have fiberwise average zero and are fiberwise $L^2$ orthonormal, with $G_{i,p,k}=0$ when $i=0$. The main goal is to find smooth functions $G_{j,p,k}$ which satisfy a certain property that we now describe.

We are also given sequences of real numbers $t_\ell\to+\infty$ and $\delta_\ell>0$ with $\delta_\ell\to 0$ and $\lambda_\ell:=\delta_\ell e^{\frac{t_\ell}{2}}\to+\infty$. Consider the diffeomorphisms
\begin{equation}\label{sigmat}
\Sigma_\ell: B_{e^{\frac{t_\ell}{2}}} \times Y\times [-e^{t_\ell}t_\ell,0] \to B \times Y\times [0,t_\ell], \;\, (z,y,t) = \Sigma_\ell(\check{z},\check{y},\check{t}) =(e^{-\frac{t_\ell}{2}}\check{z},\check{y},t_\ell+e^{-t_\ell}\check{t}),
\end{equation}
where $B_R:=B_{\C^m}(0,R)$, and for any  function $u$ on $B\times Y\times [0,t_\ell]$, we will write $\check{u}_\ell=\Sigma_\ell^*u$, and for a time-dependent $2$-form $\alpha$ (with $t\in [0,t_\ell]$) we will write $\check{\alpha}_\ell=e^{t_\ell}\Sigma_\ell^*\alpha$. In particular, note that $\check{\omega}_{\ell, {\rm can}}=e^{t_\ell}\Sigma_\ell^*\omega_{\rm can}$ is a (time-independent) K\"ahler metric on $B_{e^{\frac{t_\ell}{2}}}$ uniformly equivalent to Euclidean (independent of $\ell$).

We will also need to factor $\Sigma_\ell = \Psi_\ell\circ \Xi_\ell$, where
\begin{equation}\label{xi}
\Xi_\ell: B_{e^{\frac{t_\ell}{2}}} \times Y\times [-e^{t_\ell}t_\ell,0] \to B_{\lambda_\ell} \times Y\times [-\lambda_\ell^2t_\ell,0], \;\, (\hat{z},\hat{y},\hat{t}) = \Xi_\ell(\check{z},\check{y},\check{t}) =(\delta_\ell\check{z},\check{y},\delta_\ell^2\check{t}),
\end{equation}
\begin{equation}
\Psi_\ell:  B_{\lambda_\ell} \times Y\times [-\lambda_\ell^2t_\ell,0] \to B \times Y\times [0,t_\ell], \;\, (z,y,t) = \Psi_\ell(\hat{z},\hat{y},\hat{t}) =(\lambda_\ell^{-1}\hat{z},\hat{y},t_\ell+\lambda_\ell^{-2}\hat{t}),
\end{equation}
and given a function $u$ on $B\times Y\times [0,t_\ell]$ we will write $\hat{u}=\Psi_\ell^*u$, and given a time-dependent $2$-form $\alpha$ we will write $\hat{\alpha}=\lambda_\ell^2\Psi_\ell^*\alpha$. We will also use the notation
\begin{equation}
\hat{Q}_R:=B_R\times Y\times [-R^2,0]\ni (\hat{z},\hat{y},\hat{t}), \quad \check{Q}_R:=B_R\times Y\times [-R^2,0]\ni (\check{z},\check{y},\check{t}),
\end{equation}
so that for example $\Xi_\ell(\check{Q}_R)=\hat{Q}_{R\delta_\ell^{-1}}$.

Let also $\eta^\ddagger_\ell$ be an arbitrary sequence of $(1,1)$-forms on $B\times [0, t_\ell]$ with coefficients (spacetime) polynomials of degree at most $2j$ which satisfy $\hat{\eta}^\ddagger_\ell\to 0$ locally smoothly in spacetime (which implies that $\check{\eta}^\ddagger_\ell\to 0$ locally smoothly as well),
let $B^\sharp_\ell$ be an arbitrary sequence of smooth functions on $B\times [0,t_\ell]$ such that $\hat{B}^\sharp_\ell\to 0$ locally smoothly,
and for $1\leq i\leq j$ let $A^\sharp_{i,p,k}$ be arbitrary (spacetime) polynomials of degree at most $2j$ on $B$ such that $\hat{A}^\sharp_{\ell,i,p,k}=\lambda_\ell^2\Psi_\ell^*A^\sharp_{i,p,k}$ satisfy that there is some $0<\alpha_0<1$ such that given any $R>0$ there is $C>0$ with
\begin{equation}\label{linftyyy0}
\|\mathfrak{D}^\iota \hat{A}_{\ell,i,p,k}^\sharp\|_{\infty,\hat Q_{R},\hat{g}_\ell(0)}\leq C \delta_\ell^2 e^{-\alpha_0\frac{t_\ell}{2}},
\end{equation}
for all $0\leq \iota\leq 2j$, or equivalently that $\check{A}^\sharp_{\ell,i,p,k}=e^{t_\ell}\Sigma_\ell^*A^\sharp_{i,p,k}=\delta_\ell^2\Xi_\ell^*\hat{A}^\sharp_{\ell,i,p,k}$ satisfy
\begin{equation}\label{linftyyy}
\|\mathfrak{D}^\iota \check{A}_{\ell,i,p,k}^\sharp\|_{\infty,\check Q_{R\delta_\ell^{-1}},\check g_\ell(0)}\leq C \delta_\ell^\iota e^{-\alpha_0\frac{t_\ell}{2}}.
\end{equation}
With these, we define for $1\leq i\leq j$
\begin{equation}\label{qdugger}
\gamma^{\sharp}_{t,i,k}=\sum_{p=1}^{N_{i,k}}\ddbar\mathfrak{G}_{t,k}(A^\sharp_{i,p,k},G_{i,p,k}),
\end{equation}
so $\gamma^{\sharp}_{t,j,k}$ depends on how we choose the functions $G_{j,p,k}$. It is proved by the argument in \eqref{dager-estimate} below that our assumption \eqref{linftyyy0} on $\hat{A}^\sharp_{\ell,i,p,k}$ implies that for any $R>0$ there is $C>0$ with
\begin{equation}\label{qcrocifisso4checkhatto}
\|\DD^\iota\hat{\gamma}^{\sharp}_{\ell,i,k}\|_{\infty,\hat{Q}_{R},\hat{g}_\ell(0)}\leq C\delta_\ell^{-\iota}e^{-\alpha_0 \frac{t_\ell}{2}},
\end{equation}
or equivalently
\begin{equation}\label{qcrocifisso4check}
\|\DD^\iota\check{\gamma}^{\sharp}_{\ell,i,k}\|_{\infty,\check Q_{R\delta_\ell^{-1}},\check g_\ell(0)}\leq Ce^{-\alpha_0 \frac{t_\ell}{2}},
\end{equation}
for all $\iota\geq 0$ and all $1\leq i\leq j$ (these can be just taken as assumptions for now). Observe that the constants $C$ in \eqref{qcrocifisso4checkhatto} and \eqref{qcrocifisso4check} depend on the choice of the functions $G_{j,p,k}$ but the exponent $\alpha_0$ does not. We also define
\begin{equation}
\eta^\dagger_t=\sum_{i=1}^j\gamma^{\sharp}_{t,i,k},
\end{equation}
and
\begin{equation}
\omega^\sharp_t=(1-e^{-t})\omega_{\rm can}+e^{-t}\omega_F+\eta^\dagger_t+\eta^\ddagger_t,
\end{equation}
which has the property that $\hat{\omega}^\sharp_\ell$ is a K\"ahler metric on $\hat{Q}_R$ for all $R>0$, and $\ell$ sufficiently large, using that $\hat{\eta}^\ddagger_\ell$ is pulled back from $B$ and goes to zero locally uniformly, and the estimate \eqref{qcrocifisso4checkhatto} with $\iota=0$ for $\hat{\eta}^\dagger_\ell$. Passing to the check picture we obtain
\begin{equation}
\check{\omega}^\sharp_\ell=(1-e^{-t_\ell -e^{-t_\ell}\check{t}})\check{\omega}_{\ell, {\rm can}}+e^{-e^{-t_\ell}\check{t}}\Sigma_\ell^*\omega_F+\check{\eta}^\dagger_\ell+\check{\eta}^\ddagger_\ell.
\end{equation}

The key quantity we are interested in is then
\begin{equation}\label{stronzo}
\delta_\ell^{-2j-\alpha}\left(\log \frac{\binom{m+n}{n}\check\omega_{\ell,\mathrm{can}}^m\wedge (\Sigma_\ell^*\omega_F)^n}{(\check\omega_\ell^\sharp)^{m+n}}
+ \sum_{i=1}^j\sum_{p=1}^{N_{i,k}} \check{\mathfrak{G}}_{\check{t},k}( \de_{\check{t}}{\check{A}}_{\ell,i,p,k}^\sharp+e^{-t_\ell}\check{A}_{\ell,i,p,k}^\sharp,\check G_{\ell,i,p,k})
+\Sigma_\ell^*B^\sharp_\ell
\right),
\end{equation}
which can be compared to the corresponding quantity \cite[(3.47)]{HT3} in the elliptic setting. Observe that by definition we have $\Sigma_\ell^*B^\sharp_\ell=\Xi_\ell^*\hat{B}^\sharp_\ell$.

To clarify, when we will apply the Selection Theorem later, the functions $B^\sharp_\ell$ will be defined by
\begin{equation}
\hat{B}^\sharp_\ell=e^{-\lambda_\ell^{-2}\hat{t}}\de_{\hat{t}}\underline{{\hat{\chi}}_\ell^\sharp}-n\lambda_\ell^{-2}\hat{t},
\end{equation}
where $\de_{\hat{t}}\underline{{\hat{\chi}}_\ell^\sharp}$ is a spacetime polynomial of degree at most $2j$.
The fact that $\hat{B}^\sharp_\ell\to 0$ locally smoothly will follow from \eqref{bdd-phi-sharp-underline}. We will also later define
\begin{equation}
\underline{\check{\chi}}_\ell^\sharp=\delta_\ell^{-2}\Xi_\ell^*(\underline{{\hat{\chi}}_\ell^\sharp}),
\end{equation}
hence we will have
\begin{equation}
\Sigma_\ell^*B^\sharp_\ell=\Xi_\ell^*\hat{B}^\sharp_\ell=e^{-e^{-t_\ell}\check t}\partial_{\check t}\underline{\check\chi_\ell^\sharp}-ne^{-t_\ell}\check{t}.
\end{equation}

Given these preliminaries, the following is then the key result:

\begin{thm}[Selection Theorem]\label{thm:Selection} Suppose we are given $0\leq j\leq k$ and when $j>1$, we are also given smooth function $G_{i,p,k},1\leq i\leq j-1,1\leq p\leq N_{i,k}$, on $B\times Y$ which are fiberwise $L^2$ orthonormal  and have fiberwise average zero. Then there are a concentric ball $B'=B_{\mathbb{C}^m}(0,r)\subset B$ and smooth functions  $G_{j,p,k},1\leq p\leq N_{j,k}$ on $B'\times Y$ (identically zero if $j=0$), with fiberwise average zero so that $G_{i,p,k},1\leq p\leq N_{i,k},1\leq i\leq j,$ are all fiberwise $L^2$ orthonormal with the following property: if $\delta_\ell,t_\ell>0$ are any sequences with $t_\ell\to+\infty,$ $\delta_\ell\to 0$ and $\delta_\ell e^{\frac{t_\ell}{2}}\to +\infty,$ and if $A_{i,p,k}^\sharp,B^\sharp_\ell,\eta_\ell^\dagger,\eta_\ell^\ddagger,\omega_\ell^\sharp$ are as above, and if \eqref{stronzo}
converges locally uniformly on $\mathbb{C}^m\times Y\times (-\infty,0]$ to some limiting function $\mathcal{F}$ as $\ell\to\infty$, then on $\Sigma_\ell^{-1}(B'\times Y\times [-r^2,0])$, we can write \eqref{stronzo} as
\begin{equation}\label{cuccuruzzu}
\delta_\ell^{-2j-\a}\Sigma_\ell^*\left( f_{\ell,0}+\sum_{i=1}^j \sum_{p=1}^{N_{i,k}}f_{\ell,i,p} G_{i,p,k}\right)+o(1),
\end{equation}
where $f_{\ell,0}, f_{\ell,i,p}$ are functions pulled back from $B'\times (-r^2,0]$ such that $\hat{f}_{\ell, 0}=\Psi_\ell^*f_{\ell,0},\hat{f}_{\ell,i,p}=\Psi_\ell^*f_{\ell,i,p}$  converge locally smoothly to zero, and $o(1)$ is a term that converges locally smoothly to zero. Lastly, \eqref{stronzo} converges to $\mathcal{F}$ locally smoothly.
\end{thm}
\begin{rmk}
The argument follows closely the proof of the Selection Theorem 3.11 in \cite{HT3}, but apart from the obvious change from space to space-time, there are some other differences that we now briefly discuss, which arise from the different structure of the parabolic complex Monge-Amp\`ere equation that we have compared to its elliptic counterpart in \cite{HT3}. The first term in \eqref{stronzo} is reminescent of \cite[(3.47)]{HT3}, but it now has a logarithm. This change will be quite immaterial, since $\log(x)\approx x-1$ for $x\approx 1$. Next, compared to \cite[(3.47)]{HT3}, the quantity in \eqref{stronzo} also contains two more pieces. The last term with $B^\sharp_\ell$ is trivially acceptable since it can be absorbed into $f_{\ell,0}$, while the term involving $\check{\mathfrak{G}}_{\check{t},k}$ will have to be dealt with, and it will turn out to be ``lower order'' compared to the first term in \eqref{stronzo}. Putting these all together will allow us to follow the proof of \cite[Theorem 3.11]{HT3} very closely.
\end{rmk}
\begin{proof}
The proof is by induction on $j$. First, we treat the base case $j=0$. In this case, by definition, there are no obstruction functions and the quantity in \eqref{stronzo} reduces to
\begin{equation}\label{stronzo2}
\delta_\ell^{-\alpha}\left(\log \frac{\binom{m+n}{n}\check\omega_{\ell,\mathrm{can}}^m\wedge (\Sigma_\ell^*\omega_F)^n}{(\check\omega_\ell^\sharp)^{m+n}}
+\Sigma_\ell^*B^\sharp_\ell
\right).
\end{equation}
As in \cite{HT3}, we introduce the notation
\begin{equation}\label{imprecise}
\Sigma_\ell^*\omega_F=(\Sigma_\ell^*\omega_F)_{\mathbf{bb}}+(\Sigma_\ell^*\omega_F)_{\mathbf{bf}}+(\Sigma_\ell^*\omega_F)_{\mathbf{ff}}=:
e^{-t_\ell}\check\omega_{F,\mathbf{bb}}+e^{-\frac{t_\ell}{2}}\check\omega_{F,\mathbf{bf}}+\check\omega_{F,\mathbf{ff}},
\end{equation}
where the functions $\check\omega_{F,\mathbf{bb}},\check\omega_{F,\mathbf{bf}},\check\omega_{F,\mathbf{ff}}$ so defined are uniformly bounded on $B_{e^{\frac{t_\ell}{2}}}\times Y\times [-e^{t_\ell}t_\ell,0]$.
Following \cite[(3.50)]{HT3}, we then compute
\begin{equation}\begin{split}
(\check\omega_\ell^\sharp)^{m+n}&=((1-e^{-t_\ell -e^{-t_\ell}\check{t}})\check{\omega}_{\ell, {\rm can}}+e^{-e^{-t_\ell}\check{t}}\Sigma_\ell^*\omega_F+\check{\eta}^\ddagger_\ell)^{m+n}\\
&=(\check{\omega}_{\ell, {\rm can}}+\check\omega_{F,\mathbf{ff}}+\check{\eta}^\ddagger_\ell)^{m+n}+O(e^{-{t_\ell}})\\
&=\binom{m+n}{n}(\check{\omega}_{\ell, {\rm can}}+\check{\eta}^\ddagger_\ell)^{m}(\Sigma_\ell^*\omega_F)_{\mathbf{ff}}^n+O(e^{-{t_\ell}}),
\end{split}\end{equation}
where the $O(e^{-t_\ell})$ is in the locally smooth topology, and so
\begin{equation}\label{checkko}\begin{split}
\log \frac{(\check\omega_\ell^\sharp)^{m+n}}{\binom{m+n}{n}\check\omega_{\ell,\mathrm{can}}^m\wedge (\Sigma_\ell^*\omega_F)^n}&=\log\frac{(\check{\omega}_{\ell, {\rm can}}+\check{\eta}^\ddagger_\ell)^{m}}{\check\omega_{\ell,\mathrm{can}}^m}+O(e^{-t_\ell})\\
&=\Sigma_\ell^*f_\ell+ O(e^{-{t_\ell}}),
\end{split}\end{equation}
where $f_\ell$ is some sequence of smooth functions pulled back from $B$. Passing to the hat picture, i.e. letting $\hat{f}_\ell=\Psi_\ell^*f_\ell$, our assumption that $\hat{\eta}^\ddagger_\ell\to 0$ locally smoothly implies that $\hat{f}_\ell\to 0$ locally smoothly. It thus follows that
\begin{equation}
\log \frac{\binom{m+n}{n}\check\omega_{\ell,\mathrm{can}}^m\wedge (\Sigma_\ell^*\omega_F)^n}{(\check\omega_\ell^\sharp)^{m+n}}+\Sigma_\ell^*B^\sharp_\ell=\Sigma_\ell^*(B^\sharp_\ell-f_\ell)+O(e^{-{t_\ell}}),
\end{equation}
and recalling that $\hat{f}_\ell\to 0$ and $\hat{B}^\sharp_\ell\to 0$ locally smoothly, as well as $\delta_\ell^{-\alpha}e^{-{t_\ell}}=o(1)$, we see that \eqref{cuccuruzzu} holds. Lastly, since by assumption \eqref{stronzo2} converges locally uniformly to $\mathcal{F}$, the same is true for $\delta_\ell^{-\alpha}\Sigma_\ell^*(B^\sharp_\ell-f_\ell)$, and the same argument as in \cite[(3.52)--(3.54)]{HT3} shows that this convergence is locally smooth. This concludes the proof of Theorem \ref{thm:Selection} in the case $j=0$.

We then treat the inductive step, and assume $j\geq 1$. By assumption, the obstruction functions $G_{i,p,k}$ with $1\leq i<j$ have already been selected on $B'\times Y$ (recall that $B'=B_r$), and we need to select the $G_{j,p,k}$'s. As in \cite{HT3}, this will be done via an iterative procedure, with iteration parameter $\kappa$, initially set at $\kappa=0$, and at each step assuming we have already selected some obstruction functions $G^{[q]}_{j,p,k}, 1\leq q\leq\kappa$ (this being the empty list when $\kappa=0$), and with the iterative step consisting of selecting some new obstruction functions $G^{[\kappa+1]}_{j,p,k}$ to add to these. After this will be achieved, we will then show that if we perform this iterative step $\ov{\kappa}$ times (for some uniform $\ov{\kappa}$) and define the obstruction function $G_{j,p,k}$ by putting together all the $G^{[q]}_{j,p,k}$'s obtained at all iterations $1\leq q\leq\ov{\kappa}+1$, then the desired conclusion \eqref{cuccuruzzu} holds.

To start the proof, we give a couple of definitions following \cite{HT3}.
We will say that a sequence of functions on $\Sigma_\ell^{-1}(B_r\times Y\times (-r^2,0])$ satisfies condition $(\star)$ if it equals
\begin{equation}\label{qstarr}
\delta_\ell^{-2j-\alpha}\Sigma_\ell^*\left(f_{\ell,0}+\sum_{i=1}^N f_{\ell,i}h_i\right) + o(1),
\end{equation}
for some $N\in\mathbb{N}$, where the functions $f_{\ell,0},f_{\ell,i}$ are smooth and pulled back from $B_r\times (-r^2,0]$ and $\hat{f}_{\ell,0}=\Psi_\ell^*f_{\ell,0}, \hat{f}_{\ell,i}=\Psi_\ell^*f_{\ell,i}$ converge locally smoothly to zero, the time-independent functions $h_i$ are smooth on $B_r\times Y$ with fiberwise average zero, and the $o(1)$ is a term that converges locally smoothly to zero. This definition is tailored to our desired conclusion in \eqref{cuccuruzzu}. As in \cite[Remark 3.13]{HT3}, we see that if a sequence of functions satisfies $(\star)$ and converges locally uniformly to some limit, then this convergence is actually smooth.

For $\kappa\geq 0$, given also arbitrary spacetime polynomials $A^{\sharp,[q]}_{j,p,k},1\leq q\leq\kappa,$ of degree at most $2j$ such that $\check{A}^{\sharp,[q]}_{\ell,j,p,k}=e^{t_\ell}\Sigma_\ell^*A^{\sharp,[q]}_{j,p,k}$ satisfy \eqref{linftyyy}, we construct as in \eqref{qdugger}
\begin{equation}\label{qdugger2}
\gamma^{\sharp,[q]}_{t,j,k}=\sum_{p=1}^{N_{j,k}^{[q]}}\ddbar\mathfrak{G}_{t,k}(A^{\sharp,[q]}_{j,p,k},G_{j,p,k}),
\end{equation}
for $1\leq q\leq \kappa$ (setting $\gamma^{\sharp,[0]}_{t,j,k}=0$) and let
\begin{equation}
\check{\omega}^{\sharp,[\kappa]}_\ell=(1-e^{-t_\ell -e^{-t_\ell}\check{t}})\check{\omega}_{\ell, {\rm can}}+e^{-e^{-t_\ell}\check{t}}\Sigma_\ell^*\omega_F+\sum_{r=1}^{j-1}\check{\gamma}^{\sharp}_{\ell,r,k}+
\sum_{q=1}^\kappa\check{\gamma}^{\sharp,[q]}_{\ell,j,k}+\check{\eta}^\ddagger_\ell.
\end{equation}
This is a K\"ahler metric on $B_r\times Y\times (-r^2,0]$, and we can then consider the function
\begin{equation}\label{qkulprit}\begin{split}
\mathcal{B}_\ell^{[\kappa]}:=\delta_\ell^{-2j-\alpha}\bigg(&\log \frac{\binom{m+n}{n}\check\omega_{\ell,\mathrm{can}}^m\wedge (\Sigma_\ell^*\omega_F)^n}{(\check\omega_\ell^{\sharp,[\kappa]})^{m+n}}
+ \sum_{i=1}^{j-1}\sum_{p=1}^{N_{i,k}} \check{\mathfrak{G}}_{\check{t},k}( \de_{\check{t}}{\check{A}}_{\ell,i,p,k}^\sharp+e^{-t_\ell}\check{A}_{\ell,i,p,k}^\sharp,\check G_{\ell,i,p,k})\\
&+\sum_{q=1}^{\kappa}\sum_{p=1}^{N_{j,k}^{[q]}} \check{\mathfrak{G}}_{\check{t},k}( \de_{\check{t}}{\check{A}}^{\sharp,[q]}_{\ell,j,p,k}+e^{-t_\ell}\check{A}^{\sharp,[q]}_{\ell,j,p,k},\check G_{\ell,j,p,k})
\bigg).
\end{split}\end{equation}
The following is the analog of \cite[Lemma 3.14]{HT3}:
\begin{lma}\label{qsolomon}
Suppose either $\kappa=0$ or $\kappa\geq 1$ and we have selected the functions $G_{j,p,k}^{[q]}$ as above for $1\leq q\leq \kappa.$ Then the sequence of functions $\mathcal{B}_\ell^{[\kappa]}$ satisfies $(\star)$. Furthermore, we have
\begin{equation}\label{qwulprit}
\delta_\ell^{2j+\alpha}\mathcal{B}_\ell^{[\kappa]}=O(e^{-\alpha_0\frac{t_\ell}{2}})+o(1)_{\rm from\ base},
\end{equation}
where the term $O(e^{-\alpha_0\frac{t_\ell}{2}})$ is in $L^\infty_{\rm loc}$, while the last term is a function from the base which goes to zero locally smoothly.
\end{lma}
\begin{proof}
For ease of notation, define
\begin{equation}
\omega^\square_t=(1-e^{-t})\omega_{\rm can}+e^{-t}\omega_F+\eta^\ddagger_t,
\end{equation}
which have the property that $\hat{\omega}^\square_\ell$ are K\"ahler metrics on $\hat{Q}_R$ for all $R>0$ and $\ell$ large,
and which in the check picture become
\begin{equation}
\check{\omega}^\square_\ell=(1-e^{-t_\ell -e^{-t_\ell}\check{t}})\check{\omega}_{\ell, {\rm can}}+e^{-e^{-t_\ell}\check{t}}\Sigma_\ell^*\omega_F+\check{\eta}^\ddagger_\ell.
\end{equation}
We first consider (in the original undecorated picture)
\begin{equation}\begin{split}
(\omega^\square_t)^{m+n}&=((1-e^{-t})\omega_{\rm can}+e^{-t}\omega_F+\eta^\ddagger_t)^{m+n}\\
&=\binom{m+n}{n}e^{-nt}((1-e^{-t})\omega_{\rm can}+e^{-t}\omega_{F,\mathbf{bb}}+\eta^\ddagger_t)^m\wedge\omega_{F,\mathbf{ff}}^n\\
&+\sum_{q=1}^me^{-(n+q)t}\binom{m+n}{n+q}((1-e^{-t})\omega_{\rm can}+e^{-t}\omega_{F,\mathbf{bb}}+\eta^\ddagger_t)^{m-q}\wedge(\omega_{F,\mathbf{ff}}+\omega_{F,\mathbf{bf}})^{n+q}\\
&=:\binom{m+n}{n}e^{-nt}\left(((1-e^{-t})\omega_{\rm can}+e^{-t}\omega_{F,\mathbf{bb}}+\eta^\ddagger_t)^m\wedge\omega_{F,\mathbf{ff}}^n+e^{-t}\mathcal{D}_t\right),
\end{split}
\end{equation}
and so
\begin{equation}\begin{split}
&\frac{(\omega^\square_t)^{m+n}}{\binom{m+n}{n}\omega_{\rm can}^m\wedge (e^{-t}\omega_F)^n}\\
&=\frac{((1-e^{-t})\omega_{\rm can}+e^{-t}\omega_{F,\mathbf{bb}}+\eta^\ddagger_t)^m\wedge\omega_{F,\mathbf{ff}}^n+e^{-t}\mathcal{D}_t}{\omega_{\rm can}^m\wedge \omega_F^n}\\
&=(1-e^{-t})^m+\sum_{0<p+q\leq m}\frac{m!}{p!q!(m-p-q)!}\frac{\omega_{\rm can}^pe^{-qt}\omega^q_{F,\mathbf{bb}}(\eta_t^\ddagger)^{m-p-q}}{\omega_{\rm can}^m}
+\frac{e^{-t}\mathcal{D}_t}{\omega_{\rm can}^m\wedge \omega_F^n}.
\end{split}\end{equation}
Note that the functions
\begin{equation}\label{qculprit}
\frac{\omega_{\rm can}^{p}e^{-qt}\omega_{F,\mathbf{bb}}^q(\eta^\ddagger_t)^{m-p-q}}{\omega_{\rm can}^m},
\end{equation}
with $q=0$ are $o(1)$ and pulled back from $B$, while when $q>0$ they are not pulled back from $B$, but they are visibly of the form
$f_{t,0}+\sum_{i=1}^N f_{t,i}h_i$ with the same notation as above, where the $\hat{f}_{t,0},\hat{f}_{t,i}$ converge smoothly to zero at least as $O(e^{-qt})$, and the functions $h_i$ have fiberwise average zero and do not depend on the choice of $\eta^\ddagger_t$. An analogous statement holds for
\begin{equation}
\frac{e^{-t}\mathcal{D}_t}{\omega_{\rm can}^m\wedge \omega_F^n},
\end{equation}
and if we take the logarithm, and use the Taylor expansion of $\log(1+x)$, this shows that
\begin{equation}
\log\frac{(\omega^\square_t)^{m+n}}{\binom{m+n}{n}\omega_{\rm can}^m\wedge (e^{-t}\omega_F)^n}=f_{t,0}+\sum_{i=1}^N f_{t,i}h_i+O(e^{-(j+1)t}),
\end{equation}
for some possibly different functions $f_{t,0}, f_{t,i}$, and where the term $O(e^{-(j+1)t})$ is in $C^p_{\rm loc}$ for all $p\geq 0$. Passing to the check picture we have
\begin{equation}
\log\frac{(\check{\omega}^\square_\ell)^{m+n}}{\binom{m+n}{n}\check\omega_{\ell,\mathrm{can}}^m\wedge (\Sigma_\ell^*\omega_F)^n}=-ne^{-t_\ell}\check{t}+\Sigma_\ell^*\log\frac{((1-e^{-t})\omega_{\rm can}+e^{-t}\omega_F+\eta^\ddagger_t)^{m+n}}{\binom{m+n}{n}\omega_{\rm can}^m\wedge (e^{-t}\omega_F)^n},
\end{equation}
and so
\begin{equation}\label{intermedia}
\delta_\ell^{-2j-\alpha}\log\frac{(\check{\omega}^\square_\ell)^{m+n}}{\binom{m+n}{n}\check\omega_{\ell,\mathrm{can}}^m\wedge (\Sigma_\ell^*\omega_F)^n}=\delta_\ell^{-2j-\alpha}\Sigma_\ell^*\left(f_{\ell,0}+\sum_{i=1}^N f_{\ell,i}h_i\right)+O(\delta_\ell^{-2j-\alpha}e^{-(j+1)t_\ell}),
\end{equation}
and $\delta_\ell^{-2j-\alpha}e^{-(j+1)t_\ell}=o(1)$ by assumption, namely the LHS of \eqref{intermedia} satisfies $(\star)$, as well as
\begin{equation}
\log\frac{(\check{\omega}^\square_\ell)^{m+n}}{\binom{m+n}{n}\check\omega_{\ell,\mathrm{can}}^m\wedge (\Sigma_\ell^*\omega_F)^n}=O(e^{-t_\ell})+o(1)_{\rm from\ base}.
\end{equation}

The next step is to show that the quantity
\begin{equation}\label{qqulprit}\begin{split}
\delta_\ell^{-2j-\alpha}\log \frac{(\check\omega_\ell^{\sharp,[\kappa]})^{m+n}}{\binom{m+n}{n}\check\omega_{\ell,\mathrm{can}}^m\wedge (\Sigma_\ell^*\omega_F)^n}
=
\delta_\ell^{-2j-\alpha} \log\frac{(\check{\omega}^\square_\ell)^{m+n}}{\binom{m+n}{n}\check\omega_{\ell,\mathrm{can}}^m\wedge (\Sigma_\ell^*\omega_F)^n}
+\delta_\ell^{-2j-\alpha}\log \frac{(\check\omega_\ell^{\sharp,[\kappa]})^{m+n}}{(\check{\omega}^\square_\ell)^{m+n}}
\end{split}
\end{equation}
also satisfies $(\star)$. We have just discussed the first piece, and the second piece equals
\begin{equation}\label{qwerylong2}\begin{split}
\delta_\ell^{-2j-\alpha}\log \left(1+\sum_{i=1}^{m+n}\binom{m+n}{i}\frac{(\check{\gamma}^\sharp_{\ell,1,k}+\cdots+\check{\gamma}^\sharp_{\ell,j-1,k}+
\check{\gamma}^{\sharp,[1]}_{\ell,j,k}+\cdots+\check{\gamma}^{\sharp,[\kappa]}_{\ell,j,k})^i\wedge(\check{\omega}^\square_\ell)^{m+n-i}}{(\check{\omega}^\square_\ell)^{m+n}}\right),
\end{split}
\end{equation}
and for the terms with $\check{\gamma}^\sharp_{\ell,1,k}+\cdots+\check{\gamma}^\sharp_{\ell,j-1,k}+\check{\gamma}^{\sharp,[1]}_{\ell,j,k}+\cdots+\check{\gamma}^{\sharp,[\kappa]}_{\ell,j,k}$ we recall from \eqref{qdugger} and \eqref{qdugger2} that we have
\begin{equation}
\check{\gamma}^{\sharp}_{\ell,i,k}=\sum_{p=1}^{N_{i,k}}\ddbar\check{\mathfrak{G}}_{\check{t},k}(\check{A}^\sharp_{\ell,i,p,k},\check{G}_{i,p,k}),\quad
\check{\gamma}^{\sharp,[q]}_{\ell,j,k}=\sum_{p=1}^{N_{j,k}^{[q]}}\ddbar\check{\mathfrak{G}}_{\check{t},k}(\check{A}^{\sharp,[q]}_{\ell,j,p,k},\check{G}_{j,p,k}),
\end{equation}
with the bounds \eqref{qcrocifisso4check}, where the approximate Green operator $\check{\mathfrak{G}}_{\check{t},k}$ is given schematically by
\begin{equation}\label{blobus}
\check{\mathfrak{G}}_{\check{t},k}(\check{A},\check{G}) = \sum_{\iota=0}^{j}\sum_{r=\lceil \frac{\iota}{2} \rceil}^{k} e^{-\left(r-\frac{\iota}{2}\right)t_\ell}e^{-re^{-t_\ell}\check{t}}(\check{\Phi}_{\iota,r}(\check{G})\circledast \D^\iota \check{A}).
\end{equation}
by \eqref{plop}. Plugging this into \eqref{qwerylong2} and arguing as we did above reveals that the quantity in \eqref{qqulprit} satisfies $(\star)$ and that furthermore
\begin{equation}
\log \frac{(\check\omega_\ell^{\sharp,[\kappa]})^{m+n}}{\binom{m+n}{n}\check\omega_{\ell,\mathrm{can}}^m\wedge (\Sigma_\ell^*\omega_F)^n}=O(e^{-\alpha_0\frac{t_\ell}{2}})+o(1)_{\rm from\ base}.
\end{equation}
Lastly, to prove that $\mathcal{B}_\ell^{[\kappa]}$ satisfies $(\star)$, it remains to consider the piece
\begin{equation}\label{qqulprit2}
\delta_\ell^{-2j-\alpha}\left(\sum_{i=1}^{j-1}\sum_{p=1}^{N_{i,k}} \check{\mathfrak{G}}_{\check{t},k}( \de_{\check{t}}{\check{A}}_{\ell,i,p,k}^\sharp+e^{-t_\ell}\check{A}_{\ell,i,p,k}^\sharp,\check G_{\ell,i,p,k})
+\sum_{q=1}^{\kappa}\sum_{p=1}^{N_{j,k}^{[q]}} \check{\mathfrak{G}}_{\check{t},k}( \de_{\check{t}}{\check{A}}^{\sharp,[q]}_{\ell,j,p,k}+e^{-t_\ell}\check{A}^{\sharp,[q]}_{\ell,j,p,k},\check G_{\ell,j,p,k})\right).
\end{equation}
The fact that this term satisfies $(\star)$ again follows immediately from \eqref{blobus} together with \eqref{linftyyy}, which also give that
\begin{equation}\begin{split}
&\sum_{i=1}^{j-1}\sum_{p=1}^{N_{i,k}} \check{\mathfrak{G}}_{\check{t},k}( \de_{\check{t}}{\check{A}}_{\ell,i,p,k}^\sharp+e^{-t_\ell}\check{A}_{\ell,i,p,k}^\sharp,\check G_{\ell,i,p,k})
+\sum_{q=1}^{\kappa}\sum_{p=1}^{N_{j,k}^{[q]}} \check{\mathfrak{G}}_{\check{t},k}( \de_{\check{t}}{\check{A}}^{\sharp,[q]}_{\ell,j,p,k}+e^{-t_\ell}\check{A}^{\sharp,[q]}_{\ell,j,p,k},\check G_{\ell,j,p,k})\\
&=
O(e^{-\alpha_0\frac{t_\ell}{2}})+o(1)_{\rm from\ base},
\end{split}\end{equation}
and this completes the proof that $\mathcal{B}_\ell^{[\kappa]}$ satisfies $(\star)$ and that \eqref{qwulprit} holds.
\end{proof}
Now that Lemma \ref{qsolomon} is established, we can start the first step of the iteration, when $\kappa=0$ and we need to select the obstruction functions $G^{[1]}_{j,p,k}$. To do this, we consider $\mathcal{B}_\ell^{[0]}$, which by Lemma \ref{qsolomon} satisfies $(\star)$, and let $\{h_i\}$ be the corresponding functions in its expansion \eqref{qstarr}. Applying the approximate fiberwise Gram-Schmidt \cite[Proposition 3.1]{HT3} to the functions $h_i$ together with the $G_{i,p,k}, 1\leq i<j$, produces our desired list $G^{[1]}_{j,p,k}$ (on $B_r\times Y$, up to shrinking $r$), so that we may assume that the functions $h_i$ in \eqref{qstarr} lie in the fiberwise linear span of the $G_{i,p,k}, 1\leq i<j$ together with the $G^{[1]}_{j,p,k}$. This completes the first step ($\kappa=0$).

Next, we consider a subsequent step $\kappa\geq 1$ of the iteration, so we assume we are given the lists $G_{i,p,k}, 1\leq i<j$ and $G^{[q]}_{j,p,k}, 1\leq q\leq \kappa$, hence we have the function $\mathcal{B}_\ell^{[\kappa]}$ in \eqref{qkulprit}, and we want to construct the obstruction functions $G^{[\kappa+1]}_{j,p,k}$. In order to do this, we must compare $\mathcal{B}_\ell^{[\kappa]}$ and
$\mathcal{B}_\ell^{[\kappa-1]}$. We have
\begin{equation}\label{kebello}\begin{split}
\delta_\ell^{2j+\alpha}(\mathcal{B}_\ell^{[\kappa]}-\mathcal{B}_\ell^{[\kappa-1]})&=-\log\frac{(\check\omega_\ell^{\sharp,[\kappa]})^{m+n}}{(\check\omega_\ell^{\sharp,[\kappa-1]})^{m+n}}
+\sum_{p=1}^{N_{j,k}^{[\kappa]}} \check{\mathfrak{G}}_{\check{t},k}( \de_{\check{t}}{\check{A}}^{\sharp,[\kappa]}_{\ell,j,p,k}+e^{-t_\ell}\check{A}^{\sharp,[\kappa]}_{\ell,j,p,k},\check G_{\ell,j,p,k})\\
&=-\log\left(1+\sum_{i=1}^{m+n}\binom{m+n}{i}\frac{(\check{\gamma}^{\sharp,[\kappa]}_{\ell,j,k})^i\wedge(\check\omega_\ell^{\sharp,[\kappa-1]})^{m+n-i}}{(\check\omega_\ell^{\sharp,[\kappa-1]})^{m+n}}\right)\\
&+\sum_{p=1}^{N_{j,k}^{[\kappa]}} \check{\mathfrak{G}}_{\check{t},k}( \de_{\check{t}}{\check{A}}^{\sharp,[\kappa]}_{\ell,j,p,k}+e^{-t_\ell}\check{A}^{\sharp,[\kappa]}_{\ell,j,p,k},\check G_{\ell,j,p,k}).
\end{split}\end{equation}
As in \cite[(3.75)--(3.76)]{HT3}, for $1\leq i\leq m+n$ we have
\begin{equation}
\frac{(\check{\gamma}^{\sharp,[\kappa]}_{\ell,j,k})^i\wedge(\check\omega_\ell^{\sharp,[\kappa-1]})^{m+n-i}}{(\check\omega_\ell^{\sharp,[\kappa-1]})^{m+n}}=
\frac{(\check{\gamma}^{\sharp,[\kappa]}_{\ell,j,k})^i\wedge(\check{\omega}_{\ell, {\rm can}}+(\Sigma_\ell^*\omega_F)_{\mathbf{ff}})^{m+n-i}}{\binom{m+n}{n}\check{\omega}_{\ell, {\rm can}}^m\wedge(\Sigma_\ell^*\omega_F)^n_{\mathbf{ff}}}
(1+O(e^{-\alpha_0\frac{t_\ell}{2}})+o(1)_{\rm from\ base}),
\end{equation}
where the $O(\cdot), o(\cdot)$ are in the locally smooth topology. Using \eqref{blobus}, we can write (for any $1\leq q\leq\kappa$)
\begin{equation}\label{blob}
\check{\gamma}^{\sharp,[q]}_{\ell,j,k} = \sum_{p=1}^{N^{[q]}_{j,k}}\sum_{\iota=0}^{j}\sum_{r=\lceil \frac{\iota}{2} \rceil}^{k} e^{-\left(r-\frac{\iota}{2}\right) t_\ell}e^{-re^{-t_\ell}\check{t}}i\partial\overline\partial(\check{\Phi}_{\iota,r}(\check{G}^{[q]}_{j,p,k})\circledast \D^\iota \check{A}^{\sharp,[q]}_{\ell,j,p,k}).
\end{equation}
As in \cite[(3.80)]{HT3} we decompose \eqref{blob} schematically as the sum of $6$ pieces
\begin{equation}\label{blobbe}\begin{split}
\check{\gamma}^{\sharp,[q]}_{\ell,j,k}&= \sum_{p=1}^{N^{[q]}_{j,k}}\sum_{\iota=0}^{j}\sum_{r=\lceil \frac{\iota}{2} \rceil}^{k} e^{-\left(r-\frac{\iota}{2}\right) t_\ell}e^{-re^{-t_\ell}\check{t}}\bigg\{(i\partial\overline\partial\check{\Phi}_{\iota,r}(\check{G}^{[q]}_{j,p,k}))_{\mathbf{ff}}\circledast \D^\iota \check{A}^{\sharp,[q]}_{\ell,j,p,k}
+
(i\partial\overline\partial\check{\Phi}_{\iota,r}(\check{G}^{[q]}_{j,p,k}))_{\mathbf{bf}}\circledast \D^\iota \check{A}^{\sharp,[q]}_{\ell,j,p,k}\\
&+
(i\partial\overline\partial\check{\Phi}_{\iota,r}(\check{G}^{[q]}_{j,p,k}))_{\mathbf{bb}}\circledast \D^\iota \check{A}^{\sharp,[q]}_{\ell,j,p,k}
+(i\partial\check{\Phi}_{\iota,r}(\check{G}^{[q]}_{j,p,k}))_{\mathbf{b}}\circledast \db\D^\iota \check{A}^{\sharp,[q]}_{\ell,j,p,k}
+(i\partial\check{\Phi}_{\iota,r}(\check{G}^{[q]}_{j,p,k}))_{\mathbf{f}}\circledast \db\D^\iota\check{A}^{\sharp,[q]}_{\ell,j,p,k}\\
&+\check{\Phi}_{\iota,r}(\check{G}^{[q]}_{j,p,k})\circledast i\de\db\D^\iota \check{A}^{\sharp,[q]}_{\ell,j,p,k}\bigg\}=:\sum_{\iota=0}^{j}\sum_{r=\lceil \frac{\iota}{2} \rceil}^{k} \left({\rm I}^{[q]}_{\iota,r}+\cdots+{\rm VI}^{[q]}_{\iota,r}\right),
\end{split}\end{equation}
(which depend on $\ell,j,k$, but we omit this from the notation for simplicity). Observe here that for all $(\check{z},\check{y},\check{t})\in \check{Q}_{R\delta_\ell^{-1}}$ we have
\begin{equation}\label{goooo2}
-R^2 e^{-t_\ell}\delta_\ell^{-2}\leq e^{-t_\ell}\check{t}\leq 0,
\end{equation}
where, by assumption, $e^{-t_\ell}\delta_\ell^{-2}\to 0$, so the term $e^{-re^{-t_\ell}\check{t}}$ in \eqref{blobbe} is $1+o(1)$. Now, as in \cite[(3.85)--(3.91)]{HT3}, we see that ${\rm I}^{[q]}_{0,0}$ is the dominant term, in the sense that
\begin{equation}\label{zatan}
\|\bullet^{[q]}_{\iota,r}\|_{\infty,\check{Q}_{R\delta_\ell^{-1}},\check{g}_\ell(0)}\leq C\delta_\ell\|{\rm I}^{[q]}_{0,0}\|_{\infty,\check{Q}_{R\delta_\ell^{-1}},\check{g}_\ell(0)},
\end{equation}
whenever $\bullet\neq {\rm I}$ or $(\iota,r)\neq (0,0)$. This together with \eqref{qcrocifisso4check} imply that
\begin{equation}\label{qzulprit}\begin{split}
&\sum_{i=1}^{m+n}\binom{m+n}{i}\frac{(\check{\gamma}^{\sharp,[q]}_{\ell,j,k})^i\wedge(\check\omega_\ell^{\sharp,[q-1]})^{m+n-i}}{(\check\omega_\ell^{\sharp,[q-1]})^{m+n}}\\
&=\frac{(m+n)}{\binom{m+n}{n}}\frac{{\rm I}^{[q]}_{0,0}\wedge(\check{\omega}_{\ell, {\rm can}}+(\Sigma_\ell^*\omega_F)_{\mathbf{ff}})^{m+n-1}}{\check{\omega}_{\ell, {\rm can}}^m\wedge(\Sigma_\ell^*\omega_F)^n_{\mathbf{ff}}}
(1+O(e^{-\alpha_0\frac{t_\ell}{2}})+o(1)_{\rm from\ base})+F^{[q]}_\ell,
\end{split}
\end{equation}
where given any $R>0$ there is a $C>0$ such that for all $\ell$
\begin{equation}\label{qzulprit2}
\|F^{[q]}_\ell\|_{\infty,\check{Q}_{R\delta_\ell^{-1}}}\leq C\delta_\ell^{\alpha_0}\|I^{[q]}_{0,0}\|_{\infty,\check{Q}_{R\delta_\ell^{-1}},\check{g}_\ell(0)}.
\end{equation}
Relation \eqref{qsangennaro2} shows that
\begin{equation}\label{qzulprit3}
\frac{(m+n)}{\binom{m+n}{n}}\frac{{\rm I}^{[q]}_{0,0}\wedge(\check{\omega}_{\ell, {\rm can}}+(\Sigma_\ell^*\omega_F)_{\mathbf{ff}})^{m+n-1}}{\check{\omega}_{\ell, {\rm can}}^m\wedge(\Sigma_\ell^*\omega_F)^n_{\mathbf{ff}}}=\sum_{p=1}^{N^{[q]}_{j,k}}\check{G}^{[q]}_{j,p,k}\check{A}^{\sharp,[q]}_{\ell,j,p,k},
\end{equation}
while from the definition of $I^{[q]}_{0,0}$ and \eqref{linftyyy} we have
\begin{equation}\label{qzulprit4}
\|I^{[q]}_{0,0}\|_{\infty,\check{Q}_{R\delta_\ell^{-1}},\check{g}_\ell(0)}\leq C \sum_{p=1}^{N^{[q]}_{j,k}}\|\check{A}^{\sharp,[q]}_{\ell,j,p,k}\|_{\infty,\check{Q}_{R\delta_\ell^{-1}}}\leq C\delta_\ell^{\alpha_0},
\end{equation}
while the argument in \cite[(3.83)]{HT3} gives the reverse bound
\begin{equation}\label{qzulprit5}
\sum_{p=1}^{N^{[q]}_{j,k}}\|\check{A}^{\sharp,[q]}_{\ell,j,p,k}\|_{\infty,\check{Q}_{R\delta_\ell^{-1}}}\leq C\|I^{[q]}_{0,0}\|_{\infty,\check{Q}_{R\delta_\ell^{-1}},\check{g}_\ell(0)},
\end{equation}
and \eqref{qzulprit}, \eqref{qzulprit2}, \eqref{qzulprit3}, \eqref{qzulprit4} imply in particular that
\begin{equation}\label{qqulprit3}
\left\|\sum_{i=1}^{m+n}\binom{m+n}{i}\frac{(\check{\gamma}^{\sharp,[q]}_{\ell,j,k})^i\wedge(\check\omega_\ell^{\sharp,[q-1]})^{m+n-i}}{(\check\omega_\ell^{\sharp,[q-1]})^{m+n}}\right\|_{\infty,\check{Q}_{R\delta_\ell^{-1}}}\leq  C\|I^{[q]}_{0,0}\|_{\infty,\check{Q}_{R\delta_\ell^{-1}},\check{g}_\ell(0)}\leq C\delta_\ell^{\alpha_0}.
\end{equation}
From \eqref{qzulprit}, \eqref{qzulprit2}, \eqref{qzulprit3}, \eqref{qqulprit3} and the Taylor expansion of $\log(1+x)$ we see that
\begin{equation}\label{qqulprit4}
\log\left(1+\sum_{i=1}^{m+n}\binom{m+n}{i}\frac{(\check{\gamma}^{\sharp,[\kappa]}_{\ell,j,k})^i\wedge(\check\omega_\ell^{\sharp,[\kappa-1]})^{m+n-i}}{(\check\omega_\ell^{\sharp,[\kappa-1]})^{m+n}}\right)=
\left(\sum_{p=1}^{N^{[q]}_{j,k}}\check{G}^{[q]}_{j,p,k}\check{A}^{\sharp,[q]}_{\ell,j,p,k}\right)(1+o(1)_{\rm from\ base})+F^{[q]}_\ell,
\end{equation}
where $F^{[q]}_\ell$ satisfies \eqref{qzulprit2}.

This deals with the term on the second line of \eqref{kebello}. As for the last line, we can use \eqref{blobus} to expand
\begin{equation}\label{zattan}
\check{\mathfrak{G}}_{\check{t},k}( \de_{\check{t}}{\check{A}}^{\sharp,[\kappa]}_{\ell,j,p,k}+e^{-t_\ell}\check{A}^{\sharp,[\kappa]}_{\ell,j,p,k},\check G_{\ell,j,p,k})=
\sum_{\iota=0}^{j}\sum_{r=\lceil \frac{\iota}{2} \rceil}^{k} e^{-\left(r-\frac{\iota}{2}\right) t_\ell}e^{-re^{-t_\ell}\check{t}}(\check{\Phi}_{\iota,r}(\check{G}_{\ell,j,p,k})\circledast \D^\iota (\de_{\check{t}}{\check{A}}^{\sharp,[\kappa]}_{\ell,j,p,k}+e^{-t_\ell}\check{A}^{\sharp,[\kappa]}_{\ell,j,p,k})),
\end{equation}
and applying the obvious parabolic extension of \cite[Lemma 3.10]{HT3} to balls of radius $R\delta_\ell^{-1}$ gives for $\iota\geq 0$,
\begin{equation}
\|\DD^\iota \check{A}^{\sharp,[\kappa]}_{\ell,j,p,k}\|_{\infty,\check{Q}_{R\delta_\ell^{-1}}}\leq C\delta_\ell^{\iota}\| \check{A}^{\sharp,[\kappa]}_{\ell,j,p,k}\|_{\infty,\check{Q}_{R\delta_\ell^{-1}}},
\end{equation}
and so
\begin{equation}\label{goooo}
\|\D^\iota (\de_{\check{t}}{\check{A}}^{\sharp,[\kappa]}_{\ell,j,p,k}+e^{-t_\ell}\check{A}^{\sharp,[\kappa]}_{\ell,j,p,k})\|_{\infty,\check{Q}_{R\delta_\ell^{-1}}}\leq
C\delta_\ell^2 \|\DD^\iota \check{A}^{\sharp,[\kappa]}_{\ell,j,p,k}\|_{\infty,\check{Q}_{R\delta_\ell^{-1}}}\leq C\delta_\ell^{\iota+2}\| \check{A}^{\sharp,[\kappa]}_{\ell,j,p,k}\|_{\infty,\check{Q}_{R\delta_\ell^{-1}}}.
\end{equation}
Inserting \eqref{goooo} into \eqref{zattan}, and using also \eqref{qzulprit5} and the fact that  $e^{-re^{-t_\ell}\check{t}}=1+o(1)$ from \eqref{goooo2}, we see that
\begin{equation}
\left\|\check{\mathfrak{G}}_{\check{t},k}( \de_{\check{t}}{\check{A}}^{\sharp,[\kappa]}_{\ell,j,p,k}+e^{-t_\ell}\check{A}^{\sharp,[\kappa]}_{\ell,j,p,k},\check G_{\ell,j,p,k})\right\|_{\infty,\check{Q}_{R\delta_\ell^{-1}}}\leq C\delta_\ell^2 \|I^{[q]}_{0,0}\|_{\infty,\check{Q}_{R\delta_\ell^{-1}},\check{g}_\ell(0)},
\end{equation}
and combining this with \eqref{kebello} and \eqref{qqulprit4} we finally obtain
\begin{equation}\label{zatttan}
\mathcal{B}_\ell^{[\kappa]}=\mathcal{B}_\ell^{[\kappa-1]}-\delta_\ell^{-2j-\alpha}\left(\sum_{p=1}^{N^{[\kappa]}_{j,k}}\check{G}^{[\kappa]}_{j,p,k}\check{A}^{\sharp,[\kappa]}_{\ell,j,p,k}\right)(1+o(1)_{\rm from\ base})+\delta_\ell^{-2j-\alpha}E^{[\kappa]}_\ell,
\end{equation}
where $E^{[\kappa]}_\ell$ satisfies
\begin{equation}\label{qzulprit2a}
\|E^{[\kappa]}_\ell\|_{\infty,\check{Q}_{R\delta_\ell^{-1}}}\leq C\delta_\ell^{\alpha_0}\|I^{[\kappa]}_{0,0}\|_{\infty,\check{Q}_{R\delta_\ell^{-1}},\check{g}_\ell(0)}.
\end{equation}
Now, from Lemma \ref{qsolomon} we know that both $\mathcal{B}_\ell^{[\kappa]}$ and $\mathcal{B}_\ell^{[\kappa-1]}$ satisfy $(\star)$, and from \eqref{zatttan} we see that so does
$\delta_\ell^{-2j-\alpha}E^{[\kappa]}_\ell$, and so it has an expansion of the form \eqref{qstarr}. As in \cite{HT3}, we apply the approximate fiberwise Gram-Schmidt \cite[Proposition 3.1]{HT3} to the functions $h_i$ together with the $G_{i,p,k}, 1\leq i<j$ and $G^{[q]}_{j,p,k}, 1\leq q\leq\kappa$, produces our desired list $G^{[\kappa+1]}_{j,p,k}$ (on $B_r\times Y$, up to shrinking $r$), so that we may assume that the functions $h_i$ in \eqref{qstarr} lie in the fiberwise linear span of the $G_{i,p,k}, 1\leq i<j$ together with the $G^{[q]}_{j,p,k}, 1\leq q\leq \kappa+1$. This completes the step from $\kappa$ to $\kappa+1$ in our iterative procedure.

Iterating \eqref{zatttan} shows that for every $\kappa\geq 1$ we have
\begin{equation}\label{ratttan}
\mathcal{B}_\ell^{[\kappa]}=\mathcal{B}_\ell^{[0]}-\delta_\ell^{-2j-\alpha}\left(\sum_{q=1}^\kappa\sum_{p=1}^{N^{[q]}_{j,k}}\check{G}^{[q]}_{j,p,k}\check{A}^{\sharp,[q]}_{\ell,j,p,k}\right)
(1+o(1)_{\rm from\ base})+\delta_\ell^{-2j-\alpha}\sum_{q=1}^\kappa E^{[q]}_\ell,
\end{equation}
with
\begin{equation}\label{ratttan2}
\|E^{[q]}_\ell\|_{\infty,\check{Q}_{R\delta_\ell^{-1}}}\leq C\delta_\ell^{\alpha_0}\|I^{[q]}_{0,0}\|_{\infty,\check{Q}_{R\delta_\ell^{-1}},\check{g}_\ell(0)},
\end{equation}
for $1\leq q\leq \kappa$, and also
\begin{equation}\label{ratttan3}
\|\mathcal{B}_\ell^{[0]}\|_{\infty,\check{Q}_{R\delta_\ell^{-1}}}\leq C\delta_\ell^{-2j-\alpha},
\end{equation}
which follows immediately from \eqref{qwulprit}.

We can now repeat the iterative step $\ov{\kappa}:=\lceil\frac{2j+\alpha}{\alpha_0}\rceil$ and then we stop, so the last set of functions which are added to the list are the $G_{j,p,k}^{[\ov{\kappa}+1]}$. Our choice of $\ov{\kappa}$ is made so that
\begin{equation}\label{ratttan4}
\delta_\ell^{-2j-\alpha}\delta_\ell^{(\ov{\kappa}+1)\alpha_0}\to 0.
\end{equation}
The resulting $G_{j,p,k}^{[q]}$ with $1\leq q\leq\ov{\kappa}+1$ are then renamed simply $G_{j,p,k}$. These, together with the $G_{i,p,k}, 1\leq i<j$, are the obstruction functions that we seek. It remains to show that the statement of the Selection Theorem \ref{thm:Selection} holds with this choice of obstruction functions. By definition, the quantity in \eqref{stronzo} equals $\mathcal{B}_\ell^{[\ov{\kappa}+1]}$ (up to the term with $\Sigma_\ell^*B^\sharp_\ell$, which we can ignore since it can be absorbed into $f_{\ell,0}$ in \eqref{cuccuruzzu}). We know that $\mathcal{B}_\ell^{[\ov{\kappa}+1]}$ satisfies $(\star)$ thanks to Lemma \ref{qsolomon}. As mentioned earlier, because of this we know that if it converges locally uniformly then it converges locally smoothly, which is the last claim in the Selection Theorem \ref{thm:Selection}. The last thing to prove is that if $\mathcal{B}_\ell^{[\ov{\kappa}+1]}$ converges locally uniformly, then \eqref{cuccuruzzu} holds, and thanks to \eqref{ratttan} (with $\kappa=\ov{\kappa}+1$) and to our choice of obstruction functions, it suffices to show that $\delta_\ell^{-2j-\alpha}E^{[\ov{\kappa}+1]}_\ell$ is $o(1)$ in the locally smooth topology. Since this term satisfies $(\star)$ (as mentioned earlier), it suffices to show that it is $o(1)$ in the $L^\infty_{\rm loc}$ topology, and this follows from \eqref{ratttan}, \eqref{ratttan2}, \eqref{ratttan3}, \eqref{ratttan4} and our main assumption that $\mathcal{B}_\ell^{[\ov{\kappa}+1]}=O(1)$ in $L^\infty_{\rm loc}$, by using the same iteration argument as \cite[(3.101)--(3.112)]{HT3}. This completes the proof of Theorem \ref{thm:Selection}.
\end{proof}

\section{Asymptotic expansion}\label{sekt}
In this section, we will  prove our main technical result, Theorem \ref{thm:decomposition}, which gives an asymptotic expansion for the metrics $\omega^\bullet(t)$ which evolve under the K\"ahler-Ricci flow \eqref{kkrf}.
Recall that $\omega^\bullet(t)=\omega^\natural(t)+\ddbar\vp(t)$, where the potentials $\vp(t)$ solve the parabolic complex Monge-Amp\`ere equation \eqref{ma}, which we can write as
\begin{equation}\label{ma2}
(\omega^\bullet(t))^{m+n}=(\omega^\natural(t)+\ddbar\vp(t))^{m+n}=e^{\vp(t)+\dot{\vp}(t)-nt}\binom{m+n}{n}\omega_{\rm can}^m\wedge\omega_F^n.
\end{equation}

\subsection{Known estimates}
First, let us recall a few of the known estimates for the \KR flow \eqref{kkrf} and its equivalent formulation \eqref{ma}. There are many other facts that are known about this flow (see e.g. \cite[\S 5]{To} or \cite[\S 7]{To2} for overviews), but the following are the only ones that we will need:
\begin{lma}\label{lma:earlier work}
Assume the setup in Section \ref{set}. Then there exists $C>0$ such that on $B\times Y\times [0,+\infty)$ we have
\begin{enumerate}\setlength{\itemsep}{1mm}
\item[(i)] $C^{-1}\omega^\natural(t)\leq \omega^\bullet(t)\leq C\omega^\natural(t)$,
\item[(ii)]$|\vp(t)|+|\dot\varphi(t)|\to 0$ as $t\to+\infty$,
\item[(iii)]$|\omega^\bullet(t)- \omega_{\mathrm{can}}|_{g^\bullet(t)}\to 0$ as $t\to+\infty,$
\item[(iv)] $|R({g^\bullet(t)})|\leq C$,
\item[(v)] $|\dot{\vp}(t)+\ddot{\vp}(t)|\leq C$,
\item[(vi)] $|\nabla(\vp(t)+\dot{\vp}(t))|_{g^\bullet(t)}\leq C.$
\end{enumerate}
\end{lma}
\begin{proof}
Item (i) is proved in \cite{FZ} (and is an adaptation of \cite{To4}, see also \cite{ST} for the case $m=n=1$). Item (ii) is proved in \cite[Lemma 3.1]{TWY}, and item (iii) in \cite[Theorem 1.2]{TWY} (see especially the very end of its proof on p.685).
Item (iv) is the main theorem of \cite{ST3}, and this implies (v) thanks to the relation \cite[p.345]{To}
\begin{equation}
\dot{\vp}(t)+\ddot{\vp}(t)=-R(g^\bullet(t))-m.
\end{equation}
To prove (vi), we use \cite[Proposition 3.1]{ST3} which gives
\begin{equation}\label{allokko}
\left|\nabla\log \frac{e^{nt}(\omega^\bullet(t))^{m+n}}{\Omega}\right|_{g^\bullet(t)}\leq C,
\end{equation}
where $\Omega$ is a smooth positive volume form on $X$ such that $-\ddbar\log\Omega=-\omega_B$ is the pullback of a K\"ahler form on $B$. On the other hand by \cite[Proposition 5.9]{To} we have $-\ddbar\log(\omega_{\rm can}^m\wedge\omega_F^n)=-\omega_{\rm can},$ so on $B\times Y$ we have
\begin{equation}
\ddbar\log \frac{\binom{m+n}{n}\omega_{\rm can}^m\wedge\omega_F^n}{\Omega}=-\omega_B+\omega_{\rm can},
\end{equation}
which is a $(1,1)$-form pulled back from $B$, hence the logarithm of the ratio of these two volume forms restricted to the fibers $\{z\}\times Y$ is pluriharmonic, hence constant. Thus,
\begin{equation}
\binom{m+n}{n}\omega_{\rm can}^m\wedge\omega_F^n=e^G\Omega,
\end{equation}
for some smooth function $G$ on $B$, and so using \eqref{ma2} and \eqref{allokko} we get
\begin{equation}
|\nabla(\vp(t)+\dot{\vp}(t))|_{g^\bullet(t)}=\left|\nabla\left(\log \frac{e^{nt}(\omega^\bullet(t))^{m+n}}{\Omega}-G\right) \right|_{g^\bullet(t)}\leq C+|\nabla G|_{g^\bullet(t)}\leq C,
\end{equation}
as desired.
\end{proof}

\subsection{Statement of the asymptotic expansion}
Given any $j\in \mathbb{N}, 0\leq 2j\leq k$ and $z\in B$,  during the course of the proof of Theorem \ref{thm:decomposition} below we will work by induction on $j$. By applying repeatedly the Selection Theorem \ref{thm:Selection}, and consequently shrinking our ball at each step, we will obtain in particular a collection $G_{i,p,k}, 1\leq i\leq j, 1\leq p\leq N_{i,k}$ of smooth function on $B\times Y$ with fiberwise average zero, which are fiberwise $L^2$ orthonormal. For each such $G_{i,p,k}$ and $t>0$, as in \cite[(3.6)]{HT3} we define $P_{t,i,p,k}=P_{t,G_{i,p,k}}$ where
\begin{equation}\label{defn:P}
P_{t,H}(\a)= n(\mathrm{pr}_{B})_* (H \a \wedge \omega_F^{n-1})+e^{-t} \tr_{\omega_{\mathrm{can}}}(\mathrm{pr}_{B})_*(H \a \wedge \omega_F^n),
\end{equation}
for any $(1,1)$ form $\a$ on $B\times Y$ and $H$ with $\int_{\{z\}\times Y}H\omega_F^n=0$ for all $z\in \mathbb{C}^m$. Throughout the proof, we will fix a reference shrinking product metric $g(t)=g_{\mathbb{C}^m}+e^{-t} g_{Y,z_0}$. It will only be used to measure the norms and distance but not the connection, and thus the exact choice of $z_0$ is unimportant thanks to \eqref{equiv-product-ref} (we will usually take $z_0=0$). We will also need the $t$-dependent approximate Green operator $\mathfrak{G}_{t,k}$ defined in \cite[\S 3.2]{HT3}, to which we refer for its basic properties.

\begin{thm}\label{thm:decomposition} For all $j,k\in \mathbb{N},0\leq 2j\leq k$, $z\in B$, there exists $B'=B_{\mathbb{C}^m}(z,R)\Subset B$ and
functions $G_{i,p,k}, 1\leq i\leq j, 1\leq p\leq N_{i,k}$ as above, such that on $B'\times Y$ we have a decomposition
\begin{equation}\label{expa}
\omega^\bullet(t)=\omega^\natural(t)+\gamma_0(t)+\gamma_{1,k}(t)+\dots+\gamma_{j,k}(t)+\eta_{j,k}(t),
\end{equation}
with the following properties. For all $\a\in (0,1)$ and $r<R$, there is $C>0$ such that for all $t\geq 0$,
\begin{equation}\label{infty1}
\|\mathfrak{D}^\iota \eta_{j,k}\|_{\infty,Q_r(z,t),g(t)}
\leq Ce^{\frac{\iota-2j-\a}{2}t}\,\;\;\text{for all}\;\; 0\leq \iota\leq 2j,\\
\end{equation}
\begin{equation}\label{holder1}
[\mathfrak{D}^{2j}\eta_{j,k}]_{\a,\a/2, Q_r(z,t),g(t)}\leq C,
\end{equation}
where $Q_r(z,t)=\left(B_{\mathbb{C}^m}(z,r)\times Y\right)\times [t-r^2,t]$.
Furthermore, we have
\begin{equation}
\gamma_0(t)=\ddbar \underline{\varphi},\;\;\; \gamma_{i,k}(t)=\sum_{p=1}^{N_{i,k}}\ddbar \mathfrak{G}_{t,k}\left(A_{i,p,k}(t),G_{i,p,k} \right),
\end{equation}
for $1\leq i\leq j$ where $A_{i,p,k}(t)=P_{t,i,p,k}(\eta_{i-1,k}(t))$ are functions from the base, and we have
\begin{equation}\label{infty2}
\|\mathfrak{D}^\iota\gamma_{0}\|_{\infty, Q_r(z,t),g_{\mathbb{C}^m}} = o(1)\,\;\;\text{for all}\;\; 0\leq \iota\leq 2j,
\end{equation}
\begin{equation}\label{holder2}
\quad [\mathfrak{D}^{2j}\gamma_{0}]_{\a,\a/2, Q_r(z,t),g_{\mathbb{C}^m}} \leq C,
\end{equation}
\begin{equation}\label{infty5}
\|\mathfrak{D}^\iota(\de_t\underline{\vp}+\underline{\vp})\|_{\infty, Q_r(z,t),g_{\mathbb{C}^m}} = o(1)\,\;\;\text{for all}\;\; 0\leq \iota\leq 2j,
\end{equation}
\begin{equation}\label{holder5}
\quad [\mathfrak{D}^{2j}(\de_t\underline{\vp}+\underline{\vp})]_{\a,\a/2, Q_r(z,t),g_{\mathbb{C}^m}} \leq C,
\end{equation}
\begin{equation}\label{infty3}
\|\mathfrak{D}^\iota A_{i,p,k}\|_{\infty,Q_r(z,t),g_{\C^m}}\leq Ce^{-(2i+\a)(1-\frac{\iota}{2j+2+\a})\frac{t}2},\;\;\text{for all} \;\; 0\leq \iota\leq 2j+2,1\leq i \leq j,1\leq p\leq N_{i,k},
\end{equation}
\begin{equation}\label{infty4}
\|\mathfrak{D}^{2j+2+\iota} A_{i,p,k}\|_{\infty,Q_r(z,t),g_{\C^m}}\leq Ce^{\left(-\frac{\a(2i+\a)}{\iota+\a}(1-\frac{2j+2}{2j+2+\a})+\frac{\iota^2}{\iota+\a}\right) \frac{t}{2}},\,\;\;\text{for all}\;\; 0\leq \iota\leq 2k,\;\, 1\leq i \leq j,\;\,1\leq p\leq N_{i,k},
\end{equation}
\begin{equation}\label{holder3}
\begin{split}
\sup_{(x,s),(x',s')\in Q_r(z,t)}\sum_{i=1}^{j} \sum_{p=1}^{N_{i,k}}\sum_{\iota = -2}^{2k} &e^{-\iota\frac{t}{2}}\Bigg(\frac{|\mathfrak{D}^{2j+2+\iota} A_{i,p,k}(x,s)-\mathbb{P}_{x'x}(\mathfrak{D}^{2j+2+\iota}A_{i,p,k}(x',s'))|_{g(t)}}{ \left( d^{g(t)}(x,x') +|s-s'|^{\frac{1}{2}} \right)^\alpha}\Bigg)\leq C.
\end{split}
\end{equation}
\end{thm}
\begin{rmk}
A key difference between Theorem \ref{thm:decomposition} and \cite[Theorem 4.1]{HT3} are the estimates in \eqref{infty5}, \eqref{holder5}. These will be crucial for us in the proof, to deal with the term $\de_t\vp+\vp$ in the complex Monge-Amp\`ere equation \eqref{ma2}. Another difference is that the bounds in \eqref{infty3} are worse than those in \cite[(4.12)]{HT3}, due to the fact that in this paper we can only consider even order H\"older norms.
\end{rmk}

\subsection{Setup of induction scheme}

We start with a given $z\in B$. For any given $k\in \mathbb{N}$, we prove the Theorem by induction on $j$. We treat both the base case and induction case together, although they will have to be considered separately at certain steps of the proof.   Given $k$ with $0\leq 2j\leq k$, if $j>0$ we assume Theorem \ref{thm:decomposition} holds at the $(j-1)$-th step, so there exists $B_{\mathbb{C}^m}(z,r)\subset B$ such that we already have the decomposition of $\omega^\bullet(t)$ at the $(j-1)$-th step satisfying the desired estimates on $B_{\mathbb{C}^m}(z,r)\times Y\times [0,+\infty)$.  We aim to refine the decomposition at the $j$-th step as well as define it for $j=0$.

As mentioned in the Introduction, we can write $\omega^\bullet=\omega^\natural+\ddbar \varphi$. When $j=0$, we take $\gamma_0=\ddbar\underline{\varphi}$ and $\eta_{0,k}=\ddbar (\varphi-\underline{\varphi})$ so that $\omega^\bullet=\omega^\natural+\gamma_0+\eta_{0,k}$. If $j\geq 1$, suppose we already have the decomposition
\begin{equation}
\omega^\bullet=\omega^\natural+\gamma_0+\gamma_{1,k}+\dots+\gamma_{j-1,k}+\eta_{j-1,k},
\end{equation}
on $\left(B_{\mathbb{C}^m}(z,r)\times Y\right)\times [0,+\infty)$. We further decompose $\eta_{j-1,k}$ into $\gamma_{j,k}+\eta_{j,k}$ as follows. When $j>1$, up to shrinking $r>0$ we can assume that we already have selected smooth functions $G_{i,p,k},1\leq i\leq j-1, 1\leq p\leq N_{i,k}$ on $B_{\mathbb{C}^m}(z,r)\times Y$, which are fiberwise $L^2$ orthonormal and have fiberwise average zero. When $j\geq 1$, we then apply the Selection Theorem \ref{thm:Selection} which up to shrinking $r$ further, gives us a list of functions $G_{j,p,k}, 1\leq p\leq N_{j,k}$ on $B_{\mathbb{C}^m}(z,r)\times Y$,  which are fiberwise $L^2$ orthonormal and have fiberwise average zero, so that the conclusion of the Selection Theorem \ref{thm:Selection} holds for the  collection $G_{i,p,k}, 1\leq p\leq N_{i,k},1\leq i\leq j$. With this collection of function, we define
\begin{equation}
A_{j,p,k}(t):=P_{t,j,p,k}(\eta_{j-1,k}(t)),
\end{equation}
where $P$ is given by \eqref{defn:P} and
\begin{equation}
\gamma_{j,k}:=\sum_{p=1}^{N_{j,k}}\ddbar \mathfrak{G}_{t,k}(A_{j,p,k},G_{j,p,k}),
\end{equation}
where $\mathfrak{G}_{t,k}$ is defined in \cite[\S 3.2]{HT3}.
Finally, we define $\eta_{j,k}:=\eta_{j-1,k}-\gamma_{j,k}$ so that \begin{equation}\label{gorgulu}
\omega^\bullet=\omega^\natural+\gamma_0+\gamma_{1,k}+\dots+\gamma_{j,k}+\eta_{j,k},
\end{equation}
on $B_{\mathbb{C}^m}(z,r)\times Y\times [0,+\infty)$. For ease of notation, by scaling and translation of our coordinates, we may assume without loss that we have this decomposition  $B_{\mathbb{C}^m}(z,r)=B_{\mathbb{C}^m}(1)=B$.

\subsubsection{The base case of the induction $j=0$}\label{burzumm}
The base case of the induction, where $j=0$, needs to be treated separately, and although the overall scheme of proof is the same as when $j\geq 1$, there will be some crucial differences.

First of all, let us examine the estimates that we need to prove in order to establish Theorem \ref{thm:decomposition} for $j=0$. The estimates \eqref{infty3}, \eqref{infty4}, \eqref{holder3} are vacuous by definition. By Lemma \ref{lma:earlier work} (iii) we have that $\|\ddbar\vp\|_{\infty, B\times Y\times [t-1,t],g(t)}=o(1)$ as $t\to +\infty$, and the fiber integration argument in \cite[p.436]{To4} then gives $\|\gamma_0\|_{\infty, B\times [t-1,t],g_{\mathbb{C}^m}}=o(1)$ as well, which implies \eqref{infty2}. Similarly, Lemma \ref{lma:earlier work} (ii) implies that $\|\de_t\vp+\vp\|_{\infty, B\times Y\times [t-1,t]}=o(1)$, and taking fiberwise average this easily implies that $\|\de_t\underline{\vp}+\underline{\vp}\|_{\infty, B\times [t-1,t]}=o(1)$ too, which implies \eqref{infty5}.

Next, using the bounds $\|\partial_{t}^2\vp+\partial_{t}\vp\|_{\infty, B\times Y\times [t-1,t]}\leq C$ and $\|\nabla(\partial_{t}\vp+\vp)\|_{\infty, B\times Y\times [t-1,t], g(t)}\leq C$ from Lemma \ref{lma:earlier work} (v), (vi), which together with the $L^\infty$ bound for $\de_t\vp+\vp$  imply the same bounds for the fiber average
\begin{equation}
\|\partial_{t}^2\underline{\vp}+\partial_{t}\underline{\vp}\|_{\infty, B\times [t-1,t]}\leq C,\quad \|\nabla(\partial_{t}\underline{\vp}+\underline{\vp})\|_{\infty, B\times [t-1,t],g_{\C^m}}\leq C,
\end{equation}
we can bound for any $x,x'\in B$ and $t\geq 0$ and $s\in [t-1,t]$,
\begin{equation}
|(\partial_{t}\underline{\vp}+\underline{\vp})(x,t)-(\partial_{t}\underline{\vp}+\underline{\vp})(x',s)|\leq C(|x-x'|+|t-s|)\leq C(|x-x'|+|t-s|^{\frac{1}{2}})^\a,
\end{equation}
which gives
\begin{equation}
[\partial_{t}\underline{\vp}+\underline{\vp}]_{\a,\a/2,B\times [t-1,t],g_{\C^m}}\leq C,
\end{equation}
which implies \eqref{holder5}.

Thus, when $j=0$ it suffices to establish \eqref{infty1}, \eqref{holder1} and \eqref{holder2}. The final claim we will need is that if we suppose we have proved that for all $t\geq 0$ we have
\begin{equation}\label{agogn}
[\ddbar\vp]_{\a,\a/2, Q_r(z,t),g(t)}\leq C,
\end{equation}
where $Q_r(z,t)$ is as in the statement of Theorem \ref{thm:decomposition}, then the estimates \eqref{infty1}, \eqref{holder1} and \eqref{holder2} will all hold. To prove this claim, we use the following ``non-cancellation'' inequality
\begin{equation}\label{noncanzel}
[\ddbar\underline{\vp}]_{\a,\a/2,B_{\C^m}(z,r)\times[t-r^2,t],g_{\C^m}}\leq C[\ddbar\vp]_{\a,\a/2, \textrm{base}, Q_r(z,t),g_X}+C\|\ddbar\vp\|_{\infty, Q_r(z,t),g_X},
\end{equation}
which is straightforward to prove using \cite[(4.215)]{HT3}, except that here there is no stretching involved. Plugging \eqref{agogn} and Lemma \ref{lma:earlier work} (iii) into \eqref{noncanzel} gives
\begin{equation}\label{akogn}
[\ddbar\underline{\vp}]_{\a,\a/2,B_{\C^m}(z,r)\times[t-r^2,t],g_{\C^m}}\leq C,
\end{equation}
which is exactly \eqref{holder2}, and recalling that $\eta_{0,k}=\ddbar\vp-\ddbar\underline{\vp}$, we can use \eqref{agogn}, \eqref{akogn}, the triangle inequality and the boundedness of $\P$ to estimate
\begin{equation}
[\eta_{0,k}]_{\a,\a/2, Q_r(z,t),g(t)}\leq C,
\end{equation}
which proves \eqref{holder1}, and lastly \eqref{infty1}  follows from this and Proposition \ref{prop:braindead}, using that the potential $\vp-\underline{\vp}$ of $\eta_{0,k}$ has fiberwise average zero. This completes the proof of the claim, and shows that in order to establish Theorem \ref{thm:decomposition} for $j=0$ it suffices to prove the single estimate \eqref{agogn}.

\subsubsection{Estimates from induction hypothesis}

Suppose $j\geq 1$ and the conclusion holds at the $(j-1)$-th step.  We first observe that by \cite[(4.16)]{HT3}, the operator $P_{t,j,p,k}$ satisfies
\begin{equation}\label{induction-bdd-A-reason}
\|P_{t,j,p,k}(\a)\|_{\infty,B}\leq Ce^{-t} \|\a\|_{\infty,B\times Y,g(t)},
\end{equation}
for any $(1,1)$ form $\a$ on the total space and $t\geq 0$. In particular, we can put $\a=\eta_{i-1,k}(t)$ for $t\geq 0$ and $1\leq i\leq j$ and use also \eqref{infty1} to see that
\begin{equation}\label{induction-bdd-A}
\|A_{i,p,k}\|_{\infty,B\times [t-1,t]}
\leq C e^{-t}\|\eta_{i-1,k}\|_{\infty,B\times [t-1,t]}\leq C_\b e^{-(i+\b/2)t }.
\end{equation}
for all $\b\in (0,1)$, $1\leq i\leq j$, $1\leq p\leq N_{i,k}$ and $t\geq 0$.

\subsubsection{Reduction to estimating the H\"older seminorms, when $j\geq 1$}
Suppose again that $j\geq 1$ and fix a real number $\a\in (0,1)$.  We first show that  \eqref{infty1}, \eqref{infty2}, \eqref{infty5}, \eqref{infty3} and \eqref{infty4} on $B_{\mathbb{C}^m}(\rho)\times Y\times [0,+\infty)$ (for some $\rho<1$), would follow immediately once we establish the H\"older seminorm bounds \eqref{holder1}, \eqref{holder2}, \eqref{holder5} and \eqref{holder3} on a slightly larger domain.

We first address \eqref{infty1} and \eqref{infty2}. Since the potential of $\eta_{j,k}$ has fiberwise average zero,  \eqref{infty1}  follows directly from Proposition \ref{prop:braindead} and \eqref{holder1}. Next, as in section \ref{burzumm}, we observe that the estimate in Lemma \ref{lma:earlier work} (iii) implies that $\|\gamma_0\|_{\infty, B\times [t-1,t],g_{\mathbb{C}^m}}=o(1)$. Then, \eqref{infty2} follows by interpolating between this and \eqref{holder2} using Proposition \ref{prop:interpolation}.
Similarly, the estimate in Lemma \ref{lma:earlier work} (ii) implies that $\|\de_t\underline{\vp}+\underline{\vp}\|_{\infty, B\times [t-1,t]}=o(1)$, and interpolating between this and \eqref{holder5} via Proposition \ref{prop:interpolation}, we obtain \eqref{infty5}.
The remaining task is to show \eqref{infty3} and \eqref{infty4}.  By \eqref{holder3},  \eqref{induction-bdd-A}, we can interpolate from $Q_\rho(0,t)$ to $Q_R(0,t)$ ($\rho<R<1)$ using Proposition \ref{prop:interpolation}, and get
\begin{equation}
\begin{split}
\sum_{\iota=1}^{2j+2}(R-\rho)^{\iota} \| \mathfrak{D}^\iota A_{i,p,k} \|_{\infty,  Q_\rho(0,t) ,g(t)}&\leq C_k(R-\rho)^{2j+2+\a}+Ce^{-(i+\a/2)t}.
\end{split}
\end{equation}
By choosing
\begin{equation}
R-\rho\approx e^{-\frac{2i+\a}{2j+2+\a}\frac{t}{2}},
\end{equation}
(which is small) we see that for each $1\leq \iota\leq 2j+2$,
\begin{equation}
\| \mathfrak{D}^\iota A_{i,p,k} \|_{\infty,  Q_\rho(0,t),g(t)}\leq Ce^{-(2i+\a)(1-\frac{\iota}{2j+2+\a})\frac{t}2},
\end{equation}
which is \eqref{infty3}.
Finally for \eqref{infty4},  by interpolating \eqref{infty3} (with $\iota=2j+2$) with \eqref{holder3} using Proposition \ref{prop:interpolation}, we obtain
\begin{equation}
(R-\rho)^\iota \|\mathfrak{D}^{2j+2+\iota} A_{i,p,k}\|_{\infty,  Q_\rho(0,t),g(t)}\leq C(R-\rho)^{\iota+\a}e^{\iota \frac{t}{2}}+C e^{-(2i+\a)(1-\frac{\iota}{2j+2+\a})\frac{t}2}.
\end{equation}
By choosing
\begin{equation}
R-\rho\approx  e^{-\left((2i+\a)(1-\frac{\iota}{2j+2+\a})+\iota\right)\frac{t}{2(\iota+\a)}},
\end{equation}
we arrive at
\begin{equation}
\|\mathfrak{D}^{2j+2+\iota} A_{i,p,k}\|_{\infty,  Q_\rho(0,t),g(t)}\leq Ce^{\left(-\frac{\a(2i+\a)}{\iota+\a}(1-\frac{2j+2}{2j+2+\a})+\frac{\iota^2}{\iota+\a}\right) \frac{t}{2}}.
\end{equation}
This shows \eqref{infty4}. Thus, to prove Theorem \ref{thm:decomposition} it suffices to prove \eqref{agogn} when $j=0$, and to prove  \eqref{holder1}, \eqref{holder2}, \eqref{holder5} and \eqref{holder3} when $j\geq 1$.

\subsection{Setup of primary blowup quantity}
To this end, we denote by
\begin{equation}
\psi_{j,k}:=\varphi-\underline{\varphi}-\sum_{i=1}^{j}\sum_{p=1}^{N_{i,k}}\mathfrak{G}_{t,k}\left(A_{i,p,k}, G_{i,p,k} \right),
\end{equation}
which by definition satisfies $\eta_{j,k}=\ddbar\psi_{j,k}$. Of course, when $j=0$ we have by definition $\psi_{0,k}=\varphi-\underline{\varphi}$.
 For $x=(z,y),x'=(z',y')\in B\times Y$ which are either horizontally or vertically joined and $0<t'<t$, we  consider the quantities
\begin{equation}
\mathcal{H}_0(x,x',t,t')=\frac{|\ddbar\vp(x,t)-\P_{x'x}\ddbar\vp(x',t')|_{g(t)}}{ (d^{g(t)}(x,x')+\sqrt{t-t'})^\a},
\end{equation}
and for $j\geq 1$,
\begin{equation}
\begin{split}
\mathcal{H}_{j}(x,x',t,t')&=\frac{|\mathfrak{D}^{2j} \ddbar\underline{\varphi}(x,t)-\P_{x'x}\mathfrak{D}^{2j}\ddbar\underline{\varphi}(x',t')|_{g(t)}}{ (d^{g(t)}(x,x')+\sqrt{t-t'})^\a}\\
&\quad +\frac{|\mathfrak{D}^{2j} (\de_t\underline{\varphi}+\underline{\varphi})(x,t)-\P_{x'x}\mathfrak{D}^{2j}(\de_t\underline{\varphi}+\underline{\varphi})(x',t')|_{g(t)}}{ (d^{g(t)}(x,x')+\sqrt{t-t'})^\a}\\
&\quad +\frac{|\mathfrak{D}^{2j+2}\psi_{j,k}(x,t)-\P_{x'x}\mathfrak{D}^{2j+2} \psi_{j,k}(x',t')|_{g(t)}}{ (d^{g(t)}(x,x')+\sqrt{t-t'})^\a}\\
&\quad +\sum_{i=1}^j \sum_{p=1}^{N_{i,k}}\sum_{\iota = -2}^{2k} e^{-\iota\frac{t}{2}}\Bigg(\frac{|\mathfrak{D}^{2j+2+\iota} A_{i,p,k}(x,t)-\mathbb{P}_{x'x}(\mathfrak{D}^{2j+2+\iota}A_{i,p,k}(x',t'))|_{g(t)}}{ (d^{g(t)}(x,x')+\sqrt{t-t'})^\a}\Bigg).
\end{split}
\end{equation}
as well as
\begin{equation}
\mathcal{D}_{j}(x,x',t)=\sup_{t'\in [t-1,t]}\mathcal{H}_{j}(x,x',t,t'), \quad j\geq 0.
\end{equation}
For each $x=(z,y)\in B\times Y$ and $t\geq 0$, we define the blowup quantity
\begin{equation}
\mu_{j}(x,t)=\left|1-|z| \right|^{2j+\a} \sup \mathcal{D}_j(x,x',t),
\end{equation}
where the sup is taken over all $x'=(z',y')\in B\times Y$ with $|z'-z|<\frac14 \left| |z|-1 \right|$ and $x'$ is either horizontally or vertically joined with $x$. We want to show that there is $C>0$ such that for all $t\geq 0$,
\begin{equation}\label{krkr}
\sup_{B\times Y} \mu_{j}(x,t)\leq C.
\end{equation}
Since $g(t)$ is uniformly comparable to $g(s)$ if $|t-s|<1$,
a bound on $\mu_j$ implies \eqref{agogn} when $j=0$, and implies \eqref{holder1}, \eqref{holder2}, \eqref{holder5} and \eqref{holder3} when $j\geq 1$, and would thus conclude the proof of Theorem \ref{thm:decomposition}

Observe that the quantity $\mathcal{H}_0$ is closer in spirit to the one used in our earlier works \cite[(5.7)]{HT2} and \cite[(3.10)]{CL} (which dealt only with the case $j=0$), rather than the one used in \cite[(4.29)]{HT3} (which dealt with all $j\geq 0$ at once).

We now setup the contradiction argument, so suppose that \eqref{krkr} fails. We can then find a sequence $t_\ell>0$ such that $\sup_{B\times Y\times [0,t_\ell]} \mu_j(x,t)\to +\infty$ as $\ell\to +\infty$.  Since the solution of the flow is smooth on any compact time interval, we must have $t_\ell\to +\infty$. Moreover, there exists $s_\ell\in [0,t_\ell]$ such that $\sup_{B\times Y}\mu_{j}(x,s_\ell)=\sup_{B\times Y\times [0,t_\ell]} \mu_j(x,t)$.  Without loss of generality, we can assume $s_\ell=t_\ell\to +\infty$ and
\begin{equation}
\sup_{B\times Y} \mu_j(x,t_\ell)\to +\infty.
\end{equation}

For each $\ell$, we choose $x_\ell=(z_\ell,y_\ell)\in B\times Y$ such that $\mu_j(x_\ell,t_\ell)=\sup_{B\times Y} \mu_j(x,t_\ell)$. We also define $\lambda_\ell$ by
\begin{equation}\lambda_\ell^{2j+\a}:=\sup_{x'} \mathcal{D}_{j}(x_\ell,x',t_\ell),
\end{equation}
so that
\begin{equation}
\mu_j(x_\ell,t_\ell)= \left| |z_\ell|-1\right|^{2j+\a} \lambda_\ell^{2j+\a}\to +\infty,
\end{equation}
and hence $\lambda_\ell \to+\infty$. Let $x_\ell'\in B\times Y$ be the point realizing $\sup_{x'} \mathcal{D}_j(x_\ell,x',t_\ell)$ and $t_\ell'\in [t_\ell-1,t_\ell]$ realizing $\sup_{t'}\mathcal{H}_j(x_\ell,x_\ell',t_\ell,t')$.  Without loss of generality,  we can also assume $x_\ell\to x_\infty\in \overline{B\times Y}$.

Consider the diffeomorphisms
\begin{equation}
\Psi_\ell:  B_{\lambda_\ell} \times Y\times [-\lambda_\ell^2t_\ell,0] \to B \times Y\times [0,t_\ell], \;\, (z,y,t) = \Psi_\ell(\hat{z},\hat{y},\hat{t}) =(\lambda_\ell^{-1}\hat{z},\hat{y},t_\ell+\lambda_\ell^{-2}\hat{t}).
\end{equation}
Let $\hat{x}_\ell:=(\hat{z}_\ell,\hat{y}_\ell)$, where
\begin{equation}
(\hat{z}_\ell,\hat{y}_\ell,\hat{t}):=\Psi_\ell^{-1}(z_\ell,y_\ell,t),
\end{equation}
so that $\hat{t}_\ell=0, \hat{t}'_\ell=\lambda_\ell^2(t_\ell'-t_\ell)$, and $\hat{t}=\lambda_\ell^2(t-t_\ell)\in [-\lambda_\ell^2 t_\ell,0]$. Given a (time-dependent) contravariant $2$-tensor $\alpha$ (such as $\omega^\bullet(t), g(t)$, etc.), we define $\hat{\alpha}_\ell:=\lambda_\ell^2\Psi_\ell^*\alpha$.
Thus, for example, $\hat \omega_\ell^\bullet(\hat t)=\lambda_\ell^2\Psi_\ell^* \omega^\bullet(t_\ell+\lambda_\ell^{-2}\hat t).$ The pullback complex structure will be denoted by $\hat{J}_\ell$. Given a (time-dependent) scalar function $F$, we will also denote by $\hat{F}_\ell:=\Psi_\ell^*F$, so that for example $\hat{G}_{\ell,i,p,k}=\Psi_\ell^*G_{i,p,k}$. However, for the two functions $A_{i,p,k}$ and $\vp$ we will define instead
\begin{equation}
\hat A_{\ell,i,p,k}(\hat t):=\lambda_\ell^2 \Psi_\ell^* A_{i,p,k}(t_\ell+\lambda_\ell^{-2}\hat t),\quad \hat{\vp}_\ell(\hat{t}):=\lambda_\ell^2\Psi_\ell^*\vp(t_\ell+\lambda_\ell^{-2}\hat t),
\end{equation}
where  $\hat{t}\in [-t_\ell \lambda_\ell^2,0]$. We define also
\begin{equation}
\delta_\ell:=\lambda_\ell e^{-\frac{t_\ell}{2}}.
\end{equation}
Observe that from \eqref{induction-bdd-A} we have that
\begin{equation}\label{hat-A-estimate}
\|\hat A_{\ell,i,p,k}(\hat t)\|_{L^\infty(\hat B_{\lambda_\ell})}\leq C \delta_\ell^2 e^{\frac{-2i+2-\a}{2}t_\ell-\frac{2i+\a}{2}\lambda_\ell^{-2}\hat t},
 \end{equation}
for all $1\leq i\leq j$ and $1\leq p\leq N_{i,k}$ and $-t_\ell \lambda_\ell^2\leq \hat t\leq 0$.
For notational convenience, we will still use $\mathfrak{D}$ and $\P$ to denote their pullbacks via $\Psi_\ell$. In particular,
 $\hat\omega_\ell^\bullet=\hat \omega_\ell^\natural+\ddbar \hat \varphi_\ell$ satisfies the following \KR flow
\begin{equation}
\partial_{\hat t} \hat\omega_\ell^\bullet=-\Ric(\hat\omega_\ell^\bullet)-\lambda_\ell^{-2}\hat\omega_\ell^\bullet,
\end{equation}
and we can equivalently write the complex Monge-Amp\`ere equation \eqref{ma2} as
\begin{equation}\label{mak}
(\hat \omega_\ell^\bullet)^{m+n}= e^{\de_{\hat{t}}{\hat\varphi}_\ell+\lambda_\ell^{-2}\hat\varphi_\ell-n \lambda_\ell^{-2}\hat t} \binom{m+n}{m} \hat \omega_{\ell,\mathrm{can}} ^m \wedge (\delta_\ell^2 \Psi_\ell^* \omega_F)^n,
\end{equation}
where (following the above convention) $\hat\omega_{\ell,\mathrm{can}}=\lambda_\ell^2\Psi_\ell^* \omega_{\mathrm{can}}$. It is then straightforward to see that for all $\ell\geq 0$ we have for $j=0,$
\begin{equation}\label{sup-realized-10}
1=\frac{|\ddbar\hat\vp_\ell(\hat x_\ell,0)-\P_{\hat x_\ell'\hat x_\ell}\ddbar\hat\vp_\ell(\hat x_\ell',\hat t_\ell')|_{\hat g_\ell(0)}}{ (d^{\hat g_\ell(0)}(\hat x_\ell,\hat x_\ell')+|\hat t_\ell'|^{\frac{1}{2}})^\a},
\end{equation}
and for $j\geq 1,$
\begin{equation}\label{sup-realized-1}
\begin{split}
1&=\frac{|\mathfrak{D}^{2j}\ddbar\underline{\hat \varphi_\ell}(\hat x_\ell,0)-\P_{\hat x_\ell'\hat x_\ell}\mathfrak{D}^{2j}\ddbar\underline{\hat\varphi_\ell}(\hat x_\ell',\hat t_\ell')|_{\hat g_\ell(0)}}{ (d^{\hat g_\ell(0)}(\hat x_\ell,\hat x_\ell')+|\hat t_\ell'|^{\frac{1}{2}})^\a}\\
&\quad +\frac{|\mathfrak{D}^{2j} (\partial_{\hat t}\underline{\hat \varphi_\ell}+\lambda_\ell^{-2}\underline{\hat \varphi_\ell} )(\hat x_\ell,0)-\P_{\hat x_\ell'\hat x_\ell}\mathfrak{D}^{2j}(\partial_{\hat t}\underline{\hat \varphi_\ell}+\lambda_\ell^{-2}\underline{\hat \varphi_\ell} )(\hat x_\ell',\hat t_\ell')|_{\hat g_\ell(0)}}{ (d^{\hat g_\ell(0)}(\hat x_\ell,\hat x_\ell')+|\hat t_\ell'|^{\frac{1}{2}})^\a}\\
&\quad +\frac{|\mathfrak{D}^{2j+2} \hat \psi_{\ell,j,k}(\hat x_\ell,0)-\P_{\hat x_\ell'\hat x_\ell}\mathfrak{D}^{2j+2} \hat \psi_{\ell,j,k}(\hat x_\ell',\hat t_\ell')|_{\hat g_\ell(0)}}{ (d^{\hat g_\ell(0)}(\hat x_\ell,\hat x_\ell')+|\hat t_\ell'|^{\frac{1}{2}})^\a}\\
&\quad +\sum_{i=1}^j \sum_{p=1}^{N_{i,k}}\sum_{\iota = -2}^{2k} \delta_\ell^\iota\Bigg(\frac{|\mathfrak{D}^{2j+2+\iota} \hat A_{\ell,i,p,k}(\hat x_\ell,\hat t_\ell)- \P_{\hat x_\ell'\hat x_\ell}(\mathfrak{D}^{2j+2+\iota}\hat A_{\ell,i,p,k}(\hat x_\ell',\hat t_\ell'))|_{\hat g_\ell(0)}}{(d^{\hat g_{\ell}(0)}(\hat x_\ell,\hat x'_\ell) +|\hat t_\ell'|^{\frac{1}{2}})^\a}\Bigg),
\end{split}
\end{equation}
and $\hat{x}'_\ell$ was chosen to maximize the difference quotients in \eqref{sup-realized-10} and \eqref{sup-realized-1} (which we can call $\hat{\mathcal{D}}_j(\hat{x}_\ell,\hat{x}',\hat{t}_\ell)$) among all points
$\hat{x}'=(\hat{z}',\hat{y}')\in B_{\mathbb{C}^m}(\lambda_\ell)\times Y$,
with $|\hat{z}'-\hat{z}_\ell|<\frac14 \left| |\hat{z}_\ell|-\lambda_\ell \right|$ which are horizontally or vertically joined to $\hat{x}_\ell$. Moreover, the points $\hat{x}_\ell$ and $\hat{t}_\ell$ themselves maximize the quantity
\begin{equation}
\left\|\hat{z}|-\lambda_\ell \right|^{2j+\a} \sup_{\substack{\hat{x}'=(\hat{z}',\hat{y}')\ s.t.\ |\hat{z}'-\hat{z}|<\frac{1}{4}\|\hat{z}|-\lambda_\ell|\\\hat{x}'\ \mathrm{and}\ \hat{x}\ \mathrm{horizontally}\ \mathrm{or}\ \mathrm{vertically}\ \mathrm{joined}}} \hat{\mathcal{D}}_j(\hat{x},\hat{x}',\hat{t}),
\end{equation}
among all $\hat{x}=(\hat z,\hat y)\in B_{\mathbb{C}^m}(\lambda_\ell)\times Y$ and $\hat t\in [-t_\ell\lambda_\ell^2,0]$, hence for all such $\hat{x},\hat{x}',\hat{t}$ we have
\begin{equation}\label{kakka}
\hat{ \mathcal{D}}_{j} (\hat x,\hat x',\hat t) \leq \left(\frac{\left| |\hat z_\ell|-\lambda_\ell\right|}{\left| |\hat z|-\lambda_\ell\right|} \right)^{2j+\a}.
\end{equation}
Using $\left| |\hat z_\ell|-\lambda_\ell\right|=\lambda_\ell \left| |z_\ell|-1\right|\to +\infty$ (hence the pointed limit centered at $\hat x_\ell$ will be complete), and
\begin{equation}
\|\hat{z}'_\ell|-\lambda_\ell|\geq\frac{3}{4}\|\hat{z}_\ell|-\lambda_\ell|\to+\infty,
\end{equation}
together with the triangle inequality, we see from \eqref{kakka} that there exists $C>0$ such that for any fixed $R>0$ there exists $\ell_R\in\mathbb{N}$ such that
for all $\ell\geq\ell_R$ and $\hat{z}^\bullet_\ell=\hat{z}_\ell$ or $\hat{z}'_\ell$, we have
\begin{equation}
\sup_{\substack{\hat{x},\hat{x}' \in \hat{B}_{\C^m}(\hat{z}^\bullet_\ell,R)\times Y,\ \hat t\in [-t_\ell\lambda_\ell^2,0]\\\hat{x}'\ \mathrm{and}\ \hat{x}\ \mathrm{horizontally}\ \mathrm{or}\ \mathrm{vertically}\ \mathrm{joined}}} \hat{ \mathcal{D}}_{j} (\hat x,\hat x',\hat t) \leq C,
\end{equation}
where here and in the following, the hat decoration over $B_{\C^m}$ is just to remind the reader that we are in the hat picture.
This in particular implies there exists $C>0$ so that for all fixed $R$, $\hat t\in  [-t_\ell\lambda_\ell^2,0]$ and all sufficiently large $\ell$, we have for $j=0,$
\begin{equation}\label{a-priori Control0}
[\ddbar\hat\vp_\ell]_{\a,\a/2,\hat Q_R(\hat z,\hat t),\hat g_\ell(\hat t)}\leq C,
\end{equation}
and for $j\geq 1,$
\begin{equation}\label{a-priori Control}
\begin{split}
&\quad [\mathfrak{D}^{2j} \ddbar \underline{\hat \varphi_\ell}]_{\a,\a/2,\hat Q_R(\hat z,\hat t),\hat g_\ell(\hat t)}+ [\mathfrak{D}^{2j} (\partial_{\hat t}\underline{\hat \varphi_\ell}+\lambda_\ell^{-2}\underline{\hat \varphi_\ell})]_{\a,\a/2,\hat Q_R(\hat z,\hat t),\hat g_\ell(\hat t)}\\
&\quad + [\mathfrak{D}^{2j+2} \hat \psi_{\ell,j,k} ]_{\a,\a/2,\hat Q_R(\hat z,\hat t),\hat g_\ell(\hat t)}+\sum_{i=1}^j \sum_{p=1}^{N_{i,k}}\sum_{\iota = -2}^{2k} \delta_\ell^\iota\left[\mathfrak{D}^{2j+2+\iota} \hat A_{\ell,i,p,k}\right]_{\a,\a/2,\hat Q_R(\hat z,\hat t),\hat g_\ell(\hat t)}\leq C,
\end{split}
\end{equation}
 where $\hat Q_R(\hat z,\hat t)=\left(B_{\mathbb{C}^m}(\hat z,R)\times Y\right)\times [-R^2+\hat t,\hat t]$ with $\hat z$ being either $\hat z_\ell$ or $\hat z_\ell'$.

When $j=0$ we will need the following ``non-cancellation'' inequality
\begin{equation}\label{noncanzel2}
[\ddbar\underline{\hat{\vp}_\ell}]_{\a,\a/2,\hat Q_R(\hat z,\hat t),\hat g_\ell(\hat t)}\leq C[\ddbar\hat{\vp}_\ell]_{\a,\a/2, \textrm{base},\hat Q_R(\hat z,\hat t),g_X}
+C\left(\frac{R}{\lambda_\ell}\right)^{1-\alpha}\|\ddbar\hat{\vp}_\ell\|_{\infty, \hat Q_R(\hat z,\hat t),g_X},
\end{equation}
which is again straightforward to prove using \cite[(4.215)]{HT3}, and plugging in the H\"older bound in \eqref{a-priori Control0} and the $L^\infty$ bound $\|\ddbar\hat{\vp}_\ell\|_{\infty, \hat Q_R(\hat z,\hat t),\hat{g}_\ell(\hat{t})}=o(1)$ which comes from Lemma \ref{lma:earlier work} (iii), we see that
\begin{equation}
[\ddbar\underline{\hat\vp_\ell}]_{\a,\a/2,\hat Q_R(\hat z,\hat t),\hat g_\ell(\hat t)}\leq C,
\end{equation}
and combining this with \eqref{a-priori Control0}, using that $\hat{\eta}_{0,k}=\ddbar\hat{\psi}_{\ell,0,k}=\ddbar\hat{\vp}_\ell-\ddbar\underline{\hat{\vp}_\ell}$, and using the triangle inequality and the boundedness of $\P$, we get
\begin{equation}\label{a-priori Control00}
[\ddbar\underline{\hat\vp_\ell}]_{\a,\a/2,\hat Q_R(\hat z,\hat t),\hat g_\ell(\hat t)}+[\ddbar\hat{\psi}_{\ell,0,k}]_{\a,\a/2,\hat Q_R(\hat z,\hat t),\hat g_\ell(\hat t)}\leq C.
\end{equation}

After passing to subsequence,  we will split the rest of the proof into three cases, according to the behavior of $\delta_\ell=\lambda_\ell e^{-t_\ell/2}$: {\bf Case 1:} $\delta_\ell\to +\infty$, {\bf Case 2:} $\delta_\ell\to \delta_\infty>0$ and {\bf Case 3:} $\delta_\ell\to 0$.

\subsection{Non-escaping property}\label{nonesc}
In this subsection, we will show that in {\bf Cases 2} and {\bf 3}, the distance between the two blowup points $\hat{x}_\ell$ and $\hat x_\ell'$ will not go to infinity. This proof does not apply to {\bf Case 1} (at least when $j\geq 1$), but we will nevertheless establish the same result in that case in \eqref{eskape} below.
\begin{prop}\label{prop:non-escaping}
Suppose $\delta_\ell\leq C$ for some $C>0$, then there exists $C'>0$ such that for all $\ell>0$,
\begin{equation}\label{eskapenot}
d^{\hat g_\ell(0)}(\hat x_\ell,\hat x_\ell')+|\hat t_\ell'|^{\frac{1}{2}}\leq C'.
\end{equation}
\end{prop}
\begin{proof}
First we can easily deal with the case $j=0$. By Lemma \ref{lma:earlier work} (iii) we know that
\begin{equation}\label{kork}
\sup_{B\times Y\times [t-1,t]}|\ddbar\vp|_{g(t)}=o(1),
\end{equation}
which implies
\begin{equation}
\|\ddbar\hat{\vp}_\ell\|_{\infty,\hat Q_\rho(\hat z,\hat t),g_\ell(\hat t)}=o(1),
\end{equation}
as $\ell\to\infty$, for fixed $\hat{z},\hat{t},\rho$, and applying this to $\hat{z}=\hat{z}_\ell$ or $\hat{z}'_\ell$ and $\hat{t}=0$ or $\hat{t}'_\ell$, we see that the numerator on the RHS of
\eqref{sup-realized-10} is going to zero, which gives us the even stronger statement that
\begin{equation}\label{gecko}
d^{\hat g_\ell(0)}(\hat x_\ell,\hat x_\ell')+|\hat t_\ell'|^{\frac{1}{2}}=o(1),
\end{equation}
and for $j=0$ we do not even need the assumption that $\delta_\ell\leq C$.

Next, we assume $j\geq 1$. The argument is a modification of \cite[Proposition 4.5]{HT3}, and the goal is to estimate each of the terms in the blowup quantity in \eqref{sup-realized-1}.  In the following,  we denote  $\hat Q_R(\hat z,\hat t)=\left(B_{\mathbb{C}^m}(\hat z,R)\times Y\right)\times [-R^2+\hat t,\hat t]$ with $\hat z$ being either $\hat z_\ell$ or $\hat z_\ell'$ and $\hat t$ is either $\hat t_\ell=0$ or $\hat t_\ell'$.

We first handle the terms involving $A_{\ell,i,p,k}$. Recall from \eqref{a-priori Control} that for all $-2\leq \iota\leq 2k$ and each given $R>0$,
\begin{equation}\label{kol}
\left[\mathfrak{D}^{2j+2+\iota} \hat A_{\ell,i,p,k}\right]_{\a,\a/2,\hat Q_R(\hat z,\hat t),\hat g_\ell(\hat t)}\leq C\delta_\ell^{-\iota},
\end{equation}
while from \eqref{hat-A-estimate}
\begin{equation}
\|\hat A_{\ell,i,p,k}\|_{\infty,\hat Q_R(\hat z,\hat t)}\leq C \delta_\ell^2 e^{\frac{-2i+2-\a}{2}t_\ell+\frac{2i+\a}{2}\lambda_\ell^{-2}(R^2-\hat t)}\leq C \delta_\ell^2 e^{\frac{-2i+2-\a}{2}t_\ell},
\end{equation}
since $\hat{t}=0$ or $\hat{t}_\ell'$ satisfies $\hat{t}\geq -\lambda_\ell^2$.
By Proposition \ref{prop:interpolation}, for all $1\leq r\leq 2j$,
\begin{equation}
\begin{split}
(R-\rho)^r \|\mathfrak{D}^r\hat A_{\ell,i,p,k}\|_{\infty,\hat Q_\rho(\hat z,\hat t),\hat g_\ell(\hat t)}\leq C(R-\rho)^{2j+\a}\delta_\ell^2+C \delta_\ell^2 e^{\frac{-2i+2-\a}{2}t_\ell}.
\end{split}
\end{equation}

We choose
\begin{equation}
R-\rho\approx \left( e^{\frac{-2i+2-\a}{2}t_\ell}\right)^\frac{1}{2j+\a},
\end{equation}
which is small for any given $R$.
We get for all $0\leq r\leq 2j$ that
\begin{equation}\label{lower-A-Linfty}
\|\mathfrak{D}^r\hat A_{\ell,i,p,k}\|_{\infty,\hat Q_\rho(\hat z,\hat t),\hat g_\ell(\hat t)}\leq  C\delta_\ell^2\left( e^{\frac{-2i+2-\a}{2}t_\ell}\right)^{1-\frac{r}{2j+\a}}=o(1),
\end{equation}
where $\hat t$ is either $\hat t_\ell=0$ or $\hat t_\ell'$.

If we let $r=\iota+2>0$, then for $0<r\leq 2k+2$ we can interpolate again, using \eqref{kol}, and get
\begin{equation}
\begin{split}
(R-\rho)^r \|\mathfrak{D}^{2j+r}\hat A_{\ell,i,p,k}\|_{\infty,\hat Q_\rho(\hat z,\hat t),\hat g_\ell(\hat t)}\leq C(R-\rho)^{r+\a}\delta_\ell^{-r+2}+C\delta_\ell^2\left( e^{\frac{-2i+2-\a}{2}t_\ell}\right)^{\frac{\a}{2j+\a}}.
\end{split}
\end{equation}

By taking
\begin{equation}
R-\rho\approx \left(\delta_\ell^r\left( e^{\frac{-2i+2-\a}{2}t_\ell}\right)^{\frac{\a}{2j+\a}}\right)^\frac{1}{r+\a},
\end{equation}
which is small thanks to our assumption that $\delta_\ell\leq C$, we conclude that for all $0<r\leq 2k+2$,
\begin{equation}\label{higher-A-Linfty}
\delta_\ell^{r-2} \|\mathfrak{D}^{2j+r}\hat A_{\ell,i,p,k}\|_{\infty,\hat Q_\rho(\hat z,\hat t),\hat  g_\ell(\hat t)}\leq C  \left(\delta_\ell^r\left( e^{\frac{-2i+2-\a}{2}t_\ell}\right)^{\frac{\a}{2j+\a}}\right)^\frac{\a}{r+\a} =o(1).
\end{equation}

Applying \eqref{lower-A-Linfty} and \eqref{higher-A-Linfty} to balls centered at $(\hat z_\ell,\hat t_\ell)$ and $(\hat z_\ell',\hat t_\ell')$ (of any radius, e.g. $1$),  together with the boundedness of operator norm of $\mathbb{P}$ from \cite[\S 2.1.1]{HT3}, this gives
\begin{equation}\label{dist:A-contr}
\sum_{i=1}^j \sum_{p=1}^{N_{i,k}}\sum_{\iota = -2}^{2k} \delta_\ell^\iota\Bigg({|\mathfrak{D}^{2j+2+\iota} \hat A_{\ell,i,p,k}(\hat x_\ell,0)- \P_{\hat x_\ell'\hat x_\ell}(\mathfrak{D}^{2j+2+\iota}\hat A_{\ell,i,p,k}(\hat x_\ell',\hat t_\ell'))|_{\hat g_\ell(0)}}\Bigg)=o(1).
\end{equation}

We now treat $\hat \psi_{\ell,j,k}$ which has fiberwise average zero.
By Proposition \ref{prop:braindead} (in case $\delta_\ell$ does not converge to $0$, we choose $\rho$ to be  sufficiently large) and \eqref{a-priori Control},  we have
\begin{equation}
 \|\mathfrak{D}^{2j+2} \hat \psi_{\ell,j,k}\|_{\infty,\hat Q_\rho(\hat z,\hat t),\hat g_\ell(\hat t)}\leq C\delta_\ell^\a.
\end{equation}

Applying this to balls centered at $\hat z_\ell$ and $\hat z_\ell'$ and invoking \cite[\S 2.1.1]{HT3} again, we have
\begin{equation}\label{dist:psi-contr}
\begin{split}
{|\mathfrak{D}^{2j+2} \hat \psi_{\ell,j,k}(\hat x_\ell,0)-\P_{\hat x_\ell'\hat x_\ell}\mathfrak{D}^{2j+2} \hat \psi_{\ell,j,k}(\hat x_\ell',\hat t_\ell')|_{\hat g_\ell(0)}}\leq C\delta_\ell^\a.
\end{split}
\end{equation}

It remains to consider the fiberwise average of the potential,  i.e. $\underline{\hat \varphi_\ell}$.  Recalling \eqref{kork}
and taking fiber average (using the fiber integration argument in \cite[p.436]{To4}) gives in particular
\begin{equation}
\|\ddbar\underline{\vp}\|_{\infty, B\times [t-1,t], g_{\C^m}}=o(1),
\end{equation}
as $t\to+\infty$, which implies
\begin{equation}
\|\ddbar\underline{\hat{\vp}_\ell}\|_{\infty,\hat Q_\rho(\hat z,\hat t),g_\ell(\hat t)}=o(1),
\end{equation}
as $\ell\to\infty$, for fixed $\hat{z},\hat{t},\rho$, and interpolating between this and \eqref{a-priori Control} gives
\begin{equation}
\|\mathfrak{D}^{2j}\ddbar\underline{\hat \varphi_\ell}\|_{\infty, \hat Q_\rho(\hat z,\hat t),g_\ell(\hat t)}=o(1),
\end{equation}
for $\hat{z}=\hat{z}_\ell$ or $\hat{z}'_\ell$ and $\hat{t}=0$ or $\hat{t}'_\ell$. Using again the boundedness of the operator norm of $\P$, this implies that
\begin{equation}\label{dist:varphi-contr}
|\mathfrak{D}^{2j}\ddbar\underline{\hat \varphi_\ell}(\hat x_\ell,0)-\P_{\hat x_\ell'\hat x_\ell}\mathfrak{D}^{2j}\ddbar\underline{\hat\varphi_\ell}(\hat x_\ell',\hat t_\ell')|_{\hat g_\ell(0)}=o(1).
\end{equation}
Similarly, from Lemma \ref{lma:earlier work} (ii) and taking fiber average we know that
\begin{equation}
\sup_{B\times Y\times [t-1,t]}|\de_t\underline{\vp}+\underline{\vp}|=o(1),
\end{equation}
as $t\to+\infty$, which implies
\begin{equation}
\|\de_{\hat{t}}\underline{\hat \varphi_\ell}+\lambda_\ell^{-2}\underline{\hat \varphi_\ell}\|_{\infty,\hat Q_\rho(\hat z,\hat t),g_\ell(\hat t)}=o(1),
\end{equation}
as $\ell\to\infty$, for fixed $\hat{z},\hat{t},\rho$, and again interpolating between this and \eqref{a-priori Control} gives
\begin{equation}
\|\mathfrak{D}^{2j}(\partial_{\hat t}\underline{\hat \varphi_\ell}+\lambda_\ell^{-2}\underline{\hat \varphi_\ell} )\|_{\infty, \hat Q_\rho(\hat z,\hat t),g_\ell(\hat t)}=o(1),
\end{equation}
for $\hat{z}=\hat{z}_\ell$ or $\hat{z}'_\ell$ and $\hat{t}=0$ or $\hat{t}'_\ell$, and hence
\begin{equation}\label{dist:varphi2-contr}
{|\mathfrak{D}^{2j} (\partial_{\hat t}\underline{\hat \varphi_\ell}+\lambda_\ell^{-2}\underline{\hat \varphi_\ell} )(\hat x_\ell,0)-\P_{\hat x_\ell'\hat x_\ell}\mathfrak{D}^{2j}(\partial_{\hat t}\underline{\hat \varphi_\ell}+\lambda_\ell^{-2}\underline{\hat \varphi_\ell} )(\hat x_\ell',\hat t_\ell')|_{\hat g_\ell(0)}}=o(1).
\end{equation}

Combining \eqref{sup-realized-1} with \eqref{dist:A-contr}, \eqref{dist:psi-contr}, \eqref{dist:varphi-contr} and \eqref{dist:varphi2-contr}, we obtain the desired bound \eqref{eskapenot}.
\end{proof}

We are now in position to study the flows obtained as complete pointed limits of
\begin{equation}
\left( \hat B_{\mathbb{C}^m}(\hat z_\ell,\lambda_\ell)\times Y, \hat g_\ell(\hat{t}),\hat x_\ell\right),
\end{equation}
as $\ell\to +\infty$, where as usual $\hat{t}\in [-\lambda_\ell^2t_\ell,0]$. By translation, we can assume $\hat x_\ell=(\hat z_\ell,\hat y_\ell)=(0,\hat y_\ell)\in \mathbb{C}^m\times Y$ and $\hat y_\ell\to \hat y_\infty\in Y$ by compactness of $Y$ after passing to a subsequence.

\subsection{Blowup analysis in Case 1: $\delta_\ell\to +\infty$}

In this case the metrics $\hat{g_\ell}(0)$ are blowing up in the fiber directions, so that their pointed blowup limit (modulo local diffeomorphisms that stretch the fibers) would be $\mathbb{C}^{m+n}$. While this is the approach taken in our earlier works \cite{CL,FL, HT2}, it turns out that we need a different approach instead (at least when $j\geq 1$). So, following \cite{HT3}, we
consider the diffeomorphisms
\begin{equation}
\Xi_\ell: B_{e^{\frac{t_\ell}{2}}} \times Y\times [-e^{t_\ell}t_\ell,0] \to B_{\lambda_\ell} \times Y\times [-\lambda_\ell^2t_\ell,0], \;\, (\hat{z},\hat{y},\hat{t}) = \Xi_\ell(\check{z},\check{y},\check{t}) =(\delta_\ell\check{z},\check{y},\delta_\ell^2\check{t}),
\end{equation}
pull back time-dependent $2$-tensors via $\Xi_\ell$, rescale them by $\delta_\ell^{-2}$ and denote the new tensors with a check, so for example
$\check \omega^\bullet_\ell(\check t)=\delta_\ell^{-2} \Xi_\ell^*\hat \omega_\ell^\bullet(\delta_\ell^2 \check t)$.
We also apply the same  pullback and rescaling procedure to the scalar functions $\hat{A}_{\ell,i,p,k}$ and $\hat{\vp}_\ell$.

In this case, $\check g_\ell(\check{t})$ is locally uniformly Euclidean in space-time and $\check \omega_\ell^\bullet $ satisfies the \KR flow equation:
\begin{equation}\partial_{\check t}\check \omega_\ell^\bullet  = -\Ric(\check \omega_\ell^\bullet )-e^{-t_\ell}\check \omega_\ell^\bullet ,\end{equation}
and the Monge-Amp\`ere equation \eqref{mak} becomes
\begin{equation}\label{mma}
(\check \omega_\ell^\bullet)^{m+n}= e^{\de_{\check{t}}{\check\varphi}_\ell+e^{-t_\ell}\check\varphi_\ell-n e^{-t_\ell}\check t} \binom{m+n}{m} \check \omega_{\ell,\mathrm{can}} ^m \wedge (\Xi_\ell^* \Psi_\ell^* \omega_F)^n.
\end{equation}
Thanks to Lemma \ref{lma:earlier work} (i) we know that $\check \omega_\ell^\bullet (\check t)$ is uniformly equivalent to
\begin{equation}
\check \omega_\ell^\natural(\check t)=(1-e^{-t_\ell-e^{-t_\ell}\check t}) \check \omega_{\ell,\mathrm{can}}+e^{-e^{-t_\ell}\check t}\Xi_\ell^*\Psi_\ell^*\omega_F,
\end{equation}
which in turns is locally uniformly equivalent to the Euclidean metric. The pullback of the complex structure also converges locally uniformly smoothly to the Euclidean product complex structure due to the stretching.  We want to apply the local higher order estimates in \cite[Proposition 2.1]{CL} on $\check Q_1(\check z_\ell,0)$ and $\check Q_1(\check z_\ell',\check t_\ell')$, but we do not know whether $\check{B}_1(\check{z}_\ell),\check{B}_1(\check{z}'_\ell)$ are contained in $B_{e^{\frac{t_\ell}{2}}}$, as we don't have any relation between $\delta_\ell$ and $\|\hat{z}_\ell|-\lambda_\ell|$. However, these are compactly contained in the slightly larger ball
$B_{(1+\sigma)e^{\frac{t_\ell}{2}}}$ for any fixed $\sigma>0$ and all $\ell$ sufficiently large, and we may assume without loss that $\check \omega_\ell^\bullet (\check t)$ is uniformly equivalent to Euclidean on $B_{(1+\sigma)e^{\frac{t_\ell}{2}}}\times Y\times [-(1+\sigma)^2e^{t_\ell}t_\ell,0]$. Thus, the local higher order estimates give us uniform $C^\infty$ estimates for $\check \omega_\ell^\bullet$ on  $\check Q_1(\check z_\ell,0)$ and $\check Q_1(\check z_\ell',\check t_\ell')$. Thus on these sets we have uniform $C^\infty$ bounds for $\ddbar\check{\vp}$, hence on $\ddbar\underline{\check{\vp}}$ (by fiber averaging), hence on $\check{A}_{\ell,1,p,k}$ (from its definition), hence on $\check{\gamma}_{\ell, 1,k}$ (also from its definition), hence on $\check{\eta}_{\ell,1,k}$ (from its definition), and so forth until $\check{A}_{\ell,j,p,k},\check{\gamma}_{\ell,j,k},\check{\eta}_{\ell,j,k}$. From the PDE \eqref{mma} we also get uniform $C^\infty$ bounds for
$\partial_{\check t} \check \varphi_\ell+e^{-t_\ell} \check \varphi_\ell$, and so by fiber averaging also on $\partial_{\check t} \underline{\check \varphi_\ell}+e^{-t_\ell} \underline{\check \varphi_\ell}$. Also, since $\check{\eta}_{\ell,j,k}=\ddbar\psi_{\ell,j,k}$ is locally smoothly bounded, and $\psi_{\ell,j,k}$ has fiberwise average zero, then fiberwise Moser iteration gives us a uniform $L^\infty$ bound for $\psi_{\ell,j,k}$, and elliptic estimates show that $\psi_{\ell,j,k}$ has uniform $C^\infty$ bounds. Putting these all together we get in particular

\begin{equation}
\begin{split}
&\|\mathfrak{D}^{2j} \ddbar \check \varphi_\ell\|_{\infty,\check Q_{1},\check g_\ell(0)}+\|\mathfrak{D}^{2j} \ddbar \underline{\check \varphi_\ell}\|_{\infty,\check Q_{1},\check g_\ell(0)}+\|\mathfrak{D}^{2j} (\partial_{\check t} \underline{\check \varphi_\ell}+e^{-t_\ell} \underline{\check \varphi_\ell}    )\|_{\infty,\check Q_{1},\check g_\ell(0)}\\
&+  \|\mathfrak{D}^{2j+2}  \check \psi_{\ell,j,k}  \|_{\infty, \check Q_{1},\check g_\ell(0)}
+\sum_{i=1}^j \sum_{p=1}^{N_{i,k}}\sum_{\iota = -2}^{2k}  \|\mathfrak{D}^{2j+2+\iota} \check A_{\ell,i,p,k}\|_{\infty, \check Q_{1},\check g_\ell(0)}\leq C,
\end{split}
\end{equation}
\begin{equation}
\begin{split}
&[\mathfrak{D}^{2j} \ddbar \check \varphi_\ell]_{\a,\a/2, \check Q_{1},\check g_\ell(0)}+[\mathfrak{D}^{2j} \ddbar \underline{\check \varphi_\ell}]_{\a,\a/2, \check Q_{1},\check g_\ell(0)}+[\mathfrak{D}^{2j} (\partial_{\check t} \underline{\check \varphi_\ell}+e^{-t_\ell} \underline{\check \varphi_\ell}    ) ]_{\a,\a/2, \check Q_{\delta_\ell},\check g_\ell(0)}\\
&+    [\mathfrak{D}^{2j+2} \check \psi_{\ell,j,k}   ]_{\a,\a/2, \check Q_{1},\check g_\ell(0)}
+\sum_{i=1}^j \sum_{p=1}^{N_{i,k}}\sum_{\iota = -2}^{2k} \left[\mathfrak{D}^{2j+2+\iota} \check A_{\ell,i,p,k}\right]_{\a,\a/2, \check Q_{1},\check g_\ell(0)}\leq C,
\end{split}
\end{equation}
where  $\check Q_{1}$ is either $\check Q_{1}(\check z_\ell,0)$ or $\check Q_{1}(\check z_\ell',\check t_\ell')$.
Transferring these back to the hat picture gives
\begin{equation}\label{case1 Control-sup}
\begin{split}
&\|\mathfrak{D}^{2j} \ddbar \hat \varphi_\ell\|_{\infty,\hat Q_{\delta_\ell},\hat g_\ell(0)}+\|\mathfrak{D}^{2j} \ddbar \underline{\hat \varphi_\ell}\|_{\infty,\hat Q_{\delta_\ell},\hat g_\ell(0)}+\|\mathfrak{D}^{2j} (\partial_{\hat t} \underline{\hat \varphi_\ell}+\lambda_\ell^{-2} \underline{\hat \varphi_\ell}    )\|_{\infty,\hat Q_{\delta_\ell},\hat g_\ell(0)}\\
&+  \|\mathfrak{D}^{2j+2}  \hat \psi_{\ell,j,k}  \|_{\infty, \hat Q_{\delta_\ell},\hat g_\ell(0)}
+\sum_{i=1}^j \sum_{p=1}^{N_{i,k}}\sum_{\iota = -2}^{2k} \delta_\ell^\iota \|\mathfrak{D}^{2j+2+\iota} \hat A_{\ell,i,p,k}\|_{\infty, \hat Q_{\delta_\ell},\hat g_\ell(0)}\leq C\delta_\ell^{-2j},
\end{split}
\end{equation}
\begin{equation}\label{case1 Control}
\begin{split}
&[\mathfrak{D}^{2j} \ddbar \hat \varphi_\ell]_{\a,\a/2, \hat Q_{\delta_\ell},\hat g_\ell(0)}+ [\mathfrak{D}^{2j} \ddbar \underline{\hat \varphi_\ell}]_{\a,\a/2, \hat Q_{\delta_\ell},\hat g_\ell(0)}+[\mathfrak{D}^{2j} (\partial_{\hat t} \underline{\hat \varphi_\ell}+\lambda_\ell^{-2} \underline{\hat \varphi_\ell}    ) ]_{\a,\a/2, \hat Q_{\delta_\ell},\hat g_\ell(0)}\\
&+    [\mathfrak{D}^{2j+2} \hat \psi_{\ell,j,k}   ]_{\a,\a/2, \hat Q_{\delta_\ell},\hat g_\ell(0)}
+\sum_{i=1}^j \sum_{p=1}^{N_{i,k}}\sum_{\iota = -2}^{2k} \delta_\ell^\iota\left[\mathfrak{D}^{2j+2+\iota} \hat A_{\ell,i,p,k}\right]_{\a,\a/2, \hat Q_{\delta_\ell},\hat g_\ell(0)}\leq C\delta_\ell^{-2j-\a},
\end{split}
\end{equation}
where $\hat Q_{\delta_\ell}$ is either $\hat Q_{\delta_\ell}(\hat z_\ell,0)$ or $\hat Q_{\delta_\ell}(\hat z_\ell',\hat t_\ell')$. Using \eqref{case1 Control-sup} and the triangle inequality (and the usual bound on the operator norm of $\mathbb{P}$) we obtain a uniform upper bound for the numerators of \eqref{sup-realized-10} and \eqref{sup-realized-1} and so for $j\geq 1$ we conclude that
\begin{equation}\label{eskape}d^{\hat g_\ell(0)}(\hat x_\ell,\hat x_\ell') +|\hat t_\ell'|^{\frac{1}{2}}\leq C\delta_\ell^{-2j},\end{equation}
so that the two points $(\hat z_\ell,\hat t_\ell)$ and $(\hat z_\ell',\hat t_\ell')$ are approaching each other (we already know this when $j=0$ thanks to \eqref{gecko}). Thus $(\hat x_\ell',\hat t_\ell')\in \hat Q_{\delta_\ell}(\hat z_\ell,0)$ for all $\ell$ large, and so applying \eqref{case1 Control} shows that the quantities in \eqref{sup-realized-10} and \eqref{sup-realized-1}, which both equal $1$, are also bounded above by $C\delta_\ell^{-2j-\a}\to 0$, which is a contradiction.

\subsection{Blowup analysis in Case 2: $\delta_\ell\to \delta_\infty\in (0,+\infty)$}

Without loss of generality, we can assume $\delta_\infty=1$. In this case, the blowup model is $\mathbb{C}^m\times Y$ and
\begin{equation}\hat g_\ell^\natural(\hat t)\to g_{\mathrm{can}}|_{z=0}+g_{Y,z_\infty=0}=g_P,\end{equation}
as $\ell\to +\infty$ in $C^\infty_{\rm loc}(\mathbb{C}^m\times Y\times (-\infty,0])$. Moreover, the complex structure also converges to a product.

As in the {\bf Case 1}, Lemma \ref{lma:earlier work} (i) implies that $\hat \omega_\ell^\bullet(\hat t)$ is locally uniformly equivalent to product metric on $\mathbb{C}^m\times Y$. Moreover, since $\hat\omega_\ell^\bullet$ satisfies the \KR flow equation
\begin{equation}\partial_{\hat t} \hat\omega_\ell^\bullet=-\Ric(\hat\omega_\ell^\bullet) -\lambda_\ell^{-2}\hat\omega_\ell^\bullet,\end{equation}
we can again apply \cite[Proposition 2.1]{CL} on $\hat Q_R(\hat z_\ell,0)$ for $R$ sufficiently large to obtain $C^\infty_{\rm loc}$ regularity of $\hat\omega_\ell^\bullet$. Arguing as in {\bf Case 1} we obtain $C^\infty$ bounds for all the pieces in the decomposition, and using Proposition \ref{prop:non-escaping},  we can assume $(\hat x_\ell',\hat t_\ell')\to (\hat x_\infty',\hat t_\infty')\in (\mathbb{C}^m\times Y)\times (-\infty,0]$. Estimating the $C^\alpha$ difference quotients in \eqref{sup-realized-1} by $C^\beta$ ones for any $\beta>\alpha$, we see that $d^{\hat g_\ell(0)}(\hat x_\ell,\hat x_\ell')+|\hat t_\ell'|^{\frac{1}{2}}\geq C^{-1}$ for all $\ell$, which when $j=0$ is already a contradiction to \eqref{gecko}.

Assuming then that $j\geq 1$, we see that  the limit $(\hat x_\infty,0)$  of $(\hat x_\ell,0)$ is different from $(\hat x_\infty',\hat t_\infty')$.  By the local uniform higher order regularity,  the geometric quantities smoothly subconverge as $\ell\to +\infty$. In particular, the limit $\hat{\omega}^\bullet_\infty(\hat{t})$  is an ancient solution on $\mathbb{C}^m\times Y \times (-\infty,0]$ of the \KR flow
\begin{equation}\label{lkk}\partial_{\hat t} \hat\omega_\infty^\bullet=-\Ric(\hat\omega_\infty^\bullet).\end{equation}
Smooth convergence also implies that \eqref{sup-realized-1} still holds in the limit.

By \eqref{hat-A-estimate}, we have that $\hat A_{\ell,i,p,k}\to 0$ locally uniformly, hence locally smoothly, so its contributions to  \eqref{sup-realized-1} vanish in the limit as  $\ell\to+\infty$.  On the other hand, Lemma \ref{lma:earlier work} (ii),(iii) also implies that the limits of $\ddbar\underline{\hat \varphi_\ell}$ and $\partial_{\hat t}{\hat \varphi_\ell}+\lambda_\ell^{-2}{\hat \varphi_\ell}$ (and hence also of $\partial_{\hat t}\underline{\hat \varphi_\ell}+\lambda_\ell^{-2}\underline{\hat \varphi_\ell}$) vanish so that their contributions in \eqref{sup-realized-1} also vanish in the limit.

We are left with killing the limit of the contribution of $\hat\psi_{\ell,j,k}$.
For this, recall that  from Lemma \ref{lma:earlier work} (iv) we have $|R(g^\bullet(t))|\leq C$, and so $|R(\hat{g}^\bullet(\hat{t}))|\leq C\lambda_\ell^{-2} \to 0$, thus the limiting metrics $\hat{\omega}^\bullet_\infty(\hat{t})$ are scalar-flat, hence Ricci-flat and static (using the well-known evolution equation of the scalar curvature under the \KR flow \eqref{lkk}). Also, since $\hat A_{\ell,i,p,k}\to 0$ locally smoothly, this implies that $\hat{\gamma}_{\ell,i,k}\to 0, 1\leq i\leq j,$ and so $\hat\omega_\infty^\bullet=\omega_P+\hat \eta_{\infty,j,k}$, and in particular $\hat\eta_{\infty,j,k}$ is also static.  The Liouville Theorem from \cite{He} shows that $\nabla^{g_P}\hat\eta_{\infty,j,k}=\nabla^{g_P}\hat\omega_\infty^\bullet\equiv 0$. Thus $\hat{\eta}_{\infty,j,k}=\ddbar\hat{\psi}_{\infty,j,k}$ is parallel, with bounded $g_P$ norm (from Lemma \ref{lma:earlier work} (i)), so by \cite[Proposition 3.12]{HT2} we have $\ddbar\hat{\psi}_{\infty,j,k}=\ddbar p$ for some quadratic polynomial $p$ on $\C^m$. This means that $\hat{\psi}_{\infty,j,k}-p$ is pluriharmonic on $\C^m\times Y$, and hence it is also pulled back from $\C^m$ (since $Y$ is compact). This clearly implies that $\hat{\psi}_{\infty,j,k}$ is pulled back from $\C^m$, and since it also has fiberwise average zero, it must vanish identically. Thus the contribution of $\hat \psi_{\ell,j,k}$ to \eqref{sup-realized-1} also vanish in the limit, and this gives a contradiction to \eqref{sup-realized-1}.

\subsection{Blowup analysis in Case 3: $\delta_\ell\to  0$}
In this case, the blowup model is $\mathbb{C}^m$ which is still collapsed.  This is the most difficult case and will occupy most of the rest of the paper.

By Proposition \ref{prop:non-escaping}, we know that $(\hat x_\ell',\hat t_\ell')$ remains at bounded distance from $(\hat x_\ell,0)$.  The key claim is the following non-colliding estimate, whose proof will take substantial work:
\begin{claim}\label{claim:noncoll}
There exists $\e_0>0$ such that for all $\ell\geq 1$ we have
\begin{equation}\label{kirk}d^{\hat g_\ell(0)}(\hat x_\ell,\hat x_\ell')+|\hat t_\ell'|^{\frac{1}{2}}\geq \e_0>0.\end{equation}
\end{claim}
First we show how to quickly complete the proof of Theorem \ref{thm:decomposition} assuming Claim \ref{claim:noncoll} holds. When $j=0$ it is clear that \eqref{kirk} is incompatible with \eqref{gecko}, while if $j\geq 1$ then \eqref{sup-realized-1} implies
\begin{equation}
\begin{split}
\e_0^\a &\leq {|\mathfrak{D}^{2j}\ddbar\underline{\hat \varphi_\ell}(\hat x_\ell,0)-\P_{\hat x_\ell'\hat x_\ell}\mathfrak{D}^{2j}\ddbar\underline{\hat\varphi_\ell}(\hat x_\ell',\hat t_\ell')|_{\hat g_\ell(0)}}\\
&\quad +{|\mathfrak{D}^{2j} (\partial_{\hat t}\underline{\hat \varphi_\ell}+\lambda_\ell^{-2}\underline{\hat \varphi_\ell} )(\hat x_\ell,0)-\P_{\hat x_\ell'\hat x_\ell}\mathfrak{D}^{2j}(\partial_{\hat t}\underline{\hat \varphi_\ell}+\lambda_\ell^{-2}\underline{\hat \varphi_\ell} )(\hat x_\ell',\hat t_\ell')|_{\hat g_\ell(0)}}\\
&\quad +{|\mathfrak{D}^{2j+2} \hat \psi_{\ell,j,k}(\hat x_\ell,0)-\P_{\hat x_\ell'\hat x_\ell}\mathfrak{D}^{2j+2} \hat \psi_{\ell,j,k}(\hat x_\ell',\hat t_\ell')|_{\hat g_\ell(0)}}\\
&\quad +\sum_{i=1}^j \sum_{p=1}^{N_{i,k}}\sum_{\iota = -2}^{2k} \delta_\ell^\iota\Bigg({|\mathfrak{D}^{2j+2+\iota} \hat A_{\ell,i,p,k}(\hat x_\ell,\hat t_\ell)- \P_{\hat x_\ell'\hat x_\ell}(\mathfrak{D}^{2j+2+\iota}\hat A_{\ell,i,p,k}(\hat x_\ell',\hat t_\ell'))|_{\hat g_\ell(0)}}\Bigg),
\end{split}
\end{equation}
while the right hand side is of $o(1)$ thanks to \eqref{dist:A-contr}, \eqref{dist:psi-contr}, \eqref{dist:varphi-contr} and \eqref{dist:varphi2-contr} as $\ell\to +\infty$, since $\delta_\ell\to 0$.  This is a contradiction.  Therefore, to complete the proof of Theorem \ref{thm:decomposition}, it remains to prove Claim \ref{claim:noncoll}.

\subsection{Setup of secondary blowup in proving Claim \ref{claim:noncoll}}
For each $\ell\geq 1$, let
\begin{equation}
d_\ell:=d^{\hat g_\ell(0)}(\hat x_\ell,\hat x_\ell')+|\hat t_\ell'|^{\frac{1}{2}}>0.
\end{equation}
If Claim \ref{claim:noncoll} fails to be true, then we may assume that $d_\ell\to 0$ as $\ell\to +\infty$.  Define a new parameter
\begin{equation}
\e_\ell:=d_\ell^{-1}\delta_\ell= d_\ell^{-1}\lambda_\ell e^{-\frac{t_\ell}{2}},
\end{equation}
and consider the diffeomorphism
\begin{equation}
\Theta_\ell:  B_{d_\ell^{-1}\lambda_\ell} \times Y\times [-d_\ell^{-2}\lambda_\ell^2t_\ell,0] \to B_{\lambda_\ell}\times Y\times [-\lambda_\ell^2t_\ell,0], \;\, (\hat{z},\hat{y},\hat{t}) = \Theta_\ell(\ti{z},\ti{y},\ti{t}) =(d_\ell\ti{z},\ti{y},d_\ell^2\ti{t}),
\end{equation}
As usual, we pull back time-dependent $2$-tensors via $\Theta_\ell$, rescale them by $d_\ell^{-2}$ and denote the new tensors with a tilde, so for example
$\ti{\omega}^\bullet_\ell(\ti{t})=d_\ell^{-2} \Theta_\ell^*\hat\omega_\ell^\bullet(d_\ell^2 \ti{t})$.
We also apply the same  pullback and rescaling procedure to the scalar functions $\hat{A}_{\ell,i,p,k}$ and $\hat{\vp}_\ell$, so $\tilde A_{\ell,i,p,k}(\tilde t)=d_\ell^{-2}\Theta_\ell^* \hat A_{\ell,i,p,k}(d_\ell^2\tilde t), \ti{\vp}_\ell(\ti{t})=d_\ell^{-2}\Theta_\ell^*\hat{\vp}_\ell(d_\ell^2\ti{t})$.

The decomposition \eqref{gorgulu} becomes
\begin{equation}\label{gorgulu2}
\ti{\omega}^\bullet_\ell=\ti{\omega}^\natural_\ell+\ti{\gamma}_{\ell,0}+\ti{\gamma}_{\ell,1,k}+\dots+\ti{\gamma}_{\ell,j,k}+\ti{\eta}_{\ell, j,k},
\end{equation}
and the parabolic
complex Monge-Amp\`ere equation \eqref{mak} becomes
\begin{equation}\label{PMA-tilde}
(\ti \omega_\ell^\bullet)^{m+n}= e^{\de_{\ti{t}}{\ti\varphi}_\ell+d_\ell^2\lambda_\ell^{-2}\ti\varphi_\ell-n d_\ell^2\lambda_\ell^{-2}\ti t} \binom{m+n}{m} \ti \omega_{\ell,\mathrm{can}} ^m \wedge (\e_\ell^2 \Theta_\ell^*\Psi_\ell^* \omega_F)^n.
\end{equation}
From \eqref{a-priori Control00} (when $j=0$) and \eqref{a-priori Control} (when $j\geq 1$),  we immediately see that for any fixed $R>0$ and $\tilde t\in [-\lambda_\ell^2d_\ell^{-2}t_\ell,0]$ we have
\begin{equation}\label{a-priori Control-20}
[\ddbar \underline{\tilde \varphi_\ell}]_{\a,\a/2,\tilde Q_{Rd_\ell^{-1}}(\tilde t),\tilde g_\ell(\tilde t)}
+[\ddbar\ti{\psi}_{\ell,0,k}]_{\a,\a/2,\tilde Q_{Rd_\ell^{-1}}(\tilde t),\tilde g_\ell(\tilde t)}\leq Cd_\ell^{\a},
\end{equation}
if $j=0$, and
\begin{equation}\label{a-priori Control-2}
\begin{split}
&\quad [\mathfrak{D}^{2j} \ddbar \underline{\tilde \varphi_\ell}]_{\a,\a/2,\tilde Q_{Rd_\ell^{-1}}(\tilde t),\tilde g_\ell(\tilde t)}+ [\mathfrak{D}^{2j} (\partial_{\tilde t}\underline{\tilde \varphi_\ell}+d_\ell^2\lambda_\ell^{-2}\underline{\tilde \varphi_\ell})]_{\a,\a/2,\tilde Q_{Rd_\ell^{-1}}(\tilde t),\tilde g_\ell(\tilde t)}\\
& + [\mathfrak{D}^{2j+2} \tilde \psi_{\ell,j,k} ]_{\a,\a/2,\tilde Q_{Rd_\ell^{-1}}(\tilde t),\tilde g_\ell(\tilde t)}+\sum_{i=1}^j \sum_{p=1}^{N_{i,k}}\sum_{\iota = -2}^{2k} \e_\ell^\iota\left[\mathfrak{D}^{2j+2+\iota} \tilde A_{\ell,i,p,k}\right]_{\a,\a/2,\tilde Q_{Rd_\ell^{-1}}(\tilde t),\tilde g_\ell(\tilde t)}\leq C d_\ell^{2j+\a},
\end{split}
\end{equation}
when $j\geq 1$, where $\tilde Q_R(\tilde t)=B_{\mathbb{C}^m}(\tilde z_\ell,R)\times Y\times [-R^2+\tilde t,\tilde t]$. Moreover, \eqref{sup-realized-10} and \eqref{sup-realized-1} become respectively
\begin{equation}\label{sup-realized-20}
1=d_\ell^{-\alpha}\frac{|\ddbar \tilde\varphi_\ell(\tilde x_\ell,0)- \P_{\tilde x'_\ell\tilde x_\ell}\ddbar\tilde\varphi_\ell(\tilde x'_\ell,\tilde t_\ell')|_{\tilde g_\ell(0)}}{(d^{\tilde g_\ell(0)}(\tilde x_\ell,\tilde x'_\ell) +|\tilde t_\ell'|^{\frac{1}{2}})^\a},
\end{equation}
for $j=0$ and
\begin{equation}\label{sup-realized-2}
\begin{split}
1&=d_\ell^{-2j-\a}\frac{|\mathfrak{D}^{2j} \ddbar \underline{\tilde\varphi_\ell}(\tilde x_\ell,0)- \P_{\tilde x'_\ell\tilde x_\ell}\mathfrak{D}^{2j}\ddbar \underline{\tilde\varphi_\ell}(\tilde x'_\ell,\tilde t_\ell')|_{\tilde g_\ell(0)}}{(d^{\tilde g_\ell(0)}(\tilde x_\ell,\tilde x'_\ell) +|\tilde t_\ell'|^{\frac{1}{2}})^\a}\\
&+d_\ell^{-2j-\a}\frac{|\mathfrak{D}^{2j}  (\partial_{\tilde t}\underline{\tilde\varphi_\ell}+d_\ell^2\lambda_\ell^{-2} \underline{\tilde\varphi_\ell})(\tilde x_\ell,0)- \P_{\tilde x'_\ell\tilde x_\ell}\mathfrak{D}^{2j}(\partial_{\tilde t}\underline{\tilde\varphi_\ell}+d_\ell^2\lambda_\ell^{-2} \underline{\tilde\varphi_\ell})(\tilde x'_\ell,\tilde t_\ell')|_{\tilde g_\ell(0)}}{(d^{\tilde g_\ell(0)}(\tilde x_\ell,\tilde x'_\ell) +|\tilde t_\ell'|^{\frac{1}{2}})^\a}\\
&+d_\ell^{-2j-\a}\frac{|\mathfrak{D}^{2j+2}  \tilde\psi_{\ell,j,k}(\tilde x_\ell,0)- \P_{\tilde x'_\ell\tilde x_\ell}\mathfrak{D}^{2j+2}  \tilde\psi_{\ell,j,k}(\tilde x'_\ell,\tilde t_\ell')|_{\tilde g_\ell(0)}}{(d^{\tilde g_\ell(0)}(\tilde x_\ell,\tilde x'_\ell) +|\tilde t_\ell'|^{\frac{1}{2}})^\a}\\
&+d_\ell^{-2j-\a}\sum_{i=1}^j \sum_{p=1}^{N_{i,k}}\sum_{\iota = -2}^{2k} \e_\ell^\iota\Bigg(\frac{|\mathfrak{D}^{2j+2+\iota} \tilde A_{\ell,i,p,k}(\tilde x_\ell,0)- \P_{\tilde x_\ell'\tilde x_\ell}(\mathfrak{D}^{2j+2+\iota}\tilde A_{\ell,i,p,k}(\tilde x_\ell',\tilde t_\ell'))|_{\tilde g_\ell(0)}}{(d^{\tilde g_{\ell}(0)}(\tilde x_\ell,\tilde x'_\ell) +|\tilde t_\ell'|^{\frac{1}{2}})^\a}\Bigg),
\end{split}
\end{equation}
for $j\geq 1$, and by definition we also have
\begin{equation}\label{beaner}
d^{\tilde g_\ell(0)}(\tilde x_\ell,\tilde x_\ell')+|\tilde t'_\ell|^{\frac{1}{2}}=1,
\end{equation}
for all $\ell>0$.

For convenience, from now on $\tilde Q_r$ will always denote $B_{\mathbb{C}^m}(0,r)\times Y\times [-r^2,0]$, where recall that we have translated the first blowup point $(\ti{x}_\ell,\ti{t}_\ell)$ in the $\C^m$ directions so that $(\tilde z_\ell,\ti{t}_\ell)=(0,0)\in \mathbb{C}^m\times \mathbb{R}$. As long as $r>1$, the cylinder $\ti{Q}_r$ always contains the other blowup point $(\ti{x}'_\ell,\ti{t}'_\ell)$, because of \eqref{beaner}.  Following \cite[\S 4.8]{HT3}, our goal is to obtain a contradiction by passing to the limit as $\ell\to+\infty$ the various pieces of the decomposition, after scaling them by $d_\ell^{-2j-\a}$. To do this, we need to perform a jet subtraction centered at $(0,0)\in \mathbb{C}^m\times (-\infty,0]$ for functions pulled back from the base. Given $r\in\mathbb{N}$, the parabolic $r$-jet at $(0,0)$ of a function $u$ in $\mathbb{C}^m\times (-\infty,0]$ given by
\begin{equation}
u^\sharp:=\sum_{|p|+2q\leq r} \D^p\partial_{\ti{t}}^q  u |_{(0,0)} \frac{\ti{z}^p}{p!}\frac{{\ti{t}}^q}{q!},
\end{equation}
using standard multiindex notation, where here we treat $\ti{z}$ as real variables. The parabolic degree of such a polynomial is defined by letting the $\ti{z}$ variables have degree $1$ and the ${\ti{t}}$ variable degree $2$. Thus, the degree of $u^\sharp$ is at most $r$. We will also write $u^*:=u-u^\sharp$, whose $r$-jet at $(0,0)$ thus vanishes.

With this notation, we define $\tilde A_{\ell,i,p,k}^\sharp$ as the parabolic $2j$-jet of $\tilde A_{\ell,i,p,k}$ at $(0,0)$, and define $\tilde A_{\ell,i,p,k}^*=\tilde A_{\ell,i,p,k}-\tilde A_{\ell,i,p,k}^\sharp$ so that $\D^{p} \partial_{\ti{t}}^q \tilde A_{\ell,i,p,k}^*|_{(0,0)}=0$ for any $p+2q\leq 2j$.
As for the potential $\underline{\tilde\varphi_\ell}$, since the PDE \eqref{PMA-tilde} and the blowup quantity in \eqref{a-priori Control-2} (for $j\geq 1$) both contain the term $\de_{\ti{t}}{\ti\varphi}_\ell+d_\ell^2\lambda_\ell^{-2}\ti\varphi_\ell$, it will be more convenient to consider instead
\begin{equation}
\tilde\chi_{\ell}:=e^{d_\ell^2\lambda_\ell^{-2}\tilde t} {\tilde\varphi_\ell},
\end{equation}
so that
\begin{equation}
\de_{\ti{t}}\tilde\chi_{\ell}=e^{d_\ell^2\lambda_\ell^{-2}\tilde t}(\de_{\ti{t}}{\ti\varphi}_\ell+d_\ell^2\lambda_\ell^{-2}\ti\varphi_\ell).
\end{equation}
We claim that the fiber average $\underline{\tilde\chi_{\ell}}$ satisfies estimates similar to those satisfied by $ \underline{\tilde\varphi_\ell}$ in \eqref{a-priori Control-20} and \eqref{a-priori Control-2}:
\begin{claim}\label{sub-claim-phi-to-varphi}
For any fixed $R>0$, we have
\begin{align}
d_\ell^{-2j-\a}\left|[\mathfrak{D}^{2j} \partial_{\tilde t}\underline{\tilde \chi_\ell}]_{\a,\a/2,\tilde Q_{Rd_\ell^{-1}},\tilde g_\ell(0)}-
[\mathfrak{D}^{2j}  \left( \partial_{\tilde t}\underline{\tilde\varphi_\ell}+d_\ell^2\lambda_\ell^{-2}\underline{\tilde\varphi_\ell}  \right) ]_{\a,\a/2,\tilde Q_{Rd_\ell^{-1}},\tilde g_\ell(0)}\right|&=o(1),\label{r1}\\
d_\ell^{-2j-\a}\left|[\mathfrak{D}^{2j} \ddbar \underline{ \tilde \chi_\ell}]_{\a,\a/2,\tilde Q_{Rd_\ell^{-1}},\tilde g_\ell(0)}-[\mathfrak{D}^{2j} \ddbar \underline{ \tilde \vp_\ell}]_{\a,\a/2,\tilde Q_{Rd_\ell^{-1}},\tilde g_\ell(0)}\right|&=o(1),\label{r2}
\end{align}
as $\ell\to+\infty$.
\end{claim}
\begin{proof}[Proof of Claim \ref{sub-claim-phi-to-varphi}]
We first treat $\partial_{\tilde t}\underline{\tilde\chi_\ell}$, and since this will only be used when $j\geq 1$, we only prove it here in this case (but see \eqref{wulpo} and \eqref{wulpo2} below for a stronger statement when $j=0$). Observe first that for any $r\in\mathbb{N}$ we have
\begin{equation}
\left\|\de_{\ti{t}}^r \left(e^{d_\ell^2\lambda_\ell^{-2}\tilde t}\right)\right\|_{\infty,\tilde Q_{Rd_\ell^{-1}}}=(1+o(1))(d_\ell\lambda_\ell^{-1})^{2r}.
\end{equation}
Using this, we bound
\begin{equation}
\begin{split}
&\left|[\mathfrak{D}^{2j} \partial_{\tilde t}\underline{\tilde \chi_\ell}]_{\a,\a/2,\tilde Q_{Rd_\ell^{-1}},\tilde g_\ell(0)}-
[\mathfrak{D}^{2j}  \left( \partial_{\tilde t}\underline{\tilde\varphi_\ell}+d_\ell^2\lambda_\ell^{-2}\underline{\tilde\varphi_\ell}  \right) ]_{\a,\a/2,\tilde Q_{Rd_\ell^{-1}},\tilde g_\ell(0)}\right|\\
&\leq \left|e^{d_\ell^2\lambda_\ell^{-2}\tilde t}[\mathfrak{D}^{2j}  \left( \partial_{\tilde t}\underline{\tilde\varphi_\ell}+d_\ell^2\lambda_\ell^{-2}\underline{\tilde\varphi_\ell}  \right) ]_{\a,\a/2,\tilde Q_{Rd_\ell^{-1}},\tilde g_\ell(0)}-
[\mathfrak{D}^{2j}  \left( \partial_{\tilde t}\underline{\tilde\varphi_\ell}+d_\ell^2\lambda_\ell^{-2}\underline{\tilde\varphi_\ell}  \right) ]_{\a,\a/2,\tilde Q_{Rd_\ell^{-1}},\tilde g_\ell(0)}\right|\\
&+\sum_{p+2q=2j}\sum_{r=1}^q \left\|\de_{\ti{t}}^r \left(e^{d_\ell^2\lambda_\ell^{-2}\tilde t}\right)\right\|_{\infty,\tilde Q_{Rd_\ell^{-1}}}  [\D^p\partial_{\tilde t}^{q-r}  \left( \partial_{\tilde t}\underline{\tilde\varphi_\ell}+d_\ell^2\lambda_\ell^{-2}\underline{\tilde\varphi_\ell}  \right) ]_{\a,\a/2,\tilde Q_{Rd_\ell^{-1}},\tilde g_\ell(0)}\\
&+ \sum_{p+2q=2j}\sum_{r=0}^q \left[\de_{\ti{t}}^r \left(e^{d_\ell^2\lambda_\ell^{-2}\tilde t}\right)\right]_{\a,\a/2,\tilde Q_{Rd_\ell^{-1}},\tilde g_\ell(0)}  \|\D^p\partial_{\tilde t}^{q-r}  \left( \partial_{\tilde t}\underline{\tilde\varphi_\ell}+d_\ell^2\lambda_\ell^{-2}\underline{\tilde\varphi_\ell}  \right) \|_{\infty,\tilde Q_{Rd_\ell^{-1}},\tilde g_\ell(0)}\\
&\leq o(1)\left|[\mathfrak{D}^{2j}  \left( \partial_{\tilde t}\underline{\tilde\varphi_\ell}+d_\ell^2\lambda_\ell^{-2}\underline{\tilde\varphi_\ell}  \right) ]_{\a,\a/2,\tilde Q_{Rd_\ell^{-1}},\tilde g_\ell(0)}\right|\\
&+C\sum_{r=1}^j (d_\ell\lambda_\ell^{-1})^{2r} [\mathfrak{D}^{2j-2r}  \left( \partial_{\tilde t}\underline{\tilde\varphi_\ell}+d_\ell^2\lambda_\ell^{-2}\underline{\tilde\varphi_\ell}  \right) ]_{\a,\a/2,\tilde Q_{Rd_\ell^{-1}},\tilde g_\ell(0)}\\
&+ C\sum_{r=0}^j (d_\ell\lambda_\ell^{-1} )^{2r+\a} \|\mathfrak{D}^{2j-2r}  \left( \partial_{\tilde t}\underline{\tilde\varphi_\ell}+d_\ell^2\lambda_\ell^{-2}\underline{\tilde\varphi_\ell}  \right) \|_{\infty,\tilde Q_{Rd_\ell^{-1}},\tilde g_\ell(0)}.
\end{split}
\end{equation}
Thanks to \eqref{a-priori Control-2}, the third line from the bottom is $o(d_\ell^{2j+\a})$. As for the last two lines, we interpolate between \eqref{a-priori Control-2} and the estimate
 $\|  \partial_{\tilde t}\underline{\tilde\varphi_\ell}+d_\ell^2\lambda_\ell^{-2}\underline{\tilde\varphi_\ell} \|_{\infty,\tilde Q_{Rd_\ell^{-1}}}=o(1)$ which comes from Lemma \ref{lma:earlier work} (ii), and we get for each $1\leq \iota\leq 2j$,
\begin{equation}
\begin{split}
&\quad  d_\ell^{-\iota}(R-\rho)^{\iota} \|\mathfrak{D}^{\iota}  \left( \partial_{\tilde t}\underline{\tilde\varphi_\ell}+d_\ell^2\lambda_\ell^{-2}\underline{\tilde\varphi_\ell}  \right) \|_{\infty,\tilde Q_{\rho d_\ell^{-1}},\tilde g_\ell(0)}\leq  C(R-\rho)^{2j+\a} +o(1).
 \end{split}
\end{equation}
By choosing $R-\rho\approx 1$ (which is clearly allowed) and replacing $R$ by a slightly smaller one,  we see that for all $1\leq \iota\leq 2j$,
\begin{equation}
\|\mathfrak{D}^{\iota}  \left( \partial_{\tilde t}\underline{\tilde\varphi_\ell}+d_\ell^2\lambda_\ell^{-2}\underline{\tilde\varphi_\ell}  \right) \|_{\infty,\tilde Q_{R d_\ell^{-1}},\tilde g_\ell(0)}\leq   C d_\ell^{\iota}.
\end{equation}We can then estimate
\begin{equation}\label{phi-to-varphi-mixed-term-1}
\begin{split}
&\quad \sum_{r=0}^j (d_\ell\lambda_\ell^{-1} )^{2r+\a} \|\mathfrak{D}^{2j-2r}  \left( \partial_{\tilde t}\underline{\tilde\varphi_\ell}+d_\ell^2\lambda_\ell^{-2}\underline{\tilde\varphi_\ell}  \right) \|_{\infty,\tilde Q_{Rd_\ell^{-1}},\tilde g_\ell(0)}\\
&\leq C\sum_{r=0}^j (d_\ell\lambda_\ell^{-1} )^{2r+\a}  d_\ell^{2j-2r}=C \lambda_\ell^{-\a} d_\ell^{2j+\a}=o(d_\ell^{2j+\a}),
\end{split}
\end{equation}
\begin{equation}\label{phi-to-varphi-mixed-term-2}
\begin{split}
&\quad \sum_{r=1}^j (d_\ell\lambda_\ell^{-1})^{2r} [\mathfrak{D}^{2j-2r}  \left( \partial_{\tilde t}\underline{\tilde\varphi_\ell}+d_\ell^2\lambda_\ell^{-2}\underline{\tilde\varphi_\ell}  \right) ]_{\a,\a/2,\tilde Q_{Rd_\ell^{-1}},\tilde g_\ell(0)}\\
&\leq C \sum_{r=1}^j (d_\ell\lambda_\ell^{-1})^{2r}  d_\ell^{2j-2r+\a}= C\lambda_\ell^{-2} d_\ell^{2j+\a}=o(d_\ell^{2j+\a}).
\end{split}
\end{equation}
Putting these together proves \eqref{r1}.

Next, we treat $\ddbar\underline{ \tilde\chi_\ell}$ in a similar fashion (allowing now also $j=0$),
\begin{equation}
\begin{split}
&\left|
[\DD^{2j} \ddbar \underline{\tilde\chi_\ell}]_{\a,\a/2,\tilde Q_{Rd_\ell^{-1}},\tilde g_\ell(0)}
-[\DD^{2j}\ddbar\underline{\tilde\vp_\ell}]_{\a,\a/2,\tilde Q_{Rd_\ell^{-1}},\tilde g_\ell(0)}\right|\\
&\leq \left|e^{d_\ell^2\lambda_\ell^{-2}\tilde t}[\DD^{2j} \ddbar \underline{\tilde\vp_\ell}]_{\a,\a/2,\tilde Q_{Rd_\ell^{-1}},\tilde g_\ell(0)}
-[\DD^{2j}\ddbar\underline{\tilde\vp_\ell}]_{\a,\a/2,\tilde Q_{Rd_\ell^{-1}},\tilde g_\ell(0)}\right|\\
&+\sum_{p+2q=2j}\sum_{r=1}^q  \left\|\de_{\ti{t}}^r \left(e^{d_\ell^2\lambda_\ell^{-2}\tilde t}\right)\right\|_{\infty,\tilde Q_{Rd_\ell^{-1}}}  [\D^p\partial_{\tilde t}^{q-r}  \ddbar \underline{\tilde\varphi_\ell}]_{\a,\a/2,\tilde Q_{Rd_\ell^{-1}},\tilde g_\ell(0)}\\
&+\sum_{p+2q=2j}\sum_{r=0}^q \left[\de_{\ti{t}}^r \left(e^{d_\ell^2\lambda_\ell^{-2}\tilde t}\right)\right]_{\a,\a/2,\tilde Q_{Rd_\ell^{-1}},\tilde g_\ell(0)} \|\D^p\partial_{\tilde t}^{q-r}  \ddbar \underline{\tilde\varphi_\ell}\|_{\infty,\tilde Q_{Rd_\ell^{-1}},\tilde g_\ell(0)}\\
&\leq o(1)\left|[\DD^{2j}\ddbar\underline{\tilde\vp_\ell}]_{\a,\a/2,\tilde Q_{Rd_\ell^{-1}},\tilde g_\ell(0)}\right|
+C\sum_{r=1}^j  (d_\ell\lambda_\ell^{-1})^{2r}   [\DD^{2j-2r}  \ddbar \underline{\tilde\varphi_\ell}]_{\a,\a/2,\tilde Q_{Rd_\ell^{-1}},\tilde g_\ell(0)}\\
&+C\sum_{r=0}^j (d_\ell\lambda_\ell^{-1} )^{2r+\a}  \|\DD^{2j-2r}  \ddbar \underline{\tilde\varphi_\ell}\|_{\infty,\tilde Q_{Rd_\ell^{-1}},\tilde g_\ell(0)},
\end{split}
\end{equation}
and then we continue the argument exactly as above, using the bound \eqref{a-priori Control-20} when $j=0$. This proves \eqref{r2}.
\end{proof}

With this modification in mind,  we let $\underline{ \tilde \chi_\ell^\sharp}$ be the (parabolic) jet of $\underline{ \tilde \chi_\ell}=e^{d_\ell^2\lambda_\ell^{-2}\tilde t}\underline{\tilde\varphi_\ell}$ at $(\tilde z,\tilde t)=(0,0)$  of order $2j+2$, and define $\underline{\tilde\chi_\ell^*}:=\underline{ \tilde\chi_\ell}-\underline{ \tilde\chi_\ell^\sharp}$.   Define also
\begin{equation}\label{def:eta-phi-notvarphi}
\tilde\eta_\ell^\ddagger:=e^{-d_\ell^2\lambda_\ell^{-2}\tilde t}\ddbar \underline{\tilde\chi_\ell^\sharp},\quad\quad \tilde\eta_\ell^\diamond:=e^{-d_\ell^2\lambda_\ell^{-2}\tilde t}\ddbar \underline{  \tilde\chi_\ell^*},
\end{equation}
so that by definition we have
\begin{equation}
\ddbar\underline{\ti{\vp}_\ell}=\tilde\eta_\ell^\ddagger+\tilde\eta_\ell^\diamond.
\end{equation}
Recall that for $1\leq i\leq j$ we had defined
\begin{equation}
\tilde \gamma_{\ell,i,k}=\sum_{p=1}^{N_{i,k}} \ddbar \tilde{\mathfrak{G}}_{\ti{t},k}(\tilde A_{\ell,i,p,k},\tilde G_{\ell,i,p,k} ),
\end{equation}
(see also \eqref{Gtilde} below for an explicit formula).
We shall further denote
\begin{equation}
\tilde\eta_\ell^\circ :=\sum_{i=1}^j \sum_{p=1}^{N_{i,k}}\ddbar \tilde{\mathfrak{G}}_{\ti{t},k}(\tilde A_{\ell,i,p,k}^*,\tilde G_{\ell,i,p,k}),\quad \tilde\eta_\ell^\dagger :=\sum_{i=1}^j \sum_{p=1}^{N_{i,k}}\ddbar \tilde{\mathfrak{G}}_{\ti{t},k}(\tilde A_{\ell,i,p,k}^\sharp,\tilde G_{\ell,i,p,k}),
\end{equation}
so that
\begin{equation}\label{decomposition-tilde-KRF}
\begin{split}
\tilde\omega_\ell^\bullet&=\tilde \omega_\ell^\natural +\ddbar \underline{ \tilde\varphi_{\ell}}+\sum_{i=1}^j\sum_{p=1}^{N_{i,k}}\ddbar   \tilde{\mathfrak{G}}_{\ti{t},k}(\tilde A_{\ell,i,p,k},\tilde G_{\ell,i,p,k})+\ddbar \tilde\psi_{\ell,j,k}\\
&=\tilde \omega_\ell^\natural +\tilde\eta_\ell^\dagger+\tilde\eta_\ell^\ddagger+\tilde\eta_\ell^\circ+\tilde\eta_\ell^\diamond+\tilde\eta_{\ell,j,k},
\end{split}
\end{equation}
and we will write $\tilde\omega_\ell^\sharp:=\tilde \omega_\ell^\natural +\tilde\eta_\ell^\dagger+\tilde\eta_\ell^\ddagger$.

Thanks to these jet subtractions, and to Claim \ref{sub-claim-phi-to-varphi}, from \eqref{a-priori Control-20}, \eqref{a-priori Control-2}, \eqref{sup-realized-2} and \eqref{beaner} we obtain
\begin{equation}\label{a-priori Control-30}
[\ddbar\underline{\tilde\chi_\ell^*}]_{\a,\a/2,\tilde Q_{Rd_\ell^{-1}},\tilde g_\ell(0)}+[\ddbar\tilde \psi_{\ell,0,k} ]_{\a,\a/2,\tilde Q_{Rd_\ell^{-1}},\tilde g_\ell(0)} \leq C d_\ell^{\a},
\end{equation}
when $j=0$, and
\begin{equation}\label{a-priori Control-3}
\begin{split}
&\quad [\mathfrak{D}^{2j} \ddbar\underline{\tilde\chi_\ell^*}]_{\a,\a/2,\tilde Q_{Rd_\ell^{-1}},\tilde g_\ell(0)}+ [\mathfrak{D}^{2j} \partial_{\tilde t}\underline{\tilde\chi_\ell^*}]_{\a,\a/2,\tilde Q_{Rd_\ell^{-1}},\tilde g_\ell(0)}\\
& + [\mathfrak{D}^{2j+2}  \tilde \psi_{\ell,j,k} ]_{\a,\a/2,\tilde Q_{Rd_\ell^{-1}},\tilde g_\ell(0)} +\sum_{i=1}^j \sum_{p=1}^{N_{i,k}}\sum_{\iota = -2}^{2k} \e_\ell^\iota\left[\mathfrak{D}^{2j+2+\iota} \tilde A^*_{\ell,i,p,k}\right]_{\a,\a/2,\tilde Q_{Rd_\ell^{-1}},\tilde g_\ell( 0)}\leq C d_\ell^{2j+\a},
\end{split}
\end{equation}
when $j\geq 1$, as well as
\begin{equation}\label{sup-realized-3}
\begin{split}
& (1+o(1))=d_\ell^{-2j-\a}|\mathfrak{D}^{2j} \ddbar \underline{\tilde\chi_\ell^*}(\tilde x_\ell,0)- \P_{\tilde x'_\ell\tilde x_\ell}\mathfrak{D}^{2j} \ddbar \underline{\tilde\chi_\ell^*}(\tilde x'_\ell,\tilde t_\ell')|_{\tilde g_\ell(0)}\\
&+d_\ell^{-2j-\a}|\mathfrak{D}^{2j}  \partial_{\tilde t} \underline{\tilde\chi_\ell^*}(\tilde x_\ell,0)- \P_{\tilde x'_\ell\tilde x_\ell}\mathfrak{D}^{2j} \partial_{\tilde t} \underline{\tilde\chi_\ell^*}(\tilde x'_\ell,\tilde t_\ell')|_{\tilde g_\ell(0)}+d_\ell^{-2j-\a}|\mathfrak{D}^{2j+2}  \tilde \psi_{\ell,j,k}(\tilde x_\ell,0)- \P_{\tilde x'_\ell\tilde x_\ell}\mathfrak{D}^{2j+2} \tilde\psi_{\ell,j,k}(\tilde x'_\ell,\tilde t_\ell')|_{\tilde g_\ell(0)}\\
& +d_\ell^{-2j-\a}\sum_{i=1}^j \sum_{p=1}^{N_{i,k}}\sum_{\iota = -2}^{2k} \e_\ell^\iota\Bigg({|\mathfrak{D}^{2j+2+\iota} \tilde A^*_{\ell,i,p,k}(\tilde x_\ell,0)- \P_{\tilde x_\ell'\tilde x_\ell}(\mathfrak{D}^{2j+2+\iota}\tilde A^*_{\ell,i,p,k}(\tilde x_\ell',\tilde t_\ell'))|_{\tilde g_\ell(0)}}\Bigg),
\end{split}
\end{equation}
while on the other hand \eqref{sup-realized-20} remains the same for $j=0$.

\subsection{Estimates on each component}
Our next task is then to obtain precise estimates on all the pieces of the decomposition \eqref{decomposition-tilde-KRF}, which will allow us to later expand and linearize the Monge-Amp\`ere equation \eqref{PMA-tilde}. In the following, the radii $R$ and $S$ will be any fixed radii, unless otherwise specified. Some of the estimates are analogous to those in \cite[\S 4.9]{HT3}, replacing balls $\ti{B}_r$ by parabolic cylinders $\ti{Q}_r=B_{\mathbb{C}^m}(0,r)\times Y\times [-r^2,0]$. We follow closely the arguments there.

\subsubsection{Estimates for $\tilde\psi_{\ell,j,k}$}
First, we assume that $j=0$. Then from the H\"older seminorm bound for $\ddbar\hat{\psi}_{\ell,0,k}$ in \eqref{a-priori Control00}, together with Proposition \ref{prop:braindead}, we obtain
\begin{equation}\label{jorza}
\| \ddbar\hat\psi_{\ell,0,k} \|_{\infty,\hat{Q}_{R},\hat g_\ell(0)}\leq C \delta_\ell^{\a},
\end{equation}
and since $\hat{\psi}_{\ell,0,k}$ has fiberwise average zero, we can apply fiberwise Moser iteration to this, and get
\begin{equation}\label{jarza}
\|\hat\psi_{\ell,0,k} \|_{\infty,\hat{Q}_{R},\hat g_\ell(0)}\leq C \delta_\ell^{2+\a}.
\end{equation}
and using the bounds $\|\partial_{t}^2\vp+\partial_{t}\vp\|_{\infty, B\times Y\times[t-1,t]}\leq C$ and $\|\D(\partial_{t}\vp+\vp)\|_{\infty,  B\times Y\times[t-1,t],g(t)}\leq C$ from Lemma \ref{lma:earlier work} (v), (vi), we can bound for any $x,x'\in B_{R\lambda_\ell^{-1}}\times Y$ and $t,s\in (t_\ell-\lambda_\ell^{-2}R^2,t_\ell]$,
\begin{equation}
|(\partial_{t}\vp+\vp)(x,t)-(\partial_{t}\vp+\vp)(x',s)|\leq C(d^{g(t)}(x,x')+|t-s|)\leq o(1) (d^{g(t)}(x,x')+|t-s|^{\frac{1}{2}})^\a,
\end{equation}
which gives
\begin{equation}\label{wompo}
[\de_{\hat{t}}\hat{\vp}_\ell+\lambda_\ell^{-2}\hat{\vp}_\ell]_{\a,\a/2,\hat{Q}_R,\hat{g}_\ell(0)}\leq \lambda_\ell^{-\alpha}[\partial_{t}\vp+\vp]_{\a,\a/2,B_{R \lambda_\ell^{-1}}\times Y\times(t_\ell-\lambda_\ell^{-2}R^2,t_\ell],g(t_\ell)}=o(1),
\end{equation}
Repeating the argument with $\underline{\vp}$ gives
\begin{equation}\label{wulpo}
[\de_{\hat{t}}\underline{\hat{\vp}_\ell}+\lambda_\ell^{-2}\underline{\hat{\vp}_\ell}]_{\a,\a/2,\hat{Q}_R,\hat{g}_\ell(0)}=o(1),
\end{equation}
and combining these we see that
\begin{equation}
[\de_{\hat{t}}\hat{\psi}_{\ell,0,k}+\lambda_\ell^{-2}\hat{\psi}_{\ell,0,k}]_{\a,\a/2,\hat{Q}_R,\hat{g}_\ell(0)}=o(1),
\end{equation}
so from this, the bound $[\ddbar\hat{\psi}_{\ell,0,k}]_{\a,\a/2,\hat{Q}_R,\hat{g}_\ell(0)}\leq C$ from \eqref{a-priori Control00}, and the bounds \eqref{jorza} and \eqref{sharp-metric-decay-0} below (on $\hat{\omega}^\sharp_\ell$), we see that
\begin{equation}\label{ljk}
\left[\left(\de_{\hat{t}}-\Delta_{\hat{\omega}^\sharp_\ell}\right)\hat{\psi}_{\ell,0,k}\right]_{\a,\a/2,\hat{Q}_R,\hat{g}_\ell(0)}\leq C\lambda_\ell^{-2}[\hat{\psi}_{\ell,0,k}]_{\a,\a/2,\hat{Q}_R,\hat{g}_\ell(0)}+C.
\end{equation}
We wish to use the Schauder estimates in Proposition \ref{prop:schauder}, and for this we need to pass to the check picture via the diffeomorphism $\Xi_\ell$ in \eqref{xi},
pulling back all geometric quantities and scaling $2$-forms (as well as $\hat{\psi}_{\ell,0,k}$)
by $\delta_\ell^{-2}$. We can then apply Proposition \ref{prop:schauder} to $\check\psi_{\ell,0,k}$ and then transfer the result back to the hat picture. This shows that for every $\rho<R$ (where $R$ is fixed) we have
\begin{equation}\label{kjl}
[\DD^2\hat{\psi}_{\ell,0,k}]_{\a,\a/2,\hat{Q}_\rho,\hat{g}_\ell(0)}\leq C\left[\left(\de_{\hat{t}}-\Delta_{\hat{\omega}^\sharp_\ell}\right)\hat{\psi}_{\ell,0,k}\right]_{\a,\a/2,\hat{Q}_{\frac{R+\rho}{2}},\hat{g}_\ell(0)}
+C(R-\rho)^{-2-\a}\|\hat{\psi}_{\ell,0,k}\|_{\infty, \hat{Q}_R},
\end{equation}
and employing the interpolation inequality in \eqref{interpolation-Holder}, and \eqref{jarza}, \eqref{ljk} we can bound the RHS of \eqref{kjl} by
\begin{equation}\label{jlk}\begin{split}
&C\lambda_\ell^{-2}[\hat{\psi}_{\ell,0,k}]_{\a,\a/2,\hat{Q}_{\frac{R+\rho}{2}},\hat{g}_\ell(0)}+C+C(R-\rho)^{-2-\a}\delta_\ell^{2+\alpha}\\
&\leq C\lambda_\ell^{-2}(R-\rho)^{2}[\DD^2\hat{\psi}_{\ell,0,k}]_{\a,\a/2,\hat{Q}_R,\hat{g}_\ell(0)}+C\lambda_\ell^{-2}(R-\rho)^{-\a}\|\hat{\psi}_{\ell,0,k}\|_{\infty, \hat{Q}_R}\\
&+C+C(R-\rho)^{-2-\a}\delta_\ell^{2+\alpha}\\
&\leq\frac{1}{2}[\DD^2\hat{\psi}_{\ell,0,k}]_{\a,\a/2,\hat{Q}_R,\hat{g}_\ell(0)}+C+C\lambda_\ell^{-2}(R-\rho)^{-\a}\delta_\ell^{2+\alpha}+C(R-\rho)^{-2-\a}\delta_\ell^{2+\alpha},
\end{split}
\end{equation}
and after combining \eqref{kjl} and \eqref{jlk}, we can apply the iteration lemma in \cite[Lemma 2.9]{HT3} and deduce that for every $\rho<R$ we have
\begin{equation}
[\DD^2\hat{\psi}_{\ell,0,k}]_{\a,\a/2,\hat{Q}_\rho,\hat{g}_\ell(0)}\leq C+C\lambda_\ell^{-2}(R-\rho)^{-\a}\delta_\ell^{2+\alpha}+C(R-\rho)^{-2-\a}\delta_\ell^{2+\alpha},
\end{equation}
so in particular for any fixed $R$ we deduce that
\begin{equation}
[\DD^2\hat{\psi}_{\ell,0,k}]_{\a,\a/2,\hat{Q}_R,\hat{g}_\ell(0)}\leq C,
\end{equation}
and finally applying Proposition \ref{prop:braindead}
and translating to the tilde picture, we get
\begin{equation}\label{estimate-psi-tilde-infty-10}
d_\ell^{-\a}\| \mathfrak{D}^{\iota}\tilde \psi_{\ell,j,k} \|_{\infty,\tilde Q_{Rd_\ell^{-1}},\tilde g_\ell(0)}\leq C \e_\ell^{2+\a-\iota},\quad
d_\ell^{-\a}[ \mathfrak{D}^{\iota}\tilde \psi_{\ell,j,k} ]_{\a,\a/2,\tilde Q_{Rd_\ell^{-1}},\tilde g_\ell(0)}\leq C \e_\ell^{2-\iota},
\end{equation}
for any $0\leq \iota\leq 2$.

Next, when $j\geq 1$, from the H\"older seminorm bound for $\DD^{2j+2}\hat{\psi}_{\ell, j,k}$ in \eqref{a-priori Control}, together with Proposition \ref{prop:braindead}, we obtain bounds for the lower-order derivatives of $\hat{\psi}_{\ell,j,k}$, which in the tilde picture translate to
\begin{equation}\label{estimate-psi-tilde-infty-1}
d_\ell^{-2j-\a}\| \mathfrak{D}^{\iota}\tilde \psi_{\ell,j,k} \|_{\infty,\tilde Q_{Rd_\ell^{-1}},\tilde g_\ell(0)}\leq C \e_\ell^{2j+2+\a-\iota},\quad
d_\ell^{-2j-\a}[ \mathfrak{D}^{\iota}\tilde \psi_{\ell,j,k} ]_{\a,\a/2,\tilde Q_{Rd_\ell^{-1}},\tilde g_\ell(0)}\leq C \e_\ell^{2j+2-\iota},
\end{equation}
for any $0\leq \iota\leq 2j+2$. Observe that these are formally the same as \eqref{estimate-psi-tilde-infty-10} when $j=0$.

\subsubsection{Estimates for $\underline{\tilde\chi_\ell}$}
By Lemma \ref{lma:earlier work} (ii),(iii),
\begin{equation}
\left\{
\begin{array}{ll}\label{rr}
\|\ddbar \underline{\tilde\chi_\ell} \|_{\infty, \tilde Q_{Rd_\ell^{-1}},\tilde g_\ell(0)}= \|e^{d_\ell^2\lambda_\ell^{-2}\tilde t}\ddbar \underline{\tilde\varphi_\ell} \|_{\infty, \tilde Q_{Rd_\ell^{-1}},\tilde g_\ell(0)}= o(1),\\[2mm]
\|\partial_{\tilde t}\underline{\tilde\chi_\ell} \|_{\infty, \tilde Q_{Rd_\ell^{-1}},\tilde g_\ell(0)}= \|e^{d_\ell^2\lambda_\ell^{-2}\tilde t} (\partial_{\tilde t} \underline{\tilde\varphi_\ell}+d_\ell^2\lambda_\ell^{-2} \underline{\tilde\varphi_\ell}) \|_{\infty, \tilde Q_{Rd_\ell^{-1}},\tilde g_\ell(0)}= o(1).
\end{array}
\right.
\end{equation}
When $j=0$ we have
\begin{equation}\label{wulpo2}\begin{split}
d_\ell^{-\a}[\partial_{\tilde t}\underline{\tilde\chi_\ell}]_{\a,\a/2,\tilde Q_{Rd_\ell^{-1}},\tilde g_\ell(0)}&=d_\ell^{-\a}\left[e^{d_\ell^2\lambda_\ell^{-2}\tilde t} (\partial_{\tilde t} \underline{\tilde\varphi_\ell}+d_\ell^2\lambda_\ell^{-2} \underline{\tilde\varphi_\ell})\right]_{\a,\a/2,\tilde Q_{Rd_\ell^{-1}},\tilde g_\ell(0)}\\
&\leq d_\ell^{-\a}\left[e^{d_\ell^2\lambda_\ell^{-2}\tilde t} \right]_{\a,\a/2,\tilde Q_{Rd_\ell^{-1}},\tilde g_\ell(0)}\|\partial_{\tilde t} \underline{\tilde\varphi_\ell}+d_\ell^2\lambda_\ell^{-2} \underline{\tilde\varphi_\ell}\|_{\infty, \tilde Q_{Rd_\ell^{-1}},\tilde g_\ell(0)}\\
&+d_\ell^{-\a}\left\|e^{d_\ell^2\lambda_\ell^{-2}\tilde t} \right\|_{\infty,\tilde Q_{Rd_\ell^{-1}},\tilde g_\ell(0)}[\partial_{\tilde t} \underline{\tilde\varphi_\ell}+d_\ell^2\lambda_\ell^{-2} \underline{\tilde\varphi_\ell}]_{\a,\a/2, \tilde Q_{Rd_\ell^{-1}},\tilde g_\ell(0)}\\
&\leq o(1) Cd_\ell^{-\a}(d_\ell \lambda_\ell^{-1})^\a+o(1)=o(1),
\end{split}\end{equation}
where we used $d_\ell^{-\a}[\partial_{\tilde t} \underline{\tilde\varphi_\ell}+d_\ell^2\lambda_\ell^{-2} \underline{\tilde\varphi_\ell}]_{\a,\a/2, \tilde Q_{Rd_\ell^{-1}},\tilde g_\ell(0)}=o(1)$, which follows from \eqref{wulpo}. Similarly,
\begin{equation}\begin{split}
d_\ell^{-\a}[\ddbar\underline{\tilde\chi_\ell}]_{\a,\a/2,\tilde Q_{Rd_\ell^{-1}},\tilde g_\ell(0)}&=d_\ell^{-\a}\left[e^{d_\ell^2\lambda_\ell^{-2}\tilde t} \ddbar \underline{\tilde\varphi_\ell}\right]_{\a,\a/2,\tilde Q_{Rd_\ell^{-1}},\tilde g_\ell(0)}\\
&\leq d_\ell^{-\a}\left[e^{d_\ell^2\lambda_\ell^{-2}\tilde t} \right]_{\a,\a/2,\tilde Q_{Rd_\ell^{-1}},\tilde g_\ell(0)}\|\ddbar \underline{\tilde\varphi_\ell}\|_{\infty, \tilde Q_{Rd_\ell^{-1}},\tilde g_\ell(0)}\\
&+d_\ell^{-\a}\left\|e^{d_\ell^2\lambda_\ell^{-2}\tilde t} \right\|_{\infty,\tilde Q_{Rd_\ell^{-1}},\tilde g_\ell(0)}[\ddbar \underline{\tilde\varphi_\ell}]_{\a,\a/2, \tilde Q_{Rd_\ell^{-1}},\tilde g_\ell(0)}\\
&\leq o(1) Cd_\ell^{-\a}(d_\ell \lambda_\ell^{-1})^\a+C\leq C,
\end{split}\end{equation}
where we used $d_\ell^{-\a}[\ddbar\underline{\ti{\vp}_\ell}]_{\a,\a/2, \tilde Q_{Rd_\ell^{-1}},\tilde g_\ell(0)}\leq C$, which follows from \eqref{a-priori Control-20}. Thus, when $j=0$, we  have
\begin{equation}\label{rrr0}
 [\ddbar \underline{\tilde\chi_\ell}]_{\a,\a/2,\tilde Q_{Rd_\ell^{-1}},\tilde g_\ell(0)}\leq Cd_\ell^{\a},\quad [\partial_{\tilde t} \underline{\tilde\chi_\ell}]_{\a,\a/2,\tilde Q_{Rd_\ell^{-1}},\tilde g_\ell(0)}=o(d_\ell^{\a}).
\end{equation}
On the other hand, when $j\geq 1$, from \eqref{a-priori Control-2}, \eqref{r1} and \eqref{r2} we obtain the analogous seminorm bounds
\begin{equation}\label{rrr}
 [\mathfrak{D}^{2j} \ddbar \underline{\tilde\chi_\ell}]_{\a,\a/2,\tilde Q_{Rd_\ell^{-1}},\tilde g_\ell(0)}\leq Cd_\ell^{2j+\a},\quad [\mathfrak{D}^{2j} \partial_{\tilde t} \underline{\tilde\chi_\ell}]_{\a,\a/2,\tilde Q_{Rd_\ell^{-1}},\tilde g_\ell(0)}\leq Cd_\ell^{2j+\a},
\end{equation}
and we can interpolate between these and \eqref{rr} to conclude  that
\begin{equation}\label{rough-bdd-phi}
\left\{
\begin{array}{ll}
d_\ell^{-\iota} \|\mathfrak{D}^\iota \ddbar \underline{\tilde\chi_\ell}\|_{\infty, \tilde Q_{Rd_\ell^{-1}},\tilde g_\ell(0)}=o(1),\quad 0\leq \iota\leq 2j,\\[2mm]
d_\ell^{-\iota} \|\mathfrak{D}^\iota \partial_{\tilde t} \underline{\tilde\chi_\ell}\|_{\infty, \tilde Q_{Rd_\ell^{-1}},\tilde g_\ell(0)}=o(1),\quad 0\leq \iota\leq 2j,\\[2mm]
d_\ell^{-\iota-\a} [\mathfrak{D}^\iota \ddbar \underline{\tilde\chi_\ell}]_{\a,\a/2, \tilde Q_{Rd_\ell^{-1}},\tilde g_\ell(0)}=o(1),\quad  0\leq \iota<2j, \\[2mm]
d_\ell^{-\iota-\a} [\mathfrak{D}^\iota \partial_{\tilde t} \underline{\tilde\chi_\ell}]_{\a,\a/2, \tilde Q_{Rd_\ell^{-1}},\tilde g_\ell(0)}=o(1),\quad 0\leq \iota<2j,
\end{array}
\right.
\end{equation}
We now treat $\tilde\chi_\ell^*$. From \eqref{rrr0}, \eqref{rrr} and the definition of $\tilde{\chi}_\ell^*$ we see that for all $j\geq 0$,
\begin{equation}\label{bakry}
 [\mathfrak{D}^{2j} \ddbar \underline{\tilde\chi_\ell^*}]_{\a,\a/2,\tilde Q_{Rd_\ell^{-1}},\tilde g_\ell(0)}\leq Cd_\ell^{2j+\a},\quad
[\mathfrak{D}^{2j} \partial_{\tilde t} \underline{\tilde\chi_\ell^*}]_{\a,\a/2,\tilde Q_{Rd_\ell^{-1}},\tilde g_\ell(0)}\leq Cd_\ell^{2j+\a}.
\end{equation}
Using these and Lemma \ref{lma:jet-interpolation}, we get
\begin{equation}\label{bdd-phi-star-underline}
\left\{
\begin{array}{ll}
d_\ell^{-2j-\a} \|\mathfrak{D}^\iota \ddbar \underline{\tilde\chi_\ell^*}\|_{\infty, \tilde Q_{S},\tilde g_\ell(0)}\leq CS^{2j+\a-\iota},\\[2mm]
d_\ell^{-2j-\a} \|\mathfrak{D}^\iota \partial_{\tilde t} \underline{\tilde\chi_\ell^*}\|_{\infty, \tilde Q_{S},\tilde g_\ell(0)}\leq CS^{2j+\a-\iota} ,\\[2mm]
d_\ell^{-2j-\a} [\mathfrak{D}^\iota \ddbar \underline{\tilde\chi_\ell^*}]_{\a,\a/2, \tilde Q_{S},\tilde g_\ell(0)} \leq CS^{2j-\iota},\\ [2mm]
d_\ell^{-2j-\a} [\mathfrak{D}^\iota \partial_{\tilde t} \underline{\tilde\chi_\ell^*}]_{\a,\a/2, \tilde Q_{S},\tilde g_\ell(0)} \leq CS^{2j-\iota},
\end{array}
\right.
\end{equation}
for all $0\leq \iota\leq 2j$ and $0<S\leq Rd_\ell^{-1}$.

We will also need a bound for the $L^\infty$ norm of derivatives of $d_\ell^{-2j-\a}\underline{\tilde\chi_\ell^*}$ of order up to $2j+2$, which in general may blow up, but which will nevertheless be useful later. To start, from Lemma \ref{lma:earlier work} (ii) we have $\|\underline{\hat{\chi}_\ell}\|_{\infty, \hat{Q}_R}=o(1)\lambda_\ell^2,$ while from \eqref{bakry} we see in particular that
\begin{equation}
\left[\DD^{2j}\left(\de_{\hat{t}}-\Delta_{\omega_{\C^m}}\right)\underline{\hat{\chi}_\ell}\right]_{\a,\a/2,\hat{Q}_R,g_{\C^m}}\leq C.
\end{equation}
Standard Euclidean Schauder estimates then imply that
\begin{equation}
\left[\DD^{2j+2}\underline{\hat{\chi}_\ell}\right]_{\a,\a/2,\hat{Q}_R,g_{\C^m}}\leq C+o(1)\lambda_\ell^2,
\end{equation}
and since we may assume without loss that $o(1)\lambda_\ell^2\geq C$, passing to the tilde picture we get
\begin{equation}
d_\ell^{-2j-\a}\left[\DD^{2j+2}\underline{\ti{\chi}^*_\ell}\right]_{\a,\a/2,\ti{Q}_{Rd_\ell^{-1}},g_{\C^m}}
=d_\ell^{-2j-\a}\left[\DD^{2j+2}\underline{\ti{\chi}_\ell}\right]_{\a,\a/2,\ti{Q}_{Rd_\ell^{-1}},g_{\C^m}}\leq o(1)\lambda_\ell^2,
\end{equation}
and we can then apply Lemma \ref{lma:jet-interpolation} to $\ti{Q}_R$ and get
\begin{equation}\label{ainf}
d_\ell^{-2j-\a}\|\DD^\iota\underline{\ti{\chi}^*_\ell}\|_{\infty,\ti{Q}_R,g_{\C^m}}= o(1)\lambda_\ell^2,
\end{equation}
for $0\leq \iota\leq 2j+2$.

On the other hand, since $\underline{\ti{\chi}^\sharp_\ell}$ is a jet of $\underline{\ti{\chi}_\ell}$, it inherits from \eqref{rr} and \eqref{rough-bdd-phi} the bounds
\begin{equation}\label{bdd-phi-sharp-underline}
\left\{
\begin{array}{ll}
d_\ell^{-\iota} \|\mathfrak{D}^\iota \ddbar \underline{\tilde\chi_\ell^\sharp}\|_{\infty, \tilde Q_{Rd_\ell^{-1}},\tilde g_\ell(0)}=o(1) ,\\ [2mm]
d_\ell^{-\iota} \|\mathfrak{D}^\iota \partial_{\tilde t}  \underline{\tilde\chi_\ell^\sharp}\|_{\infty, \tilde Q_{Rd_\ell^{-1}},\tilde g_\ell(0)}=o(1)  ,\\[2mm]
d_\ell^{-\iota-\a} [\mathfrak{D}^\iota \ddbar  \underline{\tilde\chi_\ell^\sharp}]_{\a,\a/2, \tilde Q_{Rd_\ell^{-1}},\tilde g_\ell(0)} =o(1) ,\\[2mm]
d_\ell^{-\iota-\a} [\mathfrak{D}^\iota \partial_{\tilde t}  \underline{\tilde\chi_\ell^\sharp}]_{\a,\a/2, \tilde Q_{Rd_\ell^{-1}},\tilde g_\ell(0)} =o(1),
\end{array}
\right.
\end{equation}
for all $0\leq \iota\leq 2j$ and $R$ fixed. Moreover, since $\underline{\tilde\chi_\ell^\sharp}$ is a polynomial of degree at most $2j+2$, it vanishes when differentiated more than $2j+2$ times.

Furthermore, recalling the definitions of $\ti{\eta}^\ddagger_\ell$ and $\ti{\eta}^\diamond_\ell$ in \eqref{def:eta-phi-notvarphi}, from \eqref{bdd-phi-star-underline} and \eqref{bdd-phi-sharp-underline} we quickly deduce
\begin{equation}\label{gorgo}
d_\ell^{-2j-\a}\|\mathfrak{D}^\iota\ti{\eta}^\diamond_\ell\|_{\infty, \tilde Q_{S},\tilde g_\ell(0)}\leq CS^{2j+\a-\iota},\quad
d_\ell^{-2j-\a}[\mathfrak{D}^\iota\ti{\eta}^\diamond_\ell]_{\a,\a/2, \tilde Q_{S},\tilde g_\ell(0)}\leq CS^{2j-\iota},
\end{equation}
for $0\leq\iota\leq 2j, 0<S\leq Rd_\ell^{-1}$, and
\begin{equation}\label{gorgol}
d_\ell^{-\iota}\|\mathfrak{D}^\iota\ti{\eta}^\ddagger_\ell\|_{\infty, \tilde Q_{Rd_\ell^{-1}},\tilde g_\ell(0)}=o(1),\quad
d_\ell^{-\iota-\a}[\mathfrak{D}^\iota\ti{\eta}^\ddagger_\ell]_{\a,\a/2, \tilde Q_{Rd_\ell^{-1}},\tilde g_\ell(0)}=o(1),
\end{equation}
also for $0\leq\iota\leq 2j,$ $R$ fixed (and derivatives of $\ti{\eta}^\ddagger_\ell$ of order higher than $2j$ vanish, of course).

\subsubsection{Estimates for $\tilde A_{\ell,i,p,k}^*$}

From \eqref{a-priori Control-3}, we have
\begin{equation}
[\mathfrak{D}^{2j} \hat A^*_{\ell,i,p,k}]_{\a,\a/2,\hat Q_{R},\hat g_\ell( 0)} \leq C \delta_\ell^2,
\end{equation}
and hence Lemma \ref{lma:jet-interpolation} implies
\begin{equation}\label{hat-A-star-0}
\|\mathfrak{D}^{\iota} \hat A^*_{\ell,i,p,k}\|_{\infty,\hat Q_{R},\hat g_\ell( 0)} \leq C\delta_\ell^2  R^{2j+\a-\iota},\quad
[\mathfrak{D}^{\iota} \hat A^*_{\ell,i,p,k}]_{\a,\a/2,\hat Q_{R},\hat g_\ell( 0)} \leq C\delta_\ell^2 R^{2j-\iota},
\end{equation}
for all $0\leq \iota\leq 2j$ and fixed $R>0$.  Taking $R=S d_\ell$ and transferring to the tilde picture, we have
\begin{equation}\label{A-star-0}
d_\ell^{-\iota} \|\mathfrak{D}^{\iota} \tilde A^*_{\ell,i,p,k}\|_{\infty,\tilde Q_{S},\tilde g_\ell( 0)} \leq C\ve_\ell^2  (d_\ell S)^{2j+\a-\iota},\quad
d_\ell^{-\iota-\a}[\mathfrak{D}^{\iota} \tilde A^*_{\ell,i,p,k}]_{\a,\a/2,\tilde Q_{S},\tilde g_\ell( 0)} \leq C\ve_\ell^2 (d_\ell S)^{2j-\iota},
\end{equation}
for all $0\leq \iota\leq 2j$ and $S\leq R d_\ell^{-1}$.

For higher order derivatives of $\tilde A_{\ell,i,p,k}^*$, we  use interpolation. Since $\tilde A^\sharp_{\ell,i,p,k}$ is a parabolic polynomial of degree at most $2j$, \eqref{a-priori Control-3} implies
\begin{equation}\label{hat-A-star}
\left[\mathfrak{D}^{2j+\iota} \hat A^*_{\ell,i,p,k}\right]_{\a,\a/2,\hat Q_{R},\hat g_\ell( 0)} \leq C \delta_\ell^{2-\iota},
\end{equation}
for all $0\leq \iota\leq 2k+2$. For any $1\leq \iota\leq 2k+2$, Proposition \ref{prop:interpolation} in the hat picture gives for $0<\rho_1<\rho_2\leq R$,
\begin{equation}
\begin{split}
&\quad (\rho_2-\rho_1)^\iota  \|\mathfrak{D}^{2j+\iota} \hat  A^*_{\ell,i,p,k}\|_{\infty,\hat Q_{\rho_1},\hat g_\ell( 0)}\\
&\leq C\left((\rho_2-\rho_1)^{\iota+\a}  [\mathfrak{D}^{2j+\iota} \hat A^*_{\ell,i,p,k}]_{\a,\a/2,\hat Q_{\rho_2},\hat g_\ell( 0)}+\|\mathfrak{D}^{2j} \hat A^*_{\ell,i,p,k}\|_{\infty,\hat Q_{\rho_2},\hat g_\ell( 0)} \right)\\
&\leq C(\rho_2-\rho_1)^{\iota+\a}  \delta_\ell^{2-\iota}+C\delta_\ell^2 \rho_2^{\a}.
\end{split}
\end{equation}
We can follow closely the choice of $\rho_i$ in the interpolation in \cite[\S 4.9.3]{HT3} (with $j$ replaced by $2j$ here) to conclude that
\begin{equation}\label{A-star--0}
\|\mathfrak{D}^{2j+\iota} \hat A^*_{\ell,i,p,k}\|_{\infty,\hat Q_{S\delta_\ell},\hat g_\ell(0)}\leq C_S \delta_\ell^{2+\a-\iota}.
\end{equation}
This will play an important role in case $\e_\ell\geq C^{-1}$.  For later purposes, we will need to estimate the dependence of $C_S$ on $S$ as $S\to +\infty$. In fact, $C_S$ in \cite[(4.127)]{HT3} is given by $C(AS)^\frac{\a^2}{\iota+\a}$ where $A>1$ is given as a function of $S>1$ by solving the equation $(A-1)^{-1}A^\frac{\a}{\iota+\a}= S^\frac{\iota}{\iota+\a}$.  Since $\frac{\a}{\iota+\a}<1$ for $\iota\geq 1$, $A$ stays bounded as $S\to +\infty$.  Hence we can estimate $C_S$ from above by
\begin{equation}\label{CS-bdd}
C_S\leq C(AS)^\frac{\a^2}{\iota+\a}\leq CS^\a,
\end{equation}
for $S>1$. Therefore, in the tilde picture we obtain
\begin{equation}\label{A-star-1}
\|\mathfrak{D}^{2j+\iota} \tilde A^*_{\ell,i,p,k}\|_{\infty,\tilde Q_{S\e_\ell},\tilde g_\ell(0)}\leq CS^\a\e_\ell^{2-\iota+\a} d_\ell^{2j+\a},
\end{equation}
where $1\leq \iota\leq 2k+2$ and $S>1$ fixed (which is the analog of \cite[(4.129)]{HT3}).

Similarly, in the case when $\e_\ell\to 0$,  following the derivation of \cite[(4.132)]{HT3} we obtain
\begin{equation}\label{A-star-2-hat}
\|\mathfrak{D}^{2j+\iota} \hat A^*_{\ell,i,p,k}\|_{\infty,\hat Q_{S d_\ell},\hat g_\ell(0)}\leq C\delta_\ell^{2-\iota}d_\ell^\a \e_\ell^\frac{\iota\a}{\iota+\a}S^\frac{\a^2}{\iota+\a},
\end{equation}
which in the tilde picture becomes
\begin{equation}\label{A-star-2}
\|\mathfrak{D}^{2j+\iota} \tilde A^*_{\ell,i,p,k}\|_{\infty,\tilde Q_{S},\tilde g_\ell(0)}\leq C \e_\ell^{2-\iota} (S^\a \e_\ell^\iota)^\frac{\a}{\iota+\a} d_\ell^{2j+\a},
\end{equation}
for $1\leq \iota\leq 2k+2$ and $S>1$ fixed.

\subsubsection{Estimates for $\tilde A_{\ell,i,p,k}^\sharp$}

By \eqref{lower-A-Linfty}, we have for all $0\leq \iota\leq 2j$ that
\begin{equation}
\|\mathfrak{D}^\iota\hat A_{\ell,i,p,k}\|_{\infty,\hat Q_{R},\hat g_\ell(0)}\leq  C\delta_\ell^2\left( e^{\frac{-2i+2-\a}{2}t_\ell}\right)^{1-\frac{\iota}{2j+\a}},
\end{equation}
for all given $R>0$ (with $C$ independent of $R>0$).  Since $\hat A_{\ell,i,p,k}^\sharp$ is the $2j$-jet of $\hat A_{\ell,i,p,k}$ at $(0,0)$, we see that all the coefficients of the polynomial $\hat A_{\ell,i,p,k}^\sharp$ are bounded by $C\delta_\ell^2 e^{\frac{-2i+2-\a}{2}\cdot
\frac{\a }{2j+\a}t_\ell}$, and so
\begin{equation}\label{tilde-A-sharp}
\left\{
\begin{array}{ll}
\|\mathfrak{D}^\iota \hat A_{\ell,i,p,k}^\sharp\|_{\infty,\hat Q_{S},\hat g_\ell(0)}\leq C \max( 1, S^{2j-\iota}) \delta_\ell^2 e^{\frac{-2i+2-\a}{2}\cdot
\frac{\a }{2j+\a}t_\ell},\\[2mm]
\, [\mathfrak{D}^\iota \hat A_{\ell,i,p,k}^\sharp]_{\b,\b/2,\hat Q_{S},\hat g_\ell(0)}\leq C \max( 1, S^{2j-\iota-\b})\delta_\ell^2 e^{\frac{-2i+2-\a}{2}\cdot
\frac{\a }{2j+\a}t_\ell},
\end{array}
\right.
\end{equation}
for all $S>0$, $0\leq \iota\leq 2j$ and $0<\b<1$. Transferring to the tilde picture yields
\begin{equation}
\left\{
\begin{array}{ll}
d_\ell^{-\iota+2}\|\mathfrak{D}^\iota \tilde A_{\ell,i,p,k}^\sharp\|_{\infty,\tilde Q_{S},\tilde g_\ell(0)}\leq C \max( 1, (Sd_\ell)^{2j-\iota}) \delta_\ell^2 e^{\frac{-2i+2-\a}{2}\cdot
\frac{\a }{2j+\a}t_\ell},\\ [2mm]
\, d_\ell^{-\iota+2-\b}[\mathfrak{D}^\iota \tilde A_{\ell,i,p,k}^\sharp]_{\b,\b/2,\tilde Q_{S},\tilde g_\ell(0)}\leq C \max( 1, (Sd_\ell)^{2j-\iota-\b})\delta_\ell^2 e^{\frac{-2i+2-\a}{2}\cdot
\frac{\a }{2j+\a}t_\ell},
\end{array}
\right.
\end{equation}
for all $0<S\leq Rd_\ell^{-1}$, $0\leq \iota\leq 2j$ and $0<\b<1$, c.f.  \cite[(4.135)--(4.136)]{HT3}.

\subsubsection{Estimates for $\tilde\eta_{\ell}^\circ$ and its potential}
First, we recall that
\begin{equation}
\hat \eta_\ell^\circ =\sum_{i=1}^j\sum_{p=1}^{N_{i,k}}\ddbar \hat{\mathfrak{G}}_{\hat{t},k}(\hat A_{\ell,i,p,k}^*,\hat G_{\ell,i,p,k}).
\end{equation}
Here $\hat{\mathfrak{G}}_{\hat{t},k}$ and $\hat A_{\ell,i,p,k}$ are $\hat t$-dependent while $\hat G_{\ell,i,p,k}$ is independent of time.

By applying \eqref{plop} for each fixed $\hat t$, we can write
\begin{equation}
 \hat{\mathfrak{G}}_{\hat{t},k}(\hat A_{\ell,i,p,k}^*,\hat G_{\ell,i,p,k})= \sum_{\iota=0}^{2k}\sum_{q=\lceil\frac{\iota}{2}\rceil}^{k} e^{-q\lambda_\ell^{-2}\hat t-(q-\frac{\iota}{2})t_\ell} \delta_\ell^\iota   \hat \Phi_{\iota,q}(\hat G_{\ell,i,p,k})\circledast \D^\iota \hat A_{\ell,i,p,k}^*,
\end{equation}
where $\hat \Phi_{\iota,q}(\hat G_{\ell,i,p,k})$ is independent of time.  Note also that $\hat J_\ell$ is independent of time and hence \cite[(4.141), (4.144)]{HT3} can be directly carried over, so that
\begin{equation}\label{inter-estimate-0}
\left\{
\begin{array}{ll}
\| \mathfrak{D}^\iota \hat J_\ell\|_{\infty, \hat Q_R,\hat g_\ell(0)} \leq C\delta_\ell^{-\iota},\\[2mm]
\, [  \mathfrak{D}^\iota \hat J_\ell]_{\a,\a/2, \hat Q_R,\hat g_\ell(0)} \leq C\delta_\ell^{-\iota-\a},\\[2mm]
\| \mathfrak{D}^\iota \hat \Phi_{\iota,q}\|_{\infty, \hat Q_R,\hat g_\ell(0)} \leq C\delta_\ell^{-\iota},\\[2mm]
[ \mathfrak{D}^\iota \hat \Phi_{\iota,q}]_{\a,\a/2, \hat Q_R,\hat g_\ell(0)} \leq C\delta_\ell^{-\iota-\a},
\end{array}
\right.
\end{equation}
for each $\iota\geq 0,\lceil\frac{\iota}{2}\rceil\leq q\leq 2k$ and fixed $R>0$.

Schematically we have
\begin{equation}\label{expression-G-0}
\begin{split}
&\quad \mathfrak{D}^{r}\hat{\mathfrak{G}}_{\hat{t},k}(\hat A_{\ell,i,p,k}^*,\hat G_{\ell,i,p,k})\\
&= \sum_{\iota=0}^{2k}\sum_{q=\lceil\frac{\iota}{2}\rceil}^{k} \sum_{\substack{d=0\\d\in 2\mathbb{N}}}^r \sum_{i_1+i_2=r-d}e^{-q\lambda_\ell^{-2}\hat t-(q-\frac{\iota}{2})t_\ell} \left(q\lambda_\ell^{-2}\right)^{\frac{d}{2}} \delta_\ell^\iota\mathfrak{D}^{i_1}\hat \Phi_{\iota,q}(\hat G_{\ell,i,p,k}) \circledast \mathfrak{D}^{i_2+\iota} \hat A_{\ell,i,p,k}^*.
\end{split}
\end{equation}

To estimate this, we will need  \eqref{hat-A-star-0}, \eqref{hat-A-star},  and \eqref{A-star-1}:
\begin{equation}\label{inter-estimate-1}
\left\{
\begin{array}{ll}
\|\mathfrak{D}^ \iota  \hat A^*_{\ell,i,p,k}\|_{\infty,\hat Q_{S\delta_\ell},\hat g_\ell(0)} \leq CS^\a\delta_\ell^{2j+2+\a- \iota }\max( 1, S^{2j-\iota}),\\[2mm]
\,  [\mathfrak{D}^ \iota  \hat A^*_{\ell,i,p,k}]_{\a,\a/2,\hat Q_{ S\delta_\ell},\hat g_\ell(0)} \leq C \delta_\ell^{2j+2- \iota } \max(1,S^{2j-\iota}),
\end{array}
\right.
\end{equation}
for all $0\leq \iota \leq 2k+2+2j$ and $S>1$ fixed.  Now using \eqref{inter-estimate-0} and \eqref{inter-estimate-1} in \eqref{expression-G-0}, and recalling that $\delta_\ell=\lambda_\ell e^{-\frac{t_\ell}{2}},$ yields
\begin{equation}\label{potential-eta-circ-estimate-1}
\begin{split}
 \|\mathfrak{D}^{r}\hat{\mathfrak{G}}_{\hat{t},k}(\hat A_{\ell,i,p,k}^*,\hat G_{\ell,i,p,k})\|_{\infty,\hat Q_{S\delta_\ell},\hat g_\ell(0)}
 &\leq CS^\a \sum_{\iota=0}^{2k} \sum_{\substack{d=0\\d\in 2\mathbb{N}}}^r \sum_{i_1+i_2=r-d}\max(1,S^{2j-\iota})\lambda_\ell^{-d}  \delta_\ell^{\iota-i_1+2j+2+\a-i_2-\iota}\\
 &\leq CS^{2j+\a}  \delta_\ell^{-r+2j+2+\a},
\end{split}
\end{equation}
for all $0\leq r\leq 2j+2$. Similarly,
\begin{equation}
\begin{split}
 [\mathfrak{D}^{r}\hat{\mathfrak{G}}_{\hat{t},k}(\hat A_{\ell,i,p,k}^*,\hat G_{\ell,i,p,k})]_{\a,\a/2,\hat Q_{S\delta_\ell},\hat g_\ell(0)}
 &\leq C S^{2j+\a} \delta_\ell^{-r+2j+2},
\end{split}
\end{equation}
for all $0\leq r\leq 2j+2$.  Transferring to the tilde picture gives,
\begin{equation}\label{estimate-potential-G}
\left\{
\begin{array}{ll}
d_\ell^{-r+2} \|\mathfrak{D}^{r}\tilde{\mathfrak{G}}_{\ti{t},k}(\tilde A_{\ell,i,p,k}^*,\tilde G_{\ell,i,p,k})\|_{\infty,\tilde Q_{S\e_\ell},\tilde g_\ell(0)}
 &\leq CS^{2j+\a}   \delta_\ell^{-r+2j+2+\a},\\[2mm]
 d_\ell^{-r-\a+2} [\mathfrak{D}^{r}\tilde{\mathfrak{G}}_{\ti{t},k}(\tilde A_{\ell,i,p,k}^*,\tilde G_{\ell,i,p,k})]_{\a,\a/2,\tilde Q_{S\e_\ell},\tilde g_\ell(0)}
 &\leq CS^{2j+\a}  \delta_\ell^{-r+2j+2},
\end{array}
\right.
\end{equation}
for all $0\leq r\leq 2j+2$ and $S>1$ fixed.

Next, we consider the case when $\ve_\ell\leq C$, and we take only derivatives and difference quotients in the base and time directions.  The argument is similar to \cite[(4.159)]{HT3}.  We start by noting that if a certain derivative $\mathfrak{D}^\iota$ contains precisely $u$ fiber derivatives ($0\leq u\leq \iota$), then we will denote it schematically by $\mathfrak{D}_{{\bf \# f}=u}^\iota$. We then have the easy bounds from \cite[(4.153)--(4.154)]{HT3}
\begin{equation}\label{inter-estimate-2}
\left\{
\begin{array}{ll}
\| \mathfrak{D}_{{\bf \# f}=u}^\iota \hat J_\ell\|_{\infty, \hat Q_R,\hat g_\ell(0)} \leq C\lambda_\ell^{-\iota+u}\delta_\ell^{-u},\\[2mm]
\, [  \mathfrak{D}_{{\bf \# f}=u}^\iota \hat J_\ell]_{\a,\a/2,\mathrm{base}, \hat Q_R,\hat g_\ell(0)} \leq C\lambda_\ell^{-\iota-\a+u}\delta_\ell^{-u},\\[2mm]
\| \mathfrak{D}^\iota_{{\bf \# f}=u} \hat \Phi_{\iota,q}\|_{\infty, \hat Q_R,\hat g_\ell(0)} \leq C\lambda_\ell^{-\iota+u}\delta_\ell^{-u},\\[2mm]
[ \mathfrak{D}^\iota_{{\bf \# f}=u} \hat \Phi_{\iota,q}]_{\a,\a/2,\mathrm{base}, \hat Q_R,\hat g_\ell(0)} \leq C\lambda_\ell^{-\iota-\a+u}\delta_\ell^{-u}.
\end{array}
\right.
\end{equation}
Using \eqref{hat-A-star-0}, \eqref{hat-A-star},  and \eqref{A-star-2-hat} to obtain
\begin{equation}\label{inter-estimate-3}
\left\{
\begin{array}{ll}
\|\mathfrak{D}^ \iota  \hat A^*_{\ell,i,p,k}\|_{\infty,\hat Q_{Sd_\ell},\hat g_\ell(0)} \leq C\delta_\ell^2 (Sd_\ell)^{2j+\a-\iota},\\[2mm]
\,  [\mathfrak{D}^ \iota  \hat A^*_{\ell,i,p,k}]_{\a,\a/2,\hat Q_{ Sd_\ell},\hat g_\ell(0)} \leq C\delta_\ell^2 (S d_\ell)^{2j-\iota},\\[2mm]
\|\mathfrak{D}^ {2j+\iota'}  \hat A^*_{\ell,i,p,k}\|_{\infty,\hat Q_{Sd_\ell},\hat g_\ell(0)} \leq  C\delta_\ell^{2-\iota'} (d_\ell S)^\a \e_\ell^\frac{\iota' \a}{\iota'+\a},\\[2mm]
[\mathfrak{D}^{\iota'}\hat A_{\ell,i,p,k}^*]_{\a,\a/2,\hat Q_{Sd_\ell},\hat g_\ell(0)}\leq C \delta_\ell^{2-\iota'},
\end{array}
\right.
\end{equation}
for all $0\leq \iota \leq 2j$, $1\leq \iota'\leq 2k+2$ and $S>1$ fixed.   Then \eqref{expression-G-0}, \eqref{inter-estimate-2} and \eqref{inter-estimate-3} imply
\begin{equation}
\begin{split}
 \|\mathfrak{D}_{\bf bt}^{r}\hat {\mathfrak{G}}_{\hat{t},k}(\hat A_{\ell,i,p,k}^*,\hat G_{\ell,i,p,k})\|_{\infty,\hat Q_{Sd_\ell},\hat g_\ell(0)}
 &\leq C\sum_{\iota=0}^{2k} \sum_{\substack{d=0\\d\in 2\mathbb{N}}}^r \sum_{i_1+i_2=r-d} \lambda_\ell^{-d-i_1}\delta_\ell^{\iota }\cdot
 \left\{
 \begin{array}{ll}
\delta_\ell^2 (Sd_\ell)^{2j+\a-i_2-\iota},\; \text{if}\;\; i_2+\iota \leq 2j,\\
\delta_\ell^{2+2j-i_2-\iota} (S d_\ell )^\a , \; \text{if}\;\; i_2+\iota > 2j,\\
 \end{array}
 \right.\\
 &\leq C d_\ell^{2j+2+\a-r}S^{2j+\a},
\end{split}
\end{equation}
for $0\leq r\leq 2j+2$, and fixed $S>1$.  Arguing similarly for the H\"older seminorms, and transferring to the tilde picture yields
\begin{equation}\label{estimate-potential-G-bt}
\left\{
\begin{array}{ll}
& d_\ell^{-r+2}\|\mathfrak{D}_{\bf bt}^{r}\tilde {\mathfrak{G}}_{\ti{t},k}(\tilde A_{\ell,i,p,k}^*,\tilde G_{\ell,i,p,k})\|_{\infty,\tilde Q_{S},\tilde g_\ell(0)}\leq C d_\ell^{2j+2+\a-r}S^{2j+\a}, \\
 &d_\ell^{-r-\a+2}[\mathfrak{D}_{\bf bt}^{r}\tilde {\mathfrak{G}}_{\ti{t},k}(\tilde A_{\ell,i,p,k}^*,\tilde G_{\ell,i,p,k})]_{\a,\a/2,\mathrm{base},\tilde Q_{S},\tilde g_\ell(0)} \leq C  d_\ell^{2j+2-r}S^{2j+\a}.
\end{array}
\right.
\end{equation}

Observe that when $r=2j+2$, the leading term arises when $d=i_1=q=\iota=0$ and $i_2=2j+2$, which in the tilde picture is given by
\begin{equation}\label{leading-G-A*}
\mathfrak{D}_{\bf bt}^{2j+2}\tilde {\mathfrak{G}}_{\ti{t},k}(\tilde A_{\ell,i,p,k}^*,\tilde G_{\ell,i,p,k})
=\sum_{i=1}^j \sum_{p=1}^{N_{i,k}} (\Delta_{\Theta_\ell^*\Psi_\ell^*\omega_{F}|_{ \{\cdot\}\times Y}})^{-1} \tilde G_{\ell,i,p,k} \cdot  \mathfrak{D}^{2j+2}\tilde A_{\ell,i,p,k}^* +o(d_\ell^{2j+\a}).
\end{equation}

We now bound $\tilde \eta_\ell^\circ$ and its derivatives.  Using  \cite[(4.139)]{HT3}, paying extra attention to the time derivatives, we have
\begin{equation}\label{expression-G-1}
\begin{split}
\mathfrak{D}^r \hat\eta_\ell^\circ =\sum e^{-q(t_\ell+\lambda_\ell^{-2}\hat t)+\frac{\iota}{2}t_\ell} \left(q\lambda_\ell^{-2}\right)^{\frac{d}{2}}\delta_\ell^\iota (\mathfrak{D}^{r-d+1-s} \hat J_\ell)\circledast \mathfrak{D}^{i_1}\hat \Phi_{\ell,\iota,q}(\hat G_{\ell,i,p,k}) \circledast \mathfrak{D}^{i_2+\iota} \hat A_{\ell,i,p,k}^*,
\end{split}
\end{equation}
where
\begin{equation}
\sum=  \sum_{i=1}^j \sum_{p=1}^{N_{i,k}} \sum_{\iota=0}^{2k}\sum_{q=\lceil\frac{\iota}{2}\rceil}^{k}\sum_{\substack{d=0\\d\in 2\mathbb{N}}}^r \sum_{s=0}^{d+1}\sum_{i_1+i_2=s+1}.
\end{equation}

Using \eqref{inter-estimate-0}  with \eqref{inter-estimate-1}, we can estimate $\mathfrak{D}^r \hat \eta_\ell^\circ$ by
\begin{equation}\label{f1}
\begin{split}
\|\mathfrak{D}^r \hat \eta_\ell^\circ\|_{\infty, \hat Q_{S\delta_\ell},\hat g_\ell(0)}&\leq CS^{2j+\a} \lambda_\ell^{-d} \delta_\ell^{\iota -r+d-1+s-i_1+2j+2+\a-i_2-\iota} \leq C S^{2j+\a} \delta_\ell^{ -r+2j+\a},
\end{split}
\end{equation}
where we have used the fact that $\delta_\ell= \lambda_\ell e^{-t_\ell/2}$. Similarly,
\begin{equation}\label{f2}
\begin{split}
[\mathfrak{D}^r  \hat \eta_\ell^\circ]_{\a,\a/2, \hat Q_{R\delta_\ell},\hat g_\ell(0)}&\leq  CS^{2j+\a} \delta_\ell^{ -r+2j}.
\end{split}
\end{equation}

In particular in the tilde picture,
\begin{equation}\label{circ-e-nonzero}
\left\{
\begin{array}{ll}
d_\ell^{-r}\|\mathfrak{D}^r\tilde  \eta_\ell^\circ\|_{\infty, \tilde Q_{S\e_\ell},\tilde g_\ell(0)}& \leq C S^{2j+\a} \delta_\ell^{ -r+2j+\a},\\[2mm]
\,  d_\ell^{-r-\a}[\mathfrak{D}^r \tilde\eta_\ell^\circ]_{\a,\a/2, \tilde Q_{S\e_\ell},\tilde g_\ell(0)}&\leq  CS^{2j+\a} \delta_\ell^{ -r+2j},
\end{array}
\right.
\end{equation}
for all $0\leq r\leq 2j$ (which is the analog of \cite[(4.151)]{HT3}).

These estimates are only useful when $\e_\ell\geq C^{-1}$.  In the case when $\e_\ell\to 0$,  we shall only take derivatives and difference quotients in the base and time directions.  We can follow the argument to derive \cite[(4.159)]{HT3}, using \eqref{hat-A-star-0}, \eqref{hat-A-star},\eqref{A-star--0} and \eqref{CS-bdd} instead of \cite[(4.123), (4.125), (4.130)]{HT3}, and using $\mathfrak{D}$ instead of $\D$, and taking also time derivatives of $e^{-q(t_\ell+\lambda_\ell^{-2}\hat t)}$ in \eqref{expression-G-1}, we obtain
\begin{equation}\label{estimate-circ-decay}
\left\{
\begin{array}{ll}
d_\ell^{-r}\|\mathfrak{D}_{\bf bt}^r\tilde\eta_\ell^\circ \|_{\infty, \tilde Q_S,\tilde g_\ell(0)} \leq C S^{2j+\a} d_\ell^{2j+\a-r},\\[2mm]
\, d_\ell^{-r-\a}[ \mathfrak{D}_{\bf bt}^r\tilde\eta_\ell^\circ] _{\a,\a/2,\mathrm{base}, \tilde Q_S,\tilde g_\ell(0)} \leq C S^{2j+\a} d_\ell^{2j-r},\\
\end{array}
\right.
\end{equation}
for all $0\leq r\leq 2j$ and $S$ fixed.
Also, it is important to note that in the  $L^\infty$ bound in \eqref{estimate-circ-decay} with $r=2j$, which nominally is $O(d_\ell^\a)$, the only term which is not actually $o(d_\ell^\a)$ comes from the terms in the sum in \eqref{expression-G-1} with $d=0$. To see this, we follow verbatim the discussion in \cite[(4.161)]{HT3}, which gives us that
\begin{equation}\label{expansion-eta-circ}
d_\ell^{-2j-\a} \mathfrak{D}_{\bf bt}^{2j} \tilde \eta_\ell^\circ = d_\ell^{-2j-\a} \sum_{i=1}^j \sum_{p=1}^{N_{i,k}}  \left(\ddbar (\Delta_{\Theta_\ell^*\Psi_\ell^*\omega_{F}|_{ \{\cdot\}\times Y}})^{-1} \ti G_{\ell,i,p,k} \right)_{\bf ff} \mathfrak{D}^{2j}\tilde A_{\ell,i,p,k}^* +o(1),
\end{equation}
locally uniformly.

Similarly,  following the derivation of \cite[(4.162)]{HT3}, in the case when $\ve_\ell\to 0$ we have
\begin{equation}\label{qorq}
\left\{
\begin{array}{ll}
d_\ell^{-r} \| \mathfrak{D}^r \tilde\eta_\ell^\circ\|_{\infty,\tilde Q_S,g_X}\leq  CS^{2j+\a} d_\ell^{2j+\a-r},\\[2mm]
\, d_\ell^{-r-\a} [ \mathfrak{D}^r \tilde\eta_\ell^\circ]_{\a,\a/2, \tilde Q_S,g_X}\leq  CS^{2j+\a} d_\ell^{2j-r},
\end{array}
\right.
\end{equation}
for all $0\leq r\leq 2j$, fixed $S>0$ and a fixed metric $g_X$.

\subsubsection{Estimates for $\tilde\eta_\ell^\dagger$ and its potential}

Recall that by definition we have
\begin{equation}
\begin{split}
\hat \eta_\ell^\dagger
&= \ddbar \sum_{i=1}^j \sum_{p=1}^{N_{i,k}} \sum_{\iota=0}^{2k}\sum_{q=\lceil\frac{\iota}{2}\rceil}^{k} e^{-q(t_\ell+\lambda_\ell^{-2}\hat t)+\frac{\iota}{2}t_\ell} \delta_\ell^\iota \left(  \hat \Phi_{\ell,\iota,q}(\hat G_{\ell,i,p,k})\circledast \D^\iota \hat A_{\ell,i,p,k}^\sharp \right),
\end{split}
\end{equation}
where we have applied \eqref{plop}.  From \eqref{tilde-A-sharp} we have
\begin{equation}\label{recall-tilde-A-sharp}
\left\{
\begin{array}{ll}
\|\mathfrak{D}^\iota \hat A_{\ell,i,p,k}^\sharp\|_{\infty,\hat Q_{S},\hat g_\ell(0)}\leq C \max( 1, S^{2j-\iota}) \delta_\ell^2 e^{\frac{-2i+2-\a}{2}\cdot
\frac{\a }{2j+\a}t_\ell},\\ [2mm]
\, [\mathfrak{D}^\iota \hat A_{\ell,i,p,k}^\sharp]_{\b,\b/2,\hat Q_{S},\hat g_\ell(0)}\leq C \max( 1, S^{2j-\iota-\b})\delta_\ell^2 e^{\frac{-2i+2-\a}{2}\cdot
\frac{\a }{2j+\a}t_\ell},
\end{array}
\right.
\end{equation}
for all $S>0$, $0\leq \iota\leq 2j$ and $0<\b<1$ while the derivatives of (parabolic) order $>2j$ vanishes since $\hat A_{\ell,i,p,k}^\sharp$ is a (parabolic) polynomial of degree at most $2j$. By applying $\mathfrak{D}^r$ to $\hat \eta_\ell^\dagger$, we have
\begin{equation}
\begin{split}
\mathfrak{D}^r \hat\eta_\ell^\dagger =\sum e^{-q(t_\ell+\lambda_\ell^{-2}\hat t)+\frac{\iota}{2}t_\ell} (q\lambda_\ell^{-2})^{\frac{d}{2}}\delta_\ell^\iota (\mathfrak{D}^{r-d+1-s} \hat J_\ell)\circledast \mathfrak{D}^{i_1}\hat \Phi_{\iota,q}(\hat G_{\ell,i,p,k}) \circledast \mathfrak{D}^{i_2+\iota} \hat A_{\ell,i,p,k}^\sharp,
\end{split}
\end{equation}
where
\begin{equation}
\sum=  \sum_{i=1}^j \sum_{p=1}^{N_{i,k}} \sum_{\iota=0}^{2k}\sum_{q=\lceil\frac{\iota}{2}\rceil}^{k}\sum_{\substack{d=0\\d\in 2\mathbb{N}}}^r \sum_{s=0}^{d+1}\sum_{i_1+i_2=s+1},
\end{equation}
so that \eqref{inter-estimate-0} and \eqref{recall-tilde-A-sharp} imply
\begin{equation}
\begin{split}
\|\mathfrak{D}^r \hat \eta_\ell^\dagger\|_{\infty,\hat Q_S,\hat g_\ell(0)}
&\leq C\sum \lambda_\ell^{-d} \delta_\ell^{\iota -r+d-1+s-i_1+2}e^{\frac{-\a}{2}\cdot
\frac{\a }{2j+\a}t_\ell}\leq   C  \delta_\ell^{-r}e^{\frac{-\a}{2}\cdot
\frac{\a }{2j+\a}t_\ell}.
\end{split}
\end{equation}
The H\"older seminorm is similar and hence,
\begin{equation}\label{dager-estimate}
\left\{
\begin{array}{ll}
d_\ell^{-r} \|\mathfrak{D}^r \tilde\eta_\ell^\dagger\|_{\infty,\tilde Q_S,\tilde g_\ell(0)}&\leq C  \delta_\ell^{-r}e^{\frac{-\a}{2}\cdot
\frac{\a }{2j+\a}t_\ell},\\
 d_\ell^{-r-\b} [\mathfrak{D}^r \tilde\eta_\ell^\dagger]_{\b,\b/2,\tilde Q_S,\tilde g_\ell(0)}&\leq C  \delta_\ell^{-r-\b} e^{\frac{-\a}{2}\cdot
\frac{\a }{2j+\a}t_\ell},
\end{array}
\right.
\end{equation}
for all $r\geq 0,\b\in (0,1)$ and $S\leq Rd_\ell^{-1}$ with $R$ fixed (as in \cite[(4.172)]{HT3}). Likewise, if we take at most 2 fiber derivatives landing on $\hat \Phi$ and $\hat J_\ell$, then we use \eqref{recall-tilde-A-sharp} and \eqref{inter-estimate-2} to get
\begin{equation}\label{4.173}
\left\{
\begin{array}{ll}
d_\ell^{-r} \|\mathfrak{D}^r_{\bf bt} \tilde\eta_\ell^\dagger\|_{\infty,\tilde Q_S,\tilde g_\ell(0)}&\leq C  e^{\frac{-\a}{2}\cdot
\frac{\a }{2j+\a}t_\ell},\\[2mm]
 d_\ell^{-r-\b} [\mathfrak{D}^r_{\bf bt} \tilde\eta_\ell^\dagger]_{\b,\b/2,\mathrm{base}
 ,\tilde Q_S,\tilde g_\ell(0)}&\leq C  e^{\frac{-\a}{2}\cdot
\frac{\a }{2j+\a}t_\ell}.
\end{array}
\right.
\end{equation}
for all $r\geq 0,\b\in (0,1)$ and $S\leq Rd_\ell^{-1}$ with $R$ fixed (as in \cite[(4.173)]{HT3}). In particular they converge to zero. If instead we use fixed metric $g_X$, then the derivatives of $\hat \Phi$ and $\hat J_\ell$ are bounded and thus
\begin{equation}\label{eta-dagger}
\left\{
\begin{array}{ll}
 \|\mathfrak{D}^r \hat\eta_\ell^\dagger\|_{\infty,\hat Q_S,g_X}&\leq C  \delta_\ell^{2}e^{\frac{-\a}{2}\cdot
\frac{\a }{2j+\a}t_\ell},\\[2mm]
\, [\mathfrak{D}^r \hat\eta_\ell^\dagger]_{\b,\b/2,\hat Q_S,g_X}&\leq C  \delta_\ell^{2} e^{\frac{-\a}{2}\cdot
\frac{\a }{2j+\a}t_\ell},
\end{array}
\right.
\end{equation}
for all $r\geq 0,\b\in (0,1)$ and fixed $S>0$ (as in \cite[(4.174)]{HT3}). Similarly to the discussion of \cite[(4.175)]{HT3}, we also have
\begin{equation}\label{qarq}
\left\{
\begin{array}{ll}
 \|\mathfrak{D}^r \tilde\eta_\ell^\dagger\|_{\infty,\tilde Q_S,g_X}&\leq C  \e_\ell^{2}e^{\frac{-\a}{2}\cdot
\frac{\a }{2j+\a}t_\ell},\\[2mm]
\, [\mathfrak{D}^r \tilde\eta_\ell^\dagger]_{\b,\b/2,\tilde Q_S,g_X}&\leq C  \e_\ell^{2} e^{\frac{-\a}{2}\cdot
\frac{\a }{2j+\a}t_\ell},
\end{array}
\right.
\end{equation}
for all $r\geq 0,\b\in (0,1)$ and fixed $S>0$ (as in \cite[(4.175)]{HT3}).

Next, we estimate the potential of $\tilde\eta_\ell^\dagger$ when $\e_\ell\leq C$.  Given $r\geq 0$, by applying $\mathfrak{D}^r_{\bf bt}$ to $ \hat{\mathfrak{G}}_{\hat{t},k}(\hat A_{\ell,i,p,k}^\sharp,\hat G_{\ell,i,p,k})$ with \eqref{plop}
\begin{equation}
\begin{split}
&\quad \mathfrak{D}^r_{\bf bt}\hat{\mathfrak{G}}_{\hat{t},k}(\hat A_{\ell,i,p,k}^\sharp,\hat G_{\ell,i,p,k})\\
&= \sum_{\iota=0}^{2k}\sum_{q=\lceil\frac{\iota}{2}\rceil}^{k} \sum_{\substack{d=0\\d\in 2\mathbb{N}}}^r \sum_{i_1+i_2=r-d}e^{-q\lambda_\ell^{-2}\hat t-(q-\frac{\iota}{2})t_\ell} \left(q\lambda_\ell^{-2}\right)^{\frac{d}{2}} \delta_\ell^\iota\mathfrak{D}_{\bf bt}^{i_1}\hat \Phi_{\iota,q}(\hat G_{\ell,i,p,k}) \circledast \mathfrak{D}^{i_2+\iota} \hat A_{\ell,i,p,k}^\sharp,
\end{split}
\end{equation}
which can be estimated using \eqref{recall-tilde-A-sharp} and \eqref{inter-estimate-2} to get
\begin{equation}
\begin{split}
[ \mathfrak{D}^r_{\bf bt}\hat{\mathfrak{G}}_{\hat{t},k}(\hat A_{\ell,i,p,k}^\sharp,\hat G_{\ell,i,p,k})]_{\a,\a/2,\mathrm{base},\hat Q_{S d_\ell},\hat g_\ell(0)}\leq C \sum_{\iota=0}^{2k} \sum_{\substack{d=0\\d\in 2\mathbb{N}}}^r \sum_{i_1+i_2=r-d}\lambda_\ell^{-d-i_1}\delta_\ell^{\iota+2} e^{\frac{-2i+2-\a}{2}\cdot
\frac{\a }{2j+\a}t_\ell}=o(1).
\end{split}
\end{equation}

The $L^\infty$ norm is similar and hence in the tilde picture we have
\begin{equation}\label{A-sharp-potential}
\left\{
\begin{array}{ll}
d_\ell^{-r+2-\a}[ \mathfrak{D}^{r}_{\bf bt}\tilde{\mathfrak{G}}_{\ti{t},k}(\tilde A_{\ell,i,p,k}^\sharp,\tilde G_{\ell,i,p,k})]_{\a,\a/2,\mathrm{base},\tilde Q_{S},\tilde g_\ell(0)}=o(1),\\[2mm]
d_\ell^{-r+2}\| \mathfrak{D}^{r}_{\bf bt}\tilde{\mathfrak{G}}_{\ti{t},k}(\tilde A_{\ell,i,p,k}^\sharp,\tilde G_{\ell,i,p,k})\|_{\infty,\tilde Q_{S},\tilde g_\ell(0)}=o(1),\\[2mm]
\end{array}
\right.
\end{equation}
for all $r\geq 0$ and $S$ fixed.

\subsubsection{Estimates for $\tilde\omega_{\ell}^\sharp$}

Recall that
\begin{equation}
\tilde\omega_{\ell}^\sharp(\tilde t)=(1-e^{-t_\ell-d_\ell^2\lambda_\ell^{-2}\tilde t})\tilde\omega_{\ell,\mathrm{can}}+\e_\ell^2 e^{-d_\ell^2\lambda_\ell^{-2}\tilde t}\Theta_\ell^*\Psi_\ell^*\omega_{F} +\tilde\eta_\ell^\dagger+\tilde\eta_\ell^\ddagger.
\end{equation}

We can follow exactly the same discussion in \cite[\S 4.9.7]{HT3} to conclude the following. Since $\tilde\omega_{\ell,\mathrm{can}}=d_\ell^{-2}\lambda_\ell^2\Theta_\ell^*\Psi_\ell^* \omega_{\mathrm{can}}$ and $(\Psi_\ell \circ \Theta_\ell)(\tilde z,\tilde y)=(d_\ell \lambda_\ell^{-1}z,y)$ where $d_\ell \lambda_\ell^{-1}\to 0$, the spatial stretching implies
\begin{equation}\label{estimate-can-tilde-infty-1}
\left\{
\begin{array}{ll}
\| \mathfrak{D}^{\iota} \left((1-e^{-t_\ell-d_\ell^2\lambda_\ell^{-2}\tilde t})\tilde\omega_{\ell,\mathrm{can}}\right)\|_{\infty,\tilde Q_{S},\tilde g_\ell(0)}\leq C  d_\ell^\iota \lambda_\ell^{-\iota},\\[2mm]
\,[ \mathfrak{D}^{\iota}\left((1-e^{-t_\ell-d_\ell^2\lambda_\ell^{-2}\tilde t})\tilde\omega_{\ell,\mathrm{can}}\right)]_{\b,\b/2,\tilde Q_{S},\tilde g_\ell(0)}\leq C  d_\ell^{\iota+\b} \lambda_\ell^{-\iota-\b},\\[2mm]
\| \mathfrak{D}^{\iota} (\e_\ell^2 e^{-d_\ell^2\lambda_\ell^{-2}\tilde t}\Theta_\ell^*\Psi_\ell^*\omega_F ) \|_{\infty,\tilde Q_{S},\tilde g_\ell(0)}\leq C d_\ell^\iota  \delta_\ell^{-\iota},\\[2mm]
\, [\mathfrak{D}^{\iota} (\e_\ell^2  e^{-d_\ell^2\lambda_\ell^{-2}\tilde t}\Theta_\ell^*\Psi_\ell^*\omega_F ) ]_{\b,\b/2,\tilde Q_{S},\tilde g_\ell(0)}\leq C d_\ell^{\iota+\b}  \delta_\ell^{-\iota-\b}.
\end{array}
\right.
\end{equation}
 for all $\iota\geq 0,\b\in (0,1)$ and $S\leq Rd_\ell^{-1}$ with $R$ fixed (which is the analog \cite[(4.181)--(4.182)]{HT3}). Therefore,
 \begin{equation}\label{sharp-metric-decay-0}
 \left\{
 \begin{array}{ll}
 d_\ell^{-\iota} \|\mathfrak{D}^\iota \tilde\omega_\ell^\sharp\|_{\infty,\tilde Q_S,\tilde g_\ell(0)} \leq C\delta_\ell^{-\iota},\\[2mm]
\,  d_\ell^{-\iota-\b} [\mathfrak{D}^\iota \tilde\omega_\ell^\sharp]_{\b,\b/2,\tilde Q_S,\tilde g_\ell(0)} \leq C\delta_\ell^{-\iota-\b},
 \end{array}
 \right.
 \end{equation}
 for all $\iota\geq 0,\b\in (0,1)$ and $S\leq Rd_\ell^{-1}$ with $R$ fixed. Likewise, if we only differentiate in the base and time directions, then we have the following improvement:
 \begin{equation}\label{sharp-metric-decay}
 \left\{
 \begin{array}{ll}
 d_\ell^{-\iota} \|\mathfrak{D}^\iota_{\bf bt} \tilde\omega_\ell^\sharp\|_{\infty,\tilde Q_S,\tilde g_\ell(0)}=o(1),\\[2mm]
\,  d_\ell^{-\iota-\b} [\mathfrak{D}^\iota_{\bf bt} \tilde\omega_\ell^\sharp]_{\b,\b/2,\mathrm{base},\tilde Q_S,\tilde g_\ell(0)} =o(1),
 \end{array}
 \right.
 \end{equation}
  for all $\iota\geq 0,\b\in (0,1)$ with $\iota+\b>0$ and $S\leq C d_\ell^{-1}$(which is the analog of \cite[(4.183)]{HT3}) as well as
   \begin{equation}\label{F-metric-decay}
 \left\{
 \begin{array}{ll}
 d_\ell^{-\iota} \|\mathfrak{D}^\iota_{\bf bt} (\e_\ell^2 \Theta_\ell^*\Psi_\ell^* \omega_F)\|_{\infty,\tilde Q_S,\tilde g_\ell(0)}\leq C \lambda_\ell^{-\iota},\\[2mm]
\,  d_\ell^{-\iota-\b} [\mathfrak{D}^\iota_{\bf bt} (\e_\ell^2 \Theta_\ell^*\Psi_\ell^* \omega_F)]_{\b,\b/2,\tilde Q_S,\tilde g_\ell(0)}\leq C \lambda_\ell^{-\iota-\b},
 \end{array}
 \right.
 \end{equation}
  for all $\iota\geq 0,\b\in (0,1)$ and $S\leq C d_\ell^{-1}$(which is the analog of \cite[(4.184)--(4.185)]{HT3})

\subsection{Expansion of the Monge-Amp\`ere flow}

We rewrite the complex Monge-Amp\`ere equation \eqref{PMA-tilde} as
    \begin{equation}\label{PMA-tilde-new-1}
\begin{split}
&\quad e^{-d_\ell^2\lambda_\ell^{-2}\tilde t} \partial_{\tilde t}\tilde \chi_\ell\\
&= \tr_{\tilde \omega_\ell^\sharp}\left(\tilde\eta_\ell^\circ+\tilde\eta_\ell^\diamond+\tilde\eta_{\ell,j,k}\right) \\
&\quad+\left( \log \frac{(\tilde \omega_\ell^\sharp +\tilde\eta_\ell^\circ+\tilde\eta_\ell^\diamond+\tilde\eta_{\ell,j,k})^{m+n}}{ (\tilde\omega_\ell^\sharp)^{m+n}}-\tr_{\tilde \omega_\ell^\sharp}\left(\tilde\eta_\ell^\circ+\tilde\eta_\ell^\diamond+\tilde\eta_{\ell,j,k}\right)\right) \\
&\quad +\log \frac{(\tilde\omega_\ell^\sharp)^{m+n}}{ \binom{m+n}{m}\ti \omega_{\ell,\mathrm{can}} ^m \wedge (\e_\ell^2 \Theta_\ell^*\Psi_\ell^* \omega_F)^n}+nd_\ell^2\lambda_\ell^{-2}\tilde t\\
&=:\tr_{\tilde \omega_\ell^\sharp}\left(\tilde\eta_\ell^\circ+\tilde\eta_\ell^\diamond+\tilde\eta_{\ell,j,k}\right) +\mathcal{E}_1+\mathcal{E}_2,
\end{split}
\end{equation}
where the terms $\mathcal{E}_1$ and $\mathcal{E}_2$ are given respectively by the third and second line from the bottom in \eqref{PMA-tilde-new-1}.
Recall also that
\begin{equation}
\begin{split}
\tilde \chi_\ell=e^{d_\ell^2\lambda_\ell^{-2}\tilde t} \left(\tilde \psi_{\ell,j,k}+ \sum_{i=1}^j\sum_{p=1}^{N_{i,k}} \tilde{\mathfrak{G}}_{\ti{t},k}( \tilde A_{\ell,i,p,k}^*,\tilde G_{\ell,i,p,k} )\right)+ e^{d_\ell^2\lambda_\ell^{-2}\tilde t} \sum_{i=1}^j\sum_{p=1}^{N_{i,k}} \tilde{\mathfrak{G}}_{\ti{t},k}( \tilde A_{\ell,i,p,k}^\sharp,\tilde G_{\ell,i,p,k} )+\underline{\tilde\chi_\ell^*}+\underline{\tilde\chi_\ell^\sharp}.
\end{split}
\end{equation}
so that if we define
\begin{equation}\label{rho}
\ti{\rho}_\ell:=e^{-d_\ell^2\lambda_\ell^{-2}\tilde t}\underline{\tilde\chi_\ell^*}+\sum_{i=1}^j\sum_{p=1}^{N_{i,k}} \tilde{\mathfrak{G}}_{\ti{t},k}( \tilde A_{\ell,i,p,k}^*,\tilde G_{\ell,i,p,k} )+\tilde \psi_{\ell,j,k},
\end{equation}
then by definition we have
\begin{equation}
\ddbar\ti{\rho}_\ell=\tilde\eta_\ell^\diamond+ \tilde\eta_\ell^\circ+\tilde\eta_{\ell,j,k},
\end{equation}
and we can further rewrite the \eqref{PMA-tilde-new-1} as
\begin{equation}\label{PMA-tilde-new}
\begin{split}
\partial_{\tilde t} \ti{\rho}_\ell+d_\ell^2\lambda_\ell^{-2}\ti{\rho}_\ell&=\tr_{\tilde \omega_\ell^\sharp}\left(\tilde\eta_\ell^\circ+\tilde\eta_{\ell,j,k}+\tilde\eta_\ell^\diamond\right) +\mathcal{E}_1+\mathcal{E}_2\\
&\quad - e^{-d_\ell^2\lambda_\ell^{-2}\tilde t}\partial_{\tilde t} \left( e^{d_\ell^2\lambda_\ell^{-2}\tilde t} \sum_{i=1}^j\sum_{p=1}^{N_{i,k}} \tilde{\mathfrak{G}}_{\ti{t},k}( \tilde A_{\ell,i,p,k}^\sharp,\tilde G_{\ell,i,p,k} )\right)-e^{-d_\ell^2\lambda_\ell^{-2}\tilde t}\partial_{\tilde t}\underline{\tilde\chi_\ell^\sharp}\\
&=:\tr_{\tilde \omega_\ell^\sharp}\left(\tilde\eta_\ell^\circ+\tilde\eta_{\ell,j,k}+\tilde\eta_\ell^\diamond\right) +\mathcal{E}_1+\mathcal{E}_2+\mathcal{E}_3,
\end{split}
\end{equation}
where we defined
\begin{equation}\label{E3}
\mathcal{E}_3:=- e^{-d_\ell^2\lambda_\ell^{-2}\tilde t}\left(\partial_{\tilde t} \left( e^{d_\ell^2\lambda_\ell^{-2}\tilde t} \sum_{i=1}^j\sum_{p=1}^{N_{i,k}} \tilde{\mathfrak{G}}_{\ti{t},k}( \tilde A_{\ell,i,p,k}^\sharp,\tilde G_{\ell,i,p,k} )\right)+\partial_{\tilde t}\underline{\tilde\chi_\ell^\sharp}\right).
\end{equation}

The next Proposition gives us control on the error terms $\mathcal{E}_i$:
\begin{prop}\label{prop:Estimates-RHS}
For any fixed $R>0$ and $i=1,2,3$, we have
\begin{equation}\label{RHS-error1}
d_\ell^{-2j-\a} \left[\mathfrak{D}^{2j}_{\bf bt} \mathcal{E}_i \right]_{\a,\a/2,\mathrm{base},\tilde Q_R,\tilde g_\ell(0)}=o(1).
\end{equation}

In particular,
\begin{equation}\label{RHS-error-total}
\begin{split}
d_\ell^{-2j-\a}&\Biggr[ \mathfrak{D}^{2j}_{\bf bt}\biggr(\partial_{\tilde t} \ti{\rho}_\ell+d_\ell^2\lambda_\ell^{-2}\ti{\rho}_\ell - \tr_{\tilde \omega_\ell^\sharp}\left(\tilde\eta_\ell^\diamond+ \tilde\eta_\ell^\circ+\tilde\eta_{\ell,j,k}\right)  \biggl)\Biggr]_{\a,\a/2,\mathrm{base},\tilde Q_R,\tilde g_\ell(0)}=o(1).
\end{split}
\end{equation}

Furthermore, if $\e_\ell\geq C^{-1}$, then for any fixed $R>1$ and $0\leq a\leq 2j$, we have
\begin{equation}\label{error1-bad-1}
d_\ell^{-2j-\a}\|\mathfrak{D}^a\mathcal{E}_1\|_{\infty,\tilde Q_{R\e_\ell},  \tilde g_\ell(0)}\leq C\delta_\ell^{2j+\a } \e_\ell^{2j+\a-a},
\end{equation}
\begin{equation}\label{error1-bad-2}
 d_\ell^{-2j-\a}[\mathfrak{D}^a\mathcal{E}_1]_{\a,\a/2,\tilde Q_{R\e_\ell},  \tilde g_\ell(0)}\leq C\delta_\ell^{2j+\a } \e_\ell^{2j-a},
\end{equation}

\begin{equation}\label{error2-bad-1}
d_\ell^{-2j-\a}\|\mathfrak{D}^a (\mathcal{E}_2+\mathcal{E}_3)\|_{\infty,\tilde Q_{R\e_\ell},  \tilde g_\ell(0)}\leq C\e_\ell^{2j+\a-a},
\end{equation}
\begin{equation}\label{error2-bad-2}
d_\ell^{-2j-\a}[ \mathfrak{D}^a(\mathcal{E}_2+\mathcal{E}_3)]_{\a,\a/2,\tilde Q_{R\e_\ell},  \tilde g_\ell(0)}\leq C\e_\ell^{2j-a}.
\end{equation}
\end{prop}
\begin{proof}
The estimate \eqref{RHS-error1} for $\mathcal{E}_3$ follows easily from  \eqref{bdd-phi-sharp-underline} and \eqref{A-sharp-potential}.

As for $\mathcal{E}_2$, recall that
\begin{equation}
\begin{split}
\mathcal{E}_2&=\log \frac{(\tilde\omega_\ell^\sharp)^{m+n}}{ \binom{m+n}{m}\tilde\omega_{\ell,\mathrm{can}}^m\wedge (\e_\ell^2 \Theta_\ell^*\Psi_\ell^*\omega_{F})^n  }+nd_\ell^2\lambda_\ell^{-2}\tilde t.
\end{split}
\end{equation}
The term $nd_\ell^2\lambda_\ell^{-2}\tilde t$ is killed by $[\mathfrak{D}_{\bf bt}^{2j}\cdot ]_{\a,\a/2,\mathrm{base}}$ if $j>0$, while if $j=0$ we have
\begin{equation}
d_\ell^{-\alpha}[nd_\ell^2\lambda_\ell^{-2}\tilde t]_{\a,\a/2,\mathrm{base},\tilde Q_R,\tilde g_\ell(0)}\leq Cd_\ell^{2-\alpha}\lambda_\ell^{-2}=o(1).
\end{equation}
Estimate \eqref{RHS-error1} for $\mathcal{E}_2$ then follows from this together with \eqref{estimate-can-tilde-infty-1}, \eqref{sharp-metric-decay} and \eqref{F-metric-decay}.

For $\mathcal{E}_1$, we write
\begin{equation}
\begin{split}
&\quad  \log \frac{(\tilde \omega_\ell^\sharp +\tilde\eta_\ell^\circ+\tilde\eta_\ell^\diamond+\tilde\eta_{\ell,j,k})^{m+n}}{ (\tilde\omega_\ell^\sharp)^{m+n}}=:\log \left( 1+\tr_{\tilde \omega_\ell^\sharp}\left(\tilde\eta_\ell^\circ+\tilde\eta_\ell^\diamond+\tilde\eta_{\ell,j,k}\right) +\mathcal{E}_4\right),
\end{split}
\end{equation}
where we defined
\begin{equation}\label{E4}
\mathcal{E}_4:=\sum_{p=2}^{m+n}\binom{m+n}{p}\frac{\left(\tilde\eta_\ell^\circ+\tilde\eta_\ell^\diamond+\tilde\eta_{\ell,j,k}\right)^p\wedge(\tilde \omega_\ell^\sharp)^{m+n-p}}{(\tilde \omega_\ell^\sharp)^{m+n}}.
\end{equation}
Thanks to \eqref{estimate-psi-tilde-infty-1},  \eqref{gorgo}, \eqref{estimate-circ-decay}, and   \eqref{sharp-metric-decay}, for every $0\leq \iota\leq 2j$ and fixed $R>0$ we have
\begin{equation} \label{E3-trace-estimate}
\left\{
\begin{array}{ll}
d_\ell^{-\iota}\|\mathfrak{D}^{\iota}_{\bf bt}\mathcal{E}_4\|_{\infty,\tilde Q_R, \tilde g_\ell(0)}=o(1),\\[2mm]
d_\ell^{-2j-\a}[\mathfrak{D}^{2j}_{\bf bt}\mathcal{E}_4]_{\a,\a/2,\mathrm{base},\tilde Q_R, \tilde g_\ell(0)}=o(1),\\[2mm]
d_\ell^{-\iota}\|\mathfrak{D}_{\bf bt}^\iota \tr_{\tilde \omega_\ell^\sharp}\left(\tilde\eta_\ell^\circ+\tilde\eta_\ell^\diamond+\tilde\eta_{\ell,j,k}\right)\|_{\infty,\tilde Q_R,\tilde g_\ell(0)}=O(1),\\
 \,  d_\ell^{-2j-\a}\left [\mathfrak{D}^{2j}_{\bf bt}\tr_{\tilde \omega_\ell^\sharp}\left(\tilde\eta_\ell^\circ+\tilde\eta_\ell^\diamond+\tilde\eta_{\ell,j,k}\right)\right]_{\a,\a/2,\mathrm{base},\tilde Q_R,\tilde g_\ell(0)}=O(1).\\
\end{array}
\right.
\end{equation}
If we write schematically $\mathbf{A}=\tr_{\tilde \omega_\ell^\sharp}\left(\tilde\eta_\ell^\circ+\tilde\eta_\ell^\diamond+\tilde\eta_{\ell,j,k}\right), \mathbf{B}=\mathcal{E}_4$, which are both $o(1)$ locally uniformly, then we have $\mathcal{E}_1=\log(1+\mathbf{A}+\mathbf{B})-\mathbf{A}.$ For any $a\geq 0$ we can then write schematically
\begin{equation}\label{log-estimate}
\begin{split}
\DD^a\mathcal{E}_1&=\DD^a \left(\log (1+\mathbf{A}+\mathbf{B})-\mathbf{A} \right)
=-\DD^a\mathbf{A}\frac{\mathbf{A}+\mathbf{B}}{1+\mathbf{A}+\mathbf{B}}+\frac{\DD^a\mathbf{B}}{1+\mathbf{A}+\mathbf{B}}\\
&+\sum_{\substack{i_1+i_2=a-1\\ i_2>0}}\sum_{\iota=1}^{i_2}\sum_{j_1+\dots+j_{\iota}=i_2} \frac{\mathfrak{D}^{i_1+1}(1+\mathbf{A}+\mathbf{B})}{1+\mathbf{A}+\mathbf{B}}
\frac{\mathfrak{D}^{j_1}(1+\mathbf{A}+\mathbf{B})}{1+\mathbf{A}+\mathbf{B}}\cdots \frac{\mathfrak{D}^{j_\iota}(1+\mathbf{A}+\mathbf{B})}{1+\mathbf{A}+\mathbf{B}},
\end{split}
\end{equation}
and then \eqref{RHS-error1} for $\mathcal{E}_1$ follows from \eqref{E3-trace-estimate}.

Now that  \eqref{RHS-error1} is established,  \eqref{RHS-error-total} follows immediately from this and \eqref{PMA-tilde-new}.

On the other hand when $\e_\ell \geq C^{-1}$, estimates \eqref{error1-bad-1} and \eqref{error1-bad-2} follow from  \eqref{estimate-psi-tilde-infty-1},  \eqref{gorgo}, \eqref{circ-e-nonzero}, \eqref{sharp-metric-decay-0} together with \eqref{log-estimate}.  Lastly, to prove  \eqref{error2-bad-1} and \eqref{error2-bad-2},
using \eqref{PMA-tilde-new}, \eqref{error1-bad-1} and \eqref{error1-bad-2} it suffices to show that
\begin{equation}\begin{split}
d_\ell^{-2j-\a}\Bigg\|\DD^a\Bigg(&\partial_{\tilde t} \ti{\rho}_\ell+d_\ell^2\lambda_\ell^{-2}\ti{\rho}_\ell-\tr_{\tilde \omega_\ell^\sharp}\left(\tilde\eta_\ell^\circ+\tilde\eta_{\ell,j,k}+\tilde\eta_\ell^\diamond\right)\Bigg)\Bigg\|_{\infty,\tilde Q_{R\e_\ell},  \tilde g_\ell(0)}\leq C\e_\ell^{2j+\a-a},
\end{split}
\end{equation}
\begin{equation}\begin{split}
d_\ell^{-2j-\a}\Bigg[\DD^a\Bigg(&e\partial_{\tilde t} \ti{\rho}_\ell+d_\ell^2\lambda_\ell^{-2}\ti{\rho}_\ell-\tr_{\tilde \omega_\ell^\sharp}\left(\tilde\eta_\ell^\circ+\tilde\eta_{\ell,j,k}+\tilde\eta_\ell^\diamond\right)\Bigg)\Bigg]_{\a,\a/2,\tilde Q_{R\e_\ell},  \tilde g_\ell(0)}\leq C\e_\ell^{2j-a},
\end{split}
\end{equation}
for fixed $R>0$ and $0\leq a\leq 2j$, which is a direct consequence of  \eqref{estimate-psi-tilde-infty-1},  \eqref{bdd-phi-star-underline}, \eqref{estimate-potential-G},  \eqref{circ-e-nonzero} together with \eqref{sharp-metric-decay-0} and \eqref{log-estimate}.
\end{proof}

The goal is then to kill the RHS of \eqref{sup-realized-20} when $j=0$, and to kill the contributions of $d_\ell^{-2j-\a}\underline{\ti{\chi}_\ell^*}$, $d_\ell^{-2j-\a}\tilde\psi_{\ell,j,k}$ and $d_\ell^{-2j-\a}\tilde A_{\ell,i,p,k}^*$ to \eqref{sup-realized-3} when $j\geq 1$. We will split the discussion into three case (without loss of generality): {\bf Subcase A}: $\e_\ell\to +\infty$, {\bf Subcase B}: $\e_\ell\to \e_\infty>0$, {\bf Subcase C}: $\e_\ell\to 0$ as $\ell\to +\infty$ where $\e_\ell=d_\ell^{-1}\delta_\ell=d_\ell^{-1}\lambda_\ell e^{-t_\ell/2}$.
\subsection{Subcase A: $\e_\ell\to +\infty$}
In this subcase the background geometry is diverging in the fiber directions, and similarly to the analogous case in \cite[\S 4.10]{HT3} we will kill all contributions to \eqref{sup-realized-3} using \textit{parabolic} Schauder estimates for the linear heat equation. The Selection Theorem \ref{thm:Selection} will also be used crucially. The argument is quite long and involved because of the complexity of the quantitative estimates satisfied by all the pieces in the decomposition of the solution $\ti{\omega}^\bullet_\ell$. We start with the direct analog of the non-cancellation result in \cite[Proposition 4.7]{HT3}.

\begin{prop}\label{prop:noncancellation}
The following inequalities hold for all $0\leq \a<1$, $a\in \mathbb{N}$, $1\leq i\leq j$, $1\leq p\leq N_{i,k}$ and all $R>0$:
\begin{equation}\label{borgu}
[\DD^a \ddbar\underline{\hat{\vp}_\ell}]_{\a,\a/2,\hat Q_R }\leq C[\mathfrak{D}^{a}_{\bf bt}\ddbar \hat\varphi_\ell]_{\a,\a/2,\textrm{base},\hat Q_R,g_X}+C\left(\frac{R}{\lambda_\ell} \right)^{1-\a}\sum_{b=0}^{a} \lambda_\ell^{-\a} \|\mathfrak{D}^b\ddbar \hat\varphi_\ell\|_{\infty,\hat Q_R,g_X},
\end{equation}
\begin{equation}
\begin{split}
&[\mathfrak{D}^{a} \hat A_{\ell,i,p,k}]_{\a,\a/2,\hat Q_R }\leq C\delta_\ell^2 \Bigg( [\mathfrak{D}^{a}_{\bf bt}\ddbar \hat\varphi_\ell]_{\a,\a/2,\textrm{base},\hat Q_R,\hat g_\ell(0)}+\\
&+\left(\frac{R}{\lambda_\ell} \right)^{1-\a}\sum_{b=0}^{a} \lambda_\ell^{-\a}\Big( \|\mathfrak{D}^b \hat \eta_\ell^\ddagger\|_{\infty,\hat Q_R,\hat g_\ell(0)}+\|\mathfrak{D}^b(\hat\eta_\ell^\diamond+\hat\eta_\ell^\circ+\hat \eta_{\ell,j,k}) \|_{\infty,\hat Q_R,\hat g_\ell(0)}
+\delta_\ell^{-2}\|\mathfrak{D}^b\hat \eta_\ell^\dagger\|_{\infty,\hat Q_R,g_X}\Big)\Bigg)\\
&+C\sum_{\iota=1}^{i-1}\sum_{q=1}^{N_{i,k}}\left(\delta_\ell^{2k+2} [\mathfrak{D}^{a+2k+2}\hat A_{\ell,\iota,q,k}]_{\a,\a/2,\hat Q_R}+\sum_{b=0}^{a+2k+2} \left(\frac{R}{\lambda_\ell}\right)^{1-\a}e^{-(2k+2)\frac{t_\ell}{2}}\lambda_\ell^{b-a-\a}\|\mathfrak{D}^b \hat A_{\ell,\iota,q,k}\|_{\infty,\hat Q_R}\right).
\end{split}
\end{equation}
\end{prop}
\begin{proof}
The proof is very similar to that of \cite[Proposition 4.7]{HT3}, so we only highlight the differences. The starting point of the proof (Claim 1 in \cite[Proof of Proposition 4.7]{HT3}) is to express $\ddbar\underline{\hat{\vp}_\ell}$ and $\hat A_{\ell,i,p,k}$ as pushforwards of quantities on the total space. This step is essentially identical here, with the only difference being that in the formula for $\hat A_{\ell,i,p,k}$, the term  $e^{-(2k+2)\frac{t}{2}}$ in  \cite[(4.204)]{HT3} now becomes $e^{-(2k+2)(t_\ell+\lambda_\ell^{-2}\hat t)/2}.$ The extra time-dependent constant $e^{-(2k+2)\lambda_\ell^{-2}\frac{\hat t}{2}}$ will then also need to be differentiated in the analog of \cite[(4.234)]{HT3}, which gives us extra cross terms in the analog of \cite[(4.247)]{HT3}, but which can be estimated in a similar way resulting in the same upper bound as stated.

The next step is to try to commute the derivative $\DD^a$ with the pushforward, and make the commutation error terms explicit.  Recalling that $\DD^{a}$ is a sum of terms of the form $\D^p\de_{\hat{t}}^q$, we observe that $\de_{\hat{t}}^q$ trivially commutes with pushforwards (with no error terms), while the commutation of $\D^p$ gives exactly the same result as in Claim 2 \cite[Proof of Proposition 4.7]{HT3}, cf. \cite[(4.223), (4.240)]{HT3}.

The last step is to estimate the H\"older difference quotient of $\DD^a\ddbar\underline{\hat{\vp}_\ell}$ and $\DD^a\hat A_{\ell,i,p,k}$. This is now a space-time H\"older difference quotient, which we can split with the triangle inequality into a space-only difference quotient (which is estimated following the method of Claim 3 in \cite[Proof of Proposition 4.7]{HT3} verbatim), and a time-only difference quotient, which again commutes with pushforward and so can be estimated trivially without any further error terms. This completes the outline of the proof.
\end{proof}

For notation convenience, we will denote $\hat\eta_\ell=\ddbar\hat \varphi_\ell=\hat \eta_\ell^\dagger+\hat \eta_\ell^\ddagger+\hat\eta_\ell^\diamond+\hat\eta_\ell^\circ+\hat \eta_{\ell,j,k}$.

\subsubsection{The case $j=0$} Unlike \cite{HT3}, the case $j=0$ requires a separate treatment. This is due to the fact that the Monge-Amp\`ere equation \eqref{PMA-tilde-new} is naturally a parabolic PDE for the scalar potential $\ti{\rho}_\ell$, which however does not have a uniform bound on its $L^\infty$ norm, which is an issue when applying Schauder estimates. This is remedied in two different ways according to whether $j=0$ or $j>0$. In this subsection we treat the case $j=0$.

The first crucial claim is that for any fixed $R>0$ we have
\begin{equation}\label{crox}
d_\ell^{-\alpha}[(\mathcal{E}_1+\mathcal{E}_2+\mathcal{E}_3)]_{\a,\a/2,\ti{Q}_{R},\ti{g}_\ell(0)}=o(1).
\end{equation}
For the term $\mathcal{E}_1$, we have already proved an even strong result in \eqref{error1-bad-2}, so we consider $\mathcal{E}_2+\mathcal{E}_3$, which when $j=0$ equals
\begin{equation}
\log \frac{(\tilde\omega_\ell^\sharp)^{m+n}}{ \binom{m+n}{m}\ti \omega_{\ell,\mathrm{can}} ^m \wedge (\e_\ell^2 \Theta_\ell^*\Psi_\ell^* \omega_F)^n}+nd_\ell^2\lambda_\ell^{-2}\tilde t
-e^{-d_\ell^2\lambda_\ell^{-2}\ti{t}}\de_{\ti{t}}\underline{\ti{\chi}^\sharp_\ell},
\end{equation}
and since $\de_{\ti{t}}\underline{\ti{\chi}^\sharp_\ell}$ is a constant (in space and time) it is clear that
\begin{equation}
d_\ell^{-\alpha}[nd_\ell^2\lambda_\ell^{-2}\tilde t
-e^{-d_\ell^2\lambda_\ell^{-2}\ti{t}}\de_{\ti{t}}\underline{\ti{\chi}^\sharp_\ell}]_{\a,\a/2,\ti{Q}_{R},\ti{g}_\ell(0)}=o(1),
\end{equation}
and we are left with showing that
\begin{equation}\label{blrb}
d_\ell^{-\alpha}\left[\log \frac{(\tilde\omega_\ell^\sharp)^{m+n}}{ \binom{m+n}{m}\ti \omega_{\ell,\mathrm{can}} ^m \wedge (\e_\ell^2 \Theta_\ell^*\Psi_\ell^* \omega_F)^n}\right]_{\a,\a/2,\ti{Q}_{R},\ti{g}_\ell(0)}=o(1).
\end{equation}
For this, we pass to the check picture using the diffeomorphisms $\Pi_\ell$ in \eqref{pi} below, scaling geometric quantities by $\ve_\ell^{-2}$, so that the quantity in \eqref{blrb} equals
\begin{equation}\label{blrb2}
\delta_\ell^{-\alpha}\left[\log \frac{(\check\omega_\ell^\sharp)^{m+n}}{ \binom{m+n}{m}\check \omega_{\ell,\mathrm{can}} ^m \wedge (\Sigma_\ell^* \omega_F)^n}\right]_{\a,\a/2,\check{Q}_{R\ve_\ell^{-1}},\check{g}_\ell(0)},
\end{equation}
and using \eqref{checkko}, together with the facts that $[\check{\eta}^\ddagger_\ell]_{\a,\a/2,\check{Q}_{R\ve_\ell^{-1}},\check{g}_\ell(0)}=0,$ and
$\delta_\ell^{-\alpha}[\check{\omega}_{\ell,\mathrm{can}}]_{\a,\a/2,\check{Q}_{R\ve_\ell^{-1}},\check{g}_\ell(0)}\leq C\lambda_\ell^{-\alpha}=o(1),$
we see that \eqref{blrb} holds.

The next issue we face is that \eqref{estimate-psi-tilde-infty-10} does not provide us with uniform bounds on $d_\ell^{-\a}\|\DD^{2}\ti{\psi}_{\ell,0,k}\|_{\infty,\ti{Q}_R,\ti{g}_\ell(0)},$ for any fixed $R$, so we are unable to pass $d_\ell^{-\a}\de_{\ti{t}}\ti{\psi}_{\ell,0,k}$ or $d_\ell^{-\a}\ddbar\ti{\psi}_{\ell,0,k}$ to a limit. To fix this, we use the method of \cite[Subclaim 1.3]{HT2}, by replacing the whole fiber $Y$ with a coordinate chart and performing a jet subtraction to  $d_\ell^{-\a}\ti{\psi}_{\ell,0,k}$ so that the remainder is locally $C^2$ convergent.

To fix this, recall that $\tilde{g}_\ell(\ti{t}) = g_{\C^m} + \ve_\ell^2 e^{-d_\ell^2\lambda_\ell^{-2}\ti{t}} g_{Y,0}$.
Let $\mathbf{x}^{2m+1}, \ldots, \mathbf{x}^{2m+2n}$ be normal coordinates for $g_{Y,0}$ centered at $y_\ell$. Viewed as a map from $Y$ to $\R^{2n}$ these depend on $\ell$, but we prefer to instead pull back our setup to $\mathbb{R}^{2n}$ under the inverse map. In this sense we may then assume without loss that
\begin{align}
\left|\frac{\partial^\iota}{\partial\mathbf{x}^\iota}(g_{Y,0}(\mathbf{x})_{ab}-\delta_{ab})\right| \leq \frac{1}{100}|\mathbf{x}|^{2-\iota}\,\;{\rm for}\;\, |\mathbf{x}| \leq 2\;\,{\rm and}\;\,\iota=0,1.
\end{align}
This is possible thanks to the compactness of $Y$.
Define $\tilde{\mathbf{x}}^j = \ve_\ell \mathbf{x}^j$, so that $\tilde{\mathbf{x}}^{2m+1}, \ldots, \tilde{\mathbf{x}}^{2m+2n}$ are normal coordinates for $\ve_\ell^2 g_{Y,0}$ centered at $y_\ell$.
Formally also write $\tilde{\mathbf{x}}^1, \ldots, \tilde{\mathbf{x}}^{2m}$ for the standard real coordinates on $\C^m$. Then $\tilde{\mathbf{x}}^1, \ldots, \tilde{\mathbf{x}}^{2m+2n}$ are normal coordinates for $\tilde{g}_{\ell}(0)$ centered at $\tilde{x}_\ell$ with
\begin{align}\label{normk}
\left|\frac{\partial^\iota}{\partial\tilde{\mathbf{x}}^\iota}(\tilde{g}_{\ell}(0,\tilde{\mathbf{x}})_{ab}-\delta_{ab})\right| \leq \frac{\ve_\ell^{-2}}{100}|\tilde{\mathbf{x}}|^{2-\iota}\,\;{\rm for}\;\, |\tilde{\mathbf{x}}| \leq 2\ve_\ell\;\,{\rm and}\;\,\iota=0,1.
\end{align}
We then define a function $\ti{\psi}^\sharp_{\ell,0,k}$ on $\ti{Q}_{2\ve_\ell}$ to be the parabolic $2$nd order Taylor polynomial of $\ti{\psi}_{\ell,0,k}$ at the space-time origin $(0,0)$, using the spatial coordinates $\tilde{\mathbf{x}}^i,1\leq i\leq 2m+2n$, and we define also $\ti{\psi}^*_{\ell,0,k}:=\ti{\psi}_{\ell,0,k}-\ti{\psi}^\sharp_{\ell,0,k}.$

Since all Euclidean derivatives of $\ti{\psi}^*_{\ell,0,k}$ of order at most $2$ vanish at $(0,0)$, using the formula in \cite[Lemma 2.3]{HT3} relating $\D$-derivatives and ordinary derivatives, we see that $\DD^\iota\ti{\psi}^*_{\ell,0,k}\big|_{(0,0)}=0,0\leq\iota\leq 2$.

Next, we prove the following bounds for $d_\ell^{-\alpha}\ti{\psi}^\sharp_{\ell,0,k}$: for all $\iota\geq 0, 0<\beta<1$, and $0<R\leq \ve_\ell$,
\begin{equation}\label{jet1}
d_\ell^{-\alpha}[\DD^\iota\ti{\psi}^\sharp_{\ell,0,k}]_{\b,\b/2,\ti{Q}_{R},\ti{g}_\ell(0)}\leq \begin{cases}
C\ve_\ell^{1+\alpha-\iota}R^{1-\beta},\quad \text{if }\iota>0,\\
C\ve_\ell^{1+\alpha}R^{1-\beta}+C\ve_\ell^{\alpha}R^{2-\beta},\quad \text{if }\iota=0,
\end{cases}
\end{equation}
\begin{equation}\label{jet2}
d_\ell^{-\alpha}\|\DD^\iota\ti{\psi}^\sharp_{\ell,0,k}\|_{\infty,\ti{Q}_{R},\ti{g}_\ell(0)}\leq C\ve_\ell^{2+\alpha-\iota}.
\end{equation}

To prove this claim, note that for any $z\in B_1,$ the metric $g_{Y,z}$ is at bounded distance to $g_{Y,0}$ in $C^\infty(Y)$, thus
\begin{align}\label{!!!}
|\nabla^{\iota,\tilde{g}_\ell(0)}\tilde{g}_{\tilde{z}}(\ti{t})|_{\tilde{g}_\ell(0)} \leq C_\iota\ve_\ell^{-\iota} \;\,{\rm on}\;\,\ti{Q}_R,
\end{align}
for all $\iota\geq 1, R\leq\ve_\ell$.

Let us also note the following bounds for the Euclidean derivatives $\ti{\psi}_{\ell,0,k}$,
\begin{equation}\label{unnumb}
d_\ell^{-\alpha}|\de^{|\gamma|}\de_{\ti{t}}^q\ti{\psi}_{\ell,0,k}|(\ti{\mathbf{x}}_\ell,\ti{t}_\ell)\leq C\ve_\ell^{2+\alpha-|\gamma|-2q},
\end{equation}
for all multiindices $\gamma$ with $|\gamma|+2q\leq 2$. To see this, we first apply the diffeomorphism
$$\Lambda_\ell:\check{Q}_2\to \ti{Q}_{2\ve_\ell},\quad (\ti{\mathbf{x}}^1,\dots,\ti{\mathbf{x}}^{2m+2n},\ti{t})=\Lambda_\ell(\check{\mathbf{x}}^1,\dots,\check{\mathbf{x}}^{2m+2n},\check{t})
=(\ve_\ell\check{\mathbf{x}}^1,\dots,\ve_\ell\check{\mathbf{x}}^{2m+2n},\ve_\ell^2\check{t}),$$
and pull back metrics and $2$-forms, as well as $\ti{\psi}_{\ell,0,k}$, via $\Lambda_\ell$, multiply them by $\ve_\ell^{-2}$ and denote the resulting objects with a check, so for example the metrics $\check{g}_\ell(0)=\ve_\ell^{-2}\Lambda_\ell^*\ti{g}_\ell(0)$ on $\check{Q}_2$ are smoothly convergent to a fixed metric (smoothly comparable to Euclidean), and the pulled back complex structure is approaching the Euclidean one (without loss). The bounds \eqref{estimate-psi-tilde-infty-10} transform to
\[d_\ell^{-\alpha}\|\DD^\iota\check{\psi}_{\ell,0,k}\|_{\infty,\check{Q}_{2},\check{g}_\ell(0)}\leq C\ve_\ell^{2+\alpha},\quad d_\ell^{-\alpha}[\DD^\iota\check{\psi}_{\ell,0,k}]_{\a,\a/2,\check{Q}_{2},\check{g}_\ell(0)}\leq C\ve_\ell^{2+\alpha},
\]
for $0\leq \iota\leq 2$. Since $\check{g}_\ell(0)$ is approximately Euclidean, \cite[Lemma 2.6]{HT3} gives us that the Euclidean $C^{\alpha,\a/2}$ norm of $d_\ell^{-\alpha}\check{\psi}_{\ell,0,k}$ is also bounded by $C\ve_\ell^{2+\alpha}$, and translating these back to the tilde picture proves \eqref{unnumb}.

First we prove \eqref{jet1}. Given $(\ti{x},\ti{t}),(\ti{x}',\ti{t}')\in \ti{Q}_R$, call $d=d^{\ti{g}_\ell(0)}(\ti{x},\ti{x}')+|\ti{t}-\ti{t}'|^{\frac{1}{2}},$ and given $p,q\geq 0$ with $p+2q=\iota$, we can bound
\begin{equation}\label{mezzed}\begin{split}
&d_\ell^{-\alpha}|\D^p\de_{\ti{t}}^q\ti{\psi}^\sharp_{\ell,0,k}(\ti{x},\ti{t})-\P_{\ti{x}'\ti{x}}(\D^p\de_{\ti{t}}^q\ti{\psi}^\sharp_{\ell,0,k}(\ti{x}',\ti{t}'))|_{\ti{g}_\ell(0)}\\
&\leq d_\ell^{-\alpha}|\D^p\de_{\ti{t}}^q\ti{\psi}^\sharp_{\ell,0,k}(\ti{x},\ti{t})-\P_{\ti{x}'\ti{x}}(\D^p\de_{\ti{t}}^q\ti{\psi}^\sharp_{\ell,0,k}(\ti{x}',\ti{t}))|_{\ti{g}_\ell(0)}
+d_\ell^{-\alpha}|\D^p\de_{\ti{t}}^q\ti{\psi}^\sharp_{\ell,0,k}(\ti{x}',\ti{t})-\P_{\ti{x}'\ti{x}}(\D^p\de_{\ti{t}}^q\ti{\psi}^\sharp_{\ell,0,k}(\ti{x}',\ti{t}'))|_{\ti{g}_\ell(0)}\\
&\leq d^{\ti{g}_\ell(0)}(\ti{x},\ti{x}')d_\ell^{-\alpha}\|\D^{p+1}\de_{\ti{t}}^q\ti{\psi}^\sharp_{\ell,0,k}\|_{\infty,\ti{Q}_{2R},\ti{g}_\ell(0)}
+|\ti{t}-\ti{t}'|d_\ell^{-\alpha}\|\D^{p}\de_{\ti{t}}^{q+1}\ti{\psi}^\sharp_{\ell,0,k}\|_{\infty,\ti{Q}_{2R},\ti{g}_\ell(0)}\\
&\leq d\cdot d_\ell^{-\alpha}\|\D^{p+1}\de_{\ti{t}}^q\ti{\psi}^\sharp_{\ell,0,k}\|_{\infty,\ti{Q}_{2R},\ti{g}_\ell(0)}
+d^2\cdot d_\ell^{-\alpha}\|\D^{p}\de_{\ti{t}}^{q+1}\ti{\psi}^\sharp_{\ell,0,k}\|_{\infty,\ti{Q}_{2R},\ti{g}_\ell(0)}.
\end{split}
\end{equation}
Since $\ti{\psi}^\sharp_{\ell,0,k}$ is the sum of a polynomial of degree at most $2$ in the $\tilde{\mathbf{x}}$ variables (constant in time), and of a polynomial of degree at most $1$ in the $\ti{t}$ variable (constant in space), if follows that $\D^{p+1}\de_{\ti{t}}^q\ti{\psi}^\sharp_{\ell,0,k}\equiv 0$ unless $q=0$ (hence $\iota=p$), and that $\D^{p}\de_{\ti{t}}^{q+1}\ti{\psi}^\sharp_{\ell,0,k}\equiv 0$ unless $p=q=0$ (and hence $\iota=0$).

We consider these two cases separately, so we first bound the term with $\D^{\iota+1}\ti{\psi}^\sharp_{\ell,0,k}$ (which is equal to $\D^{\iota+1}$ applied to the spatial Taylor polynomial only), by converting $\D^{\iota+1}$ into $\nabla^{\iota+1}$ using \cite[Lemma 2.3]{HT3} (which involves a certain tensor $\mathbf{A}$), and estimating
\begin{align}
d_\ell^{-\alpha}\|\D^{\iota+1}\ti{\psi}^\sharp_{\ell,0,k}\|_{\infty,\ti{Q}_{2R},\ti{g}_\ell(0)} \\
\leq \biggl\|d_\ell^{-\alpha}&\left(\left(
\frac{\partial}{\partial\tilde{\mathbf{x}}} + \Gamma^{\tilde{g}_{\ti{z}}(\ti{t})}(\tilde{\mathbf{x}})\right)^{\iota+1}+\sum_{r=0}^{\iota-1}\nabla_{z,y,\ti{z}}^{\iota-1-r}\mathbf{A}\circledast\left(
\frac{\partial}{\partial\tilde{\mathbf{x}}} + \Gamma^{\tilde{g}_{\ti{z}}(\ti{t})}(\tilde{\mathbf{x}})\right)^{r}\right)\\
&\bigg(\sum_{\substack{\gamma\in\mathbb{N}^{2m+2n}\\ |\gamma|\leq 2}} \frac{1}{\gamma!} \frac{\partial^{|\gamma|}\ti{\psi}_{\ell,0,k}}{\partial\tilde{\mathbf{x}}^\gamma}(\tilde{\mathbf{x}}_\ell,\ti{t}_\ell)(\tilde{\mathbf{x}} - \tilde{\mathbf{x}}_\ell)^\gamma
\bigg)
\biggr\|_{\infty,\ti{Q}_{2R},\ti{g}_\ell(0)}
\end{align}
and estimating the big $L^\infty$ norm by $C\ve_\ell^{1+\alpha-\iota}$, as follows.

(1)
 We have $ \partial^b \Gamma$ $=$ $O( \ve_\ell^{-b-1})$ by \eqref{normk} and \eqref{!!!}.\hfill

(2) The $\mathbf{A}$-tensor in the tilde picture is bounded by $O(\ve_\ell^{-2})$, since it is schematically of the same type as $\de\Gamma$. By the same reason, $\nabla_{z,y,\ti{z}}^{r}\mathbf{A}$ is $O(\ve_\ell^{-r-2})$.

(3) Writing $d_\ell^{-\alpha}\ti{\psi}_{\ell,0,k}=\psi$, we can then estimate
\begin{equation}\label{agnusdei}\begin{split}
(\partial + \Gamma)^r((\partial^{|\gamma|}\psi)(\tilde{\mathbf{x}}_\ell,\ti{t}_\ell) (\tilde{\mathbf{x}}-\tilde{\mathbf{x}}_\ell)^\gamma)
&=(\partial^{|\gamma|}\psi)(\tilde{\mathbf{x}}_\ell,\ti{t}_\ell) \sum \partial^{a_1}\Gamma \cdots \partial^{a_\ell}\Gamma \cdot \partial^{b} (\tilde{\mathbf{x}}-\tilde{\mathbf{x}}_\ell)^\gamma,\end{split}
\end{equation}
where in the sum $a_1 + \cdots + a_\ell + \ell + b = r$  by counting the total number of $\partial$s and $\Gamma$s in each term, and $b\leq |\gamma|$. Now recall that $R\leq \ve_\ell$, so that
$\partial^b (\tilde{\mathbf{x}}-\tilde{\mathbf{x}}_\ell)^\gamma=O(\ve_\ell^{|\gamma|-b})$. Since $(\partial^{|\gamma|}\psi)(\tilde{\mathbf{x}}_\ell) = O(\ve_\ell^{2+\alpha-|\gamma|})$ by \eqref{unnumb} and using Step (1), the quantity in \eqref{agnusdei} can be estimated by $O(\ve_\ell^{2+\alpha-r})$.

(4) From Steps (2) and (3) we can bound
\begin{equation}\left(\nabla_{z,y,\ti{z}}^{\iota-1-r}\mathbf{A}\circledast(\de+\Gamma)^r\right)\sum_\gamma(\partial^{|\gamma|}\psi)(\tilde{\mathbf{x}}_\ell,\ti{t}_\ell) (\tilde{\mathbf{x}}-\tilde{\mathbf{x}}_\ell)^\gamma=O(\ve_\ell^{1+\alpha-\iota}),\end{equation}
and so using Step (3) again we obtain the desired bound of $C\ve_\ell^{1+\alpha-\iota}$ for the big $L^\infty$ norm. This gives us the desired bound
\begin{equation}\label{mezzedi}
d_\ell^{-\alpha}\|\D^{\iota+1}\ti{\psi}^\sharp_{\ell,0,k}\|_{\infty,\ti{Q}_{2R},\ti{g}_\ell(0)}\leq C\ve_\ell^{1+\alpha-\iota},
\end{equation}
for the first term in the last line of \eqref{mezzed}. As for the other term in that line, it is only nontrivial when $\iota=p=q=0$, and in that case we want to bound $d_\ell^{-\alpha}\|\de_{\ti{t}}\ti{\psi}^\sharp_{\ell,0,k}\|_{\infty,\ti{Q}_{2R},\ti{g}_\ell(0)}$. Since we have $\de_{\ti{t}}\ti{\psi}^\sharp_{\ell,0,k}=(\de_{\ti{t}}\ti{\psi}_{\ell,0,k})(\tilde{\mathbf{x}}_\ell,\ti{t}_\ell)$, we obtain from \eqref{unnumb}
\begin{equation}\label{mezzedim}
d_\ell^{-\alpha}\|\de_{\ti{t}}\ti{\psi}^\sharp_{\ell,0,k}\|_{\infty,\ti{Q}_{2R},\ti{g}_\ell(0)}\leq C\ve_\ell^\alpha,
\end{equation}
and combining \eqref{mezzed} with \eqref{mezzedi} and \eqref{mezzedim} proves \eqref{jet1}.

To prove \eqref{jet2}, given $p,q\geq 0$ with $p+2q=\iota$, to bound $d_\ell^{-\alpha}\D^p\de_{\ti{t}}^q\ti{\psi}^\sharp_{\ell,0,k}$ we again need to consider only two cases. The first case ($q=0,p=\iota$) is the one with only spatial derivatives that land on the spatial Taylor polynomial (writing again $\psi=d_\ell^{-\alpha}\ti{\psi}_{\ell,0,k}$)
\begin{equation}
\left[(\partial + \Gamma)^\iota+
\sum_{r=0}^{\iota-2}\nabla_{z,y,\ti{z}}^{\iota-2-r}\mathbf{A}\circledast(\de+\Gamma)^r\right]\sum_\gamma(\partial^{|\gamma|}\psi)(\tilde{\mathbf{x}}_\ell,\ti{t}_\ell) (\tilde{\mathbf{x}}-\tilde{\mathbf{x}}_\ell)^\gamma,
\end{equation}
whose $L^\infty$ norm on $\ti{Q}_{2R}$ is bounded by $C\ve_\ell^{2+\alpha-\iota}$ thanks to the estimates in Step (3) (with $r=\iota$) and Step (4) (with $\iota$ there replaced by $\iota-1$).
The second case only happens when $\iota=2$ and we have only 1 time derivative that lands on the $\ti{t}$-variable Taylor polynomial, which gives simply $d_\ell^{-\alpha}(\de_{\ti{t}}\ti{\psi}_{\ell,0,k})(\tilde{\mathbf{x}}_\ell,\ti{t}_\ell)$ and this is bounded by $C\ve_\ell^{\alpha}$ by \eqref{unnumb}. Putting these observations together proves \eqref{jet2}.

Combining \eqref{jet1} for $\iota=2$ with the bound $d_\ell^{-\alpha}[\DD^2\ti{\psi}_{\ell,0,k}]_{\a,\a/2,\ti{Q}_{R},\ti{g}_\ell(0)}\leq C$ from \eqref{estimate-psi-tilde-infty-10}, we see that
\begin{equation}\label{jet3}
d_\ell^{-\alpha}[\DD^2\ti{\psi}^*_{\ell,0,k}]_{\a,\a/2,\ti{Q}_{R},\ti{g}_\ell(0)}\leq C_R,
\end{equation}
and so we can apply Lemma \ref{lma:jet-interpolation} and get
\begin{equation}\label{jet4}
d_\ell^{-\alpha}\|\DD^\iota\ti{\psi}^*_{\ell,0,k}\|_{\infty,\ti{Q}_{R},\ti{g}_\ell(0)}\leq C_R,\quad
d_\ell^{-\alpha}[\DD^\iota\ti{\psi}^*_{\ell,0,k}]_{\a,\a/2,\ti{Q}_{R},\ti{g}_\ell(0)}\leq C_R,
\quad 0\leq\iota\leq 2.
\end{equation}
On the other hand, from \eqref{bdd-phi-star-underline} we also get uniform local parabolic $C^{\a,\a/2}$ bounds on $d_\ell^{-\alpha}e^{-d_\ell^2\lambda_\ell^{-2}\tilde t}\de_{\ti{t}}\underline{\tilde \chi_\ell^*}$ and $d_\ell^{-\alpha}\ddbar\left(e^{-d_\ell^2\lambda_\ell^{-2}\tilde t}\underline{\tilde \chi_\ell^* }\right)$, and so if we define
\begin{equation}
\ti{\rho}^*_\ell:=\ti{\psi}^*_{\ell,0,k}+e^{-d_\ell^2\lambda_\ell^{-2}\tilde t}\underline{\tilde \chi_\ell^* },
\end{equation}
then from \eqref{rho} we have $\ti{\rho}_\ell=\ti{\rho}^*_\ell+\ti{\psi}^\sharp_{\ell,0,k},$ and
\begin{equation}\label{krkl5}
d_\ell^{-\alpha}\left(\de_{\ti{t}}\ti{\rho}^*_\ell+d_\ell^2\lambda_\ell^{-2}\ti{\rho}^*_\ell\right)=d_\ell^{-\alpha}e^{-d_\ell^2\lambda_\ell^{-2}\tilde t}\de_{\ti{t}}\underline{\tilde \chi_\ell^*}
+d_\ell^{-\alpha}\left(\de_{\ti{t}}\ti{\psi}^*_{\ell,0,k}+d_\ell^2\lambda_\ell^{-2}\ti{\psi}^*_{\ell,0,k}\right),
\end{equation}
which by these estimates has  uniform local parabolic $C^{\a,\a/2}$ bounds. The same estimates also give us  uniform local parabolic $C^{\a,\a/2}$ bounds on $d_\ell^{-\alpha}\ddbar\ti{\rho}^*_\ell$, so passing to a subsequence, we have that
\begin{equation}
d_\ell^{-\alpha}\left(\de_{\ti{t}}\ti{\rho}^*_\ell+d_\ell^2\lambda_\ell^{-2}\ti{\rho}^*_\ell\right)\to u_\infty, \quad d_\ell^{-\alpha}\ddbar\ti{\rho}^*_\ell\to\eta_\infty,
\end{equation}
 in $C^{\gamma,\gamma/2}_{\rm loc}$ for all $0<\gamma<\a$, for a function $u_\infty\in C^{\a,\a/2}_{\rm loc}$ and a $(1,1)$-form $\eta_\infty\in C^{\a,\a/2}_{\rm loc}$  on $\C^{m+n}\times (-\infty,0]$. Moreover, thanks to \eqref{ainf} and \eqref{jet4}, we have
\begin{equation}
d_\ell^{-\a+2}\lambda_\ell^{-2}\|\ti{\rho}^*_\ell\|_{\infty, \ti{Q}_R}\leq C_Rd_\ell^2\lambda_\ell^{-2}+o(1)d_\ell^2=o(1),
\end{equation}
hence $d_\ell^{-\alpha}\de_{\ti{t}}\ti{\rho}^*_\ell\to u_\infty$ locally uniformly.

Next, we observe that by definition, for any fixed $R>0$ we have
\begin{equation}
d_\ell^{-\alpha}[e^{-d_\ell^2\lambda_\ell^{-2}\tilde t} \partial_{\tilde t}\tilde \chi_\ell ]_{\a,\a/2,\tilde Q_{R},\tilde g_\ell(0)}=
d_\ell^{-\alpha}[\partial_{\tilde t}\tilde \vp_\ell+d_\ell^2\lambda_\ell^{-2}\tilde \vp_\ell]_{\a,\a/2,\tilde Q_{R},\tilde g_\ell(0)}
=[\partial_{\hat{t}}\hat{\vp}_\ell+\lambda_\ell^{-2}\hat{\vp}_\ell]_{\a,\a/2,\hat{Q}_{R d_\ell},\hat{g}_\ell(0)}=o(1),
\end{equation}
thanks to \eqref{wompo}, and similarly from \eqref{wulpo},
\begin{equation}\label{krkl1}
d_\ell^{-\alpha}[e^{-d_\ell^2\lambda_\ell^{-2}\tilde t} \partial_{\tilde t}\underline{\tilde \chi_\ell } ]_{\a,\a/2,\tilde Q_{R},\tilde g_\ell(0)}
=[\partial_{\hat{t}}\underline{\hat{\vp}_\ell}+\lambda_\ell^{-2}\underline{\hat{\vp}_\ell}]_{\a,\a/2,\hat{Q}_{R d_\ell},\hat{g}_\ell(0)}=o(1),
\end{equation}
and since
\begin{equation}
e^{-d_\ell^2\lambda_\ell^{-2}\tilde t} \partial_{\tilde t}\tilde \chi_\ell-e^{-d_\ell^2\lambda_\ell^{-2}\tilde t} \partial_{\tilde t}\underline{\tilde \chi_\ell }=
\de_{\ti{t}}\ti{\psi}_{\ell,0,k}+d_\ell^2\lambda_\ell^{-2}\ti{\psi}_{\ell,0,k},
\end{equation}
we see that
\begin{equation}
d_\ell^{-\alpha}[\de_{\ti{t}}\ti{\psi}_{\ell,0,k}+d_\ell^2\lambda_\ell^{-2}\ti{\psi}_{\ell,0,k}]_{\a,\a/2,\tilde Q_{R},\tilde g_\ell(0)}=o(1),
\end{equation}
and using \eqref{estimate-psi-tilde-infty-10}, this gives
\begin{equation}\label{krkl3}
d_\ell^{-\alpha}[\de_{\ti{t}}\ti{\psi}_{\ell,0,k}]_{\a,\a/2,\tilde Q_{R},\tilde g_\ell(0)}=o(1),
\end{equation}
which combined with \eqref{jet1} implies
\begin{equation}\label{krkl3b}
d_\ell^{-\alpha}[\de_{\ti{t}}\ti{\psi}^*_{\ell,0,k}]_{\a,\a/2,\tilde Q_{R},\tilde g_\ell(0)}=o(1).
\end{equation}
Also, \eqref{jet4} implies that
\begin{equation}\label{krkl4}
d_\ell^{-\alpha}d_\ell^2\lambda_\ell^{-2}[\ti{\psi}^*_{\ell,0,k}]_{\a,\a/2,\tilde Q_{R},\tilde g_\ell(0)}=o(1).
\end{equation}
On the other hand, from \eqref{bdd-phi-sharp-underline} we see that
\begin{equation}
d_\ell^{-\alpha}[e^{-d_\ell^2\lambda_\ell^{-2}\tilde t} \partial_{\tilde t}\underline{\tilde \chi_\ell^\sharp } ]_{\a,\a/2,\tilde Q_{R},\tilde g_\ell(0)}=o(1),
\end{equation}
and which combined with \eqref{krkl1} gives
\begin{equation}\label{krkl2}
d_\ell^{-\alpha}[e^{-d_\ell^2\lambda_\ell^{-2}\tilde t} \partial_{\tilde t}\underline{\tilde \chi_\ell^* } ]_{\a,\a/2,\tilde Q_{R},\tilde g_\ell(0)}=o(1).
\end{equation}
Plugging \eqref{krkl3b}, \eqref{krkl4} and \eqref{krkl2} into  \eqref{krkl5} gives
\begin{equation}\label{krkl6}
d_\ell^{-\alpha}[\de_{\ti{t}}\ti{\rho}^*_\ell+d_\ell^2\lambda_\ell^{-2}\ti{\rho}^*_\ell]_{\a,\a/2,\tilde Q_{R},\tilde g_\ell(0)}=o(1),
\end{equation}
which implies that $u_\infty$ is constant in space-time and since its value at $(0,0)$ vanishes, we conclude that $u_\infty\equiv 0$ on $\mathbb{C}^{m+n}\times (-\infty,0]$. Thus, $d_\ell^{-\alpha}\de_{\ti{t}}\ti{\rho}^*_\ell\to 0$ locally uniformly.

On the other hand, given any $(\ti{x},\ti{t}),(\ti{x},\ti{t}')\in \ti{Q}_R$, we have
\begin{equation}\label{rokko}
d_\ell^{-\a}\ddbar\ti{\rho}^*_\ell(\ti{x},\ti{t})-d_\ell^{-\a}\ddbar\ti{\rho}^*_\ell(\ti{x},\ti{t}')=\ddbar \int_{\ti{t}'}^{\ti{t}}\left(d_\ell^{-\a}\de_{\ti{t}}\ti{\rho}^*_\ell\right)(\ti{x},\ti{s})d\ti{s},
\end{equation}
and since we know that $\int_{t'}^t\left(d_\ell^{-\a}\de_{\ti{t}}\ti{\rho}^*_\ell\right)(\ti{x},\ti{s})d\ti{s} \to 0$ locally uniformly (as a function of $\ti{x}$), it follows that the RHS of \eqref{rokko} converges to zero weakly as currents. Since the LHS of \eqref{rokko} converges in $C^\gamma_{\rm loc}$ to $\eta_\infty(\ti{x},\ti{t})-\eta_\infty(\ti{x},\ti{t}')$, we conclude that this H\"older continuous $(1,1)$-form (with $\ti{x}$ varying) is zero as a current, hence it is identically zero. This shows that $\eta_\infty$ is time-independent.

From \eqref{PMA-tilde-new} we have
\begin{equation}\label{crax}
\left(\de_{\ti{t}}-\Delta_{\ti{\omega}^\sharp_\ell}\right)\ti{\rho}^*_\ell + d_\ell^2\lambda_\ell^{-2}\ti{\rho}^*_\ell=\mathcal{E}_1+\mathcal{E}_2+\mathcal{E}_3-\left(\de_{\ti{t}}-\Delta_{\ti{\omega}^\sharp_\ell}\right)\ti{\psi}^\sharp_{\ell,0,k}-d_\ell^2\lambda_\ell^{-2}\ti{\psi}^\sharp_{\ell,0,k},
\end{equation}
and from \eqref{crox} and \eqref{jet1} we see that applying $d_\ell^{-\alpha}[\cdot]_{\a,\a/2,\ti{Q}_R,\ti{g}_\ell(0)}$ to the RHS of \eqref{crax} we get $o(1)$ as $\ell\to+\infty$. We can then multiply \eqref{crax} by $d_\ell^{-\alpha}$, and since the LHS converges in $C^{\gamma,\gamma/2}_{\rm loc}$ we can pass to the limit (recalling that $u_\infty\equiv 0$) and get
\begin{equation}\label{craxi}
\tr_{g_{\C^{m+n}}}\eta_\infty=c,
\end{equation}
on $\C^{m+n}\times (-\infty,0]$, for some time-independent constant $c$. Since the value of the LHS of \eqref{craxi} at $(0,0)$ is zero, this forces $c=0$, i.e.
\begin{equation}\label{craxi2}
\tr_{g_{\C^{m+n}}}\eta_\infty=0.
\end{equation}
Since the $C^{\gamma}_{\rm loc}$ form $\eta_\infty$ is time-independent and weakly closed, it can be written as $\eta_\infty=\ddbar v_\infty$ for some time-independent function $v_\infty\in C^{2+\gamma}_{\rm loc}(\mathbb{C}^{m+n})$. From \eqref{craxi2} we see that $\Delta_{g_{\mathbb{C}^{m+n}}}v_\infty=0$, so $\eta_\infty$ is smooth by elliptic regularity, and passing to the limit \eqref{bdd-phi-star-underline} and \eqref{jet3} we see that $|\ddbar v_\infty|=O(|z|^\a)$. Thus each component $\partial_\a \partial_{\bar\b}v_\infty$ of $\ddbar v_\infty$ satisfies $\Delta_{g_{\mathbb{C}^{m+n}}} \partial_\a \partial_{\bar\b}v_\infty=0$ and $|\partial_\a \partial_{\bar\b}v_\infty|=O(|z|^\a)$, so by the standard Liouville Theorem for harmonic functions we have that $\eta_\infty$ has constant coefficients, hence it vanishes identically since its value at $(0,0)$ is zero.

This implies that
\begin{equation}\label{fukoff}
d_\ell^{-\alpha}\ddbar\ti{\rho}^*_\ell\to 0,
\end{equation}
locally uniformly on $\C^{m+n}\times(-\infty,0]$ in the coordinates $(\ti{\mathbf{x}},\ti{t})$. Recall that, by definition, we have
\begin{equation}\label{deff}
\ti{\vp}_\ell=\ti{\rho}^*_\ell+\ti{\psi}^\sharp_{\ell,0,k}+e^{-d_\ell^2\lambda_\ell^{-2}\tilde t}\underline{\tilde \chi_\ell^\sharp},
\end{equation}
and that $\underline{\tilde \chi_\ell^\sharp}$ is a polynomial on $\C^m$ of degree at most $2$, while $\ti{\psi}^\sharp_{\ell,0,k}$ is a polynomial on $\C^{m+n}$ (in the $(\ti{\mathbf{x}},\ti{t})$ coordinates) of degree at most $2$. To convert $\ddbar$ into $\D$-derivatives, schematically we have $\ddbar=\ti{J}_\ell\circledast\D^2+(\D\ti{J}_\ell)\circledast \D$, with the bounds (cf. \cite[(4.304)]{HT3})
\begin{equation}\label{buk}
\|\D^\iota \ti{J}_\ell\|_{\infty,\ti{Q}_{Rd_\ell^{-1}},\ti{g}_\ell(0)}\leq C\ve_\ell^{-\iota},\quad [\D^\iota \ti{J}_\ell]_{\a,\a/2,\ti{Q}_{Rd_\ell^{-1}},\ti{g}_\ell(0)}\leq C\ve_\ell^{-\iota-\a}.
\end{equation}
Since $\ti{J}_\ell$ and $\D$ are independent of $\ti{t}$, it follows that
\begin{equation}
\de_{\ti{t}}\ddbar\ti{\psi}^\sharp_{\ell,0,k}=\ti{J}_\ell\circledast\D^2\de_{\ti{t}}\ti{\psi}^\sharp_{\ell,0,k}+(\D\ti{J}_\ell)\circledast \D\de_{\ti{t}}\ti{\psi}^\sharp_{\ell,0,k}=0,
\end{equation}
so, as in the proof of \eqref{mezzed} we can bound
\begin{equation}\label{qursq}
\begin{split}
d_\ell^{-\a}[\ddbar\ti{\psi}^\sharp_{\ell,0,k}]_{\a,\a/2,\ti{Q}_{R},\ti{g}_\ell(0)}&\leq CR^{1-\alpha}d_\ell^{-\a}\|\D\ddbar\ti{\psi}^\sharp_{\ell,0,k}\|_{\infty, \ti{Q}_{2R},\ti{g}_\ell(0)}\\
&\leq CR^{1-\alpha}d_\ell^{-\a}\big(\|\D^3\ti{\psi}^\sharp_{\ell,0,k}\|_{\infty, \ti{Q}_{2R},\ti{g}_\ell(0)}+\ve_\ell^{-1}\|\D^2\ti{\psi}^\sharp_{\ell,0,k}\|_{\infty, \ti{Q}_{2R},\ti{g}_\ell(0)}\\
&+\ve_\ell^{-2}\|\D\ti{\psi}^\sharp_{\ell,0,k}\|_{\infty, \ti{Q}_{2R},\ti{g}_\ell(0)}\big)\\
&\leq CR^{1-\alpha}\ve_\ell^{\a-2}=o(1),
\end{split}
\end{equation}
using \eqref{buk} for the second inequality and \eqref{jet2} for the third one. Also, from the bounds \eqref{bdd-phi-sharp-underline}, it follows easily that
\begin{equation}\label{qurq}
d_\ell^{-\a}\left[e^{-d_\ell^2\lambda_\ell^{-2}\tilde t}\ddbar\underline{\tilde \chi_\ell^\sharp}\right]_{\a,\a/2,\ti{Q}_{R},\ti{g}_\ell(0)}=o(1),
\end{equation}
and so using \eqref{fukoff}, \eqref{deff}, \eqref{qursq} and \eqref{qurq}, and recalling also \eqref{beaner}, we see that
\begin{equation}
d_\ell^{-\a}|\ddbar\ti{\vp}_\ell(\ti{x}_\ell,0)-\P_{\tilde x'_\ell\tilde x_\ell}\ddbar\ti{\vp}_\ell(\ti{x}_\ell',\ti{t}'_\ell)|_{\ti{g}_\ell(0)}=o(1),
\end{equation}
a contradiction to \eqref{sup-realized-20}.

\subsubsection{The case $j\geq 1$. Killing the contribution of $\tilde A_{\ell,i,p,k}$}
In the rest of this section, we will assume that $j\geq 1$.
The goal of this subsection is to prove a precise estimate on $\tilde A_{\ell,i,p,k}$:
for all $a\geq 2j$, $a\in 2\mathbb{N}$ and $\a\leq \b<1$, there is $C>0$ such that
\begin{equation}\label{improved-A-estimate-subcaseA}
d_\ell^{-2j-\a}\sum_{i=1}^j\sum_{p=1}^{N_{i,k}} \e_\ell^{a-2j-2}[\mathfrak{D}^a \tilde A^*_{\ell,i,p,k}]_{\b,\b/2,\tilde Q_{O(\e_\ell)},\tilde g_\ell(0)}\leq C \e_\ell^{\a-\b},
\end{equation}
where here and in the rest of this section we use the notation $O(\e_\ell)$ for a radius $R$ such that $\Lambda\e_\ell \leq R\leq \Lambda^2\e_\ell$ where $\Lambda>1$ is the fixed constant from \eqref{equiv-product-ref} (so that the $\ti{g}_\ell(0)$-geodesic ball centered at $\ti{x}_\ell$ with radius $R$ contains a Euclidean ball of radius $R/2$ times the whole $Y$ fiber). Note that since $j\geq 1$, we have $a\geq 2$.

Observe that once \eqref{improved-A-estimate-subcaseA} is established for all even $a\geq 2j$, the same estimate will also hold for all $a\geq 2j$ by interpolation: indeed if $a\geq 2j+1$ is odd, then
for any $0<\rho<R$,
\begin{equation}
\begin{split}
d_\ell^{-2j-\a}\e_\ell^{a-2j-2}[\mathfrak{D}^a \tilde A^*_{\ell,i,p,k}]_{\b,\b/2,\tilde Q_{\rho},\tilde g_\ell(0)}&\leq
C(R-\rho)d_\ell^{-2j-\a}\e_\ell^{a-2j-2}[\mathfrak{D}^{a+1} \tilde A^*_{\ell,i,p,k}]_{\b,\b/2,\tilde Q_{R},\tilde g_\ell(0)}\\
&+C(R-\rho)^{-1}d_\ell^{-2j-\a}\e_\ell^{a-2j-2}[\mathfrak{D}^{a-1} \tilde A^*_{\ell,i,p,k}]_{\b,\b/2,\tilde Q_{R},\tilde g_\ell(0)}\\
&\leq C(R-\rho)\ve_\ell^{\alpha-\beta-1}+C(R-\rho)^{-1}\ve_\ell^{\alpha-\beta+1},
\end{split}
\end{equation}
and taking $\rho,R=O(\ve_\ell)$ gives the claim.

Thus, once \eqref{improved-A-estimate-subcaseA} is established, taking $a=2j+2+\iota$ (with $-2\leq\iota\leq 2k$), and $\beta>\alpha$ gives an $o(1)$ bound for the $C^{\beta,\beta/2}$ seminorm on the cylinder centered at $(\ti{x}_\ell,\ti{t}_\ell)$ of radius $2$, which contains the other blowup point $(\ti{x}'_\ell,\ti{t}'_\ell)$ which lies at distance $1$, and hence an $o(1)$ bound for the $C^{\alpha,\alpha/2}$ seminorm on the same cylinder,  which kills the contribution of $\DD^{2j+2+\iota}\tilde A^*_{\ell,i,p,k}$ in \eqref{sup-realized-3}.

Apart from the fact that in our parabolic setting we only work with derivatives of even order (as was explained earlier),
the overall argument to prove \eqref{improved-A-estimate-subcaseA} will be similar to the one to prove \cite[(4.252)]{HT3},  replacing $\D$ by $\mathfrak{D}$, $j$ by $2j$ and $t$ by $t_\ell$.
As in \cite[(4.254)]{HT3}, we use \eqref{estimate-psi-tilde-infty-1}, \eqref{bdd-phi-star-underline}, \eqref{bdd-phi-sharp-underline}, \eqref{circ-e-nonzero}, \eqref{eta-dagger}  and Proposition \ref{prop:noncancellation} to conclude that for all $a\geq 2j$, $1\leq i\leq j, 1\leq p\leq N_{i,k}$, $a\in 2\mathbb{N}$ and  $\a\leq \b<1$, there is $C>0$ so that
\begin{equation}\label{main-goal-toBeKilled-1}
\begin{split}
&\quad d_\ell^{-2j-\a}\e_\ell^{a-2j-2}[\mathfrak{D}^a\tilde A_{\ell,i,p,k}]_{\b,\b/2,\tilde Q_{O(\e_\ell)},\tilde g_\ell(0)}\\
&\leq o(\e_\ell^{\a-\b})+ C \e_\ell^{a-2j}d_\ell^{-2j-\a} [\mathfrak{D}^a_{\bf bt} \tilde \eta_\ell]_{\b,\b/2,\mathrm{base},\tilde Q_{O(\e_\ell)},\tilde g_\ell(0)}\\
&\quad +Ce^{-(1-\b)\frac{t_\ell}{2}} \sum_{b=2j+1}^a \lambda_\ell^{-\b} \delta_\ell^{a-2j}d_\ell^{-b+\b-\a}\|\mathfrak{D}^b(\ti\eta_\ell^\circ+\ti\eta_\ell^\diamond+\ti\eta_{\ell,j,k}) \|_{\infty,\tilde Q_{O(\e_\ell)},\tilde g_\ell(0)}\\
&\quad +C\sum_{r=1}^{i-1}\sum_{q=1}^{N_{p,k}}\Big( d_\ell^{-2j-\a} \e_\ell^{2k+a-2j} [\mathfrak{D}^{a+2k+2}\tilde A_{\ell,r,q,k}]_{\b,\b/2,\tilde Q_{O(\e_\ell)},\tilde g_\ell(0)}\\
&\quad  +\sum_{b=0}^{a+2k+2} d_\ell^{-2j-\a} \e_\ell^{a-2j-2} d_\ell^{a-b+\b}e^{-(2k+2+1-\b)\frac{t_\ell}{2}} \lambda_\ell^{b-a-\b} \|\mathfrak{D}^b \tilde A_{\ell,r,q,k}\|_{\infty,\tilde Q_{O(\e_\ell)},\tilde g_\ell(0)} \Big).
\end{split}
\end{equation}

Our goal is to show that each terms on the right hand side of \eqref{main-goal-toBeKilled-1} is of $O(\e_\ell^{\a-\b})$ which implies \eqref{improved-A-estimate-subcaseA}.
\bigskip

We first treat the second term,  $[\mathfrak{D}^a_{\bf bt} \tilde \eta_\ell]_{\b,\b/2,\mathrm{base},\tilde Q_{O(\e_\ell)},\tilde g_\ell(0)}$.  We start with noting that we can interchange $\ti{\eta}_\ell$ with $\ti\eta_\ell^\circ+\ti\eta_\ell^\diamond+\ti\eta_{\ell,j,k}$ thanks to \eqref{gorgol} and \eqref{4.173}: for all $a\geq 2j$, $a\in 2\mathbb{N}$,
\begin{equation}\label{gorgu}
\e_\ell^{a-2j}d_\ell^{-2j-\a} [\mathfrak{D}^a_{\bf bt} \tilde\eta_\ell]_{\b,\b/2,\mathrm{base},\tilde Q_{O(\e_\ell)},\tilde g_\ell(0)}
\leq  \e_\ell^{a-2j}d_\ell^{-2j-\a} [\mathfrak{D}^a (\tilde \eta_{\ell,j,k}+\tilde \eta_\ell^\diamond+\tilde\eta_\ell^\circ)]_{\b,\b/2,\tilde Q_{O(\e_\ell)},\tilde g_\ell(0)}+o(\e_\ell^{\a-\b}).
\end{equation}
To bound the RHS of \eqref{gorgu} using the parabolic Schauder estimates in Proposition \ref{prop:schauder2}, we need to pass to the check picture via the diffeomorphism
\begin{equation}\label{pi}
\Pi_\ell:   B_{e^{\frac{t_\ell}{2}}}\times Y\times[-e^{t_\ell}t_\ell,0]  \to B_{d_\ell^{-1}\lambda_\ell}\times Y\times [-d_\ell^{-2}\lambda_\ell^2t_\ell,0], \;\, (\ti{z},\ti{y},\ti{t}) = \Pi_\ell(\check{z},\check{y},\check{t}) =(\ve_\ell\check{z},\check{y},\ve_\ell^2\check{t}),
\end{equation}
pulling back all geometric quantities and scaling $2$-forms
by $\ve_\ell^{-2}$. We can then apply Proposition \ref{prop:schauder2} to $\check\eta_{\ell,j,k}+\check \eta_\ell^\diamond+\check\eta_\ell^\circ$ and then transfer the result back to the tilde picture. This shows that given any radius $R=O(\ve_\ell)$ and $0<\rho<R$, and letting $\ti{R}=\rho+\frac{1}{2}(R-\rho)$, we have
\begin{equation}\label{luz}
\begin{split}
\e_\ell^{a-2j}d_\ell^{-2j-\a} [\mathfrak{D}^a (\tilde \eta_{\ell,j,k}+\tilde \eta_\ell^\diamond+\tilde\eta_\ell^\circ)]_{\b,\b/2,\tilde Q_{\rho},\tilde g_\ell(0)}
&\leq C\e_\ell^{a-2j}d_\ell^{-2j-\a} \left[\mathfrak{D}^{a-2}\left(\de_{\ti{t}}-\Delta_{\tilde\omega_\ell^\sharp}\right)(\tilde \eta_{\ell,j,k}+\tilde \eta_\ell^\diamond+\tilde\eta_\ell^\circ)\right]_{\b,\b/2,\tilde Q_{\ti{R}},\tilde g_\ell(0)}\\
&+C\e_\ell^{a-2j}d_\ell^{-2j-\a}(R-\rho)^{-a-\beta}\|\tilde \eta_{\ell,j,k}+\tilde \eta_\ell^\diamond+\tilde\eta_\ell^\circ\|_{\infty,\tilde Q_{\ti{R}},\tilde g_\ell(0)}\\
&\leq C\e_\ell^{a-2j}d_\ell^{-2j-\a} \left[\mathfrak{D}^{a-2}\left(\de_{\ti{t}}-\Delta_{\tilde\omega_\ell^\sharp}\right)(\tilde \eta_{\ell,j,k}+\tilde \eta_\ell^\diamond+\tilde\eta_\ell^\circ)\right]_{\b,\b/2,\tilde Q_{\ti{R}},\tilde g_\ell(0)}\\
&+C\ve_\ell^{a+\alpha}(R-\rho)^{-a-\beta},
\end{split}
\end{equation}
using \eqref{estimate-psi-tilde-infty-1}, \eqref{gorgo} and \eqref{circ-e-nonzero}, where here $\Delta_{\tilde\omega_\ell^\sharp}$ denotes the Hodge Laplacian acting on forms. This is the parabolic analog of \cite[(4.259)]{HT3}. Recalling that $\tilde \eta_{\ell,j,k}+\tilde \eta_\ell^\diamond+\tilde\eta_\ell^\circ=\ddbar\ti{\rho}_\ell,$
and using that the Hodge Laplacian of a K\"ahler metric commutes with $\ddbar$, we have
\begin{equation}
\begin{split}
&\left[\mathfrak{D}^{a-2}\left(\de_{\ti{t}}-\Delta_{\tilde\omega_\ell^\sharp}\right)(\tilde \eta_{\ell,j,k}+\tilde \eta_\ell^\diamond+\tilde\eta_\ell^\circ)\right]_{\b,\b/2,\tilde Q_{\ti{R}},\tilde g_\ell(0)}=\left[\mathfrak{D}^{a-2}\ddbar\left(\de_{\ti{t}}-\Delta_{\tilde\omega_\ell^\sharp}\right)\ti{\rho}_\ell\right]_{\b,\b/2,\tilde Q_{\ti{R}},\tilde g_\ell(0)},
\end{split}
\end{equation}
and using the PDE \eqref{PMA-tilde-new}, the triangle inequality, and the boundedness of $\P$-parallel transport, this can be bounded by
\begin{equation}
Cd_\ell^2\lambda_\ell^{-2}\left[\mathfrak{D}^{a-2}(\tilde \eta_{\ell,j,k}+\tilde \eta_\ell^\diamond+\tilde\eta_\ell^\circ)\right]_{\b,\b/2,\tilde Q_{\ti{R}},\tilde g_\ell(0)}
+C\left[\mathfrak{D}^{a-2}\ddbar(\mathcal{E}_1+\mathcal{E}_2+\mathcal{E}_3)\right]_{\b,\b/2,\tilde Q_{\ti{R}},\tilde g_\ell(0)},
\end{equation}
and inserting these into \eqref{luz} and interpolating again with Proposition \ref{prop:interpolation}, we get
\begin{equation}\label{luz2}
\begin{split}
&\e_\ell^{a-2j}d_\ell^{-2j-\a} [\mathfrak{D}^a (\tilde \eta_{\ell,j,k}+\tilde \eta_\ell^\diamond+\tilde\eta_\ell^\circ)]_{\b,\b/2,\tilde Q_{\rho},\tilde g_\ell(0)}
\leq Cd_\ell^2\lambda_\ell^{-2}\e_\ell^{a-2j}d_\ell^{-2j-\a} \left[\mathfrak{D}^{a}(\tilde \eta_{\ell,j,k}+\tilde \eta_\ell^\diamond+\tilde\eta_\ell^\circ)\right]_{\b,\b/2,\tilde Q_{R},\tilde g_\ell(0)}\\
&+C\e_\ell^{a-2j}d_\ell^{-2j-\a}\left[\mathfrak{D}^{a-2}\ddbar(\mathcal{E}_1+\mathcal{E}_2+\mathcal{E}_3)\right]_{\b,\b/2,\tilde Q_{\ti{R}},\tilde g_\ell(0)}+C\ve_\ell^{a+\alpha}(R-\rho)^{-a-\beta}\\
&\leq \frac{1}{4}\e_\ell^{a-2j}d_\ell^{-2j-\a} \left[\mathfrak{D}^{a}(\tilde \eta_{\ell,j,k}+\tilde \eta_\ell^\diamond+\tilde\eta_\ell^\circ)\right]_{\b,\b/2,\tilde Q_{R},\tilde g_\ell(0)}\\
&+C\e_\ell^{a-2j}d_\ell^{-2j-\a}\left[\mathfrak{D}^{a-2}\ddbar(\mathcal{E}_1+\mathcal{E}_2+\mathcal{E}_3)\right]_{\b,\b/2,\tilde Q_{\ti{R}},\tilde g_\ell(0)}+C\ve_\ell^{a+\alpha}(R-\rho)^{-a-\beta}.
\end{split}
\end{equation}
The main claim is then the following:
\begin{claim}\label{main-claim}For all $a\geq 2j$, $a\in 2\mathbb{N}$, $R=O(\ve_\ell),$ $0<\rho<R$, $\a\leq \b<1$, if we let $\ti{R}=\rho+\frac{1}{2}(R-\rho)$, then we have
\begin{equation}\label{claim52}
\begin{split}
C\e_\ell^{a-2j}d_\ell^{-2j-\a}\left[\mathfrak{D}^{a-2}\ddbar(\mathcal{E}_1+\mathcal{E}_2+\mathcal{E}_3)\right]_{\b,\b/2,\tilde Q_{\ti{R}},\tilde g_\ell(0)}&\leq
\frac{1}{4}\e_\ell^{a-2j}d_\ell^{-2j-\a} \left[\mathfrak{D}^{a}(\tilde \eta_{\ell,j,k}+\tilde \eta_\ell^\diamond+\tilde\eta_\ell^\circ)\right]_{\b,\b/2,\tilde Q_{R},\tilde g_\ell(0)}\\
&+
C\sum_{r=0}^a \e_\ell^{\a-\b+r}(R-\rho)^{-r}+C\sum_{r=0}^{a}\e_\ell^{\a+r} (R-\rho)^{-r-\beta}.
\end{split}
\end{equation}
\end{claim}

Before giving the proof of Claim \ref{main-claim}, we establish some of its consequences.  Suppose Claim \ref{main-claim} has been proved for some $a\geq 2j$,  then plugging it into \eqref{luz2} we get for all $R=O(\ve_\ell)$ and $0<\rho<R$,
\begin{equation}\label{kerku}
\begin{split}
\e_\ell^{a-2j}d_\ell^{-2j-\a} [\mathfrak{D}^a (\tilde \eta_{\ell,j,k}+\tilde \eta_\ell^\diamond+\tilde\eta_\ell^\circ)]_{\b,\b/2,\tilde Q_{\rho},\tilde g_\ell(0)}
&\leq\frac{1}{2}\e_\ell^{a-2j}d_\ell^{-2j-\a} \left[\mathfrak{D}^{a}(\tilde \eta_{\ell,j,k}+\tilde \eta_\ell^\diamond+\tilde\eta_\ell^\circ)\right]_{\b,\b/2,\tilde Q_{R},\tilde g_\ell(0)}\\
&+C\sum_{r=0}^a \e_\ell^{\a-\b+r}(R-\rho)^{-r}+C\sum_{r=0}^{a}\e_\ell^{\a+r} (R-\rho)^{-r-\beta},
\end{split}
\end{equation}
and then the iteration lemma in \cite[Lemma 2.9]{HT3} gives
\begin{equation}
\e_\ell^{a-2j}d_\ell^{-2j-\a} [\mathfrak{D}^a (\tilde \eta_{\ell,j,k}+\tilde \eta_\ell^\diamond+\tilde\eta_\ell^\circ)]_{\b,\b/2,\tilde Q_{\rho},\tilde g_\ell(0)}
\leq  C\sum_{r=0}^a \e_\ell^{\a-\b+r}(R-\rho)^{-r}+C\sum_{r=0}^{a}\e_\ell^{\a+r} (R-\rho)^{-r-\beta},
\end{equation}
and choosing now $\rho=O(\ve_\ell)$ gives
\begin{equation}\label{kukku}
\e_\ell^{a-2j}d_\ell^{-2j-\a} [\mathfrak{D}^a (\tilde \eta_{\ell,j,k}+\tilde \eta_\ell^\diamond+\tilde\eta_\ell^\circ)]_{\b,\b/2,\tilde Q_{O(\ve_\ell)},\tilde g_\ell(0)}
\leq  C\ve_\ell^{\a-\beta},
\end{equation}
and this can be inserted back into \eqref{gorgu} to finally give
\begin{equation}
\begin{split}
 \e_\ell^{a-2j}d_\ell^{-2j-\a} [\mathfrak{D}^a_{\bf bt}\tilde\eta_\ell]_{\b,\b/2,\mathrm{base},\tilde Q_{O(\e_\ell)},\tilde g_\ell(0)}
  &\leq C\e_\ell^{\a-\b},
\end{split}
\end{equation}
which would show that the second term on the RHS of \eqref{main-goal-toBeKilled-1} is $O(\e_\ell^{\a-\b})$.

Furthermore, we can interpolate between \eqref{kukku} and the $L^\infty$ norm bound for $\tilde \eta_{\ell,j,k}+\tilde \eta_\ell^\diamond+\tilde\eta_\ell^\circ$ that comes from \eqref{estimate-psi-tilde-infty-1}, \eqref{gorgo} and \eqref{circ-e-nonzero}, using Proposition \ref{prop:interpolation} on cylinders of radius $O(\ve_\ell)$ to see that
\begin{equation}\label{gurgu}
\|\mathfrak{D}^b (\tilde \eta_{\ell,j,k}+\tilde \eta_\ell^\diamond+\tilde\eta_\ell^\circ) \|_{\infty,\tilde Q_{O(\e_\ell)},\tilde g_\ell(0)}\leq C\e_\ell^{2j+\a-b}d_\ell^{2j+\a},
\quad [\mathfrak{D}^b (\tilde \eta_{\ell,j,k}+\tilde \eta_\ell^\diamond+\tilde\eta_\ell^\circ) ]_{\beta,\beta/2,\tilde Q_{O(\e_\ell)},\tilde g_\ell(0)}\leq C\e_\ell^{2j+\a-b-\beta}d_\ell^{2j+\a},
\end{equation}
for all $0\leq b\leq a$, and so following the discussion in \cite[(4.280)--(4.283)]{HT3} this implies that the third term on the RHS of \eqref{main-goal-toBeKilled-1} is of $O(\e_\ell^{\a-\b})$.

To summarize, we have shown that if Claim \ref{main-claim} holds, then for all $a\geq 2j$, $a\in 2\mathbb{N}$,  $\a\leq \b<1$,  $1\leq i\leq j$ and $1\leq p\leq N_{i,k}$, there is $C>0$ such that
\begin{equation}\label{induction-step-A}
\begin{split}
&\quad d_\ell^{-2j-\a}\e_\ell^{a-2j-2}[\mathfrak{D}^a\tilde A_{\ell,i,p,k}]_{\b,\b/2,\tilde Q_{O(\e_\ell)},\tilde g_\ell(0)}\\
&\leq C\e_\ell^{\a-\b} +C\sum_{r=1}^{i-1}\sum_{q=1}^{N_{p,k}}\Big( d_\ell^{-2j-\a} \e_\ell^{2k+a-2j} [\mathfrak{D}^{a+2k+2}\tilde A_{\ell,r,q,k}]_{\b,\b/2,\tilde Q_{O(\e_\ell)},\tilde g_\ell(0)}\\
&\quad  \quad +\sum_{b=0}^{a+2k+2} d_\ell^{-2j-\a} \e_\ell^{a-2j-2} d_\ell^{a-b+\b}e^{-(2k+2+1-\b)\frac{t_\ell}{2}} \lambda_\ell^{b-a-\b} \|\mathfrak{D}^b \tilde A_{\ell,r,q,k}\|_{\infty,\tilde Q_{O(\e_\ell)},\tilde g_\ell(0)} \Big),
\end{split}
\end{equation}
which is the parabolic analog of \cite[(4.283)]{HT3}.  The remaining argument follows closely that of \cite[\S 4.10.5]{HT3}.
We use induction on $1\leq i\leq j$ to show that for all $a\geq 2j$ and $a\in 2\mathbb{N}$, we have
\begin{equation}\label{goal-on-A}
d_\ell^{-2j-\a} \sum_{p=1}^{N_{i,k}} \e_\ell^{a-2j-2}[\mathfrak{D}^a \tilde A_{\ell,i,p,k}]_{\b,\b/2,\tilde Q_{O(\e_\ell)},\tilde g_\ell(0)}\leq C \e_\ell^{\a-\b}.
\end{equation}

The base case $i=1$ follows directly from \eqref{induction-step-A}. Suppose \eqref{goal-on-A} holds up to $i_0-1$ for some $i_0\geq 2$, then the first term inside the summation on \eqref{induction-step-A} can be estimated by $C\e_\ell^{\a-\b}$ since $a+2k+2\in 2\mathbb{N}$.  We now treat the second term inside the summation, namely the last
line of \eqref{induction-step-A}. First we treat the terms with $0\leq b\leq 2j$,  by using \eqref{lower-A-Linfty} and transferring it to the tilde picture which gives
\begin{equation}\label{to-goal-on-A-1}
d_\ell^{-b}\| \mathfrak{D}^b \tilde A_{\ell,i,p,k}\|_{\infty,\tilde Q_{O(\e_\ell)} , \tilde g_\ell(0)} =o(\e_\ell^{2}),
\end{equation}
and using these we can bound these terms by $o(\e_\ell^{\a-\b})$ exactly as in \cite[(4.286)]{HT3}. As for the terms with $2j<b\leq a+2k+2$,  we apply interpolation, i.e. Proposition \ref{prop:interpolation}, using \eqref{to-goal-on-A-1} and \eqref{goal-on-A} for $i\leq i_0-1$ from the induction hypothesis to shows that for $2j<b\leq a+2k+2$ and $i\leq i_0-1$,
\begin{equation}
\| \mathfrak{D}^b \tilde A_{\ell,i,p,k}\|_{\infty,\tilde Q_{O(\e_\ell)},\tilde g_\ell(0)}\leq  C\e_\ell^{2+2j+\a-b}d_\ell^{2j+\a}+C\e_\ell^{2+2j-b}d_\ell^{2j},
\end{equation}
so that the terms in the last line of \eqref{induction-step-A} with $2j<b\leq a+2k+2$ are also of $o(\e_\ell^{\a-\b})$ by the same argument as \cite[(4.290)]{HT3}. This completes the inductive proof of \eqref{goal-on-A}, and hence of \eqref{improved-A-estimate-subcaseA}, modulo the proof of Claim \ref{main-claim}, which we now turn to.

\subsubsection{Proof of Claim \ref{main-claim}}
The proof of Claim \ref{main-claim} goes along similar lines to \cite[(4.261)]{HT3}, but with some differences.  We will prove the claim by induction on $a\geq 2j$, following the discussion in \cite[\S 4.10.3--4.10.4]{HT3}.
Recall that the terms $\mathcal{E}_i, i=1,2,3,$ are defined in \eqref{PMA-tilde-new-1} and \eqref{E3}. First, we consider the term $\mathcal{E}_2+\mathcal{E}_3$, which by definition equals
\begin{equation}
\log \frac{(\tilde\omega_\ell^\sharp)^{m+n}}{ \binom{m+n}{m}\ti \omega_{\ell,\mathrm{can}} ^m \wedge (\e_\ell^2 \Theta_\ell^*\Psi_\ell^* \omega_F)^n}+nd_\ell^2\lambda_\ell^{-2}\tilde t
-  \sum_{i=1}^j\sum_{p=1}^{N_{i,k}} \tilde{\mathfrak{G}}_{\ti{t},k}\left(\partial_{\tilde t}\tilde A_{\ell,i,p,k}^\sharp
+d_\ell^2\lambda_\ell^{-2}\tilde A_{\ell,i,p,k}^\sharp,\tilde G_{\ell,i,p,k} \right)-e^{-d_\ell^2\lambda_\ell^{-2}\tilde t}\partial_{\tilde t}\underline{\tilde\chi_\ell^\sharp},
\end{equation}
and we claim that for all $b\geq 0$ and $\a\leq \beta<1$, and all fixed $R>1$, there is $C>0$ such that
\begin{equation}\label{E_2-control}
\left\{
\begin{array}{ll}
\e_\ell^{b-2j} d_\ell^{-2j-\a} \|\mathfrak{D}^b(\mathcal{E}_2+\mathcal{E}_3)\|_{\infty, \tilde Q_{R\e_\ell},\tilde g_\ell(0)}\leq C\e_\ell^\a,\\[2mm]
\e_\ell^{b-2j} d_\ell^{-2j-\a} [ \mathfrak{D}^b (\mathcal{E}_2+\mathcal{E}_3)]_{\b,\b/2, \tilde Q_{R\e_\ell},\tilde g_\ell(0)}\leq C\e_\ell^{\a-\b}.
\end{array}
\right.
\end{equation}
Observe that \eqref{E_2-control} for $a\leq 2j$ and $\b=\a$ is exactly given by \eqref{error2-bad-1} and \eqref{error2-bad-2}. To prove it for $a\geq 2j$ and $\b\geq\alpha$ we apply the diffeomorphism $\Pi_\ell$ in \eqref{pi} so that $\Pi_\ell^*(\mathcal{E}_2+\mathcal{E}_3)$ equals
\begin{equation}
\log \frac{(\check\omega_\ell^\sharp)^{m+n}}{ \binom{m+n}{m}\check\omega_{\ell,\mathrm{can}} ^m \wedge (\Sigma_\ell^* \omega_F)^n}+ne^{-t_\ell}\check t
-\sum_{i=1}^j\sum_{p=1}^{N_{i,k}} \check{\mathfrak{G}}_{\check{t},k}\left(\partial_{\check{t}}\check A_{\ell,i,p,k}^\sharp
+e^{-t_\ell}\check A_{\ell,i,p,k}^\sharp,\check G_{\ell,i,p,k} \right)-e^{-e^{-t_\ell}\check t}\partial_{\check t}\underline{\check\chi_\ell^\sharp},
\end{equation}
and for any $b\geq 0$ we have
\begin{equation}\label{luk1}
\ve_\ell^{-\alpha}\ve_\ell^{b-2j}d_\ell^{-2j-\alpha}\|\DD^b(\mathcal{E}_2+\mathcal{E}_3)\|_{\infty, \tilde Q_{R\e_\ell},\tilde g_\ell(0)}
=\delta_\ell^{-2j-\alpha}\|\DD^b(\Pi_\ell^*(\mathcal{E}_2+\mathcal{E}_3))\|_{\infty, \check Q_{R},\check g_\ell(0)},
\end{equation}
\begin{equation}\label{luk2}
\ve_\ell^{\beta-\alpha}\ve_\ell^{b-2j}d_\ell^{-2j-\alpha}[\DD^b(\mathcal{E}_2+\mathcal{E}_3)]_{\beta,\beta/2, \tilde Q_{R\e_\ell},\tilde g_\ell(0)}
=\delta_\ell^{-2j-\alpha}[\DD^b(\Pi_\ell^*(\mathcal{E}_2+\mathcal{E}_3))]_{\beta,\beta/2, \check Q_{R},\check g_\ell(0)},
\end{equation}
so thanks to \eqref{E_2-control} we see that $\delta_\ell^{-2j-\alpha}\Pi_\ell^*(\mathcal{E}_2+\mathcal{E}_3)$ is locally uniformly bounded in $C^{2j+\alpha,j+\a/2}$ with respect to the (essentially fixed) metric $\check{g}_\ell(0)$, so by Ascoli-Arzel\`a, up to passing to a subsequence, it converges locally uniformly on $\C^m\times Y\times (-\infty,0]$ to some limiting function $\mathcal{F}$. Since the quantity $-\delta_\ell^{-2j-\alpha}\Pi_\ell^*(\mathcal{E}_2+\mathcal{E}_3)$ is exactly \eqref{stronzo}, we would like to apply the Selection Theorem \ref{thm:Selection}, so we check that its hypotheses are satisfied. The functions $\hat{A}_{\ell,i,p,k}$ satisfy \eqref{linftyyy0} with $\alpha_0=\frac{\alpha^2}{2j+\alpha}$ thanks to \eqref{tilde-A-sharp}, while the function
$-e^{-\lambda_\ell^{-2}\hat t}\partial_{\hat t}\underline{\hat\chi_\ell^\sharp}-n\lambda_\ell^{-2}\hat{t}$ converges to $0$ locally smoothly thanks to \eqref{bdd-phi-sharp-underline}. We can thus apply the Selection Theorem and conclude that $\delta_\ell^{-2j-\alpha}\Pi_\ell^*(\mathcal{E}_2+\mathcal{E}_3)$ converges to $\mathcal{F}$ locally smoothly, hence its derivatives of all orders are uniformly bounded on $\check Q_{R}$. Thanks to \eqref{luk1}, \eqref{luk2}, this proves that \eqref{E_2-control} holds for all $b\geq 0$ and $\a\leq\beta<1$.
Recalling then that $\ddbar =\tilde J_\ell\circledast \D^2+\D \tilde J_\ell\circledast \D$, and using the bounds in \eqref{inter-estimate-0} for $\ti{J}_\ell$ and its derivatives, these imply directly that for all $a\geq 2, \a\leq \beta<1,$
\begin{equation}\label{benigno}
C\e_\ell^{a-2j}d_\ell^{-2j-\a}\left[\mathfrak{D}^{a-2}\ddbar(\mathcal{E}_2+\mathcal{E}_3)\right]_{\b,\b/2,\tilde Q_{R\ve_\ell},\tilde g_\ell(0)}\leq C\ve_\ell^{\alpha-\beta}.
\end{equation}
Next, we consider the term $\mathcal{E}_1$.
Recall that
\begin{equation}
\begin{split}
\mathcal{E}_1&= \log \frac{(\tilde \omega_\ell^\sharp +\tilde\eta_\ell^\circ+\tilde\eta_\ell^\diamond+\tilde\eta_{\ell,j,k})^{m+n}}{ (\tilde\omega_\ell^\sharp)^{m+n}}-\tr_{\tilde \omega_\ell^\sharp}\left(\tilde\eta_\ell^\circ+\tilde\eta_\ell^\diamond+\tilde\eta_{\ell,j,k}\right)\\
&=\log \left(1+\tr_{\tilde \omega_\ell^\sharp}(\ti\eta_\ell^\circ+\ti\eta_\ell^\diamond+\ti\eta_{\ell,j,k})+\mathcal{E}_4 \right)-\tr_{\tilde \omega_\ell^\sharp}\left(\ti\eta_\ell^\circ+\ti\eta_\ell^\diamond+\ti\eta_{\ell,j,k}\right),
\end{split}
\end{equation}
where $\mathcal{E}_4$ was defined in \eqref{E4}.
Taking $\mathfrak{D}^{a}$ derivatives, for $a\geq 2j\geq 2,a\in 2\mathbb{N}$, we again expand it schematically as in \eqref{log-estimate}.
The first step is to prove estimates for $\mathcal{E}_4$, and for this we observe that $\e_\ell^{a-2j}d_\ell^{-2j-\a}\mathcal{E}_4$ is identical to the term $\mathcal{N}$ in \cite[(4.268)]{HT3}, and we will bound it following the discussion there.
To do this, we need some basic estimates first. Using \eqref{sharp-metric-decay-0}, we have for all $\iota\geq 0$ that
\begin{equation}\label{omega-tilde-sharp}
[\mathfrak{D}^\iota \tilde\omega_\ell^\sharp]_{\b,\b/2,\tilde Q_{O(\e_\ell)},\tilde g_\ell(0)}\leq C\e_\ell^{-\b-\iota},\;\;\; \|\mathfrak{D}^\iota \tilde\omega_\ell^\sharp\|_{\infty,\tilde Q_{O(\e_\ell)},\tilde g_\ell(0)}\leq C\e_\ell^{-\iota}.
\end{equation}
Next, we observe that for all $\a\leq \b<1$ and $0\leq \iota<\max(a-1,2j)$, we have
\begin{equation}\label{induction-eta-tilde}
\|\mathfrak{D}^\iota(\ti\eta_\ell^\circ+\ti\eta_\ell^\diamond+\ti\eta_{\ell,j,k})\|_{\infty,\tilde Q_{O(\e_\ell)},\tilde g_\ell(0)}\leq C\delta_\ell^{2j+\a}\e_\ell^{-\iota},
\quad [\mathfrak{D}^\iota (\ti\eta_\ell^\circ+\ti\eta_\ell^\diamond+\ti\eta_{\ell,j,k})]_{\b,\b/2,\tilde Q_{O(\e_\ell)},\tilde g_\ell(0)}\leq C\delta_\ell^{2j+\a}\e_\ell^{-\iota-\b}.
\end{equation}
Indeed, these estimates are already known to hold for $\iota\leq 2j$ and $\beta=\a$ thanks to \eqref{estimate-psi-tilde-infty-1}, \eqref{bdd-phi-star-underline} and \eqref{estimate-potential-G}, hence using interpolation they also hold for $\iota<2j$ and $\alpha\leq\beta<1$,
while if $\iota<a-1, \a\leq\b<1$, these also hold thanks to the estimates \eqref{gurgu} which hold by the induction hypothesis. Recalling that in this section we have $j\geq 1$, it follows that $\max(a-1,2j)\geq 2$.

From \eqref{induction-eta-tilde} and the definition of $\mathcal{E}_4$ in \eqref{E4} it follows immediately that for $0\leq \iota<\max(a-1,2j)$ and $\a\leq\b<1$ we have
\begin{equation}\label{korku}
\ve_\ell^{a-2j}d_\ell^{-2j-\alpha}\|\DD^\iota\mathcal{E}_4\|_{\infty,\tilde Q_{O(\e_\ell)},\tilde g_\ell(0)}\leq C\ve_\ell^{a+\alpha-\iota} \delta_\ell^{2j+\alpha},
\quad\ve_\ell^{a-2j}d_\ell^{-2j-\alpha}[\DD^\iota\mathcal{E}_4]_{\b,\b/2,\tilde Q_{O(\e_\ell)},\tilde g_\ell(0)}\leq C\ve_\ell^{a+\alpha-\iota-\beta}\delta_\ell^{2j+\alpha}.
\end{equation}
However, we do not have the estimates \eqref{induction-eta-tilde} when $a-1\leq \iota\leq a$ and $\iota\geq 2j$, since our induction argument is only on {\em even} values of $a\geq 2j$, and this is different from the discussion in \cite[\S 4.10.3]{HT3}. To accommodate for the missing term with derivatives of order between $a-1$ and $a$, we apply Proposition \ref{prop:interpolation} with \eqref{induction-eta-tilde} so that
\begin{equation}\label{induction-eta-tilde-interpolated2}
\|\mathfrak{D}^{a-1}(\ti\eta_\ell^\circ+\ti\eta_\ell^\diamond+\ti\eta_{\ell,j,k})\|_{\infty,\tilde Q_{\tilde R},\tilde g_\ell(0)}\leq C\e_\ell^{1+\b}[\mathfrak{D}^{a}(\ti\eta_\ell^\circ+\ti\eta_\ell^\diamond+\ti\eta_{\ell,j,k})]_{\b,\b/2,\tilde Q_{R},\tilde g_\ell(0)}+C\delta_\ell^{2j+\a} (R-\rho)^{-1}\e_\ell^{-a+2},
\end{equation}
\begin{equation}\label{induction-eta-tilde-interpolated}
\begin{split}
[\mathfrak{D}^{a-1}(\ti\eta_\ell^\circ+\ti\eta_\ell^\diamond+\ti\eta_{\ell,j,k})]_{\b,\b/2,\tilde Q_{\tilde R},\tilde g_\ell(0)}&\leq C\e_\ell [\mathfrak{D}^{a}(\ti\eta_\ell^\circ+\ti\eta_\ell^\diamond+\ti\eta_{\ell,j,k})]_{\b,\b/2,\tilde Q_{R},\tilde g_\ell(0)}+C\delta_\ell^{2j+\a}  (R-\rho)^{-\b-1} \e_\ell^{-a+2},
\end{split}
\end{equation}
\begin{equation}\label{induction-eta-tilde-interpolated3}
\|\mathfrak{D}^{a}(\ti\eta_\ell^\circ+\ti\eta_\ell^\diamond+\ti\eta_{\ell,j,k})\|_{\infty,\tilde Q_{\tilde R},\tilde g_\ell(0)}\leq C\e_\ell^{\b}[\mathfrak{D}^{a}(\ti\eta_\ell^\circ+\ti\eta_\ell^\diamond+\ti\eta_{\ell,j,k})]_{\b,\b/2,\tilde Q_{R},\tilde g_\ell(0)}+C\delta_\ell^{2j+\a} (R-\rho)^{-2}\e_\ell^{-a+2}.
\end{equation}
Given these, we can use the same method as in \cite[(4.272)--(4.275)]{HT3} and see that for every $\iota$ with $\max(a-1,2j)\leq \iota\leq a$ and $\alpha\leq \beta<1$ we have
\begin{equation}\label{korku2}
\ve_\ell^{a-2j}d_\ell^{-2j-\alpha}\|\DD^\iota\mathcal{E}_4\|_{\infty,\tilde Q_{\tilde R},\tilde g_\ell(0)}\leq o(1)\ve_\ell^{a+\beta-\iota}\ve_\ell^{a-2j}d_\ell^{-2j-\alpha}[\mathfrak{D}^a (\ti\eta_\ell^\circ+\ti\eta_\ell^\diamond+\ti\eta_{\ell,j,k})]_{\b,\b/2,\tilde Q_R,\tilde g_\ell(0)}+C(R-\rho)^{-\iota}\delta_\ell^{2j+\alpha}\ve_\ell^{a+\alpha},
\end{equation}
and for $\max(a-1,2j)\leq \iota<a, \alpha\leq \beta<1$,
\begin{equation}\label{korku3}
\ve_\ell^{a-2j}d_\ell^{-2j-\alpha}[\DD^\iota\mathcal{E}_4]_{\b,\b/2,\tilde Q_{\tilde R},\tilde g_\ell(0)}\leq o(1)\ve_\ell^{a-\iota}\ve_\ell^{a-2j}d_\ell^{-2j-\alpha}[\mathfrak{D}^a (\ti\eta_\ell^\circ+\ti\eta_\ell^\diamond+\ti\eta_{\ell,j,k})]_{\b,\b/2,\tilde Q_R,\tilde g_\ell(0)}+C(R-\rho)^{-\iota-\beta}\delta_\ell^{2j+\alpha}\ve_\ell^{a+\alpha},
\end{equation}
while for $\iota=a$,
\begin{equation}\label{korku4}
\ve_\ell^{a-2j}d_\ell^{-2j-\alpha}[\DD^a\mathcal{E}_4]_{\b,\b/2,\tilde Q_{\tilde R},\tilde g_\ell(0)}\leq o(1)\ve_\ell^{a-2j}d_\ell^{-2j-\alpha}[\mathfrak{D}^a (\ti\eta_\ell^\circ+\ti\eta_\ell^\diamond+\ti\eta_{\ell,j,k})]_{\b,\b/2,\tilde Q_R,\tilde g_\ell(0)}.
\end{equation}
One important observation that we used here is that whenever we need to use \eqref{induction-eta-tilde-interpolated2}, \eqref{induction-eta-tilde-interpolated} or \eqref{induction-eta-tilde-interpolated3} for some term in $\mathcal{E}_4$, the remaining part of this summand in $\mathcal{E}_4$ is hit by at most $1+\beta$ derivatives, and since $1+\beta<2\leq \max(a-1,2j)$, for these other terms we are allowed to apply \eqref{induction-eta-tilde}.

Now that we have our estimates for $\mathcal{E}_4$, we need estimates on derivatives of $\tr_{\tilde\omega_\ell^\sharp}(\ti\eta_\ell^\circ+\ti\eta_\ell^\diamond+\ti\eta_{\ell,j,k})$. For this, from \eqref{omega-tilde-sharp} and \eqref{induction-eta-tilde} we see that for $0\leq\iota<\max(a-1,2j), \a\leq\b<1$,
\begin{equation}\label{korkuz0}
\ve_\ell^{a-2j}d_\ell^{-2j-\alpha}\|\DD^\iota\tr_{\tilde\omega_\ell^\sharp}(\ti\eta_\ell^\circ+\ti\eta_\ell^\diamond+\ti\eta_{\ell,j,k})\|_{\infty,\tilde Q_{\tilde R},\tilde g_\ell(0)}\leq
C\ve_\ell^{a+\alpha-\iota},
\end{equation}
\begin{equation}\label{korkuz1}
\ve_\ell^{a-2j}d_\ell^{-2j-\alpha}[\DD^\iota\tr_{\tilde\omega_\ell^\sharp}(\ti\eta_\ell^\circ+\ti\eta_\ell^\diamond+\ti\eta_{\ell,j,k})]_{\b,\b/2,\tilde Q_{\tilde R},\tilde g_\ell(0)}\leq
C\ve_\ell^{a+\alpha-\iota-\beta},
\end{equation}
while for derivatives of order $a-1\leq \iota\leq a$ and $\iota \geq 2j$, we can argue as above and estimate crudely
\begin{equation}\label{korkuz}
\begin{split}
\ve_\ell^{a-2j}d_\ell^{-2j-\alpha}\|\DD^\iota\tr_{\tilde\omega_\ell^\sharp}(\ti\eta_\ell^\circ+\ti\eta_\ell^\diamond+\ti\eta_{\ell,j,k})\|_{\infty,\tilde Q_{\tilde R},\tilde g_\ell(0)}&\leq C\ve_\ell^{a+\beta-\iota}\ve_\ell^{a-2j}d_\ell^{-2j-\alpha}[\mathfrak{D}^a (\ti\eta_\ell^\circ+\ti\eta_\ell^\diamond+\ti\eta_{\ell,j,k})]_{\b,\b/2,\tilde Q_R,\tilde g_\ell(0)}\\
&+C(R-\rho)^{-\iota}\ve_\ell^{a+\alpha},
\end{split}
\end{equation}
\begin{equation}\label{korkuz2}
\begin{split}
\ve_\ell^{a-2j}d_\ell^{-2j-\alpha}[\DD^\iota\tr_{\tilde\omega_\ell^\sharp}(\ti\eta_\ell^\circ+\ti\eta_\ell^\diamond+\ti\eta_{\ell,j,k})]_{\b,\b/2,\tilde Q_{\tilde R},\tilde g_\ell(0)}&\leq C\ve_\ell^{a-\iota}\ve_\ell^{a-2j}d_\ell^{-2j-\alpha}[\mathfrak{D}^a (\ti\eta_\ell^\circ+\ti\eta_\ell^\diamond+\ti\eta_{\ell,j,k})]_{\b,\b/2,\tilde Q_R,\tilde g_\ell(0)}\\
&+C(R-\rho)^{-\iota-\beta}\ve_\ell^{a+\alpha}.
\end{split}
\end{equation}
Equipped with \eqref{korku}, \eqref{korku2}, \eqref{korku3}, \eqref{korku4}, \eqref{korkuz0}, \eqref{korkuz1}, \eqref{korkuz} and \eqref{korkuz2}, we proceed to estimate derivatives of $\mathcal{E}_1$. We first consider the first line of \eqref{log-estimate}, which we can write as
\begin{equation}\label{merd}
-\mathfrak{D}^a \tr_{\tilde\omega_\ell^\sharp}(\ti\eta_\ell^\circ+\ti\eta_\ell^\diamond+\ti\eta_{\ell,j,k}) \frac{\tr_{\tilde\omega_\ell^\sharp}(\ti\eta_\ell^\circ+\ti\eta_\ell^\diamond+\ti\eta_{\ell,j,k})+\mathcal{E}_4}{1+\tr_{\tilde\omega_\ell^\sharp}(\ti\eta_\ell^\circ+\ti\eta_\ell^\diamond+\ti\eta_{\ell,j,k})+\mathcal{E}_4}
+\frac{\mathfrak{D}^a \mathcal{E}_4}{1+\tr_{\tilde\omega_\ell^\sharp}(\ti\eta_\ell^\circ+\ti\eta_\ell^\diamond+\ti\eta_{\ell,j,k})+\mathcal{E}_4},
\end{equation}
and we take $\ve_\ell^{a-2j}d_\ell^{-2j-\alpha}[\cdot]_{\b,\b/2,\tilde Q_{\tilde R},\tilde g_\ell(0)}$ of this. Since $|\tr_{\tilde\omega_\ell^\sharp}(\ti\eta_\ell^\circ+\ti\eta_\ell^\diamond+\ti\eta_{\ell,j,k})+\mathcal{E}_4|=o(1)$, when the difference quotient lands on $\mathfrak{D}^a \mathrm{tr}$ we can estimate this by $o(1)$ times \eqref{korkuz2} (with $\iota=a$). Similarly, $[\tr_{\tilde\omega_\ell^\sharp}(\ti\eta_\ell^\circ+\ti\eta_\ell^\diamond+\ti\eta_{\ell,j,k})+\mathcal{E}_4]_{C^\beta}=o(\ve_\ell^{-\beta})$, so when the difference quotient lands on $\frac{\mathrm{tr}+\mathcal{E}_4}{1+\mathrm{tr}+\mathcal{E}_4}$ we can estimate this by $o(\ve_\ell^{-\beta})$ times \eqref{korkuz} (with $\iota=a$). And when the difference quotient lands on $\frac{\DD^a\mathcal{E}_4}{1+\mathrm{tr}+\mathcal{E}_4}$ we argue similarly with \eqref{korku2} and \eqref{korku4}. So all together when we apply $\ve_\ell^{a-2j}d_\ell^{-2j-\alpha}[\cdot]_{\b,\b/2,\tilde Q_{\tilde R},\tilde g_\ell(0)}$ to \eqref{merd} we can bound it by
\begin{equation}
o(1)\ve_\ell^{a-2j}d_\ell^{-2j-\alpha}[\mathfrak{D}^a (\ti\eta_\ell^\circ+\ti\eta_\ell^\diamond+\ti\eta_{\ell,j,k})]_{\b,\b/2,\tilde Q_R,\tilde g_\ell(0)}+
C(R-\rho)^{-a-\beta}\ve_\ell^{a+\alpha}+C(R-\rho)^{-a}\ve_\ell^{a+\alpha-\beta}.
\end{equation}

Lastly we need to consider what happens when we take $\ve_\ell^{a-2j}d_\ell^{-2j-\alpha}[\cdot]_{\b,\b/2,\tilde Q_{\tilde R},\tilde g_\ell(0)}$  of the large sum in the second line of \eqref{log-estimate}. If all the derivatives that appear there are of order $<\max(a-1,2j)$, then this is bounded by $o(\ve_\ell^{\alpha-\beta})$, while if there is at least one derivative of order at least $\max(a-1,2j)$ (to which we apply \eqref{korku2}--\eqref{korku4}, \eqref{korkuz}--\eqref{korkuz2}), then all other derivatives in total are of order at most $1$ (and to these we can instead apply \eqref{korku}, \eqref{korkuz0}--\eqref{korkuz1}).  Putting all these together proves that
\begin{equation}\label{E_1-control}
\begin{split}
\e_\ell^{a-2j}d_\ell^{-2j-\a} [\mathfrak{D}^a\mathcal{E}_1]_{\b,\b/2,\tilde Q_{\tilde R},\tilde g_\ell(0)}&\leq \frac{1}{4}\e_\ell^{a-2j} d_\ell^{-2j-\a} [\mathfrak{D}^a (\ti\eta_\ell^\circ+\ti\eta_\ell^\diamond+\ti\eta_{\ell,j,k})]_{\b,\b/2,\tilde Q_R,\tilde g_\ell(0)}\\
&\quad  +C\sum_{r=0}^a \e_\ell^{\a-\b+r}(R-\rho)^{-r}+C\sum_{r=0}^{a}\e_\ell^{\a+r} (R-\rho)^{-r-\beta},
\end{split}
\end{equation}
and combining this with \eqref{benigno} completes the proof of Claim \ref{main-claim}.

For later use, observe also that the same argument gives an analogous bound for the $L^\infty$ norm of derivatives of $\mathcal{E}_1$, namely
\begin{equation}\label{E_1-control2}
\begin{split}
\e_\ell^{a-2j-\b}d_\ell^{-2j-\a} \|\mathfrak{D}^a\mathcal{E}_1\|_{\infty,\tilde Q_{\tilde R},\tilde g_\ell(0)}&\leq \frac{1}{4}\e_\ell^{a-2j} d_\ell^{-2j-\a} [\mathfrak{D}^a (\ti\eta_\ell^\circ+\ti\eta_\ell^\diamond+\ti\eta_{\ell,j,k})]_{\b,\b/2,\tilde Q_R,\tilde g_\ell(0)}\\
&\quad  +C\sum_{r=0}^a \e_\ell^{\a-\b+r}(R-\rho)^{-r}+C\sum_{r=0}^{a}\e_\ell^{\a+r} (R-\rho)^{-r-\beta}.
\end{split}
\end{equation}

\subsubsection{Killing the contributions from $\underline{\ti{\chi}^*_\ell}$}
The starting point is \eqref{borgu} with $a=2j$ and radius $C\delta_\ell$. For $0\leq b\leq 2j$, we first bound the term
\begin{equation}
\|\mathfrak{D}^b\ddbar \hat\varphi_\ell\|_{\infty,\hat Q_{C\delta_\ell},g_X}\leq C\|\mathfrak{D}^b(\hat{\eta}^\diamond_\ell+\hat{\eta}^\ddagger_\ell+\hat{\eta}^\circ_\ell+\hat{\eta}_{\ell,j,k})\|_{\infty,\hat Q_{C\delta_\ell},\hat{g}_\ell(0)}+\|\DD^b\hat{\eta}^\dagger_\ell\|_{\infty,\hat Q_{C\delta_\ell},g_X}\leq C,
\end{equation}
using \eqref{estimate-psi-tilde-infty-1}, \eqref{gorgo}, \eqref{gorgol}, \eqref{f1}, \eqref{f2} and \eqref{eta-dagger}. Using this, we transfer \eqref{borgu} to the tilde picture and multiply it by $d_\ell^{\beta-\alpha}$ and get
\begin{equation}\label{kokku}
\begin{split}
d_\ell^{-2j-\alpha}[\DD^{2j}\ti{\eta}^\diamond_\ell]_{\b,\b/2,\ti{Q}_{O(\ve_\ell)},\tilde g_\ell(0)}&=
d_\ell^{-2j-\alpha}[\DD^{2j}\ddbar\underline{\ti\varphi_\ell}]_{\b,\b/2,\ti{Q}_{O(\ve_\ell)},\tilde g_\ell(0)}\\
&\leq Cd_\ell^{-2j-\alpha}[\DD^{2j}_{\mathbf{bt}}\ti{\eta}_\ell]_{\b,\b/2,\textrm{base},\ti{Q}_{O(\ve_\ell)},\tilde g_\ell(0)}+Cd_\ell^{\b-\a}e^{-(1-\beta)\frac{t_\ell}{2}}\lambda_\ell^{-\beta}\\
&\leq Cd_\ell^{-2j-\alpha}[\mathfrak{D}^{2j} (\tilde \eta_{\ell,j,k}+\tilde \eta_\ell^\diamond+\tilde\eta_\ell^\circ)]_{\b,\b/2,\tilde Q_{O(\e_\ell)},\tilde g_\ell(0)}+o(\ve_\ell^{\a-\b})\\
&\leq C\ve_\ell^{\a-\b},
\end{split}
\end{equation}
where we used \eqref{gorgu} and \eqref{kukku}. Taking $\beta>\alpha$ gives us an $o(1)$ bound for the parabolic $C^{\beta,\beta/2}$ seminorm of $d_\ell^{-2j-\alpha}\DD^{2j}\ti{\eta}^\diamond_\ell$ on the cylinder of radius $2$ centered at $\ti{x}_\ell$ (which contains the other blowup point $\ti{x}'_\ell$), and hence an $o(1)$ bound for the parabolic $C^{\alpha,\alpha/2}$ seminorm on this same cylinder. Thanks to the bounds \eqref{gorgo}, the same conclusion holds for the  parabolic $C^{\alpha,\alpha/2}$ seminorm of
\begin{equation}
d_\ell^{-2j-\alpha}\DD^{2j}\left(e^{d_\ell^2\lambda_\ell^{-2}\ti{t}}\ti{\eta}^\diamond_\ell\right)=d_\ell^{-2j-\alpha}\DD^{2j}\ddbar\underline{\ti{\chi}^*_\ell},
\end{equation}
on the same cylinder, which kills one contribution of $\underline{\ti{\chi}^*_\ell}$ to \eqref{sup-realized-3}.

Next, using \eqref{gurgu} and the bounds \eqref{omega-tilde-sharp} for $\ti{\omega}^\sharp_\ell$, we see that for any $a\geq 0$ there is $C>0$ such that
\begin{equation}\label{korz}
d_\ell^{-2j-\a}\Biggr\| \mathfrak{D}^{a}\tr_{\tilde \omega_\ell^\sharp}\left(\tilde\eta_\ell^\diamond+ \tilde\eta_\ell^\circ+\tilde\eta_{\ell,j,k}\right)\Biggr\|_{\infty,\tilde Q_{O(\ve_\ell)},\tilde g_\ell(0)}\leq C\ve_\ell^{2j+\a-a},
\end{equation}
\begin{equation}\label{korz2}
d_\ell^{-2j-\a}\Biggr[ \mathfrak{D}^{a}\tr_{\tilde \omega_\ell^\sharp}\left(\tilde\eta_\ell^\diamond+ \tilde\eta_\ell^\circ+\tilde\eta_{\ell,j,k}\right)\Biggr]_{\b,\b/2,\tilde Q_{O(\ve_\ell)},\tilde g_\ell(0)}\leq C\ve_\ell^{2j+\a-a-\b},
\end{equation}
while from  \eqref{E_2-control}, \eqref{E_1-control}, \eqref{E_1-control2} and \eqref{kukku} we have
\begin{equation}\label{korz3}
d_\ell^{-2j-\a}\left\| \mathfrak{D}^{a}(\mathcal{E}_1+\mathcal{E}_2+\mathcal{E}_3)\right\|_{\infty,\tilde Q_{O(\ve_\ell)},\tilde g_\ell(0)}\leq C\ve_\ell^{2j+\a-a},
\end{equation}
\begin{equation}\label{korz4}
d_\ell^{-2j-\a}\left[ \mathfrak{D}^{a}(\mathcal{E}_1+\mathcal{E}_2+\mathcal{E}_3)\right]_{\b,\b/2,\tilde Q_{O(\ve_\ell)},\tilde g_\ell(0)}\leq C\ve_\ell^{2j+\a-a-\b},
\end{equation}
and so using these in the PDE \eqref{PMA-tilde-new} we get
\begin{equation}\label{korz5}
d_\ell^{-2j-\a}\left\| \mathfrak{D}^{a}(\partial_{\tilde t} \ti{\rho}_\ell+d_\ell^2\lambda_\ell^{-2}\ti{\rho}_\ell)\right\|_{\infty,\tilde Q_{O(\ve_\ell)},\tilde g_\ell(0)}\leq C\ve_\ell^{2j+\a-a},
\end{equation}
\begin{equation}\label{korz6}
d_\ell^{-2j-\a}\left[ \mathfrak{D}^{a}(\partial_{\tilde t} \ti{\rho}_\ell+d_\ell^2\lambda_\ell^{-2}\ti{\rho}_\ell)\right]_{\b,\b/2,\tilde Q_{O(\ve_\ell)},\tilde g_\ell(0)}\leq C\ve_\ell^{2j+\a-a-\b}.
\end{equation}
At this point we want to deduce from this bounds for the fiber average of $\partial_{\tilde t} \ti{\rho}_\ell+d_\ell^2\lambda_\ell^{-2}\ti{\rho}_\ell$, and this can be done using the following ``non-cancellation'' estimate, stated in the hat picture, for a smooth time-dependent function $f$ on $\hat{Q}_R$ (where the fiber average is $\underline{f}=(\mathrm{pr}_B)_*(f\Psi_\ell^*\omega_F^n)$) which states
\begin{equation}\label{nonk}
\begin{split}
[\DD^a\underline{f}]_{\b,\b/2,\hat{Q}_R}&\leq [\DD^a_{\mathbf{bt}}f]_{\b,\b/2,\textrm{base},\hat{Q}_R, g_X}
+C\left(\frac{R}{\lambda_\ell}\right)^{1-\beta}\lambda_\ell^{-\beta}\sum_{b=0}^a\|\DD^b f\|_{\infty,\hat{Q}_R,g_X}\\
&\leq [\DD^a f]_{\b,\b/2,\hat{Q}_R, \hat{g}_\ell(0)}
+C\left(\frac{R}{\lambda_\ell}\right)^{1-\beta}\lambda_\ell^{-\beta}\sum_{b=0}^a\|\DD^b f\|_{\infty,\hat{Q}_R,\hat{g}_\ell(0)},
\end{split}
\end{equation}
and which is proved exactly as \cite[(4.199)]{HT3}, using Claims 2 and 3 there. We apply this in the tilde picture with $a=2j$ to the function $\partial_{\tilde t} \ti{\rho}_\ell+d_\ell^2\lambda_\ell^{-2}\ti{\rho}_\ell$, whose fiber average is by definition $e^{-d_\ell^2\lambda_\ell^{-2}\ti{t}}\partial_{\ti{t}}\underline{\ti{\chi}^*_\ell}$, and we get
\begin{equation}\label{pork}
\begin{split}
d_\ell^{-2j-\a}\left[\DD^{2j}\left(e^{-d_\ell^2\lambda_\ell^{-2}\ti{t}}\partial_{\ti{t}}\underline{\ti{\chi}^*_\ell}\right)\right]_{\b,\b/2,\tilde Q_{O(\ve_\ell)}}
&\leq d_\ell^{-2j-\a}\left[ \mathfrak{D}^{2j}(\partial_{\tilde t} \ti{\rho}_\ell+d_\ell^2\lambda_\ell^{-2}\ti{\rho}_\ell)\right]_{\b,\b/2,\tilde Q_{O(\ve_\ell)},\ti{g}_\ell(0)}\\
&+Cd_\ell^{\b-\a}e^{-(1-\beta)\frac{t_\ell}{2}}\lambda_\ell^{-\beta}\sum_{b=0}^{2j} d_\ell^{-b}\|\DD^b(\partial_{\tilde t} \ti{\rho}_\ell+d_\ell^2\lambda_\ell^{-2}\ti{\rho}_\ell)\|_{\infty,\ti{Q}_{O(\ve_\ell)},\ti{g}_\ell(0)}\\
&\leq C\ve_\ell^{\a-\b}+Cd_\ell^{\b-\a}e^{-(1-\beta)\frac{t_\ell}{2}}\lambda_\ell^{-\beta}\sum_{b=0}^{2j}\delta_\ell^{2j+\alpha-b}\\
&\leq C\ve_\ell^{\a-\b},
\end{split}
\end{equation}
using \eqref{korz5} and \eqref{korz6}. Thus, if we take $\b>\a$ then the LHS of \eqref{pork} is $o(1)$, and so we get an $o(1)$ bound for the same $C^{\beta,\beta/2}$ seminorm on $\ti{Q}_R$ ($R>1$ fixed), and hence
\begin{equation}
d_\ell^{-2j-\a}\left[ \mathfrak{D}^{2j}\left(e^{-d_\ell^2\lambda_\ell^{-2}\ti{t}}\partial_{\ti{t}}\underline{\ti{\chi}^*_\ell}\right)\right]_{\a,\a/2,\tilde Q_R,\tilde g_\ell(0)}=o(1),
\end{equation}
which thanks to the bounds in \eqref{bdd-phi-star-underline} implies that
\begin{equation}
d_\ell^{-2j-\a}[ \mathfrak{D}^{2j}\partial_{\ti{t}}\underline{\ti{\chi}^*_\ell}]_{\a,\a/2,\tilde Q_R,\tilde g_\ell(0)}=o(1),
\end{equation}
which kills the other contribution of $\underline{\ti{\chi}^*_\ell}$ to \eqref{sup-realized-3}.

\subsubsection{Killing the contribution of $\tilde \psi_{\ell,j,k}$}
It remains to kill the contribution from $\tilde\psi_{\ell,j,k}$ to \eqref{sup-realized-3}. In contrast with \cite[\S 4.10.7]{HT3} where they had to kill the contribution from $d_\ell^{-2j-\alpha}\DD^{2j}\ddbar\ti{\psi}_{\ell,j,k}$, here we need to kill $d_\ell^{-2j-\alpha}\DD^{2j+2}\ti{\psi}_{\ell,j,k}$.

From the definition of $\ti{\rho}_\ell$ in \eqref{rho}, we have
\begin{equation}
\ti{\rho}_\ell-\underline{\ti{\rho}_\ell}=\ti{\psi}_{\ell,j,k}+\sum_{i=1}^j\sum_{p=1}^{N_{i,k}}\tilde{\mathfrak{G}}_{\ti{t},k}( \tilde A_{\ell,i,p,k}^*,\tilde G_{\ell,i,p,k} ),
\end{equation}
and so using \eqref{korz6} and \eqref{pork}, together with the triangle inequality and the boundedness of $\P$-parallel transport, gives
\begin{equation}\label{pakko}
\begin{split}
d_\ell^{-2j-\a}\Bigg[&\DD^{2j}\Bigg(\partial_{\ti{t}}\left(\ti{\psi}_{\ell,j,k}+\sum_{i=1}^j\sum_{p=1}^{N_{i,k}}\tilde{\mathfrak{G}}_{\ti{t},k}( \tilde A_{\ell,i,p,k}^*,\tilde G_{\ell,i,p,k} )\right)\\
&+d_\ell^2\lambda_\ell^{-2}\left(\ti{\psi}_{\ell,j,k}+\sum_{i=1}^j\sum_{p=1}^{N_{i,k}}\tilde{\mathfrak{G}}_{\ti{t},k}( \tilde A_{\ell,i,p,k}^*,\tilde G_{\ell,i,p,k} )\right)\Bigg)\Bigg]_{\beta,\beta/2,\ti{Q}_{O(\ve_\ell)},\ti{g}_\ell(0)}\leq C\ve_\ell^{\a-\b}.
\end{split}
\end{equation}
Now thanks to \eqref{estimate-psi-tilde-infty-1} and \eqref{estimate-potential-G} we can bound
\begin{equation}
d_\ell^2\lambda_\ell^{-2}d_\ell^{-2j-\a}\left[\DD^{2j}\left(\ti{\psi}_{\ell,j,k}+\sum_{i=1}^j\sum_{p=1}^{N_{i,k}}\tilde{\mathfrak{G}}_{\ti{t},k}( \tilde A_{\ell,i,p,k}^*,\tilde G_{\ell,i,p,k} )\right)
\right]_{\beta,\beta/2,\ti{Q}_{O(\ve_\ell)},\ti{g}_\ell(0)}\leq Cd_\ell^2\lambda_\ell^{-2}\ve_\ell^{2+\alpha-\beta}=o(\ve_\ell^{\a-\b}),
\end{equation}
and so \eqref{pakko} implies
\begin{equation}\label{pokko}
\begin{split}
d_\ell^{-2j-\a}\Bigg[&\DD^{2j}\partial_{\ti{t}}\left(\ti{\psi}_{\ell,j,k}+\sum_{i=1}^j\sum_{p=1}^{N_{i,k}}\tilde{\mathfrak{G}}_{\ti{t},k}( \tilde A_{\ell,i,p,k}^*,\tilde G_{\ell,i,p,k} )\right)\Bigg]_{\beta,\beta/2,\ti{Q}_{O(\ve_\ell)},\ti{g}_\ell(0)}\leq C\ve_\ell^{\a-\b}.
\end{split}
\end{equation}
Next, from the bound \eqref{kokku}, as well as the analogous bounds for lower derivatives of $\ti{\eta}^\diamond_\ell$ which come from \eqref{gorgo}, and the bounds
 \eqref{omega-tilde-sharp} for $\ti{\omega}^\sharp_\ell$, we see that
\begin{equation}
d_\ell^{-2j-\a}[ \mathfrak{D}^{2j}\tr_{\tilde \omega_\ell^\sharp}\tilde\eta_\ell^\diamond]_{\b,\b/2,\tilde Q_{O(\ve_\ell)},\tilde g_\ell(0)}\leq C\ve_\ell^{\a-\b},
\end{equation}
and this together with \eqref{korz2} gives
\begin{equation}
d_\ell^{-2j-\a}[ \mathfrak{D}^{2j}\tr_{\tilde \omega_\ell^\sharp}\left(\tilde\eta_\ell^\circ+\tilde\eta_{\ell,j,k}\right)]_{\b,\b/2,\tilde Q_{O(\ve_\ell)},\tilde g_\ell(0)}\leq C\ve_\ell^{\a-\b},
\end{equation}
which together with \eqref{pokko} gives
\begin{equation}\label{pekko}
\begin{split}
d_\ell^{-2j-\a}\Bigg[&\DD^{2j}\left(\partial_{\ti{t}}-\Delta_{\ti{\omega}^\sharp_\ell}\right)\left(\ti{\psi}_{\ell,j,k}+\sum_{i=1}^j\sum_{p=1}^{N_{i,k}}\tilde{\mathfrak{G}}_{\ti{t},k}( \tilde A_{\ell,i,p,k}^*,\tilde G_{\ell,i,p,k} )\right)\Bigg]_{\beta,\beta/2,\ti{Q}_{O(\ve_\ell)},\ti{g}_\ell(0)}\leq C\ve_\ell^{\a-\b}.
\end{split}
\end{equation}
This estimate can be inserted in the Schauder estimates in Proposition \ref{prop:schauder} (as usual by first going to the check picture, and then changing the result back to the tilde picture), with radii both $O(\ve_\ell)$,
\begin{equation}
\begin{split}
d_\ell^{-2j-\a}\Bigg[&\DD^{2j+2}\left(\ti{\psi}_{\ell,j,k}+\sum_{i=1}^j\sum_{p=1}^{N_{i,k}}\tilde{\mathfrak{G}}_{\ti{t},k}( \tilde A_{\ell,i,p,k}^*,\tilde G_{\ell,i,p,k} )\right)\Bigg]_{\beta,\beta/2,\ti{Q}_{O(\ve_\ell)},\ti{g}_\ell(0)}\\
&\leq
C\ve_\ell^{\a-\b}+C\ve_\ell^{-2j-2-\beta}d_\ell^{-2j-\a}\left\|\ti{\psi}_{\ell,j,k}
+\sum_{i=1}^j\sum_{p=1}^{N_{i,k}}\tilde{\mathfrak{G}}_{\ti{t},k}( \tilde A_{\ell,i,p,k}^*,\tilde G_{\ell,i,p,k} )\right\|_{\infty,\ti{Q}_{O(\ve_\ell)},\ti{g}_\ell(0)}\\
&\leq C\ve_\ell^{\a-\b},
\end{split}
\end{equation}
where we used the bounds \eqref{estimate-psi-tilde-infty-1} and \eqref{estimate-potential-G} for the $L^\infty$ norm. As usual, taking $\beta>\alpha$ this implies an $o(1)$ bound for the $C^{\b,\b/2}$ seminorm on $\ti{Q}_{O(\ve_\ell)}$, hence on $\ti{Q}_{2}$, and hence also an $o(1)$ bound for the $C^{\a,\a/2}$ seminorm on $\ti{Q}_{2}$, i.e.
\begin{equation}\label{pikko}
d_\ell^{-2j-\a}\Bigg[\DD^{2j+2}\left(\ti{\psi}_{\ell,j,k}+\sum_{i=1}^j\sum_{p=1}^{N_{i,k}}\tilde{\mathfrak{G}}_{\ti{t},k}( \tilde A_{\ell,i,p,k}^*,\tilde G_{\ell,i,p,k} )\right)\Bigg]_{\a,\a/2,\ti{Q}_{2},\ti{g}_\ell(0)}=o(1).
\end{equation}

Next we claim that
\begin{equation}\label{error-from-total-psi}
\displaystyle
\sum_{i=1}^j\sum_{p=1}^{N_{i,k}}d_\ell^{-2j-\a}  \left[\mathfrak{D}^{2j+2} \tilde{\mathfrak{G}}_{\ti{t},k}( \tilde A_{\ell,i,p,k}^*,\tilde G_{\ell,i,p,k} )\right]_{\a,\a/2,\tilde Q_2,\tilde g_\ell(0)}=o(1).
\end{equation}
The argument is identical to the deduction of  \eqref{estimate-potential-G} except we now use the improved parabolic H\"older seminorm from \eqref{improved-A-estimate-subcaseA} instead of that from blowup argument.  By applying \eqref{plop}, we have
\begin{equation}\label{Gtilde}
\tilde{\mathfrak{G}}_{\ti{t},k}( \tilde A_{\ell,i,p,k}^*,\tilde G_{\ell,i,p,k} )=\sum_{\iota=0}^{2k}\sum_{q=\lceil\frac{\iota}{2}\rceil}^ke^{-qd_\ell^2\lambda_\ell^{-2}\tilde t-(q-\frac{\iota}{2})t_\ell} \e_\ell^{\iota} \cdot \tilde \Phi_{\iota,q}(\tilde G_{i,p,k})\circledast \D^\iota \tilde A_{\ell,i,p,k}^*.
\end{equation}
so that
\begin{equation}\label{killing-contr-psi-improved-3.8}
\begin{split}
&\quad \mathfrak{D}^{2j+2}\tilde{\mathfrak{G}}_{\ti{t},k}( \tilde A_{\ell,i,p,k}^*,\tilde G_{\ell,i,p,k} )\\
&= \sum_{\iota=0}^{2k}\sum_{q=\lceil\frac{\iota}{2}\rceil}^{k} \sum_{\substack{d=0\\d\in 2\mathbb{N}}}^{2j+2} \sum_{i_1+i_2=2j+2-d} e^{-qd_\ell^2\lambda_\ell^{-2}\tilde t-(q-\frac{\iota}{2})t_\ell} (qd_\ell^2\lambda_\ell^{-2})^\frac{d}{2}\e_\ell^{\iota} \cdot \mathfrak{D}^{i_1}\tilde \Phi_{\iota,q}(\tilde G_{i,p,k})\circledast \mathfrak{D}^{i_2+\iota} \tilde A_{\ell,i,p,k}^*.
\end{split}
\end{equation}

To estimate it,  we need the following simple bounds from \cite[(4.302)--(4.303)]{HT3}
\begin{equation}
\left\{
\begin{array}{ll}
\|\mathfrak{D}^\iota \tilde \Phi_{\iota,q}(\tilde G_{i,p,k})\|_{\infty, \tilde Q_{O(\e_\ell)},\tilde g_\ell(0)}\leq C\e_\ell^{-\iota},\\[2mm]
[\mathfrak{D}^\iota \tilde \Phi_{\iota,q}(\tilde G_{i,p,k})]_{\a,\a/2, \tilde Q_{S},\tilde g_\ell(0)}\leq CS^{1-\a}\ve_\ell^{-\iota-1}=o (\e_\ell^{-\iota-\a}),
\end{array}
\right.
\end{equation}
for all $\iota\geq 0$, $\a\in (0,1)$ and  fixed $S>1$.

On the other hand, transferring  \eqref{inter-estimate-1}  to the tilde picture,
\begin{equation}
d_\ell^{-\iota+2}\| \mathfrak{D}^\iota \tilde A_{\ell,i,p,k}^*\|_{\infty,\tilde Q_{O(\e_\ell)},\tilde g_\ell(0)}\leq C\delta_\ell^{2+2j-\iota+\a}=o(\delta_\ell^{2+2j-\iota}),
\end{equation}
for all $0\leq \iota \leq 2k+2+2j$.  Putting $\b>\a$ in \eqref{improved-A-estimate-subcaseA} yields
\begin{equation}
d_\ell^{-\iota+2-\a}[\mathfrak{D}^\iota \tilde A_{\ell,i,p,k}^*]_{\b,\b/2,\tilde Q_{O(\e_\ell)},\tilde g_\ell(0)}\leq C \e_\ell^{\a-\b} \delta_\ell^{2j-\iota+2}=o(\delta_\ell^{2+2j-\iota}),
\end{equation}
for all $\iota\geq 2j$ and thus implies
\begin{equation}
d_\ell^{-\iota+2-\a}[\mathfrak{D}^\iota \tilde A_{\ell,i,p,k}^*]_{\a,\a/2,\tilde Q_{2},\tilde g_\ell(0)}\leq C \e_\ell^{\a-\b} \delta^{2+2j-\iota}_\ell =o(\delta_\ell^{2+2j-\iota}).
\end{equation}
for all $2j\leq \iota\leq 2j+2+2k$ while for $0\leq \iota\leq 2j$, we have from \eqref{A-star-0} that
\begin{equation}
d_\ell^{2-\iota-\a} [\mathfrak{D}^\iota \tilde A_{\ell,i,p,k}^*]_{\a,\a/2,\tilde Q_2,\tilde g_\ell(0)}\leq C\delta_\ell^2 d_\ell^{2j-\iota}=o(\delta_\ell^{2+2j-\iota}),
\end{equation}
since $\e_\ell\to +\infty$.

We now estimate \eqref{killing-contr-psi-improved-3.8} using the above estimates of each terms:
\begin{equation}
\begin{split}
&\quad d_\ell^{-2j-\a}\left[\mathfrak{D}^{2j+2}\tilde{\mathfrak{G}}_{\ti{t},k}( \tilde A_{\ell,i,p,k}^*,\tilde G_{\ell,i,p,k})\right]_{\a,\a/2,\tilde Q_2,\tilde g_\ell(0)}\\
&\leq o(1)\sum_{\iota,i_1,i_2,d} \Big(  (d_\ell\lambda_\ell^{-1})^{d+\a}\e_\ell^{\iota-i_1} \delta_\ell^{2+2j-\iota-i_2} d_\ell^{i_2+\iota-2} +(d_\ell\lambda_\ell^{-1})^{d}\e_\ell^{\iota-i_1-\a} \delta_\ell^{2+2j-\iota-i_2}d_\ell^{i_2+\iota-2}\\
&\quad \quad +(d_\ell\lambda_\ell^{-1})^{d}\e_\ell^{\iota-i_1} \delta_\ell^{2+2j-\iota-i_2}    d_\ell^{i_2+\iota+\a-2} \Big)\\
&=o(1),
\end{split}
\end{equation}
which proves \eqref{error-from-total-psi}. Combining it with \eqref{pikko}, and using the triangle inequality gives
\begin{equation}\label{kakku2}
d_\ell^{-2j-\alpha}[\DD^{2j+2}\ti{\psi}_{\ell,j,k}]_{\a,\a/2,\ti{Q}_{2},\tilde g_\ell(0)}=o(1),
\end{equation}
which kills the contribution of $\tilde\psi_{\ell,j,k}$ in \eqref{sup-realized-3}.  In conclusion,  all terms that appear on the right hand side of \eqref{sup-realized-3} converge to zero, which gives a contradiction and finally concludes {\bf Subcase A}.

\subsection{Subcase B: $\e_\ell\to \e_\infty>0$}
Without loss of generality, we will assume that $\e_\infty=1$.
By  \eqref{estimate-psi-tilde-infty-10}, \eqref{estimate-psi-tilde-infty-1} and \cite[Lemma 2.6]{HT3}, we see that  $d_\ell^{-2j-\a}\tilde\psi_{\ell,j,k}$  converges in $C^{2j+2+\b,j+1+\b/2}_{\rm loc}$ to a function $\tilde\psi_{\infty,j,k}\in C^{2j+2+\a,j+1+\a/2}_{\rm loc}$ defined on $\mathbb{C}^m\times Y\times (-\infty,0]$, for all $0<\b<\a$, while \eqref{bdd-phi-star-underline} gives us that
$d_\ell^{-2j-\a}\partial_{\tilde t}\underline{\tilde\chi_\ell^*}$ and $d_\ell^{-2j-\a}\ddbar\underline{\tilde\chi_\ell^*}$ converge  in $C_{\rm loc}^{2j+\b,j+\b/2}$ for all $0<\b<\a$ to a function $u_\infty\in C^{2j+\a,j+\a/2}_{\rm loc}$ and to a $(1,1)$ form $\eta_\infty\in C^{2j+\a,j+\a/2}_{\rm loc}$ respectively on $\mathbb{C}^m\times (-\infty,0]$. Using again \eqref{bdd-phi-star-underline}, we also have that $d_\ell^{-2j-\a}e^{-d_\ell^2\lambda_\ell^{-2}\tilde t}\partial_{\tilde t}\underline{\tilde \chi_\ell^* }\to u_\infty$ in $C_{\rm loc}^{2j+\b,j+\b/2}$ and
$d_\ell^{-2j-\a}e^{-d_\ell^2\lambda_\ell^{-2}\tilde t}\ddbar\underline{\tilde \chi_\ell^* }\to \eta_\infty$ in $C_{\rm loc}^{2j+\b,j+\b/2}$. Moreover, when $j\geq 1$ we have
\begin{equation}\label{trivv1}
\ddbar u_\infty=\partial_{\tilde t}\eta_\infty.
\end{equation}
From \eqref{estimate-psi-tilde-infty-10}, \eqref{estimate-psi-tilde-infty-1}, \eqref{bdd-phi-star-underline}, \eqref{circ-e-nonzero}, we see that
$\ti{\psi}_{\infty,j,k}=O(r^{2j+2+\a}),$ $u_\infty=O(r^{2j+\a}),$ $\eta_\infty=O(r^{2j+\a}),$ and $\tilde\eta_\infty^\circ =O(r^{2j+\a})$ where $r=|z|+\sqrt{|t|}$.

By \eqref{hat-A-star}, \eqref{A-star-0} and \eqref{A-star-1}, the functions $d_\ell^{-2j-\a}\tilde A_{\ell,i,p,k}^*$ converge to limiting functions $\tilde A_{\infty,i,p,k}^*$ from the base $\mathbb{C}^m$ in $C_{\rm loc}^{2j+2+2k+\b,j+1+2k+\b/2}$ for any $0<\b<\a$ while $\tilde G_{i,p,k}$ converge locally uniformly smoothly to functions $\tilde G_{\infty,i,p,k}$ pulled back from $Y$.
Using \eqref{estimate-potential-G}  we see that the functions $d_\ell^{-2j-\a}\ti{\mathfrak{G}}_{\ti{t},k}(\ti{A}_{\ell,i,p,k}^*,\ti{G}_{\ell,i,p,k})$ converge in $C^{2j+2+\b,j+1+\b/2}_{\rm loc}$ to a function which, thanks to \cite[(4.310)]{HT3}, is given by
\begin{equation}\label{futilis}
\ti{\mathfrak{G}}_{\infty,k}(\tilde A_{\infty,i,p,k}^*,\tilde G_{\infty,i,p,k})=\sum_{\ell=0}^{k}(-1)^\ell (\Delta^{\C^m})^\ell A_{\infty,i,p,k}^* (\Delta^Y)^{-\ell-1}\tilde G_{\infty,i,p,k}.
\end{equation}
By \eqref{bdd-phi-sharp-underline} and \eqref{dager-estimate}, the metrics $\tilde g^\sharp_\ell(t)\to g_P=g_{\mathbb{C}^m}+g_{Y,z_\infty=0}$ locally smoothly where $g_{\mathbb{C}^m}$ equals $g_{\mathrm{can}}|_{z=z_\infty=0}$.

Recall from \eqref{PMA-tilde-new} that
\begin{equation}\label{karpa}
d_\ell^{-2j-\a}\left(\left(\partial_{\tilde t}- \Delta_{\ti{\omega}^\sharp_\ell}\right)\ti{\rho}_\ell+d_\ell^2\lambda_\ell^{-2}\ti{\rho}_\ell\right)=d_\ell^{-2j-\a}(\mathcal{E}_1+\mathcal{E}_2+\mathcal{E}_3),
\end{equation}
where $\ti{\rho}_\ell$ is given by \eqref{rho}. From \eqref{estimate-psi-tilde-infty-10}, \eqref{estimate-psi-tilde-infty-1}, \eqref{ainf} and  \eqref{potential-eta-circ-estimate-1}, it follows that
\begin{equation}
d_\ell^{-2j-\a}d_\ell^2\lambda_\ell^{-2}\ti{\rho}_\ell\to 0,
\end{equation}
in $C^0_{\rm loc}$, so as $\ell\to+\infty$ the LHS of \eqref{karpa} converges to
\begin{equation}\label{eatalot1}
\left(\partial_{t}- \Delta_{\omega_P}\right)\left(\tilde\psi_{\infty,j,k}+\sum_{i=1}^j\sum_{p=1}^{N_{i,k}} \tilde{\mathfrak{G}}_{\infty,k}\left( \tilde A_{\infty,i,p,k}^*,\tilde G_{\infty,i,p,k} \right) \right)+u_\infty-\tr_{\omega_{\mathbb{C}^m}}\eta_\infty,
\end{equation}
while at the same time, thanks to \eqref{RHS-error-total}, the LHS of \eqref{karpa} is forced to be a polynomial $\mathcal{F}$ in the $(z,t)$ variables, of (spacetime) degree at most $2j$, with coefficients that are functions on $Y$, namely
\begin{equation}\label{eatalot2}
\mathcal{F}(z,y,t)=\sum_{|p|+2q\leq 2j} H_{p,q}(y) z^p t^q,
\end{equation}
for some smooth functions $H_{p,q}$ pullback from the fiber $Y$, where here again we treat $z$ as real variables. Plugging \eqref{eatalot2} into \eqref{eatalot1}, and taking the fiberwise average, gives
\begin{equation}
u_\infty-\tr_{\omega_{\C^m}}\eta_\infty=Q,
\end{equation}
where $Q$ is a polynomial in the $(z,t)$ variables of degree at most $2j$. But thanks to the jet subtractions, $u_\infty$ and $\eta_\infty$ have vanishing $2j$ parabolic jet at $(z,t)=(0,0)$, so
\begin{equation}\label{trivv2}
u_\infty-\tr_{\omega_{\mathbb{C}^m}}\eta_\infty=0,
\end{equation}
on $\mathbb{C}^m\times (-\infty,0]$.

We now claim that $u_\infty$ and $\eta_\infty$ are both identically zero. First, we assume that $j=0$. In this case, we can use \eqref{krkl2} and see that $u_\infty$ is a constant in space-time, hence identically zero. At this point we can argue similarly as in subcase A to show that $\eta_\infty$ also vanishes identically: first we show that $\eta_\infty$ is time-independent, by arguing similarly to \eqref{rokko}, noting that for any $\ti{z}\in\C^m$ and $\ti{t},\ti{t}'\leq 0$,
\begin{equation}\label{rokko2}
d_\ell^{-\a}\ddbar\underline{\ti{\chi}^*_\ell}(\ti{z},\ti{t})-d_\ell^{-\a}\ddbar\underline{\ti{\chi}^*_\ell}(\ti{z},\ti{t}')=\ddbar \int_{\ti{t}'}^{\ti{t}}\left(d_\ell^{-\a}\de_{\ti{t}}\underline{\ti{\chi}^*_\ell}\right)(\ti{z},\ti{s})d\ti{s},
\end{equation}
and since $\int_{\ti{t}'}^{\ti{t}}\left(d_\ell^{-\a}\de_{\ti{t}}\underline{\ti{\chi}^*_\ell}\right)(\ti{z},\ti{s})d\ti{s}\to 0$ locally uniformly (from $u_\infty\equiv 0$), it follows that the RHS of \eqref{rokko2} goes to $0$ weakly, and since the LHS converges in $C^\beta_{\rm loc}$ to $\eta_\infty(\ti{z},\ti{t})-\eta_\infty(\ti{z},\ti{t}')$, this must be identically zero as claimed.

We can then write $\eta_\infty=\ddbar v_\infty$, where $v_\infty\in C^{2+\a}_{\rm loc}(\mathbb{C}^m)$ is time-independent and $\Delta_{g_{\C^m}}v_\infty=0$, so $v_\infty$ is smooth and $|\ddbar v_\infty|=O(|z|^\a)$. The Liouville Theorem applied to each component of $\ddbar v_\infty$ then implies that $\eta_\infty$ has constant coefficients, hence it vanishes identically since its value at $(0,0)$ vanishes.

Next, we assume $j\geq 1$. We can then differentiate \eqref{trivv2} and use \eqref{trivv1} to see that
\begin{equation}
\begin{split}
\partial_{t}u_\infty &=\tr_{\mathbb{C}^m}\partial_{t} \eta_\infty=\tr_{\mathbb{C}^m}\ddbar u_\infty=\Delta_{\mathbb{C}^m}u_\infty,
\end{split}
\end{equation}
so $u_\infty$ solves the heat equation on $(-\infty,0]\times\mathbb{C}^m$, and
$|\mathfrak{D}^{\iota}u_\infty|=O(r^{2j+\a-\iota})$ where $r=|x|+\sqrt{t}$ for $0\leq \iota\leq 2j$ by \eqref{bdd-phi-star-underline}. By applying Liouville Theorem for ancient heat equation to $\mathfrak{D}^{2j}u_\infty$, we see that $u_\infty$ must be a space-time polynomial of degree at most $2j$, and hence  $u_\infty\equiv 0$ since its parabolic $2j$-jet at $(0,0)$ vanishes.
Going back to \eqref{trivv1} it then follows that $\eta_\infty$ is time-independent, it is clearly of the form $\eta_\infty=\ddbar v_\infty$ for some time-independent function $v_\infty\in C^{2j+2+\a}_{\rm loc}(\mathbb{C}^m)$, and from \eqref{trivv2} we see that $v_\infty$ is harmonic. Since $|\D^{2j}\eta_\infty|=O(|z|^\alpha)$,  the Liouville Theorem in \cite[Proposition 3.12]{HT2} then shows that the coefficients of $\eta_\infty$ are polynomials of degree at most $2j$, and since these coefficients have vanishing $2j$-jet at the origin, this implies that $\eta_\infty$ vanishes identically.

Now that we know that $u_\infty\equiv 0, \eta_\infty\equiv 0$, we can return to \eqref{eatalot1}, \eqref{eatalot2} and get
\begin{equation}\label{subcaseB-Box}
\begin{split}
&\quad \left(\de_{t}-\Delta_{\omega_P}\right)\left(\tilde\psi_{\infty,j,k}+\sum_{i=1}^j\sum_{p=1}^{N_{i,k}} \tilde{\mathfrak{G}}_{\infty,k}\left( \tilde A_{\infty,i,p,k}^*,\tilde G_{\infty,i,p,k} \right) \right)=\mathcal{F}.
\end{split}
\end{equation}
Recall from \eqref{karpa} and the discussion after it that $d_\ell^{-2j-\a}(\mathcal{E}_1+\mathcal{E}_2+\mathcal{E}_3)\to \mathcal{F}$ locally uniformly. From \eqref{error1-bad-1} we see that $d_\ell^{-2j-\a}\mathcal{E}_1\to 0$ locally uniformly, while the term $-d_\ell^{-2j-\a}(\mathcal{E}_2+\mathcal{E}_3)$ is exactly equal to \eqref{stronzo} (in this subcase, the check picture equal the tilde picture since $\ve_\ell \to 1$), so the Selection Theorem \ref{thm:Selection} shows that $\mathcal{F}$ is also equal to the limit of
\begin{equation}
d_\ell^{-2j-\a}\Sigma_\ell^*\left( f_{\ell,0}+\sum_{i=1}^j \sum_{p=1}^{N_{i,k}}f_{\ell,i,p} G_{i,p,k}\right),
\end{equation}
where $\Sigma_\ell$ was defined in \eqref{sigmat}, and where $f_{\ell,0}, f_{\ell,i,p}$ are time-dependent functions pulled back from the base, such that $\hat{f}_{\ell, 0}=\Psi_\ell^*f_{\ell,0},\hat{f}_{\ell,i,p}=\Psi_\ell^*f_{\ell,i,p}$  converge locally smoothly to zero. This implies that for any function $G$ on $\C^m\times Y$ which is fiberwise $L^2$ orthogonal to the span of the functions $\{\ti{G}_{\infty,i,p,k}\}_{1\leq i\leq j,1\leq p\leq N_{i,k}}$ together with the constants, and for any $z\in\C^m$ and $t\in (-\infty,0]$ we have
\begin{equation}
\int_{\{z\}\times Y}\mathcal{F}(z,y,t) G(z,y) \omega_{Y}^n(y)=0,
\end{equation}
which implies that we can write
\begin{equation}\label{qabrah}
\mathcal{F}(z,y,t)=g_{0}(z,t)+\sum_{i=1}^j\sum_{p=1}^{N_{i,k}}g_{i,p,k}(z,t)\ti{G}_{\infty,i,p,k}(y),
\end{equation}
for some functions $g_0,g_{i,p,k}$ on $\C^m\times (-\infty,0]$.
Since $\mathcal{F}$ is a polynomial in $(z,t)$ of degree at most $2j$, by fiberwise $L^2$ projecting $\mathcal{F}$ onto each $\ti{G}_{\infty,i,p,k}$ and onto the constants we see that the coefficients $g_0(z,t),g_{i,p,k}(z,t)$ are also polynomials of degree at most $2j$. This shows that $\mathcal{F}$ is a linear combination of the functions $\ti{G}_{\infty,i,p,k}$ together with the constant $1$, with coefficients that are polynomials in $(z,t)$ of degree at most $2j$, i.e.
\begin{equation}\label{qquamolim3}
\mathcal{F}(z,y,t)=K_0(z,t)+\sum_q K_q(z,t)H_q(y),\end{equation}
where $K_0(z,t),K_q(z,t)$ are polynomials of degree at most $2j$, and $H_q(y)$ are functions pulled back from the fiber $Y$ that lie in $\mathcal{H}$, the fiberwise span of the functions $\ti{G}_{\infty,i,p,k}, 1\leq i\leq j, 1\leq p\leq N_{i,k}$.

Following the argument in deriving \cite[(4.330)]{HT3} (which holds verbatim here), for each fixed $t\leq 0$  we obtain
 \begin{equation}\label{subcaseB-A-star}
 \begin{split}
 \tilde A^*_{\infty,i,p,k}(t)&=\int_{\{z\}\times Y} \tilde G_{\infty,i,p,k} \Delta_{\omega_P} \left(\tilde \psi_{\infty,j,k}+\sum_{i=1}^j\sum_{p=1}^{N_{i,k}} \tilde{\mathfrak{G}}_{\infty,k}\left(\tilde A_{\infty,i,p,k}^*,\tilde G_{\infty,i,p,k} \right) \right) \omega_Y^n\\
 &\quad -\sum_{q=1}^{i-1}\sum_{p=1}^{N_{q,k}}  \tilde \Phi_{2k+2,i,p,k}(\tilde G_{\infty,q,p,k}) \circledast \,  \D^{2k+2}\tilde A_{\infty,q,p,k}^*,
 \end{split}
 \end{equation}
where each quantity is evaluated at $t\leq 0$.  For notational convenience,  we denote
\begin{equation}\label{udef}
u(z,y,t)=\tilde \psi_{\infty,j,k}+\sum_{i=1}^j\sum_{p=1}^{N_{i,k}} \tilde{\mathfrak{G}}_{\infty,k}\left( \tilde A_{\infty,i,p,k}^*,\tilde G_{\infty,i,p,k} \right),
\end{equation}
which thanks to \eqref{subcaseB-Box} satisfies $ \left(\de_{t}-\Delta_{\omega_P}\right)u=\mathcal{F}$, and so given any $p,q\geq 0$ with $p+2q=2j+2$, and given $v_1,\dots,v_p$ tangent vectors to $\C^m$, from \eqref{qabrah} we see that
\begin{equation}
\left(\de_{t}-\Delta_{\omega_P}\right) \D^p_{v_1\cdots v_p}\de_t^q u =0,
\end{equation}
while, thanks to \eqref{estimate-psi-tilde-infty-10}, \eqref{estimate-psi-tilde-infty-1} and \eqref{estimate-potential-G-bt}, we also have that $\D^p_{v_1\cdots v_p}\de_t^q  u=O(r^\a)$, where $r=|z|+\sqrt{|t|}$.
Since $\omega_P$ is of Ricci-flat, we can apply the Liouville Theorem in \cite[Proposition 2.1]{FL} for ancient solutions of the heat equation, and conclude that
$\D^p_{v_1\cdots v_p}\de_t^q u$ is a constant in space and time. Since this is true for arbitrary $p,q$ with $p+2q=2j+2$, and for arbitrary $v_1,\dots,v_p$, this means that for every given $y\in Y$, the function $u(z,y,t)$ is a (parabolic) polynomial in $(z,t)$ of degree at most $2j+2$. Thus, $(\de_t u)(z,y,t)$ is a polynomial in $(z,t)$ of degree at most $2j$, hence the fiber integration
\begin{equation}
\int_{\{z\}\times Y} \tilde G_{\infty,i,p,k} (\de_t u)(z,y,t) \omega_Y^n(y),
\end{equation}
is also polynomial of degree at most $2j$. Thus, if we insert \eqref{subcaseB-Box} into \eqref{subcaseB-A-star}, we obtain
\begin{equation}\label{induction-A}
\begin{split}
\tilde A^*_{\infty,i,p,k}(\tilde t)&=-\sum_{q=1}^{i-1}\sum_{p=1}^{N_{q,k}}  \tilde \Phi_{2k+2,i,p,k}(\tilde G_{\infty,q,p,k}) \circledast \,  \D^{2k+2}\tilde A_{\infty,q,p,k}^*+Q_{i,p},
\end{split}
\end{equation}
where $Q_{i,p}$ is a (parabolic) polynomial on $\mathbb{C}^m\times (-\infty,0]$ of degree at most $2j$. We can then use this to show by induction on $0\leq i\leq j$ that $\ti{A}^*_{\infty,i,p,k}=0$ for all $i,p$. Indeed, in the base case of the induction $i=1$, the last term in \eqref{induction-A} is not present, and so $\ti{A}^*_{\infty,1,p,k}$ is a polynomial of degree at most $2j$, but since it also has vanishing $2j$-jet at $(0,0)$, it must be identically zero. The induction step is then exactly the same.

Next, since we have shown that $\tilde A_{\infty,i,p,k}^*\equiv 0$ for all $i,p$,  from \eqref{futilis} and \eqref{udef} we see that $u=\tilde\psi_{\infty,j,k}$, which by \eqref{subcaseB-A-star} satisfies
\begin{equation}
\int_{\{z\}\times Y} \tilde G_{\infty,i,p,k} \Delta_{\omega_P}u\omega_Y^n=0,
\end{equation}
for all $z,i,p$, i.e. we have
\begin{equation}\label{quip2}
\Delta_{\omega_P} u \in \mathcal{H}^\perp,
\end{equation}
the fiberwise $L^2(\omega_Y)$-orthogonal space to $\mathcal{H}$. Thanks to \eqref{subcaseB-Box} and \eqref{qabrah} we also have
\begin{equation}\label{quip1}
\left(\de_{t}-\Delta_{\omega_P}\right)u=\mathcal{F}\in \mathcal{H}.
\end{equation}
We then claim that we have $u\equiv 0$. To show this, we apply a trick from \cite[Claim 3.2]{FL}. We define a function $v$ on $\C^m\times Y\times (-\infty,0]$ as the fiberwise $L^2(\omega_Y)$-orthogonal projection of $u$ onto $\mathcal{H}^\perp$. Recalling that $|u|=O(r^{2j+2+\a})$, where $r=|z|+\sqrt{|t|}$, we claim that $v$ satisfies the same growth bound $|v|=O(r^{2j+2+\a})$. To see this, it suffices to prove this bound for the fiberwise $L^2(\omega_Y)$-orthogonal projection of $u$ onto $\mathcal{H}$, which equals
\begin{equation}
\sum_{i=1}^j\sum_{p=1}^{N_{i,k}}\left(\int_{\{z\}\times Y} u(z,\cdot,t)\ti{G}_{\infty,i,p,k}\omega_Y^n\right)\ti{G}_{\infty,i,p,k}(y),
\end{equation}
and whose supremum on $B_r(0)\times (-r^2,0]$ is clearly bounded by $C\sup_{B_r(0)\times (-r^2,0]}|u|\leq Cr^{2j+2+\a},$ as claimed.

Projecting \eqref{quip1} onto $\mathcal{H}^\perp$, and using \eqref{quip2} we see that
\begin{equation}\label{quip3}
\de_t v-\Delta_{\omega_P}u=0.
\end{equation}
Given $R>1$, which later will be taken sufficiently large, we define a function on $\mathbb{C}^m\times Y\times (-\infty,0]$ by
\begin{equation}
\phi_R(z,y,t):=e^{-\frac{|z|}{R}+\frac{t}{R}},
\end{equation}
and consider the weighted $L^2$ energy
\begin{equation}
E_R(t):=\int_{\mathbb{C}^m\times Y} v^2 \phi_R\omega_P^{m+n},
\end{equation}
which is finite since $|v|=O(r^{2j+2+\a})$ and $\phi_R$ decays fast at spatial and time infinity. Differentiating $E_R$ in time and using \eqref{quip3} we get
\begin{equation}
\begin{split}
E_R'(t)&=\int_{\mathbb{C}^m\times Y} \left( 2v\phi_R\Delta_{\omega_P} u +\frac1R v^2 \phi_R \right)\omega_P^{m+n}.
\end{split}
\end{equation}
Since $v$ is a fiberwise $L^2$-projection of $u$, we have $\int_Y v^2\omega_Y^n\leq \int_Yu^2\omega_Y^n$. Using this, together with \eqref{quip2}, we can estimate
\begin{equation}
\begin{split}
E'_R(t)&\leq (m+n)!\int_{\mathbb{C}^m}\phi_R \left( \int_{\{z\}\times Y}\left(2 v \Delta_{\omega_P}u+\frac{1}{R^2}u^2\right)\omega_Y^n \right)\omega_{\mathbb{C}^m}^m\\
&=(m+n)!\int_{\mathbb{C}^m}\phi_R \left( \int_{\{z\}\times Y}\left(2 u \Delta_{\omega_P}u+\frac{1}{R^2}u^2\right)\omega_Y^n \right)\omega_{\mathbb{C}^m}^m\\
&=\int_{\mathbb{C}^m\times Y} \left(2u\phi_R\Delta_{\omega_P} u  +\frac{1}{R^2}u^2\phi_R\right)\omega_{P}^{m+n}\\
&\leq  \int_{\mathbb{C}^m\times Y} \left( -2 |\nabla u|^2_{g_P} \phi_R +2u|\nabla u|_{g_P}|\nabla\phi_R|_{g_P}+\frac{1}{R^2} u^2 \phi_R \right)\omega_{P}^{m+n}\\
&\leq \int_{\mathbb{C}^m\times Y} \left( - |\nabla u|^2_{g_P} \phi_R +\frac{C}{R^2} u^2 \phi_R \right)\omega_{P}^{m+n},
 \end{split}
\end{equation}
where the integration by part is justified by the growth bound of $u$ and the fast decay of $\phi_R$. Since $u$ has fiberwise average zero, the Poincar\'e inequality on $(Y,\omega_Y)$ implies
\begin{equation}\begin{split}
\int_{\mathbb{C}^m\times Y} u^2 \phi_R \omega_P^{m+n}&=(m+n)!\int_{\mathbb{C}^m}\phi_R  \left( \int_{\{z\}\times Y}u^2\omega_Y^n \right)\omega_{\mathbb{C}^m}^m\\
&\leq C\int_{\mathbb{C}^m}\phi_R  \left( \int_{\{z\}\times Y}|\nabla^Y u|^2_{g_Y}\omega_Y^n \right)\omega_{\mathbb{C}^m}^m
\leq C\int_{\mathbb{C}^m\times Y} |\nabla u|^2_{g_P} \phi_R \omega_P^{m+n},
\end{split}\end{equation}
where the constant $C$ is independent of $R$, hence for all $t\leq 0$ we have
\begin{equation}
E'_R(t)\leq \int_{\mathbb{C}^m\times Y} \left( -\left(1-\frac{C}{R}\right) |\nabla u|^2_{g_P} \phi_R\right)\omega_{P}^{m+n}\leq 0,
\end{equation}
provided we choose $R$ sufficiently large. Thus, for any $s<t\leq 0$ we have $E_R(t)\leq E_R(s)$, but if we let $s\to-\infty$ then since $u$ grows at most polynomially while $\phi_R$ has exponential decay, we see that $\lim_{s\to -\infty} E_R(s)=0$, and so $E_R(t)\equiv 0$ for all $t\leq 0$, which implies that $v\equiv 0$. From \eqref{quip1} we then see that $\Delta_{\omega_P}u\equiv 0$, so $u$ is harmonic on $\C^m\times Y$ and $|u|=O(r^{2j+2+\a})$, so \cite[Proposition 3.12]{HT2} implies that for any fixed $t\leq 0$, the function $u(\cdot, t)$ is the pullback of a polynomial on $\C^m$ of degree at most $2j+2$, and since it also has fiberwise average zero, it vanishes identically.

To summarize, we have thus proved that $d_\ell^{-2j-\a}\ti{\psi}_{\ell,j,k},$ $d_\ell^{-2j-\a}\partial_{\tilde t}\underline{\tilde\chi_\ell^*},$ $d_\ell^{-2j-\a}\ddbar\underline{\tilde\chi_\ell^*},$
$d_\ell^{-2j-\a}\ti{A}^*_{\ell,i,p,k}$ all go to zero locally uniformly in the appropriate topologies, which when $j\geq 1$ shows that the RHS of \eqref{sup-realized-3} goes to zero, and gives a contradiction.  As for the case $j=0$, by definition we have
\begin{equation}\label{porp}
\ddbar\ti{\vp}_\ell=\ddbar\ti{\psi}_{\ell,0,k}+e^{-d_\ell^2\lambda_\ell^{-2}\ti{t}}\ddbar\underline{\tilde\chi_\ell^*}+e^{-d_\ell^2\lambda_\ell^{-2}\ti{t}}\ddbar\underline{\tilde\chi_\ell^\sharp},
\end{equation}
and we have just shown that $d_\ell^{-\a}\ddbar\ti{\psi}_{\ell,0,k}$ and $d_\ell^{-\a}e^{-d_\ell^2\lambda_\ell^{-2}\ti{t}}\ddbar\underline{\tilde\chi_\ell^*}$ both go to zero locally uniformly, so using also \eqref{qurq} we see that the RHS of \eqref{sup-realized-20} also converges to zero, contradiction.  This concludes {\bf Subcase B}.

\subsection{Subcase C: $\e_\ell\to 0$}
In this subcase, using repeatedly \cite[Lemma 2.6]{HT3}, thanks to \eqref{estimate-psi-tilde-infty-10} and \eqref{estimate-psi-tilde-infty-1} we have that $d_\ell^{-2j-\a}\tilde\psi_{\ell,j,k}$ converges to zero in $C^{2j+2+\b,j+1+\b/2}_{\rm loc}$.  Also, using \eqref{qorq}, we see that $d_\ell^{-2j-\a}\ti{\eta}^\circ_{\ell,j,k}$ converges in $C^{2j+\b,j+\b/2}_{\rm loc}$ to a $(1,1)$-form $\ti{\eta}^\circ_{\infty,j,k}$ on $\mathbb{C}^m\times Y\times (-\infty,0]$ which is weakly closed. The bounds in \eqref{bdd-phi-star-underline} give us that
$d_\ell^{-2j-\a}\partial_{\tilde t}\underline{\tilde\chi_\ell^*}$ and $d_\ell^{-2j-\a}\ddbar\underline{\tilde\chi_\ell^*}$ converge  in $C_{\rm loc}^{2j+\b,j+\b/2}$ for all $0<\b<\a$ to a function $u_\infty\in C^{2j+\a,j+\a/2}_{\rm loc}$ and to a $(1,1)$ form $\eta_\infty\in C^{2j+\a,j+\a/2}_{\rm loc}$ respectively on $\mathbb{C}^m\times (-\infty,0]$, and again $\eta_\infty$ is weakly closed, and these satisfy \eqref{trivv1} when $j\geq 1$. Using again \eqref{bdd-phi-star-underline}, we also have that $d_\ell^{-2j-\a}e^{-d_\ell^2\lambda_\ell^{-2}\tilde t}\partial_{\tilde t}\underline{\tilde \chi_\ell^* }\to u_\infty$ in $C_{\rm loc}^{2j+\b,j+\b/2}$ and
$d_\ell^{-2j-\a}e^{-d_\ell^2\lambda_\ell^{-2}\tilde t}\ddbar\underline{\tilde \chi_\ell^* }=d_\ell^{-2j-\a}\ti{\eta}^\diamond_\ell\to \eta_\infty$ in $C_{\rm loc}^{2j+\b,j+\b/2}$. From \eqref{gorgol} and \eqref{qarq} we see that $\ti{\omega}^\sharp_\ell\to \omega_{\C^m}$ locally smoothly, and in particular (as in \cite[(4.345)]{HT3}) this implies that
\begin{equation}\label{bizz}
d_\ell^{-2j-\a}\DD^{2j}_{\mathbf{bt}}\Delta_{\ti{\omega}^\sharp_\ell}\underline{\tilde \chi_\ell^*}=d_\ell^{-2j-\a}\DD^{2j}_{\mathbf{bt}}\Delta_{\omega_{\C^m}}\underline{\tilde \chi_\ell^*}+o(1),
\end{equation}
locally uniformly, and so
\begin{equation}\label{bizz2}
d_\ell^{-2j-\a}e^{-d_\ell^2\lambda_\ell^{-2}\tilde t}\DD^{2j}_{\mathbf{bt}}\Delta_{\ti{\omega}^\sharp_\ell}\underline{\tilde \chi_\ell^*}\to \tr_{\omega_{\C^m}}\eta_\infty,
\end{equation}
locally uniformly. Thanks to \eqref{A-star-2}, we also have that for all $-1\leq \iota\leq 2k$,
\begin{equation}
\ve_\ell^{\iota}d_\ell^{-2j-\a}\|\DD^{2j+2+\iota}\ti{A}^*_{\ell, i,p,k}\|_{\infty, \ti{Q}_R,\ti{g}_\ell(0)}\to 0,
\end{equation}
so when $j=0$, from \eqref{sup-realized-20} (using also \eqref{gorgol}) we see that
\begin{equation}\label{sup-realized-y}\begin{split}
1&\leq Cd_\ell^{-\alpha}|\ddbar \tilde\psi_{\ell,0,k}(\tilde x_\ell,0)- \P_{\tilde x'_\ell\tilde x_\ell}\ddbar\tilde\psi_{\ell,0,k}(\tilde x'_\ell,\tilde t_\ell')|_{\tilde g_\ell(0)}
+Cd_\ell^{-\alpha}|\ti{\eta}^\diamond_\ell(\tilde x_\ell,0)-\P_{\tilde x'_\ell\tilde x_\ell}\ti{\eta}^\diamond_\ell(\tilde x'_\ell,\tilde t_\ell')|_{\tilde g_\ell(0)}\\
&+Cd_\ell^{-\alpha}|\ti{\eta}^\ddagger_\ell(\tilde x_\ell,0)-\P_{\tilde x'_\ell\tilde x_\ell}\ti{\eta}^\ddagger_\ell(\tilde x'_\ell,\tilde t_\ell')|_{\tilde g_\ell(0)}\\
&\leq Cd_\ell^{-\alpha}|\ti{\eta}^\diamond_\ell(\tilde x_\ell,0)-\P_{\tilde x'_\ell\tilde x_\ell}\ti{\eta}^\diamond_\ell(\tilde x'_\ell,\tilde t_\ell')|_{\tilde g_\ell(0)}+o(1),
\end{split}
\end{equation}
while when $j\geq 1$, from \eqref{sup-realized-3} we see that
\begin{equation}\label{sup-realized-x}
\begin{split}
& (1+o(1))=d_\ell^{-2j-\a}|\mathfrak{D}^{2j} \ddbar \underline{\tilde\chi_\ell^*}(\tilde x_\ell,0)- \P_{\tilde x'_\ell\tilde x_\ell}\mathfrak{D}^{2j} \ddbar \underline{\tilde\chi_\ell^*}(\tilde x'_\ell,\tilde t_\ell')|_{\tilde g_\ell(0)}\\
&+d_\ell^{-2j-\a}|\mathfrak{D}^{2j}  \partial_{\tilde t} \underline{\tilde\chi_\ell^*}(\tilde x_\ell,0)- \P_{\tilde x'_\ell\tilde x_\ell}\mathfrak{D}^{2j} \partial_{\tilde t} \underline{\tilde\chi_\ell^*}(\tilde x'_\ell,\tilde t_\ell')|_{\tilde g_\ell(0)}\\
& +d_\ell^{-2j-\a}\sum_{i=1}^j \sum_{p=1}^{N_{i,k}}\e_\ell^{-2}\Bigg({|\mathfrak{D}^{2j} \tilde A^*_{\ell,i,p,k}(\tilde x_\ell,0)- \P_{\tilde x_\ell'\tilde x_\ell}(\mathfrak{D}^{2j}\tilde A^*_{\ell,i,p,k}(\tilde x_\ell',\tilde t_\ell'))|_{\tilde g_\ell(0)}}\Bigg).
\end{split}
\end{equation}
We can then invoke \eqref{RHS-error-total}, and use again the estimates \eqref{estimate-psi-tilde-infty-10} and \eqref{estimate-psi-tilde-infty-1}, and see that
\begin{equation}\label{qerq}
d_\ell^{-2j-\a}\Biggr[ \mathfrak{D}^{2j}_{\bf bt}\biggr(\left(\de_{\ti{t}}-\Delta_{\tilde \omega_\ell^\sharp}\right)\ti{\sigma}_\ell+d_\ell^2\lambda_\ell^{-2}\ti{\sigma}_\ell  \biggl)\Biggr]_{\a,\a/2,\mathrm{base},\tilde Q_R,\tilde g_\ell(0)}=o(1).
\end{equation}
where we have set
\begin{equation}
\ti{\sigma}_\ell:=e^{-d_\ell^2\lambda_\ell^{-2}\tilde t}\underline{\tilde\chi_\ell^*}+\sum_{i=1}^j\sum_{p=1}^{N_{i,k}} \tilde{\mathfrak{G}}_{\ti{t},k}( \tilde A_{\ell,i,p,k}^*,\tilde G_{\ell,i,p,k}).
\end{equation}
First, we dispose of the case $j=0$. Again in this case from \eqref{krkl2} we see that $u_\infty$ is a constant in space-time, hence identically zero, and as in subcase B we see that $\eta_\infty$ is time-independent (using the argument in \eqref{rokko2}). Then from \eqref{qerq}, using also \eqref{ainf}, we see that the quantity
\begin{equation}
d_\ell^{-\a}\left(\de_{\ti{t}}-\Delta_{\tilde \omega_\ell^\sharp}\right)\ti{\sigma}_\ell,
\end{equation}
is becoming asymptotically independent of the base and time directions, and so from \eqref{bizz2} we conclude that
\begin{equation}
\tr_{\omega_{\C^m}}\eta_\infty=({\rm const.}),
\end{equation}
and since $\eta_\infty$ vanishes at $(0,0)$, we actually have $\tr_{\omega_{\C^m}}\eta_\infty=0$.
As in Subcase B, we can write $\eta_\infty=\ddbar v_\infty$, where $v_\infty\in C^{2+\a}_{\rm loc}(\mathbb{C}^m)$ is time-independent and $\Delta_{\omega_{\C^m}}v_\infty=0$, so $v_\infty$ is smooth and $|\ddbar v_\infty|=O(|z|^\a)$. The Liouville Theorem applied to each component of $\ddbar v_\infty$ then implies that $\eta_\infty$ has constant coefficients, hence it vanishes identically since its value at $(0,0)$ vanishes. This means exactly that the RHS of \eqref{sup-realized-y} converges to zero, which gives us a contradiction when $j=0$.

In the remainder of this section we thus assume that $j\geq 1$.  Given any $p,q\geq 0$ with $p+2q=2j+2$, and given $v_1,\dots,v_p$ tangent vectors to $\C^m$, from \eqref{qerq}, recalling \eqref{ainf} and \eqref{estimate-potential-G}, we see that the quantity
\begin{equation}\label{rocco}
d_\ell^{-2j-\a}\D^p_{v_1\cdots v_p}\de_t^q\left(\de_{\ti{t}}-\Delta_{\tilde \omega_\ell^\sharp}\right)\left(e^{-d_\ell^2\lambda_\ell^{-2}\tilde t}\underline{\tilde\chi_\ell^*}+\sum_{i=1}^j\sum_{r=1}^{N_{i,k}} \tilde{\mathfrak{G}}_{\ti{t},k}( \tilde A_{\ell,i,r,k}^*,\tilde G_{\ell,i,r,k})\right),
\end{equation}
is asymptotically independent of the base and time directions.
Recall then that from \eqref{expansion-eta-circ} we have
\begin{equation}\label{rozzo1}
\begin{split}
&\sum_{i=1}^j\sum_{r=1}^{N_{i,k}}d_\ell^{-2j-\a} \D^p_{v_1\cdots v_p}\de_t^q \Delta_{\tilde \omega_\ell^\sharp}\tilde{\mathfrak{G}}_{\ti{t},k}( \tilde A_{\ell,i,r,k}^*,\tilde G_{\ell,i,r,k})\\
& = d_\ell^{-2j-\a} \sum_{i=1}^j \sum_{r=1}^{N_{i,k}}  \tr_{\ti{\omega}^\sharp_\ell}\left(\ddbar (\Delta_{\Theta_\ell^*\Psi_\ell^*\omega_{F}|_{ \{\cdot\}\times Y}})^{-1} \ti G_{\ell,i,r,k} \right)_{\bf ff}\D^p_{v_1\cdots v_p}\de_t^q\tilde A_{\ell,i,r,k}^* +o(1)\\
&=d_\ell^{-2j-\a} \sum_{i=1}^j \sum_{r=1}^{N_{i,k}}  \tr_{\ve_\ell^2\Theta_\ell^*\Psi_\ell^*\omega_{F}}\left(\ddbar (\Delta_{\Theta_\ell^*\Psi_\ell^*\omega_{F}|_{ \{\cdot\}\times Y}})^{-1} \ti G_{\ell,i,r,k} \right)_{\bf ff}\D^p_{v_1\cdots v_p}\de_t^q\tilde A_{\ell,i,r,k}^* +o(1)\\
&=\ve_\ell^{-2}d_\ell^{-2j-\a} \sum_{i=1}^j \sum_{r=1}^{N_{i,k}}   \ti G_{\ell,i,r,k} \D^p_{v_1\cdots v_p}\de_t^q\tilde A_{\ell,i,r,k}^* +o(1),
\end{split}
\end{equation}
where in the second equality we could exchange $\tr_{\ti{\omega}^\sharp_\ell}(\ddbar)_{\mathbf{ff}}$ with $\tr_{\ve_\ell^2\Theta_\ell^*\Psi_\ell^*\omega_{F}}(\ddbar)_{\mathbf{ff}}$ with only a $o(1)$ error since $d_\ell^{-2j-\a}\|\DD^{2j}\tilde A_{\ell,i,p,k}^*\|_{\infty,\ti{Q}_R,\ti{g}_\ell(0)}\leq C\ve_\ell^2$ by \eqref{A-star-2} and $\|\ti{\eta}^\dagger_\ell\|_{\infty,\ti{Q}_R,\ti{g}_\ell(0)}=o(1)$ by \eqref{dager-estimate}.
Similarly,  from \eqref{leading-G-A*} we see that
\begin{equation}\label{rozzo2}
\begin{split}
&\quad
d_\ell^{-2j-\a}\D^p_{v_1\cdots v_p}\de_t^{q+1} \left( \sum_{i=1}^j\sum_{r=1}^{N_{i,k}} \tilde{\mathfrak{G}}_{\ti{t},k}\left( \tilde A_{\ell,i,r,k}^*,\tilde G_{\ell,i,r,k} \right)\right)\\
& =d_\ell^{-2j-\a}\sum_{i=1}^j\sum_{r=1}^{N_{i,k}}  (\Delta_{\Theta_\ell^*\Psi_\ell^*\omega_{F}|_{ \{\cdot\}\times Y}})^{-1}  \tilde G_{\ell,i,r,k}\D^p_{v_1\cdots v_p}\de_t^{q+1}  \tilde A_{\ell,i,r,k}^*+o(1),
\end{split}
\end{equation}
while \eqref{bizz} implies that
\begin{equation}\label{rozzo3}
d_\ell^{-2j-\a}\D^p_{v_1\cdots v_p}\de_t^q\Delta_{\tilde \omega_\ell^\sharp}\left(e^{-d_\ell^2\lambda_\ell^{-2}\tilde t}\underline{\tilde\chi_\ell^*}\right)
=d_\ell^{-2j-\a}\D^p_{v_1\cdots v_p}\de_t^q\Delta_{\omega_{\C^m}}\left(e^{-d_\ell^2\lambda_\ell^{-2}\tilde t}\underline{\tilde\chi_\ell^*}\right)+o(1).
\end{equation}
and plugging \eqref{rozzo1}, \eqref{rozzo2} and \eqref{rozzo3} into \eqref{rocco}, we conclude that
\begin{equation}\label{optimus}
\begin{split}
&d_\ell^{-2j-\a}\left(\de_{\ti{t}}-\Delta_{\omega_{\C^m}}\right)\D^p_{v_1\cdots v_p}\de_t^q\left(e^{-d_\ell^2\lambda_\ell^{-2}\tilde t}\underline{\tilde\chi_\ell^*}\right)\\
&+d_\ell^{-2j-\a}\sum_{i=1}^j\sum_{r=1}^{N_{i,k}} \left( (\Delta_{\Theta_\ell^*\Psi_\ell^*\omega_{F}|_{ \{\cdot\}\times Y}})^{-1}  \tilde G_{\ell,i,r,k}\D^p_{v_1\cdots v_p}\de_t^{q+1}  \tilde A_{\ell,i,r,k}^*-\ve_\ell^{-2}\ti G_{\ell,i,r,k} \D^p_{v_1\cdots v_p}\de_t^q\tilde A_{\ell,i,r,k}^*\right),
\end{split}
\end{equation}
is asymptotically independent of the base and time directions. Observe that the first line in \eqref{optimus} is a function pulled back from $\C^m$, while the second line has fiberwise average zero. Thus, taking the fiberwise average of \eqref{optimus}, we see that
\begin{equation}\label{razzo}
d_\ell^{-2j-\a}\D^p_{v_1\cdots v_p}\de_t^q\left(\de_{\ti{t}}-\Delta_{\omega_{\C^m}}\right)\left(e^{-d_\ell^2\lambda_\ell^{-2}\tilde t}\underline{\tilde\chi_\ell^*}\right)
\end{equation}
is approaching a (time-independent) constant locally uniformly on $\C^m\times (-\infty,0]$. Recalling that $d_\ell^{-2j-\a}\de_{\ti{t}}\left(e^{-d_\ell^2\lambda_\ell^{-2}\tilde t}\underline{\tilde\chi_\ell^*}\right)\to u_\infty$ in $C^{2j+\b,j+\b/2}_{\rm loc}$  and $d_\ell^{-2j-\a}\Delta_{\omega_{\C^m}}\left(e^{-d_\ell^2\lambda_\ell^{-2}\tilde t}\underline{\tilde\chi_\ell^*}\right)\to \tr_{\omega_{\C^m}}{\eta_\infty}$ in $C^{2j+\b,j+\b/2}_{\rm loc}$, passing to the limit in \eqref{razzo} shows that $u_\infty-\tr_{\omega_{\C^m}}{\eta_\infty}$ is a parabolic polynomial on $\C^m\times (-\infty,0]$ of degree at most $2j$, and since by jet subtraction the parabolic $2j$-jets of $u_\infty$ and $\eta_\infty$ vanish at $(0,0)$,  we conclude that
\begin{equation}\label{triw2}
u_\infty-\tr_{\omega_{\mathbb{C}^m}}\eta_\infty=0,
\end{equation}
on $\mathbb{C}^m\times (-\infty,0]$. Since $j\geq 1$, we can differentiate \eqref{triw2} with respect to $t$, and use \eqref{trivv1} to see that
\begin{equation}
\begin{split}
\partial_{\tilde t}u_\infty &=\tr_{\mathbb{C}^m}\partial_{\tilde t} \eta_\infty=\tr_{\mathbb{C}^m}\ddbar u_\infty=\Delta_{\mathbb{C}^m}u_\infty,
\end{split}
\end{equation}
so $u_\infty$ solves the heat equation on $(-\infty,0]\times\mathbb{C}^m$, and
$|\mathfrak{D}^{\iota}u_\infty|=O(r^{2j+\a-\iota})$ for $0\leq \iota\leq 2j$ by \eqref{bdd-phi-star-underline}. By applying Liouville Theorem for ancient heat equation to $\mathfrak{D}^{2j}u_\infty$, we see that $u_\infty$ must be a space-time polynomial of degree at most $2j$, and hence  $u_\infty\equiv 0$ since its parabolic $2j$-jet at $(0,0)$ vanishes.  Going back to \eqref{trivv1} it then follows that $\eta_\infty$ is time-independent, it is clearly of the form $\eta_\infty=\ddbar v_\infty$ for some time-independent function $v_\infty\in C^{2j+2+\b}_{\rm loc}(\mathbb{C}^m)$, and from \eqref{triw2} we see that $v_\infty$ is harmonic. Since $|\D^{2j}\eta_\infty|=O(|z|^\alpha)$,  the Liouville Theorem in \cite[Proposition 3.12]{HT2} then shows that the coefficients of $\eta_\infty$ are polynomials of degree at most $2j$, and since these coefficients have vanishing $2j$-jet at the origin, this implies that $\eta_\infty$ vanishes identically. This kills the first two terms on the RHS of \eqref{sup-realized-x}.

At this point we return to \eqref{optimus}, and subtracting its fiber average we see that
\begin{equation}\label{subcaseC-Box}
d_\ell^{-2j-\a}\sum_{i=1}^j\sum_{r=1}^{N_{i,k}} \left( (\Delta_{\Theta_\ell^*\Psi_\ell^*\omega_{F}|_{ \{\cdot\}\times Y}})^{-1}  \tilde G_{\ell,i,r,k}\D^p_{v_1\cdots v_p}\de_t^{q+1}  \tilde A_{\ell,i,r,k}^*-\ve_\ell^{-2}\ti G_{\ell,i,r,k} \D^p_{v_1\cdots v_p}\de_t^q\tilde A_{\ell,i,r,k}^*\right)
\end{equation}
is asymptotically independent of the base and time directions. Let us then define a function $u_\ell$ on $\ti{Q}_{Rd_\ell^{-1}}$ by
\begin{equation}\label{deko1}
u_\ell(\tilde z,\tilde y,\tilde t):=d_\ell^{-2j-\a}\sum_{i=1}^j\sum_{r=1}^{N_{i,k}}  (\Delta_{\Theta_\ell^*\Psi_\ell^*\omega_{F}|_{ \{\tilde z\}\times Y}})^{-1}  \tilde G_{\ell,i,r,k} \D^p_{v_1\cdots v_p}\de_t^q\tilde A_{\ell,i,r,k}^*,
\end{equation}
so that \eqref{subcaseC-Box} is equivalent to the statement that
\begin{equation}
(\partial_{\tilde t}-\Delta_{\e_\ell^2\Theta_\ell^*\Psi_\ell^* \omega_{F}|_{ \{\cdot\}\times Y}}) u_\ell
\end{equation}
is asymptotically independent of the base and time directions.

Our next goal is to show that $u_\ell$ itself is asymptotically constant in the base and time directions.  Fix any two points $z,z'\in \mathbb{C}^m$ and times $t,t'\in (-\infty,0]$ and consider
\begin{equation}\label{deko2}
v_\ell(\tilde y,\tilde t)=u_\ell( z,\tilde y,t+\tilde t)-u_\ell( z',\tilde y,t'+\tilde t),
\end{equation}
so that in $L^\infty_{\rm loc}(Y)$ we have
\begin{equation}\label{ruzzo}
\begin{split}
&\quad \left((\partial_{\tilde t}-\Delta_{\e_\ell^2\Theta_\ell^*\Psi_\ell^* \omega_{F}|_{ \{z\}\times Y}}) v_\ell\right)(\ti{y},0)\\
&=\left((\partial_{\tilde t}-\Delta_{\e_\ell^2\Theta_\ell^*\Psi_\ell^* \omega_{F}|_{ \{z\}\times Y}})u_\ell\right)( z,\tilde y,t)-\left((\partial_{\tilde t}-\Delta_{\e_\ell^2\Theta_\ell^*\Psi_\ell^* \omega_{F}|_{ \{z\}\times Y}})  u_\ell\right)( z',\tilde y,t')\\
&=\left( \Delta_{\e_\ell^2\Theta_\ell^*\Psi_\ell^*\omega_F|_{ \{z\}\times Y}}-\Delta_{\e_\ell^2\Theta_\ell^*\Psi_\ell^*\omega_F|_{ \{z'\}\times Y}}\right) u_\ell( z',\tilde y,t')+o(1),
\end{split}
\end{equation}
and we can schematically estimate the difference of Laplacians by
\begin{equation}\label{rezzo}
\ve_\ell^{-2}\left(g_F^{\mathbf{ff}}(d_\ell\lambda_\ell^{-1}z,\ti{y})-g_F^{\mathbf{ff}}(d_\ell\lambda_\ell^{-1}z',\ti{y})\right)(\ddbar u_\ell)_{\mathbf{ff}}(z',\ti{y},t'),
\end{equation}
and since
\begin{equation}
(\ddbar u_\ell)_{\mathbf{ff}}(z',\ti{y},t')=d_\ell^{-2j-\a}\sum_{i=1}^j\sum_{r=1}^{N_{i,k}}  \left(\ddbar(\Delta_{\Theta_\ell^*\Psi_\ell^*\omega_{F}|_{ \{z'\}\times Y}})^{-1}  \tilde G_{\ell,i,r,k}\right)_{\mathbf{ff}} \D^p_{v_1\cdots v_p}\de_t^q\tilde A_{\ell,i,r,k}^*,
\end{equation}
we can use \eqref{A-star-0} with $\iota=2j$ to estimate $|(\ddbar u_\ell)_{\mathbf{ff}}|\leq C\ve_\ell^2$, and so \eqref{rezzo} can be estimated by
\begin{equation}
C\left|g_F^{\mathbf{ff}}(d_\ell\lambda_\ell^{-1}z,\ti{y})-g_F^{\mathbf{ff}}(d_\ell\lambda_\ell^{-1}z',\ti{y})\right|\leq Cd_\ell\lambda_\ell^{-1}|z-z'|=o(1),
\end{equation}
so that \eqref{ruzzo} shows that
\begin{equation}\label{mordell}
(\partial_{\tilde t}-\Delta_{\e_\ell^2\Theta_\ell^*\Psi_\ell^*\omega_F|_{ \{z\}\times Y}}) v_\ell=o(1),
\end{equation}
locally uniformly.
So the functions $v_\ell$ are approximate solutions of a fiberwise heat equation, with time parameter $\ti{t}\in(-\infty,0]$. We then employ another energy argument on the given fiber $\{z\}\times Y$. First, observe that  \eqref{A-star-0} with $\iota=2j$ implies
\begin{equation}\label{rattero}
|\e_\ell^{-2}v_\ell|\leq C.
\end{equation}
For notational convenience, we denote by $\omega_Y:=\Theta_\ell^*\Psi_\ell^*\omega_F|_{\{z\}\times Y}$ and consider then the energy
\begin{equation}
E_\ell(\ti{t})=\e_\ell^{-4}\int_Y v_\ell^2\, \omega_Y^n,
\end{equation}
which satisfies $E_\ell(\ti{t})\leq C$ for all $\ti{t}\leq 0$,
and denote by
\begin{equation}
\sigma_\ell:=\left\|\left(\de_{\ti{t}}-\Delta_{\e_\ell^2\omega_Y}\right) v_\ell\right\|_{\infty, \tilde Q_{2R}},
\end{equation}
where $R>0$ is fixed so that $(z,t),(z',t')\in Q_{2R}$. Thanks to \eqref{mordell} we have $\sigma_\ell\to 0$. We can then compute, using the Poincar\'e inequality on $(Y,\omega_Y)$ (recall that $v_\ell$ has fiberwise average zero),
\begin{equation}
\begin{split}
\frac{d}{d\ti{t}} E_\ell&=-2\e_\ell^{-6}\int_Y |\nabla^Y v_\ell|^2_{g_Y}\, \omega_Y^n+2\e_\ell^{-4}\int_Y v_\ell  \left(\de_t-\Delta_{\e_\ell^2\omega_Y}\right) v_\ell   \, \omega_Y^n\\\
&\leq -2C^{-1} \e_\ell^{-2} E_\ell(\ti{t})+2\sigma_\ell\ve_\ell^{-4}\int_Y v_\ell \omega_Y^n\\
&\leq -2C^{-1} \e_\ell^{-2} E_\ell(\ti{t})+C\sigma_\ell \e_\ell^{-2} E_\ell(\ti{t})^{\frac{1}{2}},
\end{split}
\end{equation}
and using the Young inequality $\sigma_\ell E_\ell(\ti{t})^{\frac{1}{2}}\leq C^{-1}E_\ell(\ti{t})+C\sigma_\ell^2,$ we can bound
\begin{equation}
\frac{d}{d\ti{t}} E_\ell\leq  -C^{-1} \e_\ell^{-2} E_\ell+C\sigma_\ell^2 \e_\ell^{-2}.
\end{equation}
We will compare $E_\ell$ with the real-variable function $F_\ell$ which solves the ODE
\begin{equation}
F'_\ell=-C^{-1} \e_\ell^{-2} F_\ell+C\sigma_\ell^2 \e_\ell^{-2},\quad F'_\ell(-R^2)=A,
\end{equation}
which is given explicitly by
\begin{equation}
F_\ell(\ti{t})=Ae^{-C^{-1}\ve_\ell^{-2}(R^2+\ti{t})}+C^2\sigma_\ell^2\left(1-e^{-C^{-1}\ve_\ell^{-2}(R^2+\ti{t})}\right),
\end{equation}
and if we choose $A$ large enough so that $E_\ell(\ti{t})\leq A$ for all $\ti{t}\leq 0$ (which is possible, as shown above), then we conclude that for all $\ti{t}\in [-R^2,0]$ we have
\begin{equation}
E_\ell(\ti{t})\leq F_\ell(\ti{t})=Ae^{-C^{-1}\ve_\ell^{-2}(R^2+\ti{t})}+C^2\sigma_\ell^2\left(1-e^{-C^{-1}\ve_\ell^{-2}(R^2+\ti{t})}\right),
\end{equation}
hence in particular $E_\ell(0)\to 0$ as $\ell\to+\infty$. This means that the functions $\ve_\ell^{-2}v_\ell(\cdot, 0)$, which are defined on $Y$ and are uniformly bounded by \eqref{rattero}, converge to $0$ in $L^2(Y,\omega_Y^n)$. Recalling the definitions \eqref{deko1}, \eqref{deko2}, this means that the function of $\ti{y}$ given by
\begin{equation}\label{rzzo}
\begin{split}
&d_\ell^{-2j-\a}\ve_\ell^{-2}\sum_{i=1}^j\sum_{r=1}^{N_{i,k}}  ((\Delta_{\Theta_\ell^*\Psi_\ell^*\omega_{F}|_{ \{z\}\times Y}})^{-1}  \tilde G_{\ell,i,r,k})(z,\ti{y}) (\D^p_{v_1\cdots v_p}\de_t^q\tilde A_{\ell,i,r,k}^*)(z,t)\\
&-d_\ell^{-2j-\a}\ve_\ell^{-2}\sum_{i=1}^j\sum_{r=1}^{N_{i,k}}  ((\Delta_{\Theta_\ell^*\Psi_\ell^*\omega_{F}|_{ \{z'\}\times Y}})^{-1}  \tilde G_{\ell,i,r,k})(z',\ti{y}) (\D^p_{v_1\cdots v_p}\de_t^q\tilde A_{\ell,i,r,k}^*)(z',t'),
\end{split}
\end{equation}
converges to zero in $L^2(Y,\omega_Y^n)$. However, we have
\begin{equation}\label{qorq2}
\begin{split}
((\Delta_{\Theta_\ell^*\Psi_\ell^*\omega_{F}|_{ \{z'\}\times Y}})^{-1}  \tilde G_{\ell,i,r,k})(z',\ti{y})&=
((\Delta_{\omega_{F}|_{ \{d_\ell\lambda_\ell^{-1}z'\}\times Y}})^{-1}  G_{\ell,i,r,k})(d_\ell\lambda_\ell^{-1}z',\ti{y})\\
&=((\Delta_{\omega_{F}|_{ \{d_\ell\lambda_\ell^{-1}z\}\times Y}})^{-1}  G_{\ell,i,r,k})(d_\ell\lambda_\ell^{-1}z,\ti{y})+O(d_\ell\lambda_\ell^{-1})\\
&=((\Delta_{\Theta_\ell^*\Psi_\ell^*\omega_{F}|_{ \{z\}\times Y}})^{-1}  \tilde G_{\ell,i,r,k})(z,\ti{y})+O(d_\ell\lambda_\ell^{-1}),
\end{split}
\end{equation}
where in the second equality we used that, since $G_{\ell,i,r,k}$ is a smooth function on the total space with fiberwise average zero, the function $(\Delta_{\omega_{F}|_{ \{\cdot\}\times Y}})^{-1}G_{\ell,i,r,k}$ is also smooth on the total space (by standard Schauder theory fiber-by-fiber, with continuous dependence on the base variables).
Using \eqref{qorq2} together with the bound $d_\ell^{-2j-\a}\ve_\ell^{-2}|\D^p_{v_1\cdots v_p}\de_t^q\tilde A_{\ell,i,r,k}^*|\leq C$, which comes from \eqref{A-star-0} with $\iota=2j$, we see that
\begin{equation}
\begin{split}
&d_\ell^{-2j-\a}\ve_\ell^{-2}\sum_{i=1}^j\sum_{r=1}^{N_{i,k}}  ((\Delta_{\Theta_\ell^*\Psi_\ell^*\omega_{F}|_{ \{z'\}\times Y}})^{-1}  \tilde G_{\ell,i,r,k})(z',\ti{y}) (\D^p_{v_1\cdots v_p}\de_t^q\tilde A_{\ell,i,r,k}^*)(z',t')\\
&=d_\ell^{-2j-\a}\ve_\ell^{-2}\sum_{i=1}^j\sum_{r=1}^{N_{i,k}}  ((\Delta_{\Theta_\ell^*\Psi_\ell^*\omega_{F}|_{ \{z\}\times Y}})^{-1}  \tilde G_{\ell,i,r,k})(z,\ti{y}) (\D^p_{v_1\cdots v_p}\de_t^q\tilde A_{\ell,i,r,k}^*)(z',t')+o(1),
\end{split}
\end{equation}
so from \eqref{rzzo} we see that
\begin{equation}\label{rxzzo}
d_\ell^{-2j-\a}\ve_\ell^{-2}\sum_{i=1}^j\sum_{r=1}^{N_{i,k}}  ((\Delta_{\Theta_\ell^*\Psi_\ell^*\omega_{F}|_{ \{z\}\times Y}})^{-1}  \tilde G_{\ell,i,r,k})(z,\ti{y}) \left((\D^p_{v_1\cdots v_p}\de_t^q\tilde A_{\ell,i,r,k}^*)(z,t)-(\D^p_{v_1\cdots v_p}\de_t^q\tilde A_{\ell,i,r,k}^*)(z',t')\right)
\end{equation}
also converges to zero in $L^2(Y,\omega_Y^n)$.
Now, the functions $\{(\Delta_{\Theta_\ell^*\Psi_\ell^*\omega_{F}|_{ \{z\}\times Y}})^{-1}  \tilde G_{\ell,i,r,k}\}_{i,r}$ are fiberwise linearly independent (since so are $\{\tilde G_{\ell,i,r,k}\}_{i,r}$), and the function in \eqref{rxzzo} along the fiber $\{z\}\times Y$ is expressed as a linear combination of these functions with coefficients (which are constants on $Y$) given by
\begin{equation}
d_\ell^{-2j-\a}\ve_\ell^{-2}\left((\D^p_{v_1\cdots v_p}\de_t^q\tilde A_{\ell,i,r,k}^*)(z,t)-(\D^p_{v_1\cdots v_p}\de_t^q\tilde A_{\ell,i,r,k}^*)(z',t')\right).
\end{equation}
Since the $L^2$ norm of \eqref{rxzzo} is going to zero, these coefficients must be going to zero too, which means that the functions
\begin{equation}
d_\ell^{-2j-\a}\ve_\ell^{-2}\D^p_{v_1\cdots v_p}\de_t^q\tilde A_{\ell,i,r,k}^*,
\end{equation}
are approximately constant (in space and time) as $\ell\to+\infty$. This kills the last term on the RHS of \eqref{sup-realized-x}, and gives the final contradiction, thus completing the proof of {\bf Subcase C} and of Theorem \ref{thm:decomposition}.

\section{Proof of the main theorem}\label{s4}
In this final section we give the proof of our main Theorem \ref{mainthm}, namely we prove Conjectures \ref{c1} and \ref{c2}. The asymptotic expansion in Theorem \ref{thm:decomposition} will play a crucial role.

\subsection{Higher order estimates}
To prove the higher order estimates in Conjecture \ref{c1} from the expansion in Theorem \ref{thm:decomposition} we follow the arguments in \cite[Proof of Theorem A]{HT3}, but since our estimate \eqref{infty3} is weaker than the corresponding \cite[(4.12)]{HT3}, we will have to deal with some new difficulties. As explained in the Introduction, in this section we work locally on the base (away from the image of the singular fibers), and the K\"ahler-Ricci flow that we analyze thus lives on $B\times Y\times [0,+\infty)$ (with a non-product complex structure) for some Euclidean ball $B\subset\mathbb{C}^m$. For brevity, in this section all norms and seminorms will be tacitly taken on $B\times Y\times[t-1,t]$ (or $B\times[t-1,t]$ for objects that live on the base), without making this explicit in the notation. The ball $B$ and the interval $[t-1,t]$ will also be shrunk slightly every time we use interpolation.

Given an even integer $k\geq 2$, we want to show that $\omega^\bullet(t)$ is uniformly bounded in $C^k(g_X)$. Applying Theorem \ref{thm:decomposition} with $j:=\frac{k}{2}$, up to shrinking $B$ we can write
\begin{equation}
\omega^\bullet(t)=\omega^\natural(t)+\gamma_0(t)+\gamma_{1,k}(t)+\cdots+\gamma_{\frac{k}{2},k}(t)+\eta_{\frac{k}{2},k}(t),
\end{equation}
and in this decomposition $\omega^\natural(t)$ is clearly bounded in $C^k(g_X)$, $\gamma_0(t)$ has a similar bound by \eqref{infty2}, $\eta_{\frac{k}{2},k}(t)$ is bounded in the shrinking $C^k$ norm by \eqref{infty1} (hence in the regular $C^k(g_X)$ norm by \cite[Lemma 2.6]{HT3}), so we are left with dealing with the terms $\gamma_{i,k}(t),1\leq i\leq \frac{k}{2}$. By definition and using \eqref{plop} we have
\begin{equation}\label{coot}
\gamma_{i,k}(t)=\ddbar\sum_{p=1}^{N_{i,k}}\sum_{\iota=0}^{2k}\sum_{r=\lceil\frac{\iota}{2}\rceil}^{k}e^{-r t}\Phi_{\iota,r}(G_{i,p,k})\circledast\D^\iota A_{i,p,k}(t),
\end{equation}
and so, as in \cite[(5.10)]{HT3}, for $0\leq q\leq k$,
\begin{equation}\label{juz}
\DD^q\gamma_{i,k}=\sum_{p=1}^{N_{i,k}}\sum_{\iota=0}^{2k}\sum_{r=\lceil\frac{\iota}{2}\rceil}^{k}\sum_{s=0}^{q+1}\sum_{i_1+i_2=s+1}e^{-r t}(\D^{q+1-s}J)\circledast\D^{i_1}\Phi_{\iota,r}(G_{i,p,k})\circledast\DD^{i_2+\iota} A_{i,p,k},
\end{equation}
and using the fixed metric $g_X$ we can estimate $|\D^{i_1}\Phi_{\iota,r}(G_{i,p,k})|_{g_X}\leq C$ and $|(\D^{q+1-s}J)|_{g_X}\leq C$, while from \eqref{infty3} we see that
$|\DD^{i_2+\iota} A_{i,p,k}|=o(1)$ when $i_2+\iota\leq k+2$ and from \eqref{infty4} that
$|\DD^{i_2+\iota} A_{i,p,k}|=o(e^{(i_2+\iota-k-2)\frac{t}{2}})$ when $k+2<i_2+\iota\leq k+2+2k$, and so
\begin{equation}\label{qpietrogamba}\begin{split}
|\DD^q\gamma_{i,k}|_{g_X}&\leq o(1)+ o(1)\sum_{i_1+i_2\leq q+2}\sum_{\iota=k+3-i_2}^{2k}\sum_{r=\lceil\frac{\iota}{2}\rceil}^{k}e^{-r t}e^{(i_2+\iota-k-2)\frac{t}{2}}\\
&\leq o(1)+o(e^{(q-k)\frac{t}{2}})\sum_{\iota=0}^{2k}\sum_{r=\lceil\frac{\iota}{2}\rceil}^{k}e^{-\left(r-\frac{\iota}{2}\right) t}\leq o(1)+o(e^{(q-k)\frac{t}{2}})=o(1),
\end{split}\end{equation}
since $i_2\leq q+2$ and $q\leq k$. This completes the proof of Conjecture \ref{c1}.

\subsection{Ricci curvature bounds}
Next, we prove Conjecture \ref{c2}, namely the Ricci curvature bound for $\omega^\bullet(t)$ on compact subsets of $X\backslash f(S)$, which in our setting translates to
\begin{equation}\label{riccib}
\sup_{B\times Y}|\Ric(\omega^\bullet(t))|_{g^\bullet(t)}\leq C.
\end{equation}
The argument is similar to \cite[Proof of Theorem B]{HT3}, but there are some crucial differences coming from the time evolution in the Monge-Amp\`ere equation, and from the fact that the bounds in \eqref{infty3} are worse than those in \cite[(4.12)]{HT3}.

We will use the expansion \eqref{expa} with $j=1$ and $k\geq 4$ (arbitrary), and with $\alpha$ close to $1$, and our first task is to improve the estimates \eqref{infty3}, \eqref{infty4}. These give us
\begin{equation}\label{qquack}
|\DD^iA_{1,p,k}|\leq \begin{cases} Ce^{-(2+\alpha)\left(1-\frac{i}{4+\alpha}\right)\frac{t}{2}},\quad 0\leq i\leq 4,\\
o(e^{-(4-i)\frac{t}{2}}),\quad 5\leq i\leq 4+2k,
\end{cases}
\end{equation}
and we can interpolate between $|A_{1,p,k}|\leq e^{-(2+\alpha)\frac{t}{2}}$ and $[\DD^2A_{1,p,k}]_{C^\alpha}\leq Ce^{-t}$ from \eqref{holder3} and get
\begin{equation}\label{qsquot}
|\DD^i A_{1,p,k}|\leq Ce^{-t}\left(e^{-\frac{\alpha}{2}t}\right)^{1-\frac{i}{2+\alpha}}, \quad 0\leq i\leq 2,
\end{equation}
so in particular we have $|\dot{A}_{1,p,k}|=o(e^{-t}).$
The next claim is that
\begin{equation}\label{qbrutus}
|\gamma_{1,k}|_{g_X}=o(e^{-t}).
\end{equation}
Indeed, using the decomposition in \eqref{juz}
\begin{equation}\label{qcornelius}
|\gamma_{1,k}|_{g_X}\leq C\sum_{i_2=0}^2\sum_{p=1}^{N_{1,k}}\sum_{\iota=0}^{2k}\sum_{r=\lceil\frac{\iota}{2}\rceil}^{2}e^{-r t}|\D^{i_2+\iota} A_{1,p,k}|,
\end{equation}
and we bound the RHS of \eqref{qcornelius} by $o(e^{-t})$ by considering the possible values of $i_2+\iota\in\{0,\dots,2k+2\}$: if $i_2+\iota=0,1,2$ then \eqref{qsquot} in particular gives $|\D^{i_2+\iota} A_{1,p,k}|=o(e^{-t})$, which is acceptable. If $i_2+\iota=3,4$ then necessarily $\iota\geq 1$ and so $r\geq 1$, while \eqref{qquack} in particular gives $|\D^{i_2+\iota} A_{1,p,k}|=o(1)$ so the RHS of \eqref{qcornelius} is again $o(e^{-t})$. And if $i_2+\iota\geq 5$ then we use \eqref{qquack} exactly as in \eqref{qpietrogamba} to bound the RHS of \eqref{qcornelius} by
\begin{equation}\label{korn0}
o(e^{-t})+o(1)\sum_{i_2=0}^2\sum_{p=1}^{N_{1,k}}\sum_{\iota=5-i_2}^{2k}\sum_{r=\lceil\frac{\iota}{2}\rceil}^{k}e^{-r t}e^{(i_2+\iota-4)\frac{t}{2}}
\leq o(e^{-t})+o(e^{-t})\sum_{\iota=3}^{2k}\sum_{r=\lceil\frac{\iota}{2}\rceil}^{k}e^{-\left(r-\frac{\iota}{2}\right) t}=o(e^{-t}),
\end{equation}
since $i_2\leq 2$, which concludes the proof of \eqref{qbrutus}. Next, we want to show that
\begin{equation}\label{qsatanacchio}
(\gamma_{1,k})_{\mathbf{ff}}=\sum_{p=1}^{N_{1,k}}A_{1,p,k}\de_{\mathbf{f}}\db_{\mathbf{f}}(\Delta^{\omega_F|_{\{\cdot\}\times Y}})^{-1}G_{1,p,k}+ o(e^{-2t}),
\end{equation}
where the $o(e^{-2t})$ is in $L^\infty_{\rm loc}(g_X)$. Indeed from \eqref{coot} we can write
\begin{equation}
(\gamma_{1,k})_{\mathbf{ff}}=\sum_{p=1}^{N_{1,k}}\sum_{\iota=0}^{2k}\sum_{r=\lceil\frac{\iota}{2}\rceil}^{k}e^{-r t}\de_{\mathbf{f}}\db_{\mathbf{f}}\Phi_{\iota,r}(G_{1,p,k})\circledast\D^\iota A_{1,p,k},
\end{equation}
and we can estimate each term as follows. For $\iota\geq 4$ we have $|\D^\iota A_{1,p,k}|=o(e^{-(4-\iota)\frac{t}{2}})$ from \eqref{qquack}, and so
\begin{equation}
e^{-r t}|\de_{\mathbf{f}}\db_{\mathbf{f}}\Phi_{\iota,r}(G_{1,p,k})\circledast\D^\iota A_{1,p,k}|_{g_X}\leq o(1)e^{-2t}e^{-\left(r-\frac{\iota}{2}\right)t}=o(e^{-2t}),
\end{equation}
since $r\geq \frac{\iota}{2}$. For $\iota=3$, we have $r\geq 2$ and $|\D^\iota A_{1,p,k}|=o(1)$ from \eqref{qquack}, so the term is again $o(e^{-2t})$. For $\iota=1,2$, we have $r\geq 1$ and $|\D^\iota A_{1,p,k}|=o(e^{-t})$ from \eqref{qsquot}, so the term is again $o(e^{-2t})$. And for $\iota=0$, let us first look at the terms with $r\geq 1$. For these, we have $|A_{1,p,k}|=O(e^{-(2+\alpha)\frac{t}{2}})$, and so when multiplied by $e^{-r t},r\geq 1$, these terms are indeed $o(e^{-2t})$. So we are only left with the terms where $\iota=r=0$ which equal
\begin{equation}
\sum_{p=1}^{N_{1,k}}A_{1,p,k}\de_{\mathbf{f}}\db_{\mathbf{f}}(\Delta^{\omega_F|_{\{\cdot\}\times Y}})^{-1}G_{1,p,k},
\end{equation}
since $\Phi_{0,0}(G)=(\Delta^{\omega_F|_{\{\cdot\}\times Y}})^{-1}G$ by \eqref{qsangennaro2}, thus proving \eqref{qsatanacchio}. In particular, using the bound \eqref{qquack} in \eqref{qsatanacchio} gives
\begin{equation}\label{qbrutus_better}
|(\gamma_{1,k})_{\mathbf{ff}}|_{g_X}\leq Ce^{-\frac{2+\alpha}{2}t},
\end{equation}
while tracing \eqref{qsatanacchio} fiberwise gives
\begin{equation}\label{qsatanacchio2}
\mathrm{tr}_{\omega_F|_{\{\cdot\}\times Y}}(\gamma_{1,k})_{\mathbf{ff}}=\sum_{p=1}^{N_{1,k}}A_{1,p,k}G_{1,p,k} + o(e^{-2t}).
\end{equation}

Before we continue with the proof of \eqref{riccib}, recall that from \cite[p.110]{FZ} (see also \cite[Lemma 5.13]{To}) we know that
\begin{equation}
\sup_{B\times Y}|\vp(t)-\underline{\vp}(t)|\leq Ce^{-t},
\end{equation}
for all $t\geq 0$. The argument in \cite[Lemma 3.1 (iv)]{TWY} then allows one to deduce from this that
\begin{equation}
\sup_{B\times Y} |\dot{\vp}(t)-\underline{\dot{\vp}}(t)|\leq Ce^{-\frac{t}{2}}.
\end{equation}
As a consequence of our asymptotic expansion, we can now improve both of these:
\begin{prop}\label{trew}
On $B\times Y$ we have
\begin{equation}\label{suxx}
|\vp(t)-\underline{\vp}(t)|=o(e^{-t}), \quad |\dot{\vp}(t)-\underline{\dot{\vp}}(t)|=o(e^{-t}).
\end{equation}
\end{prop}
\begin{proof}
Recall that by definition we can write
\begin{equation}\label{triw}
\vp-\underline{\vp}=\psi_{1,k}+\sum_{p=1}^{N_{1,k}}\mathfrak{G}_{t,k}(A_{1,p,k},G_{1,p,k}),
\end{equation}
and we want to bound the $L^\infty$ norm of the two terms on the RHS of \eqref{triw} and their time derivatives. For $\psi_{1,k}$, from \eqref{infty1} we have
\begin{equation}
|\ddbar\psi_{1,k}|_{g(t)}\leq e^{-(1+\frac{\alpha}{2})t},
\end{equation}
so restricting this to a fiber $\{z\}\times Y$ gives
\begin{equation}
|\ddbar\psi_{1,k}|_{\{z\}\times Y}|_{g_Y}\leq e^{-(2+\frac{\alpha}{2})t},
\end{equation}
and since $\psi_{1,k}$ has fiberwise average zero, applying Moser iteration on $\{z\}\times Y$ (with constant that does not depend on $z\in B$) gives
\begin{equation}\label{trow}
|\psi_{1,k}|\leq Ce^{-(2+\frac{\alpha}{2})t}.
\end{equation}
Arguing similarly for $\dot{\psi}$, which by \eqref{infty1} satisfies (since $k\geq 4$)
\begin{equation}
|\ddbar\dot{\psi}_{1,k}|_{g(t)}\leq e^{-\frac{\alpha}{2}t},
\end{equation}
we get
\begin{equation}\label{traw}
|\dot{\psi}_{1,k}|\leq Ce^{-(1+\frac{\alpha}{2})t}.
\end{equation}
As for the term $\mathfrak{G}_{t,k}(A_{1,p,k},G_{1,p,k})$, using again \eqref{coot} we can write it as
\begin{equation}
\sum_{\iota=0}^{2k}\sum_{r=\lceil\frac{\iota}{2}\rceil}^{k}e^{-r t}\Phi_{\iota,r}(G_{1,p,k})\circledast\D^\iota A_{1,p,k},
\end{equation}
and using \eqref{qquack} and \eqref{qsquot} we can argue similarly to the proof of \eqref{qbrutus} and bound
\begin{equation}\label{korn}
|\mathfrak{G}_{t,k}(A_{1,p,k},G_{1,p,k})|\leq C \sum_{\iota=0}^{2k}\sum_{r=\lceil\frac{\iota}{2}\rceil}^{k}e^{-r t}|\D^\iota A_{1,p,k}|,
\end{equation}
by $o(e^{-t})$ by considering the possible values of $\iota\in\{0,\dots,2k\}$: if $\iota=0$ then \eqref{qquack} in particular gives $|A_{1,p,k}|=o(e^{-t})$, which is acceptable. If $1\leq\iota\leq 4$ then necessarily $r\geq 1$, while \eqref{qquack} in particular gives $|\D^{\iota} A_{1,p,k}|=o(1)$ so the RHS of \eqref{korn} is again $o(e^{-t})$. And if $\iota\geq 5$ then we use \eqref{qquack} to bound the RHS of \eqref{korn} by
\begin{equation}
o(e^{-t})+o(1)\sum_{p=1}^{N_{1,k}}\sum_{\iota=5}^{2k}\sum_{r=\lceil\frac{\iota}{2}\rceil}^{k}e^{-r t}e^{(\iota-4)\frac{t}{2}}
\leq o(e^{-t})+o(e^{-2t})\sum_{\iota=5}^{2k}\sum_{r=\lceil\frac{\iota}{2}\rceil}^{k}e^{-\left(r-\frac{\iota}{2}\right) t}=o(e^{-t})+o(e^{-2t}),
\end{equation}
and so
\begin{equation}\label{korn3}
\mathfrak{G}_{t,k}(A_{1,p,k},G_{1,p,k})=o(e^{-t}).
\end{equation}
Similarly,
\begin{equation}\label{korn2}
\mathfrak{G}_{t,k}(\dot{A}_{1,p,k},G_{1,p,k})=\sum_{\iota=0}^{2k}\sum_{r=\lceil\frac{\iota}{2}\rceil}^{k}e^{-r t}\Phi_{\iota,r}(G_{1,p,k})\circledast\D^\iota \dot{A}_{1,p,k},
\end{equation}
and we argue as above to show that this is $o(e^{-t})$: if $\iota=0$ then \eqref{qsquot} in particular gives $|\dot{A}_{1,p,k}|=o(e^{-t})$, which is acceptable. If $1\leq \iota\leq 2$ then
$r\geq 1$, while \eqref{qquack} in particular gives $|\D^{\iota} \dot{A}_{1,p,k}|=o(1)$ so the RHS of \eqref{korn2} is again $o(e^{-t})$. And if $\iota\geq 3$ then we use \eqref{qquack} to bound the RHS of \eqref{korn2} by
\begin{equation}
o(e^{-t})+o(1)\sum_{p=1}^{N_{1,k}}\sum_{\iota=3}^{2k}\sum_{r=\lceil\frac{\iota}{2}\rceil}^{k}e^{-r t}e^{(\iota-2)\frac{t}{2}}
\leq o(e^{-t})\sum_{\iota=3}^{2k}\sum_{r=\lceil\frac{\iota}{2}\rceil}^{k}e^{-\left(r-\frac{\iota}{2}\right) t}=o(e^{-t}),
\end{equation}
and so
\begin{equation}\label{korn4}
\mathfrak{G}_{t,k}(\dot{A}_{1,p,k},G_{1,p,k})=o(e^{-t}).
\end{equation}
Combining \eqref{triw} with \eqref{trow}, \eqref{traw}, \eqref{korn3} and \eqref{korn4} we see that \eqref{suxx} holds.
\end{proof}
\begin{rmk}
If we use the stronger bounds \eqref{kurz}, \eqref{kurz2} for derivatives of $A_{1,p,k}$, which will be established below, we can then repeat the proof of Proposition \ref{trew} and we get the better bounds
\begin{equation}\label{suxx2}
|\vp(t)-\underline{\vp}(t)|\leq Ce^{-2t}, \quad |\dot{\vp}(t)-\underline{\dot{\vp}}(t)|\leq Ce^{-(1+\frac{\gamma}{2})t},
\end{equation}
where $\gamma=\frac{\alpha}{4+\alpha}>0.$
\end{rmk}

Now that \eqref{suxx} is established, we can use it to prove the next claim, which is the analog of \cite[(5.26)]{HT3}:
\begin{prop}
We have
\begin{equation}\label{qfilimi}
|A_{1,p,k}|\leq Ce^{-2t}.
\end{equation}
\end{prop}
\begin{proof}
First, observe that thanks to \eqref{infty1} we have
\begin{equation}
|\eta_{1,k}|_{g(t)}\leq e^{-(2+\alpha)\frac{t}{2}},
\end{equation}
which implies that
\begin{equation}\label{zatter}
|\eta_{1,k}|_{g_X}=o(e^{-t}),\quad |(\eta_{1,k})_{\mathbf{ff}}|_{g_X}=o(e^{-2t}),
\end{equation}
and combining this with \eqref{qbrutus} and \eqref{qbrutus_better} we get
\begin{equation}\label{zutter}
|e^t\gamma_{1,k}+e^t\eta_{1,k}|_{g_X}=o(1),\quad |(e^t\gamma_{1,k}+e^t\eta_{1,k})_{\mathbf{ff}}|_{g_X}=O(e^{-\alpha\frac{t}{2}}).
\end{equation}
The Monge-Amp\`ere equation \eqref{ma2} that describes the flow can be written as
\begin{equation}\label{q3sat}
e^{\vp(t)+\dot{\vp}(t)}e^{-nt}\binom{m+n}{n}\omega_{\rm can}^m\wedge\omega_F^n=(\omega_{\rm can}+e^{-t}\omega_F+\ddbar \underline{\vp}(t)+\gamma_{1,k}(t)+\eta_{1,k}(t))^{m+n},
\end{equation}
and multiplying this by $\frac{e^{nt}}{\binom{m+n}{n}}$, expanding it out, and using that $\ddbar \underline{\vp}(t)$ is small in $C^2$ by \eqref{infty2}, as well as \eqref{zatter}, \eqref{zutter}, we see that
\begin{equation}\label{q3sata}\begin{split}
e^{\vp+\dot{\vp}}\omega_{\rm can}^m\wedge\omega_F^n&=(\omega_{\rm can}+\ddbar\underline{\vp})^m\wedge(\omega_F+e^t\gamma_{1,k}+e^t\eta_{1,k})_{\mathbf{ff}}^n\\
&+\frac{m}{n+1}e^{-t}
(\omega_{\rm can}+\ddbar\underline{\vp})^{m-1}\wedge(\omega_F+e^t\gamma_{1,k}+e^t\eta_{1,k})^{n+1}+O(e^{-2t})\\
&=(\omega_{\rm can}+\ddbar\underline{\vp})^m\wedge\omega_F^n+
n(\omega_{\rm can}+\ddbar\underline{\vp})^m(\omega_F)_{\mathbf{ff}}^{n-1}(e^t\gamma_{1,k})_{\mathbf{ff}}\\
&+\binom{n}{2}(\omega_{\rm can}+\ddbar\underline{\vp})^m(\omega_F)_{\mathbf{ff}}^{n-2}(e^t\gamma_{1,k})_{\mathbf{ff}}^2+\frac{m}{n+1}e^{-t}
\omega_{\rm can}^{m-1}\wedge\omega_F^{n+1}+o(e^{-t}),
\end{split}\end{equation}
where the error terms $o(e^{-t})$ are in $L^\infty_{\rm loc}(g_X)$. We used here that $\alpha>\frac{2}{3}$, so that $|(e^t\gamma_{1,k})_{\mathbf{ff}}^p|=o(e^{-t})$ for $p\geq 3$ by \eqref{qbrutus_better}.
As in \cite[(5.29)]{HT3} we define a function $\mathcal{S}$ by
\begin{equation}\label{semmes}
\frac{m}{n+1}\omega_{\rm can}^{m-1}\wedge\omega_F^{n+1}=\mathcal{S}\omega_{\rm can}^{m}\wedge\omega_F^{n},
\end{equation}
so that dividing \eqref{q3sata} by the volume form $\omega_{\rm can}^m\wedge\omega_F^n$ and multiplying by $e^{-t}$ gives
\begin{equation}\label{dd1}
e^{-t}e^{\vp+\dot{\vp}}=\left(e^{-t}+\tr_{\omega_F|_{\{\cdot\}\times Y}}(\gamma_{1,k})_{\mathbf{ff}}+e^{-t}\frac{\binom{n}{2}(\omega_F)_{\mathbf{ff}}^{n-2}(e^t\gamma_{1,k})_{\mathbf{ff}}^2}{(\omega_F)_{\mathbf{ff}}^{n}}+e^{-2t}\mathcal{S}\right)(1+o(1)_{\rm base})+o(e^{-2t}),
\end{equation}
and subtracting from \eqref{dd1} its fiber average (the two terms with $\gamma_{1,k}$ have fiberwise average zero since $\gamma_{1,k}$ is $\de\db$-exact) and using \eqref{qsatanacchio2} gives
\begin{equation}\label{dd2}
e^{-t}(e^{\vp+\dot{\vp}}-\underline{e^{\vp+\dot{\vp}}})=\left(\sum_{p=1}^{N_{1,k}}A_{1,p,k}G_{1,p,k}
+e^{-t}\frac{\binom{n}{2}(\omega_F)_{\mathbf{ff}}^{n-2}(e^t\gamma_{1,k})_{\mathbf{ff}}^2}{(\omega_F)_{\mathbf{ff}}^{n}}+e^{-2t}(\mathcal{S}-\underline{\mathcal{S}})\right)(1+o(1)_{\rm base})+o(e^{-2t}).
\end{equation}
To bound the LHS of \eqref{dd2}, use the Taylor expansion of the exponential, together with Lemma \ref{lma:earlier work} (ii) and \eqref{suxx} to bound
\begin{equation}
e^{-t}|e^{\vp+\dot{\vp}}-\underline{e^{\vp+\dot{\vp}}}|\leq Ce^{-t}|\vp+\dot{\vp}-\underline{\vp+\dot{\vp}}|=o(e^{-2t}),
\end{equation}
while on the RHS of \eqref{dd2} we can bound
\begin{equation}
e^{-t}\left|\frac{\binom{n}{2}(\omega_F)_{\mathbf{ff}}^{n-2}(e^t\gamma_{1,k})_{\mathbf{ff}}^2}{(\omega_F)_{\mathbf{ff}}^{n}}\right|\leq Ce^{-t}|(e^t\gamma_{1,k})_{\mathbf{ff}}^2|\leq Ce^{-(1+\alpha)t},
\end{equation}
by \eqref{qbrutus_better}. Thus, going back to \eqref{dd2}, we learn that
\begin{equation}
\sum_{p=1}^{N_{1,k}}A_{1,p,k}G_{1,p,k}=O(e^{-(1+\alpha)t}),
\end{equation}
and taking the fiberwise $L^2$ inner product with each $G_{1,p,k}$ we see that
\begin{equation}
A_{1,p,k}=O(e^{-(1+\alpha)t}),
\end{equation}
which improves over the bound $A_{1,p,k}=O(e^{-(1+\frac{\alpha}{2})t})$ from \eqref{qquack}. We can then go back to \eqref{qsatanacchio} and see that
\begin{equation}
|(\gamma_{1,k})_{\mathbf{ff}}|_{g_X}\leq Ce^{-(1+\alpha)t},
\end{equation}
which allows us to improve \eqref{zutter} to
\begin{equation}\label{zutter2}
|(e^t\gamma_{1,k}+e^t\eta_{1,k})_{\mathbf{ff}}|_{g_X}=O(e^{-\alpha t}),
\end{equation}
and so as long as we chose $\alpha>\frac{1}{2}$ we see that
\begin{equation}
|(e^t\gamma_{1,k})_{\mathbf{ff}}^2|_{g_X}=o(e^{-t}),
\end{equation}
and so in \eqref{dd2} we have
\begin{equation}
e^{-t}\left|\frac{\binom{n}{2}(\omega_F)_{\mathbf{ff}}^{n-2}(e^t\gamma_{1,k})_{\mathbf{ff}}^2}{(\omega_F)_{\mathbf{ff}}^{n}}\right|\leq Ce^{-t}|(e^t\gamma_{1,k})_{\mathbf{ff}}^2|=o(e^{-2t}),
\end{equation}
and returning to \eqref{dd2} again we see that
\begin{equation}
\sum_{p=1}^{N_{1,k}}A_{1,p,k}G_{1,p,k}=O(e^{-2t}),
\end{equation}
and again taking the fiberwise inner product with each $G_{1,p,k}$ concludes the proof of \eqref{qfilimi}.
\end{proof}

Next, observe that from \eqref{infty1} we have in particular
\begin{equation}\label{1}
|\DD^i\eta_{1,k}|_{g(t)}\leq Ce^{\frac{i-2-\alpha}{2}t},\quad 0\leq i\leq 2.
\end{equation}
We want to show a similar estimate for $\gamma_{1,k}$ which is only slightly worse:
\begin{prop}\label{worse}
We have
\begin{equation}\label{2}
|\DD^i\gamma_{1,k}|_{g(t)}\leq Ce^{\frac{i-2}{2}t},\quad 0\leq i\leq 2,
\end{equation}
and also
\begin{equation}\label{2a}
|\dot{\gamma}_{1,k}|_{g(t)}\leq C e^{-\gamma\frac{t}{2}},
\end{equation}
where $\gamma=\frac{\alpha}{4+\alpha}>0$.
\end{prop}
\begin{proof}
The first step is to use \eqref{qfilimi} to improve the estimates in \eqref{qquack}, by interpolating it with $[\DD^4 A_{1,p,k}]_{C^\alpha}\leq C$ from \eqref{holder3} and get
\begin{equation}\label{kurz}
|\DD^iA_{1,p,k}|\leq Ce^{-(4-i+i\frac{\alpha}{4+\alpha})\frac{t}{2}}\leq\begin{cases} Ce^{-2t},\quad i=0,\\
Ce^{-\frac{4-i+\gamma}{2}t}, \quad 1\leq i\leq 4,
\end{cases}
\end{equation}
where $\gamma=\frac{\alpha}{4+\alpha}>0$. As for higher order derivatives of $A_{1,p,k}$, given $1\leq\iota\leq 2k$, we interpolate between $|\DD^4A_{1,p,k}|\leq Ce^{-\frac{\gamma}{2}t}$ from \eqref{kurz} and $[\DD^{4+\iota}A_{1,p,k}]_{C^\alpha}\leq Ce^{\frac{\iota}{2}t}$ from \eqref{holder3} and letting $i=4+\iota$ we get
\begin{equation}\label{kurz2}
|\DD^i  A_{1,p,k}|\leq Ce^{-\frac{4-i+\gamma}{2}t}, \quad 5\leq i\leq 4+2k.
\end{equation}
Finally, using \eqref{kurz} and \eqref{kurz2} we can bound $|\DD^i\gamma_{1,k}|_{g(t)}, 0\leq i\leq 2,$ by going back to \eqref{juz} and bounding
\begin{equation}
|\DD^i\gamma_{1,k}|_{g(t)}\leq \sum_{p=1}^{N_{1,k}}\sum_{\iota=0}^{2k}\sum_{r=\lceil\frac{\iota}{2}\rceil}^{k}\sum_{s=0}^{i+1}\sum_{i_1+i_2=s+1}e^{-r t}|(\D^{i+1-s}J)|_{g(t)} |\D^{i_1}\Phi_{\iota,r}(G_{1,p,k})|_{g(t)} |\DD^{i_2+\iota} A_{1,p,k}|_{g(t)},
\end{equation}
and using also $|\D^{i_1}\Phi_{\iota,r}(G_{1,p,k})|_{g(t)}\leq Ce^{i_1\frac{t}{2}},$ $|\D^\iota J|_{g(t)}\leq Ce^{\iota\frac{t}{2}},$ we get
\begin{equation}
|\DD^i\gamma_{1,k}|_{g(t)}\leq C\sum_{\iota=0}^{2k}\sum_{r=\lceil\frac{\iota}{2}\rceil}^{k}\sum_{s=0}^{i+1}\sum_{i_1+i_2=s+1}e^{-r t}e^{(i_1+i_2+i+1-s+\iota-4)\frac{t}{2}}
\leq Ce^{(i-2)\frac{t}{2}}\sum_{\iota=0}^{2k}\sum_{r=\lceil\frac{\iota}{2}\rceil}^{k}e^{-\left(r-\frac{\iota}{2}\right)t}\leq Ce^{(i-2)\frac{t}{2}},
\end{equation}
for $0\leq i\leq 2$, which proves \eqref{2}. As for \eqref{2a}, using \eqref{juz} again we can bound
\begin{equation}\begin{split}
|\dot{\gamma}_{1,k}|_{g(t)}&\leq \sum_{p=1}^{N_{1,k}}\sum_{\iota=0}^{2k}\sum_{r=\lceil\frac{\iota}{2}\rceil}^{k}\sum_{s=0}^{1}\sum_{i_1+i_2=s+1}e^{-r t}|(\D^{1-s}J)|_{g(t)} |\D^{i_1}\Phi_{\iota,r}(G_{1,p,k})|_{g(t)} |\DD^{i_2+\iota} \dot{A}_{1,p,k}|_{g(t)}\\
&\leq C\sum_{\iota=0}^{2k}\sum_{r=\lceil\frac{\iota}{2}\rceil}^{k}\sum_{s=0}^{1}\sum_{i_1+i_2=s+1}e^{-r t}e^{(i_1+i_2+1-s+\iota-2-\gamma)\frac{t}{2}}
\leq Ce^{-\gamma\frac{t}{2}}\sum_{\iota=0}^{2k}\sum_{r=\lceil\frac{\iota}{2}\rceil}^{k}e^{-\left(r-\frac{\iota}{2}\right)t}\leq Ce^{-\gamma\frac{t}{2}}.
\end{split}
\end{equation}
\end{proof}

After these preparations, we can finally give the proof of the Ricci curvature bound \eqref{riccib}. For this, we take $\ddbar\log$ of the Monge-Amp\`ere equation \eqref{ma2}, and get
\begin{equation}\label{cocco}\begin{split}
\Ric(\omega^\bullet(t))&=-\ddbar\log\det g^\bullet(t)\\
&=-\ddbar\log(\omega_{\rm can}^m\wedge\omega_F^n)-\ddbar(\vp(t)+\dot{\vp}(t))\\
&=-\omega_{\rm can}-\ddbar(\vp(t)+\dot{\vp}(t))\\
&=-\omega^\bullet(t)+e^{-t}\omega_F-\ddbar\dot{\vp}(t),
\end{split}
\end{equation}
where we used the known relation $-\ddbar\log(\omega_{\rm can}^m\wedge\omega_F^n)=-\omega_{\rm can}$ (see e.g. \cite[Proposition 5.9]{To}). Since $\omega^\bullet(t)$ is uniformly equivalent to $\omega^\natural(t)$ (by Lemma \ref{lma:earlier work} (i)), it suffices to bound $|\Ric(\omega^\bullet(t))|_{g^\natural(t)}$ on compact subsets of $X\backslash S$. We can then work on $B\times Y$ as before, where we may assume that we have the expansion \eqref{expa} with $j=1$ and $k\geq 4$,
\begin{equation}\label{k3}\omega^\bullet(t)=\omega^\natural(t)+\gamma_{0}(t)+\gamma_{1,k}(t)+\eta_{1,k}(t).\end{equation}
Using \eqref{infty2}, \eqref{1} and \eqref{2}, we thus see that
\begin{equation}\label{k4}
|\omega^\bullet(t)-\omega^\natural(t)|_{g^\natural(t)}=o(1).
\end{equation}
Taking a time derivative of \eqref{k3}, we obtain
\begin{equation}
\ddbar\dot{\vp}=\ddbar\underline{\dot{\vp}}+\dot{\gamma}_{1,k}+\dot{\eta}_{1,k},
\end{equation}
and using \eqref{infty2}, \eqref{1}, \eqref{2a} we see that
\begin{equation}\label{k1}
|\ddbar\dot{\vp}|_{g^\natural(t)}=o(1).
\end{equation}
Thus, going back to \eqref{cocco}, we can write
\begin{equation}\label{cocco2}\begin{split}
\Ric(\omega^\bullet(t))&=-\omega^\bullet(t)+e^{-t}\omega_F-\ddbar\dot{\vp}(t)\\
&=-\omega_{\rm can}+\left(-\omega^\bullet(t)+\omega^\natural(t)-\ddbar\dot{\vp}(t)\right),
\end{split}
\end{equation}
where thanks to \eqref{k4} and \eqref{k1}, we have
\begin{equation}\label{k5}
\left|-\omega^\bullet(t)+\omega^\natural(t)-\ddbar\dot{\vp}(t)\right|_{g^\natural(t)}=o(1),
\end{equation}
and since we clearly have $|\omega_{\rm can}|_{g^\natural(t)}\leq C,$ the Ricci bound $|\Ric(\omega^\bullet(t))|_{g^\natural(t)}\leq C$ follows.

Observe that \eqref{cocco2} and \eqref{k5} give us very detailed information about the Ricci curvature of $\omega^\bullet(t)$ (on compact subsets of $X\backslash S$), showing that it is asymptotic to $-\omega_{\rm can}$ in a strong sense.

\begin{rmk}
Since  $\omega_{\rm can}$ is pulled back from the base, and since $\omega^\bullet(t)\to \omega_{\rm can}$ locally uniformly (from Lemma \ref{lma:earlier work} (i), (iii)), it follows that
\begin{equation}
\tr_{\omega^\bullet(t)}\omega_{\rm can}\to \tr_{\omega_{\rm can}}\omega_{\rm can}=m,
\end{equation}
locally uniformly on $X\backslash S$. Thus, taking the trace of \eqref{cocco2} with respect to $\omega^\bullet(t)$, we see that
\begin{equation}
R(\omega^\bullet(t))+m\to 0,
\end{equation}
in $C^0_{\rm loc}(X\backslash S)$, which recovers the main theorem of \cite{Ji}.
\end{rmk}

\begin{rmk}
Continuing the above arguments along the lines of \cite[(5.37)--(5.86)]{HT3}, one can also identify the first nontrivial term in the expansion \eqref{expa} of $\omega^\bullet(t)$, whose shape is identical to the one in the elliptic setting, see \cite[Theorem B]{HT3}. For the sake of brevity, we leave the straightforward proof to the interested reader.
\end{rmk}

\end{document}